\documentclass{amsart}

\usepackage{graphicx}
\usepackage{morefloats,bm,caption,amsbsy,enumerate,amsmath,amsthm,
amssymb,mathtools,amsfonts,multirow,verbatim,tikz,tikz-cd,thmtools,thm-restate,chngcntr}
\usepackage{enumitem}
\usepackage[alphabetic]{amsrefs}
\usepackage[hidelinks,colorlinks=true,linkcolor=blue, citecolor=black,linktocpage=true]{hyperref}
\usetikzlibrary{matrix,arrows,decorations.pathmorphing}
\usepackage[top=1.25in, bottom=1in, left=1.25in, right=1.25in,marginpar=1in]{geometry}
\usepackage{url}
\counterwithin{figure}{section}
\numberwithin{equation}{section}

\counterwithin{figure}{section}

\newtheorem{innercustomthm}{Theorem}
\newenvironment{customthm}[1]
  {\renewcommand\theinnercustomthm{#1}\innercustomthm}
  {\endinnercustomthm}

\newtheorem{innercustomcor}{Corollary}
\newenvironment{customcor}[1]
  {\renewcommand\theinnercustomcor{#1}\innercustomcor}
  {\endinnercustomcor}

\newtheorem{thm}{Theorem}[section]
\newtheorem{prop}[thm]{Proposition}

\newtheorem{cor}[thm]{Corollary}
\newtheorem{lem}[thm]{Lemma}

\theoremstyle{definition}
\newtheorem{define}[thm]{Definition}

\theoremstyle{remark}
\newtheorem{rem}[thm]{Remark}
\newtheorem{example}[thm]{Example}

\newcommand{\ve}[1]{\boldsymbol{\mathbf{#1}}}
\newcommand{\R}{\mathbb{R}}

\newcommand{\Z}{\mathbb{Z}}

\newcommand{\Q}{\mathbb{Q}}
\newcommand{\C}{\mathbb{C}}

\renewcommand{\d}{\partial}
\renewcommand{\subset}{\subseteq}

\renewcommand{\tilde}{\widetilde}
\renewcommand{\bar}{\overline}
\renewcommand{\Hat}{\widehat}
\newcommand{\iso}{\cong}

\DeclareMathOperator{\Char}{{Char}}

\DeclareMathOperator{\Crit}{{Crit}}

\DeclareMathOperator{\gr}{{gr}}

\DeclareMathOperator{\id}{{id}}

\DeclareMathOperator{\Int}{{int}}

\DeclareMathOperator{\MCG}{{MCG}}

\DeclareMathOperator{\rank}{{rank}}

\DeclareMathOperator{\Spin}{{Spin}}

\DeclareMathOperator{\Span}{{Span}}
\DeclareMathOperator{\Sym}{{Sym}}

\DeclareMathOperator{\Tors}{{Tors}}

\DeclareMathOperator{\rep}{{rep}}

\newcommand{\lk}{\ell k}

\newcommand{\bA}{\mathbb{A}}

\newcommand{\bF}{\mathbb{F}}
\newcommand{\bG}{\mathbb{G}}

\newcommand{\bJ}{\mathbb{J}}
\newcommand{\bK}{\mathbb{K}}
\newcommand{\bL}{\mathbb{L}}

\newcommand{\bP}{\mathbb{P}}

\newcommand{\bS}{\mathbb{S}}
\newcommand{\bT}{\mathbb{T}}
\newcommand{\bU}{\mathbb{U}}

\newcommand{\CP}{\mathbb{CP}}

\newcommand{\cA}{\mathcal{A}}
\newcommand{\cB}{\mathcal{B}}
\newcommand{\cC}{\mathcal{C}}
\newcommand{\cD}{\mathcal{D}}

\newcommand{\cF}{\mathcal{F}}
\newcommand{\cG}{\mathcal{G}}
\newcommand{\cH}{\mathcal{H}}

\newcommand{\cL}{\mathcal{L}}
\newcommand{\cM}{\mathcal{M}}

\newcommand{\cO}{\mathcal{O}}
\newcommand{\cP}{\mathcal{P}}

\newcommand{\cR}{\mathcal{R}}
\newcommand{\cS}{\mathcal{S}}
\newcommand{\cT}{\mathcal{T}}

\newcommand{\cW}{\mathcal{W}}

\newcommand{\frd}{\mathfrak{d}}

\newcommand{\frs}{\mathfrak{s}}

\newcommand{\cCFL}{\mathcal{C\!F\!L}}
\newcommand{\cHFL}{\mathcal{H\!F\! L}}
\newcommand{\CF}{\mathit{CF}}
\newcommand{\HF}{\mathit{HF}}
\newcommand{\CFK}{\mathit{CFK}}
\newcommand{\HFK}{\mathit{HFK}}
\newcommand{\tHFK}{\mathit{tHFK}}
\newcommand{\tCFK}{\mathit{tCFK}}
\newcommand{\CFL}{\mathit{CFL}}
\newcommand{\HFL}{\mathit{HFL}}

\newcommand{\PD}{\mathit{PD}}
\newcommand{\tF}{\mathit{tF}}

\newcommand{\xs}{\ve{x}}
\newcommand{\ys}{\ve{y}}
\newcommand{\zs}{\ve{z}}
\newcommand{\ws}{\ve{w}}

\newcommand{\as}{\ve{\alpha}}
\newcommand{\bs}{\ve{\beta}}
\newcommand{\gs}{\ve{\gamma}}

\newcommand{\taus}{\ve{\tau}}
\newcommand{\sigmas}{\ve{\sigma}}

\renewcommand{\a}{\alpha}
\renewcommand{\b}{\beta}
\newcommand{\g}{\gamma}

\newcommand{\bmP}{{\bm{P}}}

\title{Link cobordisms and absolute gradings on link Floer homology}
\author{Ian Zemke}
\address{Department of Mathematics\\Princeton University\\  Princeton, NJ 08544, USA}
\email{izemke@math.princeton.edu}
\thanks{This research was supported by NSF grant DMS-1703685}

\begin{document}
\begin{abstract}We show that the link cobordism maps defined by the author are graded and satisfy a grading change formula.  Using the grading change formula, we prove a new bound for $\Upsilon_K(t)$ for knot cobordisms in negative definite 4-manifolds.  As another application, we show that the link cobordism maps associated to a connected, closed surface in $S^4$ are determined by the genus of the surface. We also prove a new adjunction relation and adjunction inequality for the link cobordism maps. Along the way, we see how many known results in Heegaard Floer homology can be proven using basic properties of the link cobordism maps, together with the grading change formula.
\end{abstract}
\maketitle

\tableofcontents

\section{Introduction}

Ozsv\'{a}th and Szab\'{o} \cite{OS4ballgenus} define a homomorphism $\tau$ from the smooth concordance group to $\Z$, which satisfies 
\begin{equation}
|\tau(K)|\le g_4(K),\label{eq:tau4-ballgenusbound}
\end{equation}
where $g_4(K)$ denotes the smooth 4-ball genus. More generally if $(W,\Sigma)\colon (S^3,K_1)\to (S^3,K_2)$ is an oriented knot cobordism with $b_1(W)=b_2^+(W)=0$, they proved that 
\begin{equation}
\tau(K_2)\le \tau(K_1)-\frac{\big|[\Sigma]\big|+[\Sigma]\cdot [\Sigma]}{2}+g(\Sigma)\label{eq:taubounds}
\end{equation}
 where $|[\Sigma]|$ denotes the integer obtained by factoring $[\Sigma]\in H_2(W,\d W;\Z)$ into $H_2(W;\Z)$ and setting
\[
|[\Sigma]|:=\max_{\substack{C\in \Char(Q_W)\\ C^2=-b_2(W)}} \langle C, [\Sigma]\rangle.
\]
In the above expression, $\Char(Q_W)$ denotes the set of characteristic vectors of $H^2(W;\Z)/\Tors$.

Another set of concordance invariants from Heegaard Floer homology are Rasmussen's local $h$-invariants \cite{RasmussenKnots}. If $k$ is a nonnegative integer, Rasmussen defines a nonnegative integer invariant $V_k(K)$, and proves   that
\begin{equation}
V_k(K)\le \left\lceil \frac{g_4(K)-k}{2} \right\rceil,\label{eq:localhinvariantbounds}
\end{equation}
whenever $k\le g_4(K)$ \cite{RasmussenGodaTeragaito}*{Theorem~2.3}.

The original proofs of Equation~\eqref{eq:taubounds} by Ozsv\'{a}th and Szab\'{o} and Equation~\eqref{eq:localhinvariantbounds} by Rasmussen used the behavior of Heegaard Floer homology with respect to surgeries on a 3-manifold. 

There are several notable examples of geometric results which have been proven using maps induced by link cobordisms. Sarkar gave a combinatorial proof of Equation~\eqref{eq:tau4-ballgenusbound} using maps associated to link cobordisms defined using grid diagrams \cite{SarkarTau}, though the proof does not extend to prove the full version of the bound in Equation~\eqref{eq:taubounds}. Rasmussen \cite{RasSliceGenus} gave a proof of the Milnor conjecture using the $s$-invariant and cobordism maps defined on Khovanov homology.

The motivation of the present paper is to extend the tools of link Floer homology to prove geometric results using the link cobordism maps from \cite{ZemCFLTQFT}.  As an example, we will show how Equations~\eqref{eq:taubounds} and \eqref{eq:localhinvariantbounds} can be proven using link cobordism techniques; See Theorems~\ref{thm:OSboundtauproof} and \ref{thm:Rasmussensbound}. In the rest of the paper, we apply our techniques to derive new geometric applications of link Floer homology.

\subsection{A bound on $\Upsilon_K(t)$}

Using the techniques we develop in this paper, we prove a new bound on Ozsv\'{a}th, Stipsicz and Szab\'{o}'s concordance invariant $\Upsilon_K(t)$ \cite{OSSUpsilon}, which is analogous to the bound on $\tau(K)$ in Equation~\eqref{eq:taubounds}.

 For a fixed knot $K$, the invariant $\Upsilon_K(t)$ is a piecewise linear function from $[0,2]$ to $\R$. The maps $\Upsilon_K(t)$ determine a homomorphism from the concordance group to the group of piecewise linear functions from $[0,2]$ to $\R$.

Suppose $W$ is a compact, oriented 4-manifold with boundary equal to two rational homology spheres.  If $[\Sigma]\in H_2(W,\d W;\Z)$ is a class whose image in $H_1(\d W;\Z)$ vanishes, then $[\Sigma]$ determines a unique element of $H_2(W;\Z)/\Tors$, for which we also write $[\Sigma]$. We define the quantity
\[
M_{[\Sigma]}(t):=\max_{C\in \Char(Q_W)} \frac{C^2+b_2(W)-2 t\langle C, [\Sigma]\rangle+2t([\Sigma]\cdot [\Sigma])}{4}.
\]

Although the expression for $M_{[\Sigma]}(t)$ looks unmotivated, we note that if $\rank (H^2(W;\Z))=1$, $Q_W=(-1)$ and $\PD[\Sigma]=n \cdot E$ for a generator $E\in H^2(W;\Z)$ and integer $n\ge 0$, then
\[
M_{[\Sigma]}(t)=\Upsilon_{T_{n,n+1}}(t),
\]
where $T_{n,n+1}$ denotes the $(n,n+1)$-torus knot; See Lemma~\ref{lem:UpsilonTnn+1}.

We now state our result about the invariant $\Upsilon_K(t)$:

\begin{thm}\label{thm:upsiloninvariant}If $(W,\Sigma)\colon (S^3,K_1)\to (S^3,K_2)$ is an oriented knot cobordism with $b_1(W)=b_2^+(W)=0$, then 
 \[
 \Upsilon_{K_2}(t)\ge \Upsilon_{K_1}(t)+ M_{[\Sigma]}(t)+ g(\Sigma)\cdot (|t-1|-1).
 \]
\end{thm}

The bound in Theorem~\ref{thm:upsiloninvariant} is sharp in the following sense: for any positive torus knot $T_{a,b}$, there is a knot cobordism $(W,\Sigma)$ from the unknot to $T_{a,b}$, with $g(\Sigma)=b_1(W)=b_2^+(W)=0$, such that
\[
\Upsilon_{T_{a,b}}(t)=M_{[\Sigma]}(t).
\]
See Proposition~\ref{prop:generaltorusknots}.

For small $t$, the function  $\Upsilon_K(t)$ satisfies $\Upsilon_K(t)=-\tau(K)\cdot t$. Correspondingly, for small $t$, our bound reads
\[
\Upsilon_{K_2}(t)\ge \Upsilon_{K_1}(t)+t\cdot \left(\frac{\big|[\Sigma]\big|+[\Sigma]\cdot [\Sigma]}{2}- g(\Sigma)\right),
\]
 reflecting Ozsv\'{a}th and Szab\'{o}'s bound on $\tau$ in Equation \eqref{eq:taubounds}.

\subsection{Alexander and Maslov grading change formulas}
\label{sec:introduction:gradingformulas}
Theorem~\ref{thm:upsiloninvariant} follows from a much more general formula for the grading change associated to the link cobordism maps from \cite{ZemCFLTQFT}, as well as some properties of the link cobordism maps. In this section, we describe a general result about the link cobordism map grading changes.

Before we state our grading formula, we recall the setup of link Floer homology.

\begin{define}An \emph{oriented multi-based link} $\bL=(L,\ve{w},\ve{z})$ in $Y^3$ is an oriented link $L$ with two disjoint collections of basepoints $\ve{w}=\{w_1,\dots, w_n\}$ and $\ve{z}=\{z_1,\dots, z_n\}$, such that as one traverses $L$, the basepoints alternate between $\ve{w}$ and $\ve{z}$. Furthermore, each component of $L$ has a positive (necessarily even) number of basepoints, and each component of $Y$ contains at least one component of $L$. 
\end{define}

Suppose $\bL$ is a multi-based link in $Y^3$, and  $\frs$ is a  $\Spin^c$ structure on $Y$. There are several algebraic variations of link Floer homology. The most general construction is a curved chain complex
\[
\cCFL^\infty(Y,\bL,\frs)
\]
 over the ring 
 \[
 \bF_2[U_{\ve{w}},V_{\ve{z}}]:=\bF_2[U_{w_1},\dots, U_{w_n}, V_{z_1}, \dots, V_{z_n}].
 \] 
The curved chain complex $\cCFL^\infty(Y,\bL,\frs)$ also has a filtration indexed by $\Z^{\ve{w}}\oplus \Z^{\ve{z}}$. The construction is outlined in Section~\ref{sec:background}. 

To simplify the notation in the introduction, we describe a slight algebraic simplification. There is a natural ring homomorphism
 \[
\sigma\colon \bF_2[U_{w_1},\dots, U_{w_n}, V_{z_1},\dots, V_{z_n}]\to \bF_2[U,V],
 \] defined by sending each $U_{w_i}$ to $U$, and each $V_{z_i}$ to $V$. The map $\sigma$ gives an action of $\bF_2[U_{\ws},V_{\zs}]$ on $\bF_2[U,V]$, allowing us to define the complex
 \[
 \cCFL^\infty(Y,\bL^\sigma,\frs):=\cCFL^\infty(Y,\bL,\frs)\otimes_{\bF_2[U_{\ws},V_{\zs}]} \bF_2[U,V].
 \]
which is a module over $\bF_2[U,V]$ and has a filtration indexed by $\Z\oplus \Z$. The simplification obtained by taking the tensor product using the homomorphism $\sigma$ is an example of a more general operation called \emph{coloring} a chain complex; See Section~\ref{sec:coloring}.

In \cite{ZemCFLTQFT}, the author described a TQFT for the complexes $\cCFL^\infty(Y,\bL^\sigma,\frs)$, modeled on the TQFT for the hat version $\Hat{\HFL}(Y,\bL)$ constructed by Juh\'{a}sz \cite{JCob}. The following is an adaptation of Juh\'{a}sz's notion of a decorated link cobordism:

\begin{define}\label{def:decoratedlinkcobordism} We say a pair $(W,\cF)$ is  a \emph{decorated link cobordism}  between two 3-manifolds with multi-based links, and write
\[
(W,\cF)\colon (Y_1,\bL_1)\to (Y_2,\bL_2),
\]
 if the following are satisfied:
\begin{enumerate}
\item $W$ is a 4-dimensional cobordism from $Y_1$ to $Y_2$.
\item $\cF=(\Sigma, \cA)$ consists of an oriented surface $\Sigma$ with a properly embedded 1-manifold $\cA$, whose complement in $\Sigma$ consists of two disjoint subsurfaces, $\Sigma_{\ve{w}}$ and $\Sigma_{\ve{z}}$. The intersection of the closures of $\Sigma_{\ws}$ and $\Sigma_{\zs}$ is $\cA$.
\item $\d \Sigma=-L_1\sqcup L_2$.
\item Each component of $L_i\setminus \cA$ contains exactly one basepoint.
\item  The $\ve{w}$ basepoints are all in $\Sigma_{\ve{w}}$ and the $\ve{z}$ basepoints are all in $\Sigma_{\ve{z}}$.
\end{enumerate}
\end{define}

In \cite{ZemCFLTQFT}, to a decorated link cobordism $(W,\cF)\colon (Y,\bL_1)\to (Y,\bL_2)$, the author associates a $\Z\oplus \Z$-filtered homomorphism of $\bF_2[U,V]$-modules
\[
F_{W,\cF,\frs}\colon \cCFL^\infty(Y_1,\bL_1^\sigma,\frs|_{Y_1})\to \cCFL^\infty(Y_2,\bL_2^\sigma,\frs|_{Y_2}),
\]
which is an invariant up to $\bF_2[U,V]$-equivariant, $\Z\oplus \Z$-filtered chain homotopy.

Adapting the construction of Alexander and Maslov gradings from \cite{OSLinks}, if $c_1(\frs)$ is torsion and $L$ is null-homologous, there are three gradings on $\cCFL^\infty(Y,\bL^\sigma,\frs)$. There are two Maslov gradings, $\gr_{\ws}$ and $\gr_{\zs}$, as well as an Alexander grading $A$. The three gradings are related by the formula
\begin{equation}
A=\frac{1}{2}(\gr_{\ws}-\gr_{\zs}). \label{eq:Alexandergradingrelation1}
\end{equation}

More generally, the Alexander grading can be defined even when $c_1(\frs)$ is non-torsion, though it will depend on a choice of Seifert surface $S$ for $L$; See  Theorem~\ref{thm:1}. In this more general situation, we will write $A_S$ for the Alexander grading, for the choice of Seifert surface $S$.

We will prove the following grading change formula:
\begin{thm}\label{thm:maingradingchangeformula}
Suppose that $(W,\cF)\colon (Y_1,\bL_1)\to (Y_2,\bL_2)$ is a decorated link cobordism, with type-$\ws$ and type-$\zs$ subsurfaces $\Sigma_{\ws}$ and $\Sigma_{\zs}$. 
\begin{enumerate}
\item \label{thm:maingradingchange:1} If $c_1(\frs|_{Y_1})$ and $c_1(\frs|_{Y_2})$ are torsion, then $F_{W,\cF,\frs}$ is a graded with respect to $\gr_{\ws}$, and satisfies
\[
\gr_{\ws}(F_{W,\cF,\frs}(\xs))-\gr_{\ws}(\xs)=\frac{c_1(\frs)^2-2\chi(W)-3\sigma(W)}{4}+\tilde{\chi}(\Sigma_{\ve{w}}),
\] 
where
\[
\tilde{\chi}(\Sigma_{\ve{w}}):=\chi(\Sigma_{\ve{w}})-\frac{1}{2}(|\ve{w}_{1}|+|\ve{w}_{2}|).
\]
\item\label{thm:maingradingchange:2} If $c_1(\frs|_{Y_1}-\PD[L_1])$ and $c_1(\frs|_{Y_2}-\PD[L_2])$ are torsion, then the map $F_{W,\cF,\frs}$ is graded with respect to $\gr_{\zs}$, and satisfies
\[
\gr_{\ve{z}}(F_{W,\cF,\frs}(\ve{x}))-\gr_{\ve{z}}(\ve{x})=\frac{c_1(\frs-\PD[\Sigma])^2-2\chi(W)-3\sigma(W)}{4}+\tilde{\chi}(\Sigma_{\ve{z}}).
\]
 \item If $L_1$ and $L_2$ are null-homologous, and $S_1$ and $S_2$ are Seifert surfaces of $L_1$ and $L_2$, respectively, then the map $F_{W,\cF,\frs}$ is graded with respect to the Alexander grading, and satisfies
 \[
A_{S_2}(F_{W,\cF,\frs}(\xs))-A_{S_1}(\xs)= \frac{\langle c_1(\frs),\hat{\Sigma}\rangle -[\hat{\Sigma}]\cdot [\hat{\Sigma}]}{2}+\frac{\chi(\Sigma_{\ws})-\chi(\Sigma_{\zs})}{2},
 \]
 where $\hat{\Sigma}=(-S_1)\cup \Sigma\cup S_2$.
\end{enumerate} 
\end{thm}

 We  also prove a more general version of Theorem~\ref{thm:maingradingchangeformula} for the Alexander multi-grading; see Theorem~\ref{thm:generalAlexandergradingformula}.
 
Theorem~\ref{thm:maingradingchangeformula} will follow from a description of the absolute Maslov and Alexander gradings in terms of surgery presentations of the link complement; See Theorem~\ref{thm:1}. Our description of the absolute gradings is modeled on the description of absolute gradings on $\HF^-(Y,\frs)$ by Ozsv\'{a}th and Szab\'{o} \cite{OSTriangles}.
 
  We note that in \cite{JMConcordance} and \cite{JMComputeCobordismMaps},  Juh\'{a}sz and Marengon   compute the Alexander and Maslov grading changes for Juh\'{a}sz's link cobordism maps on $\Hat{\HFL}$   when the underlying 4-manifold is $[0,1]\times S^3$.  Their formula for the grading changes agree with the ones from Theorem~\ref{thm:maingradingchangeformula} (though in their case, the only non-zero terms involve the Euler characteristic of the subsurfaces).

\subsection{Adjunction relations and link Floer homology}

We will show that the techniques developed in this paper naturally give simple proofs of some known adjunction relations and inequalities on the 3- and 4-manifold invariants constructed by Ozsv\'{a}th and Szab\'{o}. We then prove several new adjunction relations and inequalities for the link cobordism maps.

In addition to link Floer homology, Ozsv\'{a}th and Szab\'{o} \cite{OSDisks} described an invariant $\CF^-(Y,\frs)$ of a closed 3-manifold $Y$ equipped with a $\Spin^c$ structure $\frs$. For our purposes, the invariant $\CF^-(Y,\frs)$ is a free, finitely generated chain complex over $\bF_2[U]$. If $W\colon Y_1\to Y_2$ is a cobordism of connected 3-manifolds and $\frs\in \Spin^c(W)$, they describe a map \cite{OSTriangles}
\begin{equation}
F_{W,\frs}\colon \Lambda^*( H_1(W;\Z)/\Tors)\otimes_{\bF_2[U]} \CF^-(Y_1,\frs|_{Y_1})\to \CF^-(Y_2,\frs|_{Y_2}),
\label{eq:OScobordismmapintro}
\end{equation}
well defined up to $\bF_2[U]$-equivariant chain homotopy. We note that technically the maps $F_{W,\frs}$ depend on a choice of path from $Y_1$ to $Y_2$, though we suppress this dependency from the introduction; See Section~\ref{sec:computations}, and more generally \cite{ZemGraphTQFT}, for more on the dependency of Ozsv\'{a}th and Szab\'{o}'s cobordism maps on the choice of path.

In Section~\ref{sec:adjunctionrelations}, we show that our results about link cobordisms can be used to prove several familiar adjunction relations from Heegaard Floer homology. First, if $\Sigma$ is an oriented, connected surface with exactly one boundary component, there is a distinguished element
\[
\xi(\Sigma)\in \bF_2[U]\otimes_{\bF_2}\Lambda^*( H_1(\Sigma;\bF_2)),
\]
 constructed by picking a collection of $2g$ simple closed curves $A_1,\dots, A_g,$ $B_1,\dots, B_g$ on $\Sigma$ such that the geometric intersection number $|A_i\cap B_j|$ is $\delta_{ij}$. The element $\xi(\Sigma)$ is defined as
\[
\xi(\Sigma):=\prod_{i=1}^{g(\Sigma)} (U+A_i\wedge B_i).
\]
The element $\xi(\Sigma)$ is in fact independent of the choice of basis, and is fixed by the mapping class group of $\Sigma$; See Proposition~\ref{prop:fixedbymappingclassgroup}.

We prove the following result about the cobordism maps on $\CF^-$:

\begin{thm}\label{thm:7}Suppose that $\cF=(\Sigma,\cA)$ is an oriented, closed, decorated surface inside of the cobordism $W\colon Y_1\to Y_2$. Write $\Sigma_{\ws}$ and $\Sigma_{\zs}$ for the type-$\ws$ and type-$\zs$ subsurfaces of $\cF$. Suppose $\cA$ consists of a simple closed curve, the surfaces $\Sigma_{\ws}$ and $\Sigma_{\zs}$ are connected, and
\[
\langle c_1(\frs),[\Sigma]\rangle -[\Sigma]\cdot [\Sigma]+2g(\Sigma_{\zs})-2g(\Sigma_{\ws})=0.
\]
 Then
\[
F_{W,\frs}(\xi(\Sigma_{\ws})\otimes -)\simeq F_{W,\frs-\PD[\Sigma]}(\xi(\Sigma_{\zs})\otimes -),
\]
as maps on $\CF^-$.
\end{thm}

When $g(\Sigma)=g(\Sigma_{\zs})$ and $g(\Sigma_{\ws})=0$, Theorem~\ref{thm:7} is well known in the Seiberg-Witten setting \cite{OSSymplecticThom}*{Theorem 1.3} \cite{FintushelSternImmersedSpheres}*{Lemma 5.2}, and essentially well known in the Heegaard Floer setting \cite{OSTrianglesandSymplectic}*{Theorem 3.1}.

We also prove a generalization of Theorem~\ref{thm:7} which holds for the link cobordism maps; See Theorem~\ref{thm:adjunctionrelationlinkcob}.

Theorem~\ref{thm:7} is a very powerful relation satisfied by the Heegaard Floer cobordism maps. For example, it implies Ozsv\'{a}th and Szab\'{o}'s adjunction inequality for $\HF^+$ \cite{OSProperties}*{Theorem~7.1}, which states that  if $W\colon Y_1\to Y_2$ is a cobordism which contains a smoothly embedded surface $\Sigma$ with $g(\Sigma)>0$ and $[\Sigma]\cdot [\Sigma]\ge 0$, and the induced map $F_{W,\frs}\colon \HF^+(Y_1,\frs|_{Y_1})\to \HF^+(Y_2,\frs|_{Y_2})$ is non-trivial, then 
 \begin{equation}
 |\langle c_1(\frs),\Sigma\rangle |+[\Sigma]\cdot [\Sigma]\le 2g(\Sigma)-2;\label{eq:adjunctioninequalitymapsonHF+}
 \end{equation}  See Corollary~\ref{cor:adjunctioninequality}.

By using Theorem~\ref{thm:adjunctionrelationlinkcob}, our refinement of Theorem~\ref{thm:7} for the link cobordism maps, 
we will prove an adjunction inequality for the link cobordism maps which is analogous to Equation~\eqref{eq:adjunctioninequalitymapsonHF+}.

 If $\bL$ is a link in $Y$, we define $\CFL^-(Y,\bL,\frs)$ to be the $\bF_2[U]$-module obtained by setting $V=0$ in the chain complex $\cCFL^-(Y,\bL^\sigma,\frs)$ (which we recall is a chain complex over $\bF_2[U,V]$), i.e.,
\[
\CFL^-(Y,\bL,\frs):= \cCFL^-(Y,\bL^\sigma,\frs)\otimes_{\bF_2[U,V]} \bF_2[U,V]/(V).
\]
Our link Floer homology adjunction inequality states the following:

\begin{thm}\label{thm:linkadjunctioninequality}Suppose that $(W,\cF)\colon (Y_1,\bL_1)\to (Y_2,\bL_2)$ is a link cobordism with $b_1(W)=0$, such that $\bL_1$ and $\bL_2$ are null-homologous in $Y_1$ and $Y_2$, respectively, and suppose that the induced map 
\[
F_{W,\cF,\frs}\colon \CFL^-(Y_1,\bL_1,\frs|_{Y_1})\to \CFL^-(Y_2,\bL_2,\frs|_{Y_2})
\]
 is not $\bF_2[U]$-equivariantly chain homotopic to the zero map. If $\Sigma$ is a closed, oriented surface in the complement of $\cF$ with $g(\Sigma)>0$, then
 \[
|\langle c_1(\frs), [\Sigma]\rangle| +[\Sigma]\cdot [\Sigma] \le 2 g(\Sigma)-2.
 \] 
 \end{thm}

Several examples of Theorem~\ref{thm:linkadjunctioninequality} are described in Section~\ref{sec:adjunctioninequalityforlinkcobs}.

\subsection{Computations of the link cobordism maps}

Using techniques of this paper, we compute the maps $F_{W,\cF,\frs}$ in certain special cases. Our first computational result is the induced map on $\cHFL^\infty$ when $b_1(W)=b_2^+(W)=0$ and when the dividing set is relatively simple:

\begin{thm}\label{thm:10} Suppose that $(W,\cF)\colon (S^3,\bK_1)\to (S^3,\bK_2)$ is a knot cobordism between two doubly based knots with $b_1(W)=b_2^+(W)=0$. Suppose that the decorated surface $\cF=(\Sigma, \cA)$ is connected, and $\cA$ consists of a pair of arcs, both running from $\bK_1$ to $\bK_2$, and $\Sigma_{\ws}$ and $\Sigma_{\zs}$ are both connected. Then the induced map on homology
\[
F_{W,\cF,\frs}\colon \cHFL^\infty(S^3,\bK_1)\to \cHFL^\infty(S^3,\bK_2)
\]
is an isomorphism. In fact, under the canonical identification $\cHFL^\infty(S^3,\bK_i)\iso \bF_2[U,V,U^{-1},V^{-1}]$, it is the map
\[
1\mapsto U^{-d_1/2} V^{-d_2/2},
\]
 where
\[
d_1=\frac{c_1(\frs)^2-2\chi(W)-3\sigma(W)}{4}-2g(\Sigma_{\ve{w}})
\]
 and
\[
d_2=\frac{c_1(\frs-\PD[\Sigma])^2-2\chi(W)-3\sigma(W)}{4}-2g(\Sigma_{\ve{z}}).
\]
\end{thm}

\begin{figure}[ht!]
	\centering
\begingroup%
  \makeatletter%
  \providecommand\color[2][]{%
    \errmessage{(Inkscape) Color is used for the text in Inkscape, but the package 'color.sty' is not loaded}%
    \renewcommand\color[2][]{}%
  }%
  \providecommand\transparent[1]{%
    \errmessage{(Inkscape) Transparency is used (non-zero) for the text in Inkscape, but the package 'transparent.sty' is not loaded}%
    \renewcommand\transparent[1]{}%
  }%
  \providecommand\rotatebox[2]{#2}%
  \newcommand*\fsize{\dimexpr\f@size pt\relax}%
  \newcommand*\lineheight[1]{\fontsize{\fsize}{#1\fsize}\selectfont}%
  \ifx\svgwidth\undefined%
    \setlength{\unitlength}{144.12394485bp}%
    \ifx\svgscale\undefined%
      \relax%
    \else%
      \setlength{\unitlength}{\unitlength * \real{\svgscale}}%
    \fi%
  \else%
    \setlength{\unitlength}{\svgwidth}%
  \fi%
  \global\let\svgwidth\undefined%
  \global\let\svgscale\undefined%
  \makeatother%
  \begin{picture}(1,0.83893422)%
    \lineheight{1}%
    \setlength\tabcolsep{0pt}%
    \put(0,0){\includegraphics[width=\unitlength,page=1]{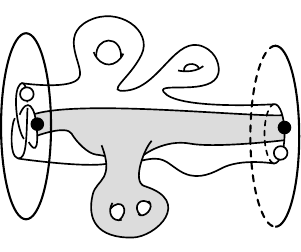}}%
    \put(0.29151437,0.39156386){\color[rgb]{0,0,0}\makebox(0,0)[lt]{\lineheight{0}\smash{\begin{tabular}[t]{l}$\Sigma_{\ve{w}}$\end{tabular}}}}%
    \put(0.25690877,0.50526779){\color[rgb]{0,0,0}\makebox(0,0)[lt]{\lineheight{0}\smash{\begin{tabular}[t]{l}$\Sigma_{\ve{z}}$\end{tabular}}}}%
    \put(0,0){\includegraphics[width=\unitlength,page=2]{fig15.pdf}}%
  \end{picture}%
\endgroup%

	\caption{\textbf{An example of the dividing sets considered in Theorems~\ref{thm:10} and~\ref{cor:2knotcomp}.} Here $\cA$ consists of two arcs, and $\Sigma\setminus \cA$ consists of two connected components, $\Sigma_{\ve{w}}$ and $\Sigma_{\ve{z}}$.}\label{fig::15}
\end{figure}

More generally, we will prove that if $(W,\cF)$ is a decorated link cobordism between two knots in $S^3$, then the induced maps $F_{W,\cF,\frs}$ on $\cHFL^\infty$ can always be computed in terms of the cobordism maps on $\HF^\infty$; see Theorem~\ref{thm:mostgeneralcompHFLinfty}.

As a consequence of Theorem~\ref{thm:10}, we will compute the maps associated to closed surfaces in $S^4$, with simple decoration:

\begin{thm}\label{cor:2knotcomp} Suppose that $\cF=(\Sigma,\cA)$ is a closed, oriented, decorated surface in $S^4$ such that $\cA$ consists of a single closed curve which divides $\Sigma$ into two connected subsurfaces, $\Sigma_{\ve{w}}$ and $\Sigma_{\ve{z}}$. The link cobordism map
\[
F_{S^4,\cF,\frs_0}\colon \cCFL^\infty(\varnothing)\to \cCFL^\infty(\varnothing) 
\]
is equal to the map
\[
1\mapsto U^{g(\Sigma_{\ve{w}})} V^{g(\Sigma_{\ve{z}})},
\]
under the canonical identification $\cCFL^\infty(\varnothing)\iso \bF_2[U,V,U^{-1},V^{-1}]$.
\end{thm}

\subsection{Translating between common conventions}
\label{subsec:usersguide}
We note that for technical reasons, some of our conventions differ slightly from more common conventions in the literature, so we now provide a brief guide to help translate between various conventions.

If $\bK=(K,w,z)$ is a doubly based knot in  an integer homology sphere $Y$, it is more common to use a different version of the full link Floer complex than the complex $\cCFL^\infty(S^3,\bK,\frs)$ considered in this paper. Instead, one often considers a $\Z\oplus \Z$-filtered chain complex $\CFK^\infty(Y,K)$ over $\bF_2[U,U^{-1}]$, defined by Ozsv\'{a}th and Szab\'{o} \cite{OSKnots}.  Using Ozsv\'{a}th and Szab\'{o}'s notation, $\CFK^\infty(Y,K)$ is generated over $\bF_2$ by elements of the form $[\xs,i,j]$ where $A(\xs)-j+i=0$. 

In the terminology of this paper, $[\xs,i,j]$ corresponds to the element $U^{-i} V^{-j} \cdot \xs$. In particular, using our description of the Alexander grading, one has
\[
\CFK^\infty(Y,K)=\cCFL^\infty(Y,\bK,\frs_0)_0,
\]
 where  $\cCFL^\infty(Y,\bK,\frs_0)_0$ denotes the homogeneous subset of zero Alexander grading.  Recalling from  Equation~\eqref{eq:Alexandergradingrelation1} that $A=\tfrac{1}{2}(\gr_{\ve{w}}-\gr_{\ve{z}})$, it follows that the two Maslov gradings $\gr_{\ve{w}}$ and $\gr_{\ve{z}}$ coincide on $\cCFL^\infty(Y,\bK,\frs_0)_0=\CFK^\infty(Y,K)$, reflecting the usual convention that  $\CFK^\infty$ has a single Maslov grading.

  The action of $U$ on $\CFK^\infty(Y,K)$ is normally described by the formula $U\cdot [\xs,i,j]=[\xs,i-1,j-1]$, and hence corresponds in our notation to the action of $UV$ on $\cCFL^\infty(Y,\bK,\frs_0)$. To disambiguate the notation, we will often write $\hat{U}$ for the product $UV$, which we think of as the standard action of $U$ on $\CFK^\infty$.

The $\Z\oplus \Z$ filtrations are similarly translated between $\CFK^\infty$ and $\cCFL^\infty$. One normally defines a subset $\cC_{(i,j)}(Y,K)\subset \CFK^\infty(Y,K)$ generated over $\bF_2$ by elements $[\xs,i',j']$ with $i'\le i$ and $j'\le i$.  The chain complex $\cCFL^\infty(Y,\bK,\frs)$ also has a $\Z\oplus \Z$-filtration, given by filtering over powers of the variables. If $i,j\in \Z$, then we define a subset $\cG_{(i,j)}(Y,\bK,\frs)\subset \cCFL^\infty(Y,\bK,\frs)$ generated over $\bF_2$ by monomials of the form $U^{i'} V^{j'} \cdot \xs$ where $i'\ge i$ and $j'\ge j$. The correspondence between the two $\Z\oplus \Z$ filtrations is that $\cC_{(i,j)}(Y,K)\subset \CFK^\infty(Y,K)$ is equal to $\cG_{(-i,-j)}(Y,\bK,\frs)\cap \CFK^\infty(Y,K)$, i.e., the subset of $\cG_{(-i,-j)}(Y,\bK,\frs)$ in Alexander grading zero.

Next, we note that if $(W,\cF)$ is a decorated link cobordism from $(S^3,\bK_1)$ to $(S^3,\bK_2)$, then $F_{W,\cF,\frs}$ may not send $\CFK^\infty(S^3,K_1)$  to $\CFK^\infty(S^3,K_2)$, since the Alexander grading change may be nonzero. Instead, $F_{W,\cF,\frs}$ will send $\CFK^\infty(S^3,K_1)$ to $\CFK^\infty(S^3,K_2)\{k\}$ for some shift $k\in \Z$ in the Alexander grading (i.e. the cobordism map sends monomials $ U^i V^j\cdot \ve{x}$ with $A(\ve{x})+(j-i)=0$ to sums of monomials of the form $U^{i'}V^{j'}\cdot \ve{y}$ with $A(\ve{y})+(j'-i')=k$).  See Theorem~\ref{thm:Rasmussensbound} for a concrete example of this phenomenon in action.

Finally, we discuss the equivalence of our grading conventions with those in the literature. In Proposition~\ref{prop:agreeswithOSconstruction}, we show that our Alexander multi-grading coincides with the description due to Ozsv\'{a}th and Szab\'{o} \cite{OSKnots} \cite{OSLinks}. For the Maslov grading, one must be somewhat careful, since there are two natural normalization conventions. The first is the \emph{invariance convention}, which is normalized by requiring that the relatively graded $\bF_2$-module 
\[
\Hat{\HF}(S^3,\ws)\iso \bigotimes_{i=1}^{|\ws|-1}\left((\bF_2)_{-\frac{1}{2}}\oplus (\bF_2)_{\frac{1}{2}}\right)
\] 
have top degree generator in grading $0$, regardless of how many basepoints are in $\ws$. There is another natural convention, the \emph{cobordism convention}, which is normalized by setting the Maslov grading of the top degree element of $\Hat{\HF}(S^3,\ws)$ to be $\tfrac{1}{2}(|\ws|-1)$. This is the convention that we take, since it is the most natural from the perspective of the grading change formulas. We note that when $\bK$ is a doubly based knot in $S^3$, the two conventions coincide, and our absolute Maslov gradings coincide with those from \cite{OSKnots} and \cite{OSLinks}.

\subsection{Organization}
In Section~\ref{sec:background} we provide some background  on link Floer homology, define some technical notions about indexings and colorings of links, and state our most general grading formulas. In Section~\ref{sec:heegaardtriplesandspinc} we describe some technical results concerning Heegaard triples and link cobordisms.  Section~\ref{sec:kirbycalculus} describes a relative version of Kirby calculus for link complements. In Section~\ref{sec:definitiongradings} we describe the relative gradings, and state the definition of the absolute gradings in terms of Kirby diagrams. In Section~\ref{sec:invarianceofgradings} we prove invariance of the absolute gradings. In Section~\ref{sec:gradingchange} we prove the grading change formulas for the link cobordism maps. In Section~\ref{sec:equivalence}, we prove that our construction is equivalent to Ozsv\'{a}th and Szab\'{o}'s construction for links in $S^3$. In Section~\ref{sec:computations} we prove some computational results about the link cobordism maps, which are useful for proving bounds. In Section~\ref{sec:bounds}, we show how the link cobordisms give conceptually simple proofs of well known bounds on $\tau$ and $V_k$.  Section~\ref{sec:Upsilon} covers our bound on $\Upsilon_K(t)$. In Section~\ref{sec:adjunctionrelations} we prove the adjunction relations and inequalities.

\subsection{Acknowledgments}

I would like to thank Kristen Hendricks, Jen Hom, Andr\'{a}s Juh\'{a}sz,  David Krcatovich, Robert Lipshitz,  Ciprian Manolescu, Marco Marengon and Peter Ozsv\'{a}th for helpful conversations and suggestions. I would also like to thank the two anonymous referees for their careful readings and very helpful suggestions.

\section{Preliminaries and statement of the full grading theorem}
\label{sec:background}

In this section we describe background material on the link Floer complexes, focusing on the terminology necessary to state the Alexander multi-grading formula.

\subsection{Link Floer homology}
\label{sec:backgroundlinkFloer}
Knot Floer homology was originally constructed by Ozsv\'{a}th and Szab\'{o} \cite{OSKnots}, and independently by Rasmussen \cite{RasmussenKnots}. Link Floer homology is a generalization to links, constructed by Ozsv\'{a}th and Szab\'{o} \cite{OSLinks}. In this section, we recall the basic construction of link Floer homology, focusing on the curved variation considered in \cite{ZemCFLTQFT}.

 To an oriented multi-based link $\bL$ in $Y$, one can construct a multi-pointed Heegaard diagram $\cH=(\Sigma, \as,\bs,\ve{w},\ve{z})$. We consider the two tori $\bT_{\alpha},\bT_{\beta}\subset \Sym^n(\Sigma)$, defined as the Cartesian products
 \[
 \bT_{\alpha}=\alpha_1\times \cdots\times \alpha_n\qquad \text{and} \qquad \bT_{\beta}=\beta_1\times \cdots\times \beta_n,
 \] 
 where $n=|\as|=|\bs|=g(\Sigma)+|\ws|-1=g(\Sigma)+|\zs|-1$. Ozsv\'{a}th and Szab\'{o} define a map 
 \[
 \frs_{\ve{w}}\colon \bT_{\alpha}\cap \bT_{\beta}\to \Spin^c(Y)
 \]
  in \cite{OSDisks}*{Section~2.6}. Recall that if $\ws=\{w_1,\dots, w_m\}$ and $\zs=\{z_1,\dots, z_n\}$,  we write $\bF_2[U_{\ws},V_{\zs}]$ for the polynomial ring $\bF_2[U_{w_1},\dots, U_{w_n}, V_{z_1},\dots, V_{z_n}]$.

  If $\frs\in \Spin^c(Y)$, then under appropriate admissibility assumptions on the diagram $\cH$, we define $\cCFL^-(\cH,\frs)$ to be the free $\bF_2[U_{\ve{w}},V_{\ve{z}}]$-module generated by intersection points $\ve{x}\in\bT_{\alpha}\cap \bT_{\beta}$ with $\frs_{\ve{w}}(\ve{x})=\frs$. We define $\cCFL^\infty(\cH,\frs)$ to be the free $\bF_2[U_{\ws}, U_{\ws}^{-1}, V_{\zs}, V_{\zs}^{-1}]$-module generated by intersection points $\xs\in \bT_{\alpha}\cap \bT_{\beta}$ with $\frs_{\ws}(\xs)=\frs$.
 
Both $\cCFL^-(\cH,\frs)$ and $\cCFL^\infty(\cH,\frs)$ have a filtration which is indexed by $\Z^{\ws}\oplus \Z^{\zs}$, i.e.,  the set of pairs $(I,J)$ where $I\colon \ws\to \Z$ and $J\colon \zs\to \Z$ are functions. We now describe the filtration in detail.  If $(I,J)\in \Z^{\ws}\oplus \Z^{\zs}$, we write $U_{\ws}^I V_{\zs}^J$ for the monomial $U_{w_1}^{I(w_1)}\cdots U_{w_n}^{I(w_n)} V_{z_1}^{J(z_1)}\cdots V_{z_n}^{J(z_n)}$. Given $(I,J)\in \Z^{\ws}\oplus \Z^{\zs}$, we define $\cG_{(I,J)}(\cH,\frs)\subset \cCFL^\infty(\cH,\frs)$ to be the $\bF_2[U_{\ws}, V_{\zs}]$-submodule
\[
\cG_{(I,J)}(\cH,\frs):=\Span_{\bF_2}\{ U_{\ws}^{I'} V_{\zs}^{J'} \cdot \xs: I'\ge I \text{ and } J'\ge J\}. 
\]
 Note that with this notation, $\cCFL^-(\cH,\frs)(\cH,\frs)=\cG_{(\ve{0},\ve{0})}(\cH,\frs)$ where $(\ve{0},\ve{0})$ denotes the zero map from $\ws\cup \zs$ to $\Z$.

  After choosing a generic 1-parameter family of almost complex structures on $\Sym^{n}(\Sigma)$, we can define an endomorphism $\d$ of $\cCFL^-(\cH,\frs)$ by counting holomorphic disks via the formula
 \begin{equation}
\d(\xs)=\sum_{\substack{\phi\in \pi_2(\xs,\ys)\\ \mu(\phi)=1}} \#\Hat{\cM}(\phi)\cdot  U_{w_1}^{n_{w_1}(\phi)}\cdots U_{w_n}^{n_{w_n}(\phi)} V_{z_1}^{n_{z_1}(\phi)}\cdots V_{z_n}^{n_{z_n}(\phi)}\cdot \ys. \label{eq:del}
 \end{equation}
  Because boundary degenerations appear in the ends of the moduli spaces of Maslov index 2 holomorphic curves, the map $\d$ does not square to zero for a general link. Instead, according to \cite{ZemQuasi}*{Lemma~2.1}, we have that
  \[
  \d^2=\omega_{\bL}\cdot \id,
\]
where $\omega_{\bL}\in \bF_2[U_{\ws}, V_{\zs}]$ denotes the scalar
\[
\omega_{\bL}=\sum_{K\in C(L)} (U_{w_{K,1}}V_{z_{K,1}}+V_{z_{K,1}}U_{w_{K,2}}+U_{w_{K,2}}V_{z_{K,2}}+\cdots+ V_{z_{K,n_K}}U_{w_{K,1}}).
  \]
  In the above expression,  $w_{K,1},$ $z_{K,1},$ $\dots,$  $w_{K,n_K},$  $z_{K,n_K}$ denote the basepoints on the link component $K\in C(L)$, in the order that they appear. In general, the curvature constant $\omega_{\bL}$ may be non-zero, though we note that if each link component has exactly two basepoints, then $\omega_{\bL}=0$.

An important property of Heegaard Floer homology is that it is natural with respect to the choice of Heegaard diagram. We will need the following result:

\begin{prop}\label{prop:naturality}If $\cH$ and $\cH'$ are two strongly $\frs$-admissible diagrams for the pair $(Y,\bL)$, then there is a map
\[
\Phi_{\cH\to \cH'}\colon \cCFL^\infty(\cH,\frs)\to \cCFL^\infty(\cH',\frs),
\]
which is filtered, and $\bF_2[U_{\ws},V_{\zs}]$-equivariant. The map $\Phi_{\cH\to \cH'}$ is a chain homotopy equivalence, and further, it is well defined up to filtered, $\bF_2[U_{\ws},V_{\zs}]$-equivariant chain homotopy.
\end{prop}

A summary of the proof can be found in \cite{ZemCFLTQFT}*{Proposition~3.5}, though most of the details  are due to other mathematicians, notably Ozsv\'{a}th, Szab\'{o}, Juh\'{a}sz and Thurston. See \cite{JTNaturality} for more on the question of naturality.

We define $\cCFL^\infty(Y,\bL,\frs)$ to be the collection of all of the complexes $\cCFL^\infty(\cH,\frs)$, together with the transition maps $\Phi_{\cH\to \cH'}$. We call the object $\cCFL^\infty(Y,\bL,\frs)$ the \emph{transitive chain homotopy type}.  There is an obvious notion of morphism between transitive chain homotopy type invariants.

\subsection{Link cobordisms, colorings, and functoriality}

\label{sec:coloring}

In this section, we state the main result from \cite{ZemCFLTQFT}, concerning the functoriality of link Floer homology.

There is an algebraic modification that one must do to the complexes to achieve functoriality, which we call \emph{coloring} the chain complexes.

\begin{define} If $\bL=(L,\ws,\zs)$ is a multi-based link in $Y$, a \emph{coloring} of $\bL$ is a function $\sigma\colon \ws\cup \zs\to \bmP$ where $\bmP$ is a finite set.
\end{define}

We call $\bmP$ the set of colors, and we note that the set $\bmP$ is part of the data of a coloring. We write $\bL^\sigma$ for a multi-based link $\bL$ equipped with a coloring $\sigma$.

If $\bmP=\{p_1,\dots, p_n\}$ is a set of colors, we define the ring $\cR_{\bmP}^-$ to be
\[
\cR_{\bmP}^-:=\bF_2[X_{p_1},\dots, X_{p_n}],
\]
the free polynomial ring generated by the formal variables $X_{p_1},\dots, X_{p_n}$. We define the ring $\cR^\infty_{\bmP}$ by adjoining the multiplicative inverses of the variables $X_{p_1},\dots, X_{p_n}$.

A coloring $\sigma\colon \ws\cup \zs\to \bmP$ gives the ring $\cR^\infty_{\bmP}$ the structure of an $\bF_2[U_{\ws}, V_{\zs}]$-module. This allows us to form the complex
\[
\cCFL^\infty(Y,\bL^\sigma,\frs):= \cCFL^\infty(Y,\bL,\frs)\otimes_{\bF_2[U_{\ws},V_{\zs}]} \cR_{\bmP}^\infty.
\]

The module $\cCFL^\infty(Y,\bL^\sigma,\frs)$ has a natural filtration by $\Z^{\bmP}$, defined by filtering by powers of the variables, similar to the filtration defined on the uncolored modules described in Section~\ref{sec:backgroundlinkFloer}.

To define functorial cobordism maps, we need the following notion of a coloring for a decorated link cobordism:

\begin{define} Suppose that $\cF=(\Sigma,\cA)$ is a surface with divides. A \emph{coloring} of $\cF$ is a map $\sigma\colon C(\Sigma\setminus \cA)\to \bmP$, where $\bmP$ is a finite set of colors.
\end{define}

We will write $\cF^\sigma$ for a decorated surface $\cF$ equipped with a coloring $\sigma$. If $(W,\cF)\colon (Y_1,\bL_1)\to (Y_2,\bL_2)$ is a decorated link cobordism (see Definition~\ref{def:decoratedlinkcobordism}), then a coloring $\sigma$ of $\cF$ naturally induces colorings of $\bL_1$ and $\bL_2$, for which we write $\sigma|_{\bL_i}$. 

The following is the main result of \cite{ZemCFLTQFT}:

\begin{thm}[\cite{ZemCFLTQFT}*{Theorem~A}] If $(W,\cF)$ is a decorated link cobordism equipped with a coloring $\sigma$, there are $\Spin^c$ functorial chain maps
\[
F_{W,\cF^\sigma,\frs}\colon \cCFL^\infty(Y_1,\bL_1^{\sigma_1}, \frs|_{Y_1})\to \cCFL^\infty(Y_2,\bL_2^{\sigma_2}, \frs|_{Y_2}),
\]
where $\sigma_i=\sigma|_{\bL_i}$. Furthermore, the maps $F_{W,\cF^\sigma,\frs}$ are $\cR_{\bmP}^\infty$-equivariant, $\Z^{\bmP}$-filtered, and are invariants up to $\cR_{\bmP}^\infty$-equivariant, $\Z^{\bmP}$-filtered chain homotopy.
\end{thm}

\subsection{Indexings of links and compatible colorings}
\label{sec:defs}

 In this section we describe some notions which are necessary for defining the Alexander multi-grading. As motivation, for an $\ell$-component link in $L\subset S^3$, the group $\Hat{\HFL}(L)$ has an $\ell$-component Alexander multi-grading. To obtain an Alexander multi-grading formula defined on $\cCFL^\infty(Y,\bL,\frs)$ for the link cobordism maps, we need to collapse certain indices of the Alexander multi-grading. We encode this notion in an \emph{indexing} of a link or link cobordism (Definition~\eqref{def:index}). Since the actions of the $U_{\ws}$ and $V_{\zs}$ have different gradings in different components of the Alexander multi-grading, we need to consider colorings which respect the indexing; These are \emph{indexed colorings} (Definition~\eqref{def:indexedcoloring}). Finally, since the $U_{\ws}$ and $V_{\zs}$ variables behave differently with respect to the gradings, we also need to consider colorings which don't identify any of the $U_{\ws}$ variables with any of the $V_{\zs}$ variables; These are the \emph{type-partitioned} colorings (Definition~\eqref{def:typepartitioned}).
 
 We begin with the formal definitions:

\begin{define}\label{def:index}If $A$ is a topological space, an \emph{indexing} of $A$ by a finite set $\bJ$ is a locally constant map
$J\colon A\to \bJ$.
\end{define}

Whenever $A$ is a finite set, we implicitly give $A$ the discrete topology, so an indexing of $A$ by $\bJ$ is the same as a map from $A$ to $\bJ$.

If $A$ is indexed over $\bJ$, we say that the map $J\colon A\to \bJ$ is the \emph{index assignment}, and the set $\bJ$ is the \emph{index set}. If $A$ is indexed over $\bJ$ and $j\in \bJ$, we write $A_j$ for
\[
A_j:=J^{-1}(j).
\]

 \begin{define}\label{def:indexedcoloring} If $\bL$ is a multi-based link, an \emph{indexed coloring} of $\bL$ is a pair $(\sigma,J)$ consisting of a coloring $\sigma\colon \ws\cup \zs\to \bmP$ together with an indexing $J\colon \bmP\to \bJ$, such that $(J\circ \sigma)(p)=(J\circ \sigma)(p')$ whenever $p,p'\in \ws\cup \zs$ are basepoints on the same component of $\bL$.
 \end{define}
 
 If $(\sigma, J)$ is an indexed coloring of $\bL$, then there is an induced indexing of $L$. Abusing notation slightly, we will also write $J\colon L\to \bJ$ for the induced indexing.

 \begin{define}\label{def:typepartitioned} A \emph{type-partitioned} coloring of a multi-based link $\bL=(L,\ws,\zs)$ is a coloring $\sigma\colon \ws\cup \zs\to \bmP$ such that the following hold:
 \begin{enumerate}
 \item $\bmP$ is partitioned as $\bmP=\bmP_{\ws}\sqcup \bmP_{\zs}$.
 \item $\sigma(\ws)\subset \bmP_{\ws}$ and $\sigma(\zs)\subset \bmP_{\zs}$.
\end{enumerate}
 \end{define}

An indexed, or type-partitioned coloring of a surface with divides is defined similarly.

\begin{example}For any multi-based link $\bL=(L,\ws,\zs)$, there are two type-partitioned, indexed colorings $(\sigma,J)$ to keep in mind:
\begin{enumerate}
\item $\bJ=\{*\}$ and $\bmP=\{U,V\}.$ The map $J\colon \bmP\to \bJ$ is constant. The map $\sigma$ sends any basepoint in $\ws$ to the element $U\in \bmP$, and similarly sends any basepoint in $\zs$ to $V$. Abusing notation slightly, we identify the ring $\cR_{\bmP}^-$ with $\bF_2[U,V]$.
\item $\bJ=C(L)$ and $\bmP=\ws\cup \zs$. The map $\sigma\colon \ws\cup \zs\to \bmP$ is the identity, and $J\colon \bmP\to C(L)$ is the natural map. The ring $\cR_{\bmP}^-$ can be identified with $\bF_2[U_{\ws},V_{\zs}]$.
\end{enumerate}
\end{example}

\subsection{$\bJ$-null-homologous links and generalized Seifert surfaces}

\begin{define}\label{def:seifertsurface} Suppose that $L$ is an oriented link in $Y$ which is indexed over $\bJ$.  We say that $L$ is $\bJ$\emph{-null-homologous} if 
	\[
	 [L_j]=0\in H_1(Y;\Z),
	 \] 
	 for each $j\in \bJ$. 
\end{define}

In this paper, we will use the following notion of a Seifert surface:

\begin{define}\label{def:orientationconventionSeifertsurface} If $L$ is an oriented link in $Y$, a \emph{generalized Seifert surface} of $L$ is an integral 2-chain $S$ in $Y$ with boundary $\d S=-L$.
\end{define}

 We will need the following notion of a Seifert surface which is indexed by $\bJ$:
 \begin{define}\label{def:generalizedJSeifertsruface} Suppose $L$ is an oriented link which is indexed by $\bJ$. A \emph{generalized $\bJ$-Seifert surface} $S=\{S_j\}_{j\in \bJ}$ is a collection of integral 2-chains $S_j$ with 
 \[
\d S_j=-L_j. 
 \]
\
\end{define}

\begin{rem}If $L$ is a $\bJ$-null-homologous link, it is possible to pick a generalized $\bJ$-Seifert surface $(S_j)_{j\in \bJ}$ so that each $S_j$ is an embedded surface. It is not always possible to pick a generalized $\bJ$-Seifert surface $\{S_j\}_{j\in \bJ}$ so that the union is an embedded Seifert surface for $L$, since, for example, $S_i$ and $S_j$ will be forced to intersect if $\lk(L_i,L_j)\neq 0$, for distinct $i,j\in \bJ$. For the purposes of this paper, it is only necessary to work with integral 2-chains, instead of embedded Seifert surfaces, since our description of the Alexander grading only uses $S$ to build an integral homology class; see Equation~\eqref{eq:defabsgrading}.
\end{rem}

\subsection{Statements of the grading theorems}
\label{sec:gradingtheorems}

Having established the necessary notation, we now state the main technical results of this paper, generalizing Theorem~\ref{thm:maingradingchangeformula} in the introduction.

Our first theorem concerns the existence of distinguished absolute gradings on link Floer homology. Using surgery presentations of link complements, we will prove the following:

\begin{thm}\label{thm:1} Suppose that $\bL=(L,\ws,\zs)$ is a multi-based link in $Y$, and $\frs\in \Spin^c(Y)$.
\begin{enumerate}[ ref= $\alph*$, label=(\alph*)]
\item\label{thm1.1a} Suppose $(\sigma, \bJ)$ is a type-partitioned, indexed coloring of $\bL$, $L$ is $\bJ$-null-homologous, and that $S$ is a generalized $\bJ$-Seifert surface of $L$. Then the chain complex $\cCFL^\infty(Y,\bL^{\sigma}, \frs)$ admits an absolute Alexander multi-grading $A_{S}$ which takes values in $\Q^{\bJ}$. The multi-grading is additive with respect to collapsing indices.
\item\label{thm1.1a'}The component $(A_S)_j$ of the Alexander multi-grading takes values in $\Z+\tfrac{1}{2}\lk(L\setminus L_j,L_j)$, where $\lk(L\setminus L_j, L_j)$ denotes $\# ((L\setminus L_j)\cap S_j)$.
\item\label{thm1.1b} If $(\sigma, \bJ)$ is a type-partitioned, indexed coloring of $\bL$, $L$ is $\bJ$-null-homologous, and $\frs$ is torsion, then the multi-grading $A_{S}$ is independent of $S$. More generally if $S$ and $S'$ are two choices of generalized $\bJ$-Seifert surfaces, then 
\[
A_{S'}(\ve{x})_j-A_{S}(\ve{x})_j=\frac{\langle c_1(\frs), [S_j'\cup -S_j]\rangle}{2}.
\]
\item\label{thm1.1c} If $\sigma$ is a type-partitioned coloring of $\bL$, and $c_1(\frs)$ is torsion, then there is a distinguished absolute Maslov grading $\gr_{\ve{w}}$ on $\cCFL^\infty(Y,\bL^\sigma,\frs)$. If $c_1(\frs-\PD[L])$ is torsion, then there is an absolute Maslov grading $\gr_{\ve{z}}$.
\item \label{thm1.1d}   If $\sigma$ is a type partitioned coloring,  $[L]=0\in H_1(Y;\Z)$ and $\frs$ is torsion, then $\gr_{\ws}$, $\gr_{\zs}$ and the collapsed Alexander grading $A$ are all defined, and $A=\tfrac{1}{2}(\gr_{\ve{w}}-\gr_{\ve{z}})$.
\end{enumerate}
\end{thm}

Theorem~\ref{thm:1} is perhaps not of particular interest on its own, since Ozsv\'{a}th and Szab\'{o} described absolute lifts of the Alexander multi-grading for links in integer homology spheres using a conjugation symmetry of the link Floer complexes, similar to the symmetry property of the Alexander polynomial. We will show that the Alexander gradings induced by surgery descriptions of the link complement in Theorem~\ref{thm:1} agree with the gradings defined by Ozsv\'{a}th and Szab\'{o}; See Proposition~\ref{sec:equivalence}.

In Theorem~\ref{thm:maingradingchangeformula} of the introduction, we stated our grading change formula for the Maslov and collapsed Alexander gradings, for link cobordisms which are given a coloring with exactly two colors. In this case, the ring $\cR_{\bmP}^-$ is isomorphic to $\bF_2[U,V]$. We will prove the following refinement of Theorem~\ref{thm:maingradingchangeformula}:

\begin{thm}\label{thm:generalAlexandergradingformula}Suppose that $(W,\cF)\colon (Y_1,\bL_1)\to (Y_2,\bL_2)$ is a decorated link cobordism. 
\begin{enumerate}
\item If $\sigma$ is an arbitrary, type-partitioned coloring of $\cF$, then the maps $F_{W,\cF^\sigma,\frs}$ satisfy the statement of Parts~\eqref{thm:maingradingchange:1}~and~\eqref{thm:maingradingchange:2} of Theorem~\ref{thm:maingradingchangeformula}.
\item If $(\sigma,\bJ)$ is a type-partitioned, indexed coloring of $\cF$, and $\bL_1$ and $\bL_2$ are $\bJ$-null-homologous with generalized $\bJ$-Seifert surfaces $S_1$ and $S_2$, then
\[
A_{S_2}(F_{W,\cF^\sigma,\frs}(\xs))_j-A_{S_1}(\xs)_j= \frac{\langle c_1(\frs),[\hat{\Sigma}_j]\rangle -[\hat{\Sigma}]\cdot [\hat{\Sigma}_j]}{2}+\frac{\chi((\Sigma_{j})_{\ws})-\chi((\Sigma_j)_{\zs})}{2},
\]
where $\hat{\Sigma}_j$ denotes $(-S_1)_j\cup \Sigma_j \cup (S_2)_j$, and $[\hat{\Sigma}]=\sum_{j\in \bJ} [\hat{\Sigma}_j]$.
\end{enumerate}
\end{thm}

\section{Heegaard triples and link cobordisms}
\label{sec:heegaardtriplesandspinc}

In this section we study multi-pointed Heegaard triples and $\Spin^c$ structures on 3- and 4-manifolds.

\subsection{Doubly multi-pointed Heegaard triples}

In this section, we study the following objects:

\begin{define}We  say that a Heegaard triple $(\Sigma, \as,\bs,\ve{\gamma},\ve{w},\ve{z})$ equipped with two collections of basepoints, $\ws$ and $\zs$,  is a \emph{doubly multi-pointed Heegaard triple} if each component of $\Sigma\setminus \taus$ has exactly one $\ve{w}$ basepoint, and exactly one $\ve{z}$ basepoint, for each $\taus\in \{\as,\bs,\gs\}$. 
\end{define}

If $(\Sigma,\as,\bs,\gs)$ is a Heegaard triple, then by adapting the construction from \cite{OSDisks}*{Section~8.1},  we can construct a 4-manifold $X_{\a\b\g}$ via the formula
\begin{equation}
X_{\alpha\beta\gamma}:=\big((\Delta \times \Sigma)\cup (e_{\a}\times U_{\alpha})\cup (e_{\b}\times U_{\beta})\cup (e_{\g}\times U_{\gamma})\big)/\sim, \label{eq:Xabgdef}
\end{equation}
In the above expression, $\Delta$ denotes a triangle with edges $e_{\a},$ $e_{\b}$ and $e_{\g}$, in clockwise order. Also, if $\tau\in \{\a,\b,\g\}$, then $U_{\tau}$ denotes the genus $g(\Sigma)$ handlebody with $\d U_{\g}=\Sigma$ which has $\ve{\tau}$ as a set of compressing curves. The 4-manifold $X_{\alpha\beta\gamma}$ has boundary
\[
\d X_{\alpha\beta\gamma}=-Y_{\alpha\beta}\sqcup -Y_{\beta\gamma}\sqcup Y_{\alpha\gamma}.
\]

If $(\Sigma, \as,\bs,\ve{\gamma},\ve{w},\ve{z})$ is a doubly multi-pointed Heegaard triple, then there is a properly embedded surface 
\[
\Sigma_{\alpha\beta\gamma}\subset X_{\alpha\beta\gamma},
\]
 which is well defined up to isotopy. It is constructed as follows. Let $f_{\alpha},$ $f_{\beta}$ and $f_{\gamma}$ denote Morse functions on $U_{\alpha},$ $U_{\beta}$ and $U_{\gamma}$ which induce the curves $\as$, $\bs$ and $\gs$ on $\Sigma$, and which have $\Sigma$ as a level set. For $\taus\in \{\as,\bs,\gs\}$, let $A_{\tau}\subset U_{\tau}$ denote the union of the flowlines of $\nabla f_{\tau}$ which terminate at a point in $\ws\cup \zs$. Note that $A_{\tau}$ is a properly embedded 1-manifold in $U_{\tau}$.

The surface $\Sigma_{\alpha\beta\gamma}$ is defined as the union
\[
\Sigma_{\alpha\beta\gamma}:=(\Delta\times  (-\ve{w}\cup \ve{z}))\cup (e_{\alpha}\times A_{\alpha})\cup (e_{\beta}\times A_{\beta})\cup (e_{\gamma}\times A_{\gamma}).
\]

For $\taus,\sigmas\in \{\as,\bs,\gs\}$ there is an oriented link $L_{\tau\sigma}\subset Y_{\tau\sigma}$, defined as
 \[
 L_{\tau\sigma}:=A_\tau\cup A_\sigma.
 \]
  The link $L_{\tau\sigma}$ can be naturally oriented by requiring that the intersections with $\Sigma$ are negative at the $\ve{w}$ basepoints, and positive at the $\ve{z}$ basepoints (note that $\Sigma$ is oriented as $\d U_\tau$ inside of $Y_{\tau\sigma}$). We include a picture in Figure~\ref{fig::1}.

\begin{figure}[ht!]
	\centering
	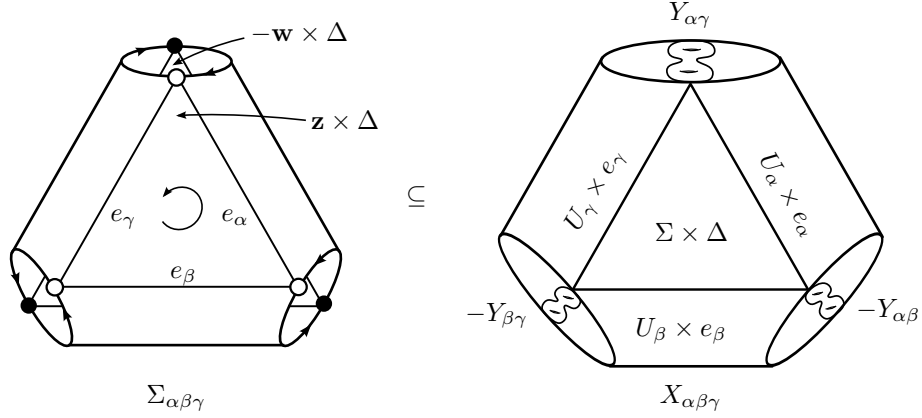
	\caption{\textbf{A schematic of the surface $\Sigma_{\alpha\beta\gamma}$ inside of $X_{\alpha\beta\gamma}$.} Orientations are shown.}\label{fig::1}
\end{figure}

Using the outward normal first convention for the boundary orientation, we have that
\[
\d \Sigma_{\alpha\beta\gamma}=-L_{\alpha\beta}\sqcup -L_{\beta\gamma}\sqcup L_{\alpha\gamma}.
\]

If $(\Sigma,\as,\bs,\gs,\ws,\zs)$ is a Heegaard triple, we will write $\cP_{\alpha\beta\gamma}$ for the set of integral 2-chains on $\Sigma$ which have boundary equal to a linear combination of the $\as,$ $\bs$ and $\gs$ curves. Elements of $\cP_{\a\b\g}$ are referred to as \emph{triply periodic domains}. There is a map
\[
H\colon \cP_{\alpha\beta\gamma}\to H_2(X_{\alpha\beta\gamma};\Z),
\]
described in \cite{OSTriangles}*{pg. 9}. The map $H$ is defined by taking  a domain $\cD\in \cP_{\alpha\beta\gamma}$ and including it into $\Sigma\times \{pt\}\subset X_{\alpha\beta\gamma}$. We then extend the 2-chain outward towards the boundary of $\Delta$, and then cap off with disks in $U_{\a}$, $U_{\b}$ and $U_{\g}$ to get a closed 2-chain $H(\cD)\in H_2(X_{\alpha\beta\gamma};\Z)$. 

Note that the 2-chains $H(\cD)$ and $\Sigma_{\alpha\beta\gamma}$ intersect only at $\{pt\}\times  (\ve{w}\cup \ve{z})$, and the multiplicity of the intersection points is given by the multiplicity of the domain $\cD$ at the basepoints. Hence
\begin{equation}
\langle \PD[H(\cD)],[\Sigma_{\a\b\g}]\rangle=\# (H(\cD)\cap \Sigma_{\alpha\beta\gamma})=(n_{\ve{z}}-n_{\ve{w}})(\cD).\label{eq:intersecttriplyperiodicdomain}
\end{equation} 
More generally, if $J\colon \Sigma_{\a\b\g}\to \bJ$ is an indexing, then the analogous equation for $\langle \PD[H(\cD)], [(\Sigma_{\a\b\g})_j] \rangle$ holds if we sum over only the basepoints mapped to $j\in \bJ$ by $J$.

\subsection{Heegaard triples subordinate to bouquets and bands}
\label{sec:triplesandbouquets}
Suppose that $\bL=(L,\ve{w},\ve{z})$ is a multi-based link in $Y$. Let $\bL_\alpha$ denote the arcs of $L\setminus (\ve{w}\cup \ve{z})$ which go from $\ve{w}$ basepoints to $\ve{z}$ basepoints. Let $\bL_\beta$ denote the arcs of $L\setminus (\ve{w}\cup \ve{z})$ which go from $\ve{z}$ basepoints to $\ve{w}$ basepoints. 

\begin{define}A \emph{$\beta$-bouquet} $\cB^\beta$ for a framed link $\bS_1$ in $Y\setminus L$ is a collection of arcs which connect links in $\bS_1$ to the interior of $\bL_\beta$. We assume that there is exactly one arc per component of $\bS_1$, and each arc has one endpoint on $\bS_1$ and one endpoint in $\bL_\beta$. 
\end{define}

An $\alpha$-bouquet can be defined analogously. We will use $\beta$-bouquets to define the grading, but we will show in Section~\ref{subsec:usealphabouqeuts} that either $\alpha$-bouquets or $\beta$-bouquets can be used.

Given a $\beta$-bouquet $\cB^\beta$ for a framed link $\bS_1\subset Y\setminus L$, we can consider the sutured manifold $Y(\bL\cup\cB^\beta)$ obtained by removing a regular neighborhood of $L\cup \cB^\beta$, and adding meridional sutures corresponding to the basepoints $\ve{w}$ and $\ve{z}$. As an adaptation of \cite{OSTriangles}*{Definition~4.2} and \cite{JCob}*{Definition~6.3}, we make the following definition:

\begin{define}Suppose that $\bL=(L,\ws,\zs)$ is a multi-based link in $Y$, and that $\bS_1\subset Y\setminus L$ is a framed 1-dimensional link. Write $\ell_1,\dots, \ell_k$ for the connected components of $\bS_1$. We say that a triple $(\Sigma,\alpha_1,\dots, \alpha_n,\beta_1,\dots, \beta_n,\beta_1',\dots, \beta_n',\ve{w},\ve{z})$ is \emph{subordinate} to the $\beta$-bouquet $\cB^\beta$ for a framed 1-dimensional link $\bS_1\subset Y\setminus L$ if the following hold:
\begin{enumerate}
\item $\Sigma\subset Y$ is an embedded Heegaard surface such that $\Sigma\cap L=\ws\cup \zs$.
\item After removing neighborhoods of $\ws$ and $\zs$ from the surface $\Sigma$, the diagram \sloppy $(\Sigma, \alpha_1,\dots, \alpha_n, \beta_{k+1},\dots, \beta_n, \ve{w},\ve{z})$ induces a sutured diagram for $Y(\bL\cup\cB^\beta)$.
\item  For $i\in \{1,\dots, k\}$, the curve $\beta_i$ is a meridian of $\ell_i$. Hence \sloppy $(\Sigma, \alpha_1,\dots, \alpha_n, \beta_{1},\dots, \beta_n, \ve{w},\ve{z})$ is a diagram for $(Y,\bL)$.
\item For $i\in \{k+1,\dots, n\}$, the curve $\beta_{i}'$ is a small Hamiltonian translate of $\beta_i$, and $|\beta_i\cap \beta'_j|=2\delta_{ij}$.
\item For $i\in \{1,\dots, k\}$, the curve $\beta'_i$ is a longitude of the link component that $\beta_i$ is a meridian of. Furthermore, for $i\in \{1,\dots, k\}$, the curves $\beta_i$ and $\beta'_i$ intersect in a single point.
\end{enumerate}
\end{define}

To prove invariance of our gradings, we will need the following result about connecting two Heegaard triples subordinate to a fixed bouquet:

\begin{lem}\label{lem:movesbetweenbetasubordinatediagrams}  Suppose that $\bL$ is a multi-based link in $Y$, and $\bS_1$ is a framed 1-dimensional link in the complement of $L$, and $\bS_1$ has components $\ell_1,\dots, \ell_k$. Any two Heegaard triples  subordinate to a fixed $\beta$-bouquet $\cB^\beta$ for $\bS_1$ can be connected via a sequence of the following moves:
	\begin{enumerate}
		\item\label{move:connecttwotriples1} An isotopy or handleslide amongst the $\as$ curves.
		\item\label{move:connecttwotriples2} An isotopy or handleslide amongst the curves $\{\beta_{k+1},\dots, \beta_n\}$, while performing the analogous isotopy or handleslide of the curve $\{\beta'_{k+1},\dots, \beta'_n\}$.
		\item\label{move:connecttwotriples3} An index 1/2-stabilization or destabilization, i.e., taking the connected sum of $\cT$ with a triple $(\bT^2,\alpha_0,\beta_0,\beta_0')$ where $\beta_0$ and $\beta_0'$ are small Hamiltonian isotopies of each other, intersecting at two points, and $|\alpha_0\cap \beta_0|=|\alpha_0\cap \beta_0'|=1$.
		\item\label{move:connecttwotriples4} For $i\in \{1,\dots, k\}$,  an isotopy of $\beta_i$, or a handleslide of $\beta_i$ across a $\beta_j$ with $j\in \{k+1,\dots, n\}$.
		\item\label{move:connecttwotriples5} For $i\in \{1,\dots, k\}$, an isotopy of $\beta'_i$, or a handleslide of $\beta'_i$ across a $\beta'_j$ with $j\in \{k+1,\dots, n\}$.
		\item\label{move:connecttwotriples6} An isotopy of the surface $\Sigma$ inside of $Y$, through surfaces whose intersection with $L\cup \cB^\beta$ is exactly $\ws\cup \zs$.
		\end{enumerate}
\end{lem}
\begin{proof}This can be proven by adapting \cite{OSTriangles}*{Lemma~4.5} or \cite{JCob}*{Lemma~6.5}.
\end{proof}

The maps from \cite{ZemCFLTQFT} use the following notion of bands:
\begin{define}\label{def:alphaband} An \emph{oriented $\alpha$-band} $B$ is an embedded rectangle in $Y$ such that $B\cap L$ consists of two opposite sides of $\d B$, and each component of $B\cap L$ is contained in a distinct component of $\bL_{\a}$,  where $\bL_{\alpha}$ denotes the subset of $L\setminus (\ws\cup \zs)$ consisting of arcs which go from $\ws$ to $\zs$.
\end{define}

We note that $\bL_{\alpha}$ can be equivalently described as the subset of $L\setminus (\ws\cup \zs)$ consisting of arcs contained in the $\alpha$-handlebody of $Y$, for any Heegaard diagram of $(Y,\bL)$.  Given a $\alpha$-band, we can form the multi-based link $\bL(B)$ in $Y$.
 
 Analogous to a Heegaard triple subordinate to a $\beta$-bouquet of a framed link in $Y\setminus L$, we need use the following notion of a Heegaard triple which is adapted to an $\alpha$-band:

\begin{define}[\cite{ZemCFLTQFT}*{Definition~6.2}]\label{def:subordinatetoalphaband} We say a triple $(\Sigma,\alpha_1',\dots, \alpha_n',\alpha_1,\dots, \alpha_n,\beta_1,\dots,\beta_n,\ws,\zs)$ is \emph{subordinate} to the $\alpha$-band $B$ if the following hold:
\begin{enumerate}
\item $\Sigma\subset Y$ is an embedded Heegaard surface such that $\Sigma\cap L=\ws\cup \zs$.
\item After removing neighborhoods of the $\ws$ and $\zs$ basepoints from 
\[
(\Sigma,\alpha_1,\dots,\alpha_{n-1},\beta_1,\dots, \beta_n,\ws,\zs),
\]
we obtain a sutured Heegaard diagram for the sutured manifold $Y\setminus N(L\cup B)$ (with meridional sutures induced by the basepoints).
\item $\alpha_1',\dots, \alpha_{n-1}'$ are small Hamiltonian isotopies of the curves $\alpha_1,\dots, \alpha_{n-1}$.
\item The curve $\alpha_n$ bounds a compressing disk in the complement of $L$, and the curve $\alpha_n'$ bounds a compressing disk in the complement of $L(B)$. Furthermore, $|\alpha_n'\cap \alpha_n|=2$.
\end{enumerate}
\end{define}

If $\bS_1$ is an $\ell$-component link, we write $W(Y,\bS_1)$ for following 4-dimensional 2-handle cobordism:
\[
W(Y,\bS_1):=([0,1]\times Y)\cup \left(\coprod_{i=1}^\ell D^2\times D^2\right)\cup ([1,2]\times Y(\bS_1)).
\]
We define the properly embedded surface $\Sigma(\bS_1)\subset W(Y,\bS_1)$ as
\[
\Sigma(\bS_1):=[0,2]\times L.
\]
We  write $\cW(Y,L,\bS_1)$ for the link cobordism
\[
\cW(Y,L,\bS_1):=(W(Y,\bS_1), \Sigma(\bS_1))\colon (Y,L)\to (Y,L).
\]

Similarly, if $B$ is an oriented band for the link $L\subset Y$, there is a well defined link cobordism
\[
\cW(Y,L,B)=([0,2]\times Y, \Sigma(B)),
\]
where $\Sigma(B)$ is the surface
\[
\Sigma(B):=([0,1]\times L)\cup (\{1\}\times B)\cup ([1,2]\times L).
\]

The following observation will be extremely important (compare \cite{OSTriangles}*{Proposition~4.3} and \cite{JCob}*{Proposition~6.6}):

\begin{lem}\label{lem:uniqueembeddingintoW}Let $(Y,\bL)$ be a 3-manifold containing a multi-based link.
\begin{enumerate}
\item\label{lem:embedpart1} Suppose $\bS_1\subset Y\setminus L$ is an $\ell$-component framed 1-dimensional link and  $(\Sigma,\as,\bs,\bs',\ws,\zs)$ is subordinate to a $\beta$-bouquet of $\bS_1$. Then there is an embedding of link cobordisms
\[
(X_{\a\b\b'}, \Sigma_{\a\b\b'})\hookrightarrow \cW(Y,L,\bS_1),
\]
which is well defined up to isotopy. The complement of $X_{\a\b\b'}$ in $W(Y,\bS_1)$ consists of a 4-dimensional 1-handlebody of genus $g(\Sigma)- \ell$. Furthermore $\Sigma_{\a\b\b'}$ intersects $Y_{\b\b'}$ in a $|\ws|$-component unlink.
\item\label{lem:embedpart2} Suppose $B$ is an $\alpha$-band for $\bL$ in $Y$, and $(\Sigma,\as',\as,\bs,\ws,\zs)$ is a Heegaard triple subordinate to $B$. Then there is an embedding of link cobordisms
\[
(X_{\a'\a\b},\Sigma_{\a'\a\b})\hookrightarrow \cW(Y,L,B),
\]
which is well defined up to isotopy. Furthermore, the complement of $X_{\a'\a\b}$ in $[0,2]\times Y$ consists of a  4-dimensional  1-handlebody of genus $g(\Sigma)$. The surface $\Sigma_{\a'\a\b}$ intersects $Y_{\a'\a}$ in a $|\ws|-1$ component unlink.
\end{enumerate}
\end{lem}
\begin{proof}Let us consider Part~\eqref{lem:embedpart1}, when $(\Sigma,\as,\bs,\bs',\ws,\zs)$ is subordinate to a bouquet of a framed link $\bS_1$. First, recall that by assumption $(\Sigma,\as,\bs,\ws,\zs)$ is an embedded Heegaard diagram for $(Y,\bL)$.  Consider 
\[
H_{\b\b'}:=([1-\epsilon,1]\times U_\b)\cup \left(\coprod_{i=1}^\ell D^2\times D^2\right)\subset W(Y,\bS_1).
\] We note that $[1-\epsilon,1]\times U_\b$ is a 4-dimensional 1-handlebody of genus $g(\Sigma)$, and we can view $H_{\b\b'}$ as being obtained by attaching $\ell$ 2-handles, which each cancel one of the 1-handles forming $[1-\epsilon,1]\times U_\b$. Hence $H_{\b\b'}$ is a 4-dimensional 1-handlebody of genus $g(\Sigma)-\ell$. 

We now claim that $X_{\a\b\b'}$ and $W(Y,\bS_1)\setminus \Int (H_{\b\b'})$ are diffeomorphic, via a diffeomorphism which is well defined up to isotopy. It is convenient to thicken the Heegaard surface and view $Y$ as $U_\alpha\cup ([-1,1]\times \Sigma) \cup U_\b$. Similarly, it is convenient to also fatten the vertices of the triangle $\Delta$ in the construction of $X_{\a\b\b'}$ and view $\Delta$ as a hexagon, with sides that alternate between being in the interior of $X_{\a\b\b'}$, and being on the boundary. We can write
\[
W(Y,\bS_1)\setminus \Int(H_{\b\b'})=([0,2]\times [-1,1]\times \Sigma)\cup ([0,2]\times U_\a)\cup ([0,1-\epsilon]\times U_\b)\cup ([1,2]\times U_{\b'}). 
\]
Up to rounding corners, this is canonically diffeomorphic to the 4-manifold $X_{\a\b\b'}$ defined in Equation~\eqref{eq:Xabgdef}, as long as we identify $[0,2]\times [-1,1]\times \Sigma$ with $\Delta\times \Sigma$. Furthermore, the surface $\Sigma_{\a\b\b'}$ is mapped into $[0,2]\times L$, by construction.

The argument for Part~\eqref{lem:embedpart2}, where $(\Sigma,\as',\as,\bs,\ws,\zs)$ is subordinate to an $\alpha$-band, is a straightforward adaptation.
\end{proof}

\subsection{Heegaard triples and $\Spin^c$ structures}

In this section we discuss $\Spin^c$ structures on 3- and 4-manifolds.

Heegaard Floer homology uses Turaev's interpretation of $\Spin^c$ structures on 3-manifolds as  homology classes of non-vanishing vector fields. Two non-vanishing vector fields on $Y$ are said to be homologous if they are homotopic on the complement of a set of 3-balls. 

If $Y$ is a closed 3-manifold, $\Spin^c(Y)$ has an affine action of $H_1(Y;\Z)$. The action has a convenient description in terms of vector fields,  referred to as \emph{Reeb surgery} \cite{Nicolasecu}*{Section~3.2}, which we describe presently.  Suppose $v$ is a non-vanishing vector field on $Y^3$, and $\gamma\subset Y$ is an oriented, simple closed curve. We can homotope $v$ so that $v|_\gamma=-\gamma'$. View a neighborhood of $\gamma$ as $S^1\times D^2$, and assume that $v=\d/\d \theta$ on this neighborhood. Viewing $D^2$ as the unit complex disk, let $[0,1]\subset D^2$ denote a radius of $D^2$, and pick a non-vanishing vector field $v(\gamma)$ on $[0,1]$ so that $(v(\gamma))(1)=v(1)=\d/\d \theta$ and $(v(\gamma))(0)=-\d/\d \theta$. We can extend $v(\gamma)$ over $D^2$ by requiring it to be invariant under the action of rotation on $D^2$. Next, we extend $v(\gamma)$ over $S^1\times D^2$ by requiring $v(\gamma)$ to be invariant under the action of $S^1$. We extend $v(\gamma)$ to all of $Y$ by setting it equal to $v$, outside of $S^1\times D^2$. This is illustrated schematically in Figure~\ref{fig::40}.

\begin{figure}[ht!]
\centering
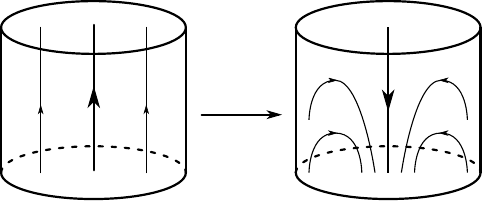
\caption{\textbf{Reeb surgery of a non-vanishing vector field $v$ along a curve $\gamma$.} On the left, $v$ is shown in a regular neighborhood of $\gamma$. We assume that $v|_{\gamma}=-\gamma'$. On the right is $v(\gamma)$, the result of Reeb surgery.}\label{fig::40}
\end{figure}

If $\frs(v)$ denotes the $\Spin^c$ structure induced by $v$, we claim that
\begin{equation}
\frs(v(\gamma))=\frs(v)+\PD[\gamma].\label{eq:Reebsurgery}
\end{equation}
To establish Equation~\eqref{eq:Reebsurgery}, one considers the set of relative $\Spin^c$ structures on a solid torus $N$, which we identify with homology classes of non-vanishing vector fields on $N$ which agree with some fixed vector field $v_0$ on $\d N$. The set $\Spin^c(N,\d N)$ is an affine space over $H^2(N,\d N;\Z)\iso H_1(N;\Z)$. If $\tau_0$ is a fixed trivialization of $v_0^\perp$ on $\d N$,  then the relative Chern class 
\[
c_1(\frs(v), \tau_0):=c_1(v^\perp, \tau_0)\in H^2(N,\d N;\Z)
\]
 gives a way of distinguishing relative $\Spin^c$ structures.  Abstractly, one knows that
\[
c_1(\frs+[\gamma],\tau_0)=c_1(\frs,\tau_0)+2\PD[\gamma].
\] 
 If $D\subset N$ is a disk such that $\langle \PD[\gamma], [D,\d D]\rangle =1$, then Equation~\eqref{eq:Reebsurgery} can be verified directly by computing
\begin{equation}
\langle c_1(v(\gamma)^{\perp},\tau_0),[D, \d D]\rangle-\langle c_1(v^\perp,\tau_0),[D, \d D]\rangle = 2\langle \PD[\gamma], [D, \d D]\rangle,\label{eq:Reebsurgeryexplicitcomputation}
\end{equation}
 as follows. We pick $\tau_0$ to be the planar trivialization induced by $D$. The 2-plane field $v^\perp$ has a non-vanishing section whose restriction to $\d D$ is constant with respect to $\tau_0$, so $\langle c_1(v^\perp, \tau_0),[D, \d D]\rangle$ vanishes. A non-vanishing section $\eta$ of $v(\gamma)^{\perp}$ can be constructed by picking a non-vanishing section along $[0,1]\subset D$, and then using the action of $S^1$ on $D$ (i.e. rotation) simultaneously with the action of $S^1$ on the fibers of $v(\gamma)^\perp$ to extend $\eta$ over all of $D$. The section $\eta$ induces a map from $\d D$ to the unit sphere bundle of $v(\gamma)^{\perp}|_{\d D}$, and the trivialization $\tau_0$ gives a well defined degree of this map. Noting that the orientation of $T D$ is opposite to the orientation of $v^\perp$ along $\d D$, the degree is $-2$. The evaluation of $c_1(v(\gamma)^\perp,\tau_0)$ can be computed by using $\tau_0$ to glue an oriented 2-plane bundle over a disk $\hat{D}$, with constant fiber, to the 2-plane bundle $v^{\perp}\to D$, and then extending $\eta$ to a generic section over $\hat{D}\cup D$. The value $\langle c_1(v(\gamma)^{\perp},\tau_0),[D, \d D]\rangle$ is the algebraic intersection number of $\eta$ on $D\cup \hat{D}$ with the zero section, which is equal to the degree of the induced map from $\d \hat{D}$ to the unit sphere bundle of $v(\gamma)^{\perp}$. The degree of this map over $\d \hat{D}$ is $+2$, since $\d \hat{D}$ has the opposite orientation as $\d D$. Hence $\langle c_1(v(\gamma)^{\perp},\tau_0),[D, \d D]\rangle=+2$, establishing Equation~\eqref{eq:Reebsurgeryexplicitcomputation}.

In \cite{OSDisks}*{Section~2.6}, to a Heegaard diagram $(\Sigma, \as,\bs,\ve{w})$ for $Y$, Ozsv\'{a}th and Szab\'{o} associate a map
\[
\frs_{\ve{w}}\colon \bT_{\alpha}\cap \bT_{\beta}\to \Spin^c(Y).
\] 
One starts with the upward gradient vector field for a Morse function inducing the Heegaard diagram. In a neighborhood of the flowlines passing through $\ve{w}$, as well as the flowlines passing through the intersection points of $\ve{x}\in \bT_{\alpha}\cap \bT_{\beta}$, one modifies the gradient vector field so that it is non-vanishing (which can be done since the chosen flowlines connect critical points of opposite parities). The construction of $\frs_{\ws}$ has the following dependence on the basepoints (compare \cite{OSDisks}*{Lemma~2.19}):

\begin{lem}\label{lem:spincpoincareduallink}If $\cH=(\Sigma, \as,\bs,\ve{w},\ve{z})$ is a diagram for $(Y,\bL)$, and $\ve{x}\in \bT_{\alpha}\cap \bT_{\beta}$, then
\[
\frs_{\ve{w}}(\ve{x})-\frs_{\ve{z}}(\ve{x})=\PD[L].
\]
\end{lem}
\begin{proof}Write $A_{\ws}$ and $A_{\zs}$ for the flowlines passing through $\ws$ and $\zs$, respectively. Note $A_{\ws}\cup A_{\zs}=L$, and that  $\frs_{\ws}(\xs)$ and $\frs_{\zs}(\xs)$ differ only on a neighborhood of $L$. Let us consider the disk $D:=N(A_{\zs})\cap \Sigma$.  As we described above, the set $\Spin^c(N(L), \d N(L))$ is an affine space over $H_1(N(L);\Z)\iso \Z$, and the difference between two $\Spin^c$ structures can be computed by evaluating their relative Chern classes on $[D,\d D]$. Note that our orientation convention is that $L$ goes from $\zs$ to $\ws$ in the handlebody $U_{\b}$, and furthermore, the Morse function defining the diagram $(\Sigma,\as,\bs,\ws,\zs)$ obtains its maximum on $U_{\b}$. The vector field corresponding to $\frs_{\ws}(\xs)$ naturally has $L$ as an orbit, while the vector field for $\frs_{\zs}(\xs)$ naturally has $-L$ as an orbit. We can pick the vector fields for $\frs_{\ws}(\xs)$ and $\frs_{\zs}(\xs)$ so that on $D$, $\frs_{\ws}(\xs)$ agrees with the Reeb surgery of $\frs_{\zs}(\xs)$ on $L$. The proof then follows from Equation~\eqref{eq:Reebsurgeryexplicitcomputation}, which implies that the homology class of $\frs_{\ws}(\xs)$ is equal to the Reeb surgery of $\frs_{\zs}(\xs)$ on $L$, completing the proof.
\end{proof}

Analogous to Turaev's description of a 3-dimensional  $\Spin^c$ structure as a homology class of non-vanishing vector fields, there is an interpretation of $\Spin^c$ structures on a 4-manifold $W$ in terms of homotopy classes of almost complex structures on the 2-skeleton of $W$, which extend over the 3-skeleton; See \cite{OSDisks}*{Section~8.1.4} and \cite{GompfSpinc}*{pg.~49}. Similar to Reeb surgery in 3-dimensions, if $\Sigma$ is a properly embedded surface in $W$, then the action of $\PD[\Sigma]\in H^2(W;\Z)$ can be described geometrically by modifying an almost complex structure in a neighborhood of $\Sigma$. Rather than give a general description, we will focus on a particular example arising from Heegaard triples, though it is straightforward to extend our result to the general case.

To a Heegaard triple $(\Sigma, \as,\bs,\ve{\gamma},\ve{w})$, in \cite{OSDisks}*{Section~8} Ozsv\'{a}th and Szab\'{o} associate a map
\[
\frs_{\ve{w}}\colon \pi_2(\ve{x},\ve{y},\ve{z})\to\Spin^c(X_{\alpha\beta\gamma}).
\] 
Similar to Lemma~\ref{lem:spincpoincareduallink}, the map $\frs_{\ws}$ has an  important dependence on the basepoints $\ws$:

\begin{lem}\label{lem:poincaredualofsurface}If $(\Sigma, \as,\bs,\ve{\gamma},\ve{w},\ve{z})$ is a doubly multi-pointed Heegaard triple, then
\[
\frs_{\ve{w}}(\psi)-\frs_{\ve{z}}(\psi)=\PD[\Sigma_{\alpha\beta\gamma}].
\]
\end{lem}
\begin{proof} We briefly recall how the maps $\frs_{\ws}$ and $\frs_{\zs}$ are defined. To a homology class of triangles $\psi$, Ozsv\'{a}th and Szab\'{o} associate a real, oriented 2-plane field on the complement of a collection of 4-balls in $X_{\alpha\beta\gamma}$.  We will write $\xi_{\ws}$ and $\xi_{\zs}$ for the oriented 2-plane fields associated to $\frs_{\ws}(\psi)$ and $\frs_{\zs}(\psi)$. We recall the construction.

One starts with a partially defined oriented 2-plane field $\xi_0$. On $e_{\a}\times U_{\a}$,   $\xi_0$ is defined on the complement of $e_{\a}\times \Crit(f_\a)$ to be $\{0\}\oplus \ker (d f_{\a})=\{0\}\oplus (\nabla f_a)^{\perp}$. On $e_\b\times U_\b$ and $e_\g\times U_\g$, $\xi_0$ is defined similarly. The 2-plane field $\xi_0$ is defined on all of $\Delta\times \Sigma$ to be $\{0\}\oplus T\Sigma$.

If $\psi\in \pi_2(\xs,\ys,\zs)$ is a homology class and $u\colon \Delta\to \Sym^n(\Sigma)$ is a topological representative, we define the immersed surface $S_u\subset \Delta\times \Sigma$ as the set of  points $(x,\sigma)$ such that $\sigma\in u(x)$. One extends $S_u$ into $e_{\a}\times U_{\a}$ by extending the image of $u$ radially over a set of compressing disks of  $U_{\a}$ which have boundary $\as$. The immersed surface $S_u$ is extended similarly into $e_{\b}\times U_{\b}$ and $e_{\g}\times U_{\g}$.

 Next, we define the surface $S_{\ws}$  by extending $\Delta\times \ws$ into $e_{\a}\times U_{\a}$ with the product of $e_{\a}$ and the flowlines of $\nabla f_{\a}$ passing through $\ws$. The surface $S_{\ws}$ is similarly extended into $e_{\b}\times U_{\b}$ and $e_{\g}\times U_{\g}$. A surface $S_{\zs}$ is defined similarly. Note that  $S_{\ws}\cup S_{\zs}=\Sigma_{\a\b\g}$.

 Ozsv\'{a}th and Szab\'{o} describe a codimension 1 singular foliation $F$ on $X_{\alpha\beta\gamma}$ \cite{OSDisks}*{Figure~5}. The intersection of each leaf of $F$ with $\Sigma_{\alpha\beta\gamma}$ induces a codimension 1 singular foliation $F_{\Sigma}$ on  $\Sigma_{\alpha\beta\gamma}$. Let $\Gamma\subset \Sigma_{\alpha\beta\gamma}$ denote the union of the singular leaves of $F_{\Sigma}$. The foliation $F_{\Sigma}$ is schematically shown in Figure~\ref{fig::39}. 

\begin{figure}[ht!]
\centering
\begingroup%
  \makeatletter%
  \providecommand\color[2][]{%
    \errmessage{(Inkscape) Color is used for the text in Inkscape, but the package 'color.sty' is not loaded}%
    \renewcommand\color[2][]{}%
  }%
  \providecommand\transparent[1]{%
    \errmessage{(Inkscape) Transparency is used (non-zero) for the text in Inkscape, but the package 'transparent.sty' is not loaded}%
    \renewcommand\transparent[1]{}%
  }%
  \providecommand\rotatebox[2]{#2}%
  \newcommand*\fsize{\dimexpr\f@size pt\relax}%
  \newcommand*\lineheight[1]{\fontsize{\fsize}{#1\fsize}\selectfont}%
  \ifx\svgwidth\undefined%
    \setlength{\unitlength}{206.76112225bp}%
    \ifx\svgscale\undefined%
      \relax%
    \else%
      \setlength{\unitlength}{\unitlength * \real{\svgscale}}%
    \fi%
  \else%
    \setlength{\unitlength}{\svgwidth}%
  \fi%
  \global\let\svgwidth\undefined%
  \global\let\svgscale\undefined%
  \makeatother%
  \begin{picture}(1,0.66007391)%
    \lineheight{1}%
    \setlength\tabcolsep{0pt}%
    \put(0,0){\includegraphics[width=\unitlength,page=1]{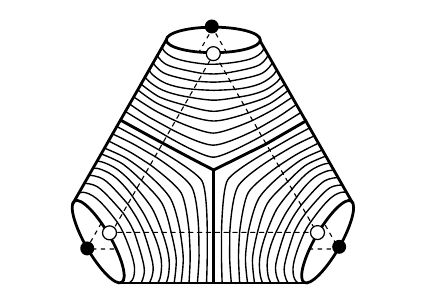}}%
    \put(0.8283691,0.043829){\color[rgb]{0,0,0}\makebox(0,0)[lt]{\lineheight{1.25}\smash{\begin{tabular}[t]{l}$-L_{\as,\bs}$\end{tabular}}}}%
    \put(0.56269506,0.61125088){\color[rgb]{0,0,0}\makebox(0,0)[lt]{\lineheight{1.25}\smash{\begin{tabular}[t]{l}$L_{\as,\gs}$\end{tabular}}}}%
    \put(0.1502686,0.04480181){\color[rgb]{0,0,0}\makebox(0,0)[rt]{\lineheight{1.25}\smash{\begin{tabular}[t]{r}$-L_{\bs,\gs}$\end{tabular}}}}%
    \put(-0.50476828,-1.17772195){\color[rgb]{0,0,0}\makebox(0,0)[lt]{\begin{minipage}{1.66256638\unitlength}\raggedright \end{minipage}}}%
    \put(0.78036671,0.32892479){\color[rgb]{0,0,0}\makebox(0,0)[lt]{\lineheight{1.25}\smash{\begin{tabular}[t]{l}$\Gamma$\end{tabular}}}}%
    \put(0,0){\includegraphics[width=\unitlength,page=2]{fig39.pdf}}%
  \end{picture}%
\endgroup%

\caption{\textbf{The codimension 1 singular foliation $F_{\Sigma}$ on $\Sigma_{\alpha\beta\gamma}$.} The union of the singular leaves $\Gamma$ is shown bold.}\label{fig::39}
\end{figure}

Let $B_{u,\rep}\subset \Delta \times \Sigma$ denote a neighborhood of the set of points $(x,\sigma)\in S_u\cap (\Delta \times \Sigma)$ where $u(x)\in \Sym^n(\Sigma)$ has a repeated entry. Define $\Gamma_{\ws}:=\Gamma\cap S_{\ws}$ and $\Gamma_{\zs}:=\Gamma\cap S_{\zs}$. Let $B_{u,\ws}\subset (\Delta\times \Sigma)$ denote a neighborhood of the set of points $(x,\sigma)\in S_u\cap (\Delta \times \Sigma)$ where $u(x)\cap \ws \neq \varnothing$. Define $B_{u,\zs}$ similarly. By choosing the map $u$ generically, we may assume that $B_{u,\rep}\cap B_{u,\ws}=\varnothing$. Furthermore, by picking $u$ and $N(\Gamma)$ appropriately, we may assume that $B_{u,\ws}\subset N(\Gamma_{\ws})$ and $B_{u,\zs}\subset N(\Gamma_{\zs})$.

 The 2-plane field $\xi_{\ws}$ is defined in the complement of $N(\Gamma_{\ws})\cup B_{u,\rep}$. It is constructed by extending the non-singular 2-plane field $\xi_0|_{X_{\a\b\g}\setminus N(S_{\ws}\cup S_u)}$ over all of $X_{\a\b\g}\setminus N(\Gamma_{\ws}\cup B_{u,\rep})$.   See \cite{OSDisks}*{Section~8.1.4} for a precise description of the extension to $X_{\a\b\g}\setminus N(\Gamma_{\ws}\cup B_{u,\rep})$. The 2-plane field $\xi_{\zs}$ is constructed analogously.

Note that $\xi_{\ws}$ and $\xi_{\zs}$ are both defined on the complement of $N(\Gamma)\cup B_{u,\rep}$, and they differ only in a neighborhood of $\Sigma_{\alpha\beta\gamma}\setminus N(\Gamma)$.

Let us write
\[
X_0:=X_{\alpha\beta\gamma}\setminus N(\Gamma) \qquad \text{and} \qquad \Sigma_0:=\Sigma_{\alpha\beta\gamma}\setminus N(\Gamma).
\]
 Note that $\Sigma_0$ is a disjoint union of properly embedded annuli in $X_0$, each with one end on $\d N(\Gamma)$ and one end on $\d X_{\alpha\beta\gamma}$. The singular foliation $F_{\Sigma}$ restricts to a non-singular, codimension 1 foliation on each connected component of $\Sigma_0$, and furthermore, each leaf is a simple closed curve, which is homologically essential in $\Sigma_0$.

 Since $\Gamma$ is a 1-dimensional cell complex, the map $\Spin^c(X_{\alpha\beta\gamma})\to \Spin^c(X_0)$ is an isomorphism. Hence it is sufficient to show that
\begin{equation}
(\frs_{\ws}(\psi)-\frs_{\zs}(\psi))|_{X_0}=\PD[\Sigma_0].\label{eq:diffspincremovegraph}
\end{equation}

Suppose that $\ell$ is a leaf of $F$ which is contained in $X_0$. Note that $\ell$ can be identified with either $Y_{\a\b}$, $Y_{\b\g}$ or $Y_{\a\g}$. Similarly the corresponding leaf $\ell_{\Sigma}=\ell\cap \Sigma_{\alpha\beta\gamma}\subset F_{\Sigma}$ can be identified with one of the links $L_{\a\b}$, $L_{\b\g}$ or $L_{\a\g}$. 

The foliations $F$ and $F_{\Sigma}$ are not orientable. Nonetheless, the non-singular leaves are canonically oriented, since if $\ell$ is a non-singular leaf of $F$, and $\ell_\Sigma\subset \ell$ is the corresponding leaf of $F_{\Sigma}$, then the pair $(\ell,\ell_\Sigma)$ is canonically identified with one of the boundary components of $(X_{\a\b\g},\Sigma_{\a\b\g})$ which we give the boundary orientation.

Now suppose that $(\ell,\ell_\Sigma)$ is a nested pair of leaves of $F$ and $F_\Sigma$. Note that $\xi_{\ws}|_\ell$ and $\xi_{\zs}|_{\ell}$ are subbundles of $T\ell$. Let $v_{\ws}$ and $v_{\zs}$ denote the orthogonal complements of $\xi_{\ws}$ and $\xi_{\zs}$ inside of $T\ell$.

Note that using the outward normal first boundary orientation convention, it is easily checked that
\begin{equation}
v_{\zs}=-T\ell_{\Sigma},\qquad v_{\ws}= T\ell_\Sigma, \qquad (\xi_{\zs})^{\perp}=-T \Sigma_0\qquad \text{and}\qquad  (\xi_{\ws})^{\perp}= T\Sigma_0. \label{eq:orientationsofbundles}
\end{equation}
On $\ell$,  the oriented 1-plane bundles $v_{\ws}$ and $v_{\zs}$ differ only in a neighborhood of $\ell_{\Sigma}$, and there they differ exactly the same as the vector fields built using the 3-dimensional $\Spin^c$ structure maps considered in Lemma~\ref{lem:spincpoincareduallink}. Hence, by Lemma~\ref{lem:spincpoincareduallink}, we can take $v_{\ws}$ to be obtained from  $v_{\zs}$ by Reeb surgery (see Figure~\ref{fig::40}) on $\ell_{\Sigma}$.

Together with  their orthogonal complements, the oriented 2-plane fields $\xi_{\ws}$ and $\xi_{\zs}$ determine almost complex structures $J_{\ws}$ and $J_{\zs}$ on $T X_0$, up to homotopy. We will show that
  \begin{equation}
(TX_0, J_{\ws})\iso (TX_0, J_{\zs})\otimes_{\C} L,\label{eq:isoofcomplexbundles}  
  \end{equation} where  $L$ is a complex line bundle on $X_0$ with $c_1(L)=\PD[\Sigma_0]$. Note that this will imply Equation~\eqref{eq:diffspincremovegraph}, and  complete the proof.

 Our strategy for showing  Equation~\eqref{eq:isoofcomplexbundles} is to find a complex 1-plane subbundle  $\zeta_{\ws}\subset (TX_0,J_{\ws})$  such that
\begin{equation}
\zeta_{\ws}\iso \xi_{\zs}\otimes_{\C} L \qquad \text{and} \qquad \zeta_{\ws}^{\perp}\iso \xi_{\zs}^{\perp}\otimes_{\C} L,\label{eq:whatzetawillsatisfy}
\end{equation}
for a line bundle $L\to X_0$ with $c_1(L)=\PD[\Sigma_0]$. Note Equation~\eqref{eq:whatzetawillsatisfy} will imply Equation~\eqref{eq:isoofcomplexbundles}.

We can trivialize $N(\Sigma_0)$ as $\Sigma_0\times D^2$, viewing $D^2$ as the unit complex disk. By construction, $\xi_{\ws}$ and $\xi_{\zs}$ are both invariant under the $S^1$ action on $D^2$, and also constant on the $\Sigma_0$ factor, under this trivialization of $N(\Sigma_0)$.

We now describe our choice of $\zeta_{\ws}\subset (T X_0,J_{\ws})$.  Write $[0,1]\subset D^2$ for a radius. Let $G_{\ws}$ denote the $\CP^1$-bundle over $X_0$ whose fiber over  $p\in X_0$ is the set of complex lines in $(T_p X_0, J_{\ws})$. We define  a bundle $G_{\zs}$ similarly. Pick  $p_0\in \Sigma_0$, and let $\gamma\colon \{p_0\}\times [0,1]\to G_{\ws}$ denote a section of the bundle $G_{\ws}|_{\{p_0\}\times [0,1]}$ which satisfies 
\[\gamma(p_0,0)=(\xi_{\ws}^{\perp})_{(p_0,0)}\qquad \text{and}\qquad  \gamma(p_0,1)=(\xi_{\ws})_{(p_0,1)}.\]
 Since the subspace $\xi_{\ws}^\perp|_{(p_0,0)}$ is fixed by the rotation action  on $D^2$, we can construct a complex 1-plane subbundle $\zeta_{\ws}$ of $(TX_0|_{\{p_0\}\times D^2},J_{\ws})$ by declaring it to be equal to $\gamma$ on $\{p_0\}\times [0,1]$ and also to be invariant under the $S^1$-action on $D^2$. Next, we extend $\zeta_{\ws}$ to all of $\Sigma_0\times D^2$ by declaring it to be constant on the $\Sigma_0$ factor of $\Sigma_0\times D^2$.   Since $\zeta_{\ws}$ agrees with $\xi_{\ws}$ on $\d N(\Sigma_0)$, we can extend $\zeta_{\ws}$ to all of $X_0$ by declaring it to be $\xi_{\ws}$ outside of $N(\Sigma_0)$.

We now describe the complex line bundle $L\to X_0$, which will feature in Equation~\eqref{eq:whatzetawillsatisfy}. We will define $L$ to be an oriented, real 2-plane subbundle of $\underline{\C}\oplus \underline{\C}\to X_0$, where $\underline{\C}$ denotes the trivial complex line bundle. By trivializing $T\Sigma_0$ as an oriented 2-plane bundle, we can write
\begin{equation}
T N(\Sigma_0)\iso T\Sigma_0\oplus TD^2\iso \underline{\C}\oplus \underline{\C}.\label{eq:thetrivialization}
\end{equation}
  By Equation~\eqref{eq:orientationsofbundles}, the 2-plane field $\zeta_{\ws}$  has constant fiber $-\{0\}\oplus \C$ along $\d N(\Sigma_0)$, under this trivialization. We define the fiber of $L$ over $x$ to be the fiber of $\zeta_{\ws}$ over $x$, when $x\in N(\Sigma_0)$, and we define the fiber of $L$ over $x$ to be $-\{0\}\oplus \C$ for $x\not \in N(\Sigma_0)$.  We define $L^{\perp}$ to be the real orthogonal complement of $L$ inside of $\underline{\C}\oplus \underline{\C}\to X_0$.

We  now show that
\begin{equation}
\zeta_{\ws}\iso \xi_{\zs}\otimes_{\C} L\qquad \text{and} \qquad \zeta_{\ws}^\perp\iso \xi_{\zs}^\perp\otimes_{\C} L^{\perp}.\label{eq:bundleisomorphisms0}
\end{equation} To establish  Equation~\eqref{eq:bundleisomorphisms0}, note that on $N(\Sigma_0)$, $\xi_{\zs}$ can be identified with $-TD\iso-\{0\}\oplus \underline{\C}$ under the trivialization in Equation~\eqref{eq:thetrivialization}. Outside of $N(\Sigma_0)$, the bundle $L$ is identified with $-\{0\}\oplus \underline{\C}$. Hence we can define a bundle isomorphism $\xi_{\zs}\otimes_{\C} L\to \zeta_{\ws}$ via the formula
\begin{equation}
(z\otimes w)_x\mapsto \begin{cases}
\bar{z}\cdot (w)_x& \text{if } x\in  N(\Sigma_0)\\
\bar{w}\cdot (z)_x & \text{if } x\not \in N(\Sigma_0).
\end{cases}
\label{eq:stranglebundleisomorphism}
\end{equation}
In Equation~\eqref{eq:stranglebundleisomorphism},  `$\cdot$' denotes multiplication with respect to the homotopically unique complex structure on $\zeta_{\ws}$, induced from its orientation. This not the same as the complex structure of $\underline{\C}\oplus \underline{\C}$.  A bundle isomorphism  $\xi_{\zs}^\perp\otimes_{\C} L^{\perp}\to \zeta_{\ws}^\perp$ is defined similarly, establishing Equation~\eqref{eq:bundleisomorphisms0}.

Finally, it remains to show
\begin{equation}
c_1(L)=c_1(L^\perp)=\PD[\Sigma_0],\ \label{eq:generalChernclassformula}
\end{equation}
which also implies $L\iso L^\perp$ as complex line bundles. Viewing $S^2$ as the quotient $D^2/\d D^2$,  we note that $L$ and $L^\perp$ are isomorphic to the pullbacks of two complex line bundles over $S^2$ under a map $\pi\colon X_0\to S^2$. Writing $p$ for the point $[\d D^2]$ in $S^2=D^2/\d D^2$, the map $\pi$ is given by $\pi(x)= p$ if $x\not \in N(\Sigma_0)$, and $\pi(z,w)= w$ if $(z,w)\in N(\Sigma_0)\iso \Sigma_0\times D^2$. Write $L_0$ and $L_0^\perp$ for these two complex line bundles, over $S^2$. Note that $\pi^*(\PD[p])=\PD[\Sigma_0]$, so it is sufficient to show that 
\begin{equation}
 c_1(L_0)=c_1(L_0^\perp)=\PD[p].\label{eq:simplifiedChernclasscomp}
\end{equation}

Equation~\eqref{eq:simplifiedChernclasscomp} can be established by the following direct computation. For the computation, it is easier to view $S^2$ as the union of $D^2$ and another disk $\hat{D}^2$, whose center is the point $p$. Over $\hat{D}^2$, we define the fibers of the bundles $L_0$ and $L_0^\perp$ to be $-\{0\}\oplus \C$ and $-\C\oplus \{0\}$, respectively. To prove Equation~\eqref{eq:simplifiedChernclasscomp}, we will prove that a generic section of each of $L_0$ and $L_0^\perp$ intersects the zero section once, algebraically, with positive sign.

 Let $\rho^\theta\colon D^2\to D^2$ denote the diffeomorphism obtained by multiplication by $e^{i\theta}$. We will write $\rho_*^\theta$ for the bundle automorphism of the trivial bundle $\underline{\C}\oplus \underline{\C}\to D^2$, which covers the diffeomorphism $\rho^\theta\colon D^2\to D^2$, and is defined by the formula
 \[
 \rho_*^\theta((v,w)_{z})=(v,e^{i\theta}\cdot w)_{e^{i\theta}\cdot z}.
 \]
 
  On $[0,1]\subset D^2$ we pick any nonvanishing section $v$ of $L_0$. We can extend $v$ to a non-vanishing section on all of $D^2$ using the formula
\[
v(re^{i\theta})=\rho^\theta_*(v(r)).
\]
 Since $v(0)\in  \C\oplus \{0\}$, which is fixed by the action of $\rho^\theta_*$, the vector field $v$ is well defined when $r=0$. The bundle $L_0$ has fiber $-\{0\}\oplus \C$, along $\d D^2$, and hence the vector field $v$ determines a map from $\d D^2$ to $S^1$ with respect to this trivialization. The map induced by $v$ is of degree $-1$, with respect to the complex orientation of $L_0$. If we extend $v$ generically over $\hat{D}^2$, then the index of $v$ over $\hat{D}^2$ is the same as the oriented intersection of a generic perturbation of $v$ with the zero section. The index of $v$ over $\hat{D}^2$ is $+1$, since it is the same as the degree of $v$ as a map from $\d \hat{D}^2$ to $S^1$, and the orientation of $\d \hat{D}^2$ is opposite to $\d D^2$. Hence  $c_1(L_0)= \PD[p].$
 
 We can analyze $L_0^\perp$ similarly. We let $\tilde{v}\colon [0,1]\to L^\perp_0$ be a nonvanishing section. Since the fiber over 0 of $L_0^\perp$ is $\{0\}\oplus \C$, we can define an extension of $\tilde{v}$ to all of $D^2$ via the formula
 \[
\tilde{v}(re^{i\theta})=e^{-i\theta}\cdot \rho_*^\theta(\tilde{v}(r)). 
 \]
 Multiplication is with respect to the complex structure of $L_0^\perp$. The bundle $L_0^{\perp}$ has constant fiber $-\{0\}\oplus \C$ on $\d D^2$, and with respect to the induced trivialization of $L_0^\perp$ on $\d D^2$, the degree of the induced map from $\d D^2$ to $S^1$ is $-1$. As with $L_0$, this implies that $c_1(L_0^\perp)=\PD[p]$. Equation~\eqref{eq:simplifiedChernclasscomp} follows, and hence so does Equation~\eqref{eq:generalChernclassformula}.
 
 Combining Equations~ \eqref{eq:bundleisomorphisms0} and \eqref{eq:generalChernclassformula} implies Equation~\eqref{eq:isoofcomplexbundles} and completes the proof.
\end{proof}

\section{Kirby calculus for manifolds with boundary}
\label{sec:kirbycalculus}
Our strategy for constructing an absolute grading on $\cCFL^\infty$ will parallel the construction of the absolute $\Q$-gradings on the groups $\HF^\infty(Y,\frs)$ in \cite{OSTriangles}. W define a notion of Kirby diagram for  a 3-manifold $Y$ with an embedded link $L$, by presenting the link complement $Y\setminus N(L)$ as surgery on the standard unlink complement $S^3\setminus N(U)$. We then consider a Kirby calculus argument for how to relate two such presentations.

It will be useful for our purposes to first define a more general notion of surgery presentations, not specific to link complements:

\begin{define}\label{def:parsurgdata} If $M$ and $M'$ are oriented 3-manifolds with boundary and $\phi\colon \d M\to \d M'$ is a fixed, orientation preserving diffeomorphism, we say that a pair $(\bS_1,f)$ is a \emph{parametrized surgery presentation} for $(M,M',\phi)$ if $\bS_1\subset \Int M$ is a framed link and $
f\colon M(\bS_1)\to M' $
is a diffeomorphism which extends $\phi$.
\end{define}

It is well known that if $M$ and $M'$ are connected, oriented 3-manifolds and $\phi\colon \d M\to \d M'$ is an orientation preserving diffeomorphism, then there exists a parametrized surgery presentation of $(M,M',\phi)$. This can be seen by the following argument (cf. \cite{RobertsRelativeKirbyCalculus}). Using the diffeomorphism $\phi$, we form the closed, oriented three manifold $-M\cup (\d M\times [0,1])\cup M'$. This bounds a compact oriented 4-dimensional manifold $W$. We can view such a manifold as a cobordism of manifolds with boundary from $M$ to $M'$. We think of $M$ and $M'$ as the ``horizontal'' parts of $\d W$ and $[0,1]\times \d M$ as the ``vertical'' part of the boundary. We can find a Morse function which is $0$ on $M$, $t$ on $\{t\}\times  \d M$ and $1$ on $M'$, which has only index 1, 2, and 3 critical points. By changing the 4-manifold, we can replace index 1 and 3 critical points with index 2 critical points, to get a cobordism from $M$ to $M'$ which has a Morse function with only index 2 critical points. If we take a gradient like vector field on $W$ which is $\d/\d t$ on $[0,1]\times \d M$, then the descending manifolds from the index 2 critical points yield a framed link $\bS_1$ in $M$, and the Morse function and gradient like vector field determine a diffeomorphism $f\colon M(\bS_1)\to M'$ which is well defined up to isotopy and extends $\phi$.

Using our terminology, Kirby's calculus of links \cite{KirbyCalculus} gives a set of moves which can relate any two parametrized surgery presentations of triples $(M,M',\phi)=(S^3,Y,\varnothing)$, when $Y$ is a closed, oriented 3-manifold. The moves are blow-ups, blow-downs, handleslides, and isotopies of $f$ or $\bS_1$. In \cite{FennRourke},  Fenn and Rourke  extended the calculus to the case that $M$ is an arbitrary, closed, oriented 3-manifold, though an extra move is required, which is supported in a solid torus; See Figure~\ref{fig::22}. In \cite{RobertsRelativeKirbyCalculus}, Roberts extends the calculus to arbitrary 3-tuples $(M,M',\phi)$.

Keeping track of the parametrization $f\colon M(\bS_1)\to M'$ in the definition of a parametrized surgery decomposition is  important for our purposes. A Kirby move between two framed links $\bS_1$ and $\bS_1'$ in $M$ canonically yields a diffeomorphism $M(\bS_1)\to M(\bS_1')$ as we describe in the following paragraph. In particular, if $(\bS_1,f)$ is parametrized surgery data, and $\bS_1'$ is the result of one of the above moves on $\bS_1$, then there is a diffeomorphism $f'\colon M(\bS_1')\to M'$ which is canonically induced, and is well defined up to isotopy. 

 We now illustrate the canonical diffeomorphism from $M(\bS_1)$ to $M(\bS_1')$ resulting from a Kirby move; See \cite{GompfStipsicz}*{pg.~160}. Suppose that $\bS_1$ and $\bS_1'$ are two framed links in $M$ and  $\bS_1'=\bS_1\cup \{U\}$, where $U$ is a $\pm 1$ framed unknot which is contained in a 3-ball in $M\setminus \bS_1$.  The manifolds $M(\bS_1)\setminus B $ and $M(\bS_1')\setminus B(U)$ are canonically diffeomorphic, via the identity map. Noting that $B$ and $B(U)$ are both 3-balls with an identification of their boundaries, the diffeomorphism can be extended over $B$. Furthermore, the extension is unique up to isotopy, since $\MCG(B^3,S^2)=\{*\}$.
 
 Similarly, consider the case that $\bS_1'\subset M$ is obtained from $\bS_1$ via a handleslide. Let $H\subset S^3(\bS_1)$ and $H'\subset S^3(\bS_1')$ denote the genus 2 handlebodies which contain the support of the handleslide. The manifolds $S^3(\bS_1)\setminus H$ and $S^3(\bS_1')\setminus H'$ are canonically diffeomorphic (via the identity map, and clearly this diffeomorphism extends over $H$). Since $\MCG(H_g,\d H_g)=\{*\}$ for the genus $g$ handlebody $H_g$, the extension is unique, up to isotopy.
 
  As a specific example, a diffeomorphism $\psi\colon M\to M$ which is the identity on $\d M$ may be presented as a sequence of Kirby moves on framed links in $M$, starting and ending at the empty link in $\Int (M)$.

 We state the following version of the main result from \cite{RobertsRelativeKirbyCalculus}:

\begin{thm}\label{thm:movesbetweenparsurdecomps} Any two parametrized surgery presentations of a triple $(M,M',\phi)$ can be connected by a sequence of the following moves:
\begin{enumerate}[leftmargin=22mm, ref= $\cO_{\arabic*}$, label=\textrm{(Move $\cO_{\arabic*}$):}]
\setcounter{enumi}{-1}
\item\label{move:O0}  Isotopies of $f$ or $\bS_1$ which fix $\d M$.
\item\label{move:O1}  Handleslides of link components amongst each other.
\item\label{move:O2}  Blow-ups or blow-downs along a $\pm 1$ framed unknot in $\Int (M)$.
\item\label{move:O3}   Addition or removal to $\bS_1$ of a two component link $K\cup \mu_K$ inside of a solid torus which is disjoint from $\bS_1$, where $K$ is a core of the solid torus, and $\mu_K$ is a meridian of $K$. Furthermore, $K$ can be given arbitrary framing, though $\mu_K$ must be given the Seifert framing. See Figure~\ref{fig::22}.
\end{enumerate}
\end{thm}

\begin{figure}[ht!]
\centering
\begingroup%
  \makeatletter%
  \providecommand\color[2][]{%
    \errmessage{(Inkscape) Color is used for the text in Inkscape, but the package 'color.sty' is not loaded}%
    \renewcommand\color[2][]{}%
  }%
  \providecommand\transparent[1]{%
    \errmessage{(Inkscape) Transparency is used (non-zero) for the text in Inkscape, but the package 'transparent.sty' is not loaded}%
    \renewcommand\transparent[1]{}%
  }%
  \providecommand\rotatebox[2]{#2}%
  \newcommand*\fsize{\dimexpr\f@size pt\relax}%
  \newcommand*\lineheight[1]{\fontsize{\fsize}{#1\fsize}\selectfont}%
  \ifx\svgwidth\undefined%
    \setlength{\unitlength}{323.81996251bp}%
    \ifx\svgscale\undefined%
      \relax%
    \else%
      \setlength{\unitlength}{\unitlength * \real{\svgscale}}%
    \fi%
  \else%
    \setlength{\unitlength}{\svgwidth}%
  \fi%
  \global\let\svgwidth\undefined%
  \global\let\svgscale\undefined%
  \makeatother%
  \begin{picture}(1,0.22179478)%
    \lineheight{1}%
    \setlength\tabcolsep{0pt}%
    \put(0.50146339,0.12931996){\color[rgb]{0,0,0}\makebox(0,0)[t]{\lineheight{0}\smash{\begin{tabular}[t]{c}\ref{move:O3}\end{tabular}}}}%
    \put(0,0){\includegraphics[width=\unitlength,page=1]{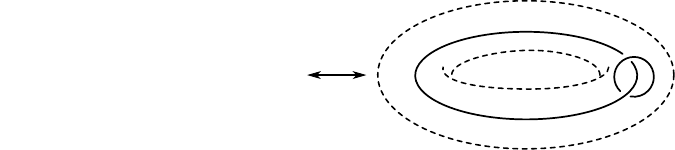}}%
    \put(0.80949706,0.18059906){\color[rgb]{0,0,0}\makebox(0,0)[lt]{\lineheight{1.25}\smash{\begin{tabular}[t]{l}$K$\end{tabular}}}}%
    \put(0.93123721,0.14658134){\color[rgb]{0,0,0}\makebox(0,0)[lt]{\lineheight{1.25}\smash{\begin{tabular}[t]{l}$\mu_K$\end{tabular}}}}%
    \put(0,0){\includegraphics[width=\unitlength,page=2]{fig22.pdf}}%
  \end{picture}%
\endgroup%

\caption{\textbf{Move~\ref{move:O3}} The move takes place in a solid torus in $M$. The framing on $K$ can be arbitrary, but the framing on $\mu_K$ is the Seifert framing.\label{fig::22}}
\end{figure}

We note that in \cite{RobertsRelativeKirbyCalculus}, the moves between framed links are presented without explicitly referencing diffeomorphism $f$, though for our purposes, it is important to keep track of the diffeomorphism $f$. 

 We now consider the implications of Theorem~\ref{thm:movesbetweenparsurdecomps} when $M$ and $M'$ are link complements.

Let $U$ denote a fixed, oriented, $\ell$-component unlink in $S^3$.
Suppose that $L$ is an $\ell$-component link in a 3-manifold $Y$. Define
\[
 Y_L:=Y\setminus N(L)\qquad \text{and}\qquad S^3_U:=S^3\setminus N(U).
 \]
Let $\phi_0\colon U\to L$ be a fixed, orientation preserving diffeomorphism of compact 1-manifolds. Together with a  choice of framing $\lambda$ of $L$, $\phi_0$ determines a diffeomorphism
\begin{equation}
\phi_\lambda\colon \d (S^3_U)\to \d (Y_L),\label{eq:philambdadef}
\end{equation} 
well defined up to isotopy.

\begin{define} We call a tuple $\bP=(\phi_0,\lambda, \bS_1,f)$ a \emph{parametrized Kirby diagram} for an oriented link $L$ in $Y$, if the following are satisfied:
\begin{enumerate}
\item $\phi_0\colon U\to L$ is an orientation preserving diffeomorphism of 1-manifolds.
\item $\lambda$ is a framing of $L$.
\item $\bS_1$ is a framed link in $S^3_U$.
\item  $f\colon  S^3_U(\bS_1)\to Y_L$ is a diffeomorphism such that $f|_{\d S^3_U}$ is isotopic to the diffeomorphism $\phi_\lambda$ from Equation~\eqref{eq:philambdadef}.
\end{enumerate}
\end{define}

We now describe  a new move, \ref{move:L3}, which we will be a convenient alternative to Move~\ref{move:O3} when we are working with parametrized Kirby diagrams for links. Given a link $L$ in $Y$, with framing $\lambda$, the move \ref{move:L3} consists of performing $\pm 1$ surgery on a knot $K$ which is a meridian of a component of $U$, as in Figure~\ref{fig::23}. Suppose $\bP=(\phi_0,\lambda, \bS_1,f)$ is a choice of parametrized Kirby diagram for $(Y,\bL)$.   The parametrizing diffeomorphism $f\colon S^3_U(\bS_1)\to Y_L$ induces a diffeomorphism  from $ S^3_U(\bS_1\cup \{K\})$ to $ Y_L(f(K))$. Furthermore, there is a canonical diffeomorphism from $Y_L(f(K))$ to $Y_L$, which is the identity outside of a solid torus containing $K$ whose boundary intersects $\d Y_L$ in an annulus. By composing the two maps, we obtain a diffeomorphism
\[
f_K\colon S^3_U(\bS_1\cup \{K\})\to Y_L,
\]
well defined up to isotopy.

On $\d S^3_U$, the map $f_K$ no longer restricts to $\phi_\lambda$, but instead $\phi_{\lambda'}$, where $\lambda'$ is a new framing which differs by $\mp 1$ on the component that $K$ encircled.

 After performing Move \ref{move:L3}, we get a new parametrized Kirby diagram $\bP_K=(\phi_0,\lambda',\bS_1\cup K,f_K)$ for $(Y,\bL)$.

\begin{figure}[ht!]
\centering
\begingroup%
  \makeatletter%
  \providecommand\color[2][]{%
    \errmessage{(Inkscape) Color is used for the text in Inkscape, but the package 'color.sty' is not loaded}%
    \renewcommand\color[2][]{}%
  }%
  \providecommand\transparent[1]{%
    \errmessage{(Inkscape) Transparency is used (non-zero) for the text in Inkscape, but the package 'transparent.sty' is not loaded}%
    \renewcommand\transparent[1]{}%
  }%
  \providecommand\rotatebox[2]{#2}%
  \newcommand*\fsize{\dimexpr\f@size pt\relax}%
  \newcommand*\lineheight[1]{\fontsize{\fsize}{#1\fsize}\selectfont}%
  \ifx\svgwidth\undefined%
    \setlength{\unitlength}{210.70255242bp}%
    \ifx\svgscale\undefined%
      \relax%
    \else%
      \setlength{\unitlength}{\unitlength * \real{\svgscale}}%
    \fi%
  \else%
    \setlength{\unitlength}{\svgwidth}%
  \fi%
  \global\let\svgwidth\undefined%
  \global\let\svgscale\undefined%
  \makeatother%
  \begin{picture}(1,0.37810653)%
    \lineheight{1}%
    \setlength\tabcolsep{0pt}%
    \put(0,0){\includegraphics[width=\unitlength,page=1]{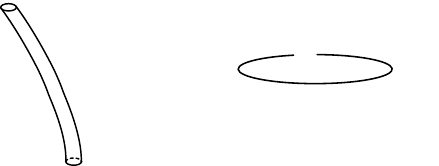}}%
    \put(0.07750579,0.35106252){\color[rgb]{0,0,0}\makebox(0,0)[lt]{\lineheight{0}\smash{\begin{tabular}[t]{l}$U$\end{tabular}}}}%
    \put(0.67797563,0.35106252){\color[rgb]{0,0,0}\makebox(0,0)[lt]{\lineheight{0}\smash{\begin{tabular}[t]{l}$U$\end{tabular}}}}%
    \put(0,0){\includegraphics[width=\unitlength,page=2]{fig23.pdf}}%
    \put(0.3782459,0.220607){\color[rgb]{0,0,0}\makebox(0,0)[t]{\lineheight{0}\smash{\begin{tabular}[t]{c}\ref{move:L3}\end{tabular}}}}%
    \put(0.85433428,0.25830839){\color[rgb]{0,0,0}\makebox(0,0)[lt]{\lineheight{0}\smash{\begin{tabular}[t]{l}$\pm 1$\end{tabular}}}}%
    \put(0.85460042,0.15871351){\color[rgb]{0,0,0}\makebox(0,0)[lt]{\lineheight{0}\smash{\begin{tabular}[t]{l}$K$\end{tabular}}}}%
    \put(0,0){\includegraphics[width=\unitlength,page=3]{fig23.pdf}}%
  \end{picture}%
\endgroup%

\caption{\textbf{The move $\cL_3$ between two parametrized Kirby diagrams for a link}. The solid tube denotes a boundary component of $S^3_U$. The knot $K$ is a new component in the framed link $\bS_1$.\label{fig::23}}
\end{figure}

We now reformulate Theorem~\ref{thm:movesbetweenparsurdecomps} to describe a sufficient set of moves between any two parametrized Kirby diagrams of a link:

\begin{prop}\label{prop:connecttwoparamsurgdecomps}Suppose $Y$ is a 3-manifold containing an oriented link $L$ with $\ell$ components. Let $U$ denote a fixed $\ell$-component unlink in $S^3$. Any two parametrized Kirby diagrams for $(Y,L)$ can be connected by a sequence of the following moves:
\begin{enumerate}[leftmargin=22mm, ref= $\cL_{\arabic*}$, label=\textrm{(Move $\cL_{\arabic*}$):}]
\setcounter{enumi}{-1}
\item  \label{move:L0} An isotopy of $f$ or $\bS_1$ which fixes $\d S^3_U$ pointwise.
\item \label{move:L1} A handleslide amongst the components of $\bS_1$.
\item \label{move:L2} A blow-up or blow-down along a $\pm 1$ framed unknot which is contained in a 3-ball in $S^3\setminus (U\cup \bS_1)$.
\item \label{move:L3} A blow-up or blow-down along a $\pm 1$ framed unknot which is a meridian of a single component of $U$, and is unlinked from all other components of $U$ and $\bS_1$.
\item  \label{move:L4} If $\psi_0\colon U\to U$ is an orientation preserving diffeomorphism, and $\psi\colon (S^3,U)\to (S^3,U)$ is an orientation preserving extension, we replace $\bP=(\phi_0,\lambda, \bS_1,f)$ with $\bP'=(\phi_0\circ \psi^{-1}_0, \lambda, \psi(\bS_1), f\circ (\psi^{\bS_1})^{-1})$, where $\psi^{\bS_1}\colon S^3_U(\bS_1)\to S^3_U(\psi(\bS_1))$ is the diffeomorphism induced by $\psi$.
\end{enumerate}
\end{prop}
\begin{proof}For fixed $\phi_0$ and $\lambda$, Theorem~\ref{thm:movesbetweenparsurdecomps} implies that Moves~\ref{move:L0}, \ref{move:L1}, \ref{move:L2}  and \ref{move:O3} suffice.

 We first claim that for fixed $\phi_0$ and $\lambda$, it is sufficient to use only instances of Move~\ref{move:O3} where $K$ is a meridian of a single component of $U$, and is unlinked from all other components of $\bS_1$ and $U$. Let us write $\cO_{3}^0$ for an instance of Move~\ref{move:O3} with this configuration. Move $\cO_3^0$ is shown in Figure~\ref{fig::55}.

\begin{figure}[ht!]
\centering
\begingroup%
  \makeatletter%
  \providecommand\color[2][]{%
    \errmessage{(Inkscape) Color is used for the text in Inkscape, but the package 'color.sty' is not loaded}%
    \renewcommand\color[2][]{}%
  }%
  \providecommand\transparent[1]{%
    \errmessage{(Inkscape) Transparency is used (non-zero) for the text in Inkscape, but the package 'transparent.sty' is not loaded}%
    \renewcommand\transparent[1]{}%
  }%
  \providecommand\rotatebox[2]{#2}%
  \newcommand*\fsize{\dimexpr\f@size pt\relax}%
  \newcommand*\lineheight[1]{\fontsize{\fsize}{#1\fsize}\selectfont}%
  \ifx\svgwidth\undefined%
    \setlength{\unitlength}{188.6322323bp}%
    \ifx\svgscale\undefined%
      \relax%
    \else%
      \setlength{\unitlength}{\unitlength * \real{\svgscale}}%
    \fi%
  \else%
    \setlength{\unitlength}{\svgwidth}%
  \fi%
  \global\let\svgwidth\undefined%
  \global\let\svgscale\undefined%
  \makeatother%
  \begin{picture}(1,0.4223457)%
    \lineheight{1}%
    \setlength\tabcolsep{0pt}%
    \put(0,0){\includegraphics[width=\unitlength,page=1]{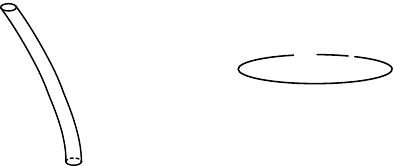}}%
    \put(0.08657411,0.39213749){\color[rgb]{0,0,0}\makebox(0,0)[lt]{\lineheight{0}\smash{\begin{tabular}[t]{l}$U$\end{tabular}}}}%
    \put(0.75730003,0.39213749){\color[rgb]{0,0,0}\makebox(0,0)[lt]{\lineheight{0}\smash{\begin{tabular}[t]{l}$U$\end{tabular}}}}%
    \put(0,0){\includegraphics[width=\unitlength,page=2]{fig55.pdf}}%
    \put(0.42250137,0.25437041){\color[rgb]{0,0,0}\makebox(0,0)[t]{\lineheight{0}\smash{\begin{tabular}[t]{c}$\cO_3^0$\end{tabular}}}}%
    \put(0,0){\includegraphics[width=\unitlength,page=3]{fig55.pdf}}%
    \put(0.93867306,0.31764093){\color[rgb]{0,0,0}\makebox(0,0)[lt]{\lineheight{1.25}\smash{\begin{tabular}[t]{l}$0$\end{tabular}}}}%
  \end{picture}%
\endgroup%

\caption{\textbf{The move $\cO_3^0$.} It is a special instance of Move~\ref{move:O3}.\label{fig::55}}
\end{figure} 
 
We will show that an arbitrary instance of Move \ref{move:O3}, performed along a knot $K$ and its meridian $\mu_K$, can instead be written as a composition of Moves \ref{move:L0}, \ref{move:L1}, \ref{move:L2} and $\cO_3^0$. Let $\Phi_{K,\cO_3}$ denote the diffeomorphism from $S^3_U(\bS_1)$ to $S^3_U(\bS_1\cup K\cup \mu_K)$ which is the identity outside of a solid torus containing $K$ and $\mu_K$.  The knot $K$ can be transformed into a meridian of a component of $U$ via a sequence of crossing changes of $K$ with components of $\bS_1$, with $U$, or with itself. Hence we will show that if $K'$ is obtained from $K$ by changing a crossing of $K$ with $\bS_1$, $U$, or itself, then  Move~\ref{move:O3}, applied along $K$, can be written as a composition of Move~\ref{move:O3}, applied along $K'$, as well as some combination of Moves \ref{move:L0}, \ref{move:L1}, \ref{move:L2} and $\cO_3^0$.

   First, suppose that $K'$ is obtained from $K$ by either a crossing change of $K$ with itself, or a crossing change of $K$ with another component of $\bS_1$. The link $\bS_1\cup K'\cup \mu_{K'}$ can be obtained from $\bS_1\cup K\cup \mu_K$ by handlesliding a link component across $\mu_K$ (if the crossing change is of $K$ with itself, then $K$ is handleslid across $\mu_K$; if the crossing change is of $K$ with another component of $\bS_1$, then the other component is handleslid across $\mu_K$) followed by an isotopy. Let $\Phi_{H}\colon S^3_U(\bS_1\cup K\cup \mu_K)\to S^3_U(\bS_1\cup K'\cup \mu_{K'})$ denote the diffeomorphism resulting from the composition of this handleslide and isotopy. We will show that
 \begin{equation}
 \Phi_{H}\circ \Phi_{K,\cO_3}\sim \Phi_{K',\cO_3},
 \label{eq:isotopicdiffeos1}
 \end{equation}
  where $\sim$ denotes isotopy. Suppose that $K'$ is obtained by changing a crossing of $K$ with $K_0\subset \bS_1$. We will handleslide $K_0$ across $\mu_K$.  Let $a$ be the handleslide arc connecting $K_0$ and $\mu_K$, and let $D$ denote a Seifert disk of $\mu_K$. Let $N\subset S^3_U$ denote a regular neighborhood of $K\cup a\cup K_0\cup D$; see Figure~\ref{fig::2}. We note $N$ is a genus 2 handlebody.  Note that $\Phi_H$ and $\Phi_{K',\cO_3}\circ \Phi_{K,\cO_3}^{-1}$ both restrict to diffeomorphisms between the surgered 3-manifolds  $N(K_0\cup K\cup \mu_K)$ and $N(K_0\cup K'\cup \mu_{K'}),$ (which are both themselves genus 2 handlebodies), and $\Phi_H$ and $\Phi_{K',\cO_3}\circ \Phi_{K,\cO_3}^{-1}$ agree on $\d N(K_0\cup K\cup \mu_K)$, we conclude that $\Phi_H$ and $\Phi_{K',\cO_3}\circ \Phi_{K,\cO_3}^{-1}$ must be isotopic  since $MCG(H_g,\d H_g)=\{*\}$, where $H_g$ denotes a genus $g$ handlebody. Hence Equation~\eqref{eq:isotopicdiffeos1} holds.   An analogous argument holds for changing a crossing of $K$ with itself. Let $D$ be a Seifert disk of $\mu_K$ (which intersects $K$ at a single point), and let $a$ be an arc from $\mu_K$ to $K$. Let $N$ denote a regular neighborhood of $K\cup a\cup D$. Noting that $N$ is a genus two handlebody, since $\Phi_H$ and $\Phi_{K',\cO_3}\circ \Phi_{K,\cO_3}^{-1}$ differ only inside of $N$, it follows that they must be isotopic.
 A similar argument establishes Equation~\eqref{eq:isotopicdiffeos1} in the the case that $K'$ is obtained by changing a crossing of $K$ with itself. 
 
 \begin{figure}[ht!]
 \centering
 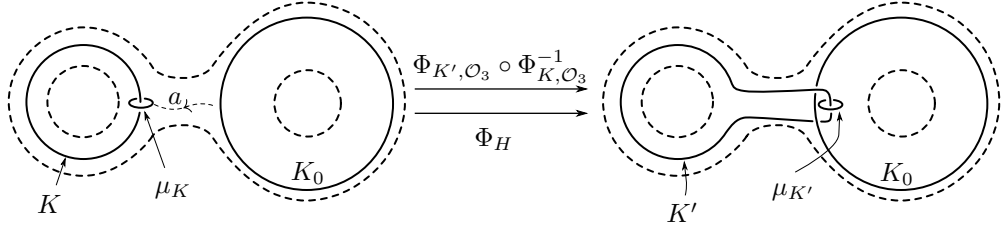
 \caption{\textbf{Move~\ref{move:O3}, applied along $K$, is equal to a composition of Move~\ref{move:O3}, applied along $K'$, and a handleslide.} The region shown is the genus 2 handlebody $N$.  The diffeomorphisms $\Phi_H$ and $\Phi_{K',\cO_3}\circ \Phi_{K,\cO_3}^{-1}$ are equal outside of $N$, and hence are isotopic.\label{fig::2}}
 \end{figure}

Next, we consider the case that $K\subset S^3_U\setminus \bS_1$ is a knot and $K'$ is the result of changing a crossing of $K$ with $U$. We wish to show that Move~\ref{move:O3}, performed along $K$, can be written as a composition of Move~\ref{move:O3}, performed on $K'$, as well as Moves~\ref{move:L0}, \ref{move:L1}, \ref{move:L2} and $\cO_3^0$. The procedure for doing this is shown in Figure~\ref{fig::11}. We perform Move $\cO_3^0$ on a meridian of $U$, then perform a sequence of handleslides, and then perform the inverse of Move $\cO_3^0$.  As before, the parameterizing diffeomorphism resulting from applying Move~\ref{move:O3} along $K$ is isotopic to the parametrized diffeomorphism resulting from applying Move~\ref{move:O3} along $K'$, and then applying a sequence of Moves~\ref{move:L0}, \ref{move:L1} and $\cO_3^0$, since they can be shown to agree outside of a genus two handlebody.
 
  \begin{figure}[ht!]
  \centering
  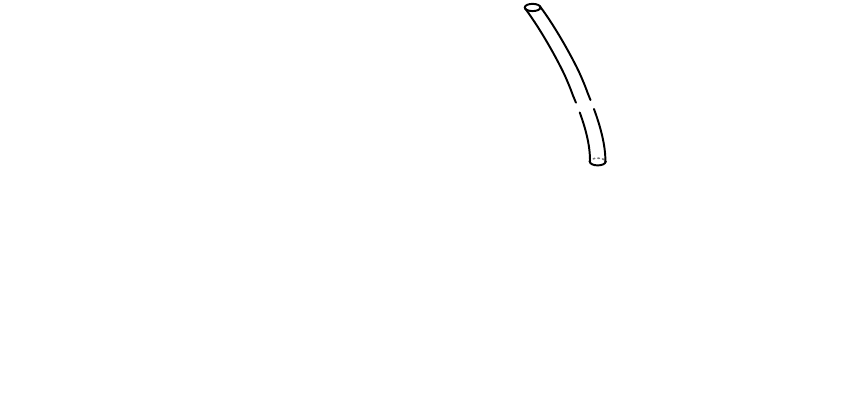
  \caption{\textbf{Writing an instance of Move~\ref{move:O3} along a knot $K'$ in terms of Move~\ref{move:O3} on $K$ as well as Moves~\ref{move:L0}, \ref{move:L1}, \ref{move:L2} and $\cO_3^0$}.\label{fig::11}}
  \end{figure}
 
 In such manner, by performing a sequence of crossing changes, we may reduce $K$ to a meridian of a single component of $U$. Hence we can write an arbitrary \ref{move:O3} move on a knot $K$ as a composition of the Moves~\ref{move:L0}, \ref{move:L1}, \ref{move:L2} and $\cO_3^0$.
 
Next, we note that using a standard trick, Move $\cO_3^0$ can be written as a composition of two \ref{move:L3} moves and possibly moves $\cL_1$ and $\cL_2$, depending on the framing of the knot $K$ in the $\cO_3^0$ move; see \cite{FennRourke}*{Figure~13}. In detail, note that by handlesliding $K$ over $\mu_K$, we can assume that $K$ has framing $0$ or $1$. If $K$ has framing $1$, then handlesliding $\mu_K$ over $K$ leaves two meridians of $U$, one with framing $+1$ and the other with framing $-1$ (i.e. two applications of Move $\cL_3$). If $K$ is instead given framing $0$, then we blow-up along a $-1$ framed unknot $K_0$ (as in Move~\ref{move:L2}), and slide both $K$ and $\mu_K$ over $K_0$. This leaves $K$ and $\mu_K$ both with framing $-1$. We then slide $K_0$ over $\mu_K$, which leaves $K$ with framing $-1$, $K_0$ with framing $0$, and $\mu_K$ with framing $-1$. Furthermore, $\mu_K$ is now unlinked with $K$ and $K_0$, and $K_0$ is now a meridian of $K$. Sliding $K_0$ over $K$ and blowing down along $\mu_K$ yields two meridians of $U$, one with framing $+1$ and the other with framing $-1$.
 
We note also that a single instance of $\cL_3$ changes the framing of one component of $L$ by $\pm 1$. Hence for fixed $\phi_0$, by applying $\cL_3$ some number of times, any framing $\lambda$ can be achieved.
 
We finally  show that any $\phi_0\colon U\to L$ can be achieved.  We note that any two $\phi_0$ maps differ by pre-composition with an orientation preserving diffeomorphism $\psi_0$ from $U$ to itself, and any such diffeomorphism extends to an orientation preserving diffeomorphism $\psi$ of $(S^3,U)$ with itself. Writing $\psi$ also for the induced automorphism of $S^3_U$, there is an induced diffeomorphism
\[
\psi^{\bS_1}\colon S^3_U(\bS_1)\to S^3_U(\psi(\bS_1)).
\]
Hence, via tautology, we get an induced parametrized Kirby decomposition
\[
\bP:=(\phi_0\circ \psi^{-1}_0,\lambda,\psi(\bS_1), f\circ (\psi^{\bS_1})^{-1}).
\]
 Hence Move~\ref{move:L4} can be used  to move between any two $\phi_0$ maps. 
\end{proof}

\section{Definition of the gradings} 
\label{sec:definitiongradings}

In this section, we give the definition of the Alexander and Maslov gradings. In Section~\ref{sec:relativegradings}, we describe the relative versions of the gradings. In Section~\ref{sec:defabsolutegradings}, we define the absolute gradings.

\subsection{Relative gradings}
\label{sec:relativegradings}

We begin with the relative Maslov gradings. For our purposes, it is convenient to describe two Maslov gradings, $\gr_{\ve{w}}$ and $\gr_{\ve{z}}$. The relative grading $\gr_{\ve{w}}$ is defined on generators by
 \begin{equation}
 \gr_{\ve{w}}(\ve{x},\ve{y})=\mu(\phi)-2\sum_{w\in \ve{w}} n_w(\phi),\label{eq:relwgrading}
 \end{equation} 
 for any disk $\phi\in \pi_2(\ve{x},\ve{y})$. By \cite{OSProperties}*{Proposition~7.5}, if $\phi\in \pi_2(\xs,\xs)$ is a class, then
 \begin{equation}
\langle c_1(\frs_{\ws}(\xs)), H(\phi)\rangle =\mu(\phi)-2\sum_{w\in \ws} n_w(\phi),\label{eq:chernclassformula}
 \end{equation}
where $H$ denotes the map $H\colon \pi_2(\xs,\xs)\to H_2(Y;\Z)$ obtained by capping the domain of $\phi$ (viewed as a 2-chain in $\Sigma$ with boundary an integral sum of $\as$ and $\bs$ curves) with a sum of compressing disks in $U_\a$ and $U_\b$ which are attached along the $\as$ and $\bs$ curves. In particular, if $c_1(\frs)$ is torsion, then the quantity $\gr_{\ws}(\xs,\ys)$ defined in Equation~\eqref{eq:relwgrading} is independent of the choice of $\phi$.
 
 We extend $\gr_{\ve{w}}$ to  $\cCFL^\infty(Y,\bL,\frs)$ by declaring all $U_{\ve{w}}$ variables to have grading $-2$, and all $V_{\ve{z}}$ variables to have grading 0.
 
 Analogously, we can define the relative grading $\gr_{\ve{z}}$ via the formula
 \begin{equation}
 \gr_{\ve{z}}(\ve{x},\ve{y})=\mu(\phi)-2\sum_{z\in \ve{z}} n_z(\phi),\label{eq:relzgrading}
 \end{equation}
  for a disk $\phi\in \pi_2(\ve{x},\ve{y})$. We extend $\gr_{\ve{z}}$ to all of $\cCFL^\infty(Y,\bL,\frs)$, by declaring all $U_{\ws}$ variables to be 0 graded, and all $V_{\zs}$ variables to be $-2$ graded. By Equation~\eqref{eq:chernclassformula}, $\gr_{\zs}(\xs,\ys)$ is independent of the choice of $\phi$ when $c_1(\frs_{\ve{z}}(\ve{x}))=c_1(\frs-\PD[L])$ is torsion.

\begin{rem}If $c_1(\frs)$ (resp. $c_1(\frs-\PD[L])$) is torsion, then $\gr_{\ws}$ (resp. $\gr_{\zs}$) will also determine a well defined relative grading on $\cCFL^\infty(Y,\bL^\sigma,\frs)$ whenever $\sigma$ is a type-partitioned coloring.
\end{rem}

We now describe the relative Alexander multi-gradings. Suppose that $J\colon L\to \bJ$ is an indexing of $L$, and $L$ is $\bJ$-null-homologous. The Alexander multi-grading is a relative $\Z^{\bJ}$ grading on $\cCFL^\infty(Y,\bL,\frs)$.

 Given a homology class $\phi\in \pi_2(\xs,\ys)$ and an index $j\in \bJ$, we define 
\[
n_{\ve{z}}(\phi)_j:=\sum_{\substack{z\in \ve{z}\\ J(z)=j}} n_z(\phi),\qquad \text{and} \qquad  n_{\ve{w}}(\phi)_j:=\sum_{\substack{w\in \ve{w}\\ J(w)=j}} n_w(\phi).
\]
 The relative multi-grading is defined by the equation
\begin{equation}
A(\ve{x}, \ve{y})_j=(n_{\ve{z}}-n_{\ve{w}})_j(\phi)
\label{eq:defalexandergrading},
\end{equation}
 for any class $\phi\in \pi_2(\ve{x},\ve{y})$.
 
We extend the grading in Equation~\eqref{eq:defalexandergrading} to $\cCFL^\infty(Y,\bL,\frs)$ by declaring $V_z$ to have grading $+1$ in index grading $J(z)$, and $U_w$ to have grading $-1$ in index $J(w)$. As a concrete example, this implies that $A(V_z\cdot \xs, \xs)_j=1$ if $J(z)=j$.

To see that Equation~\eqref{eq:defalexandergrading} is independent of the choice of class $\phi\in \pi_2(\xs,\ys)$, it is sufficient to show that the expression $(n_{\ve{z}}-n_{\ve{w}})_j(\phi)$ vanishes for any element $\phi\in\pi_2(\ve{x},\ve{x})$. Writing $H\colon \pi_2(\xs,\xs)\to H_2(Y;\Z)$ for the map obtained by capping off a periodic domain, a simple computation shows that if $\phi\in \pi_2(\xs,\xs)$, then
	 \begin{equation}
	  (n_{\ve{z}}-n_{\ve{w}})_j(\phi)=\# (H(\phi)\cap J^{-1}(j)) \label{eq:nz-nw=intersectionnumber},
	  \end{equation}
	  where $\# (H(\phi)\cap J^{-1}(j))$ denotes the oriented intersection number. If $L$ is $\bJ$-null-homologous, then by definition $L_j:=J^{-1}(j)$ is null-homologous, so the expression on the right hand side of Equation~\eqref{eq:nz-nw=intersectionnumber} vanishes.

\begin{rem}Suppose $\bL=(L,\ws,\zs)$ is a link in $Y$. If $(\sigma, J)$ is a type-partitioned,  indexed coloring of $\bL$, with index set $\bJ$, and $L$ is $\bJ$-null-homologous, then the  complex $\cCFL^\infty(Y,\bL^\sigma,\frs)$ has a well defined $\Z^{\bJ}$-valued relative Alexander grading. The type-partitioned requirement on the coloring assures that none of the $U_{\ws}$ variables are identified with one of the $V_{\zs}$ variables. Requiring that the coloring be indexed ensures that if two variables are identified, then their corresponding link components are assigned the same index by $J$.
\end{rem}

\begin{rem}It is straightforward to compute that the differential $\d$ on $\cCFL^\infty$  lowers $\gr_{\ws}$ and $\gr_{\zs}$ by 1, and preserves $A_j$, whenever they are defined.
\end{rem}

We now show that the relatively graded isomorphism type of $\cCFL^\infty$ is an invariant:

\begin{lem}\label{lem:changeofdiagramsrelativegradings}Suppose $\bL$ is a multi-based link in $Y$, with a type-partitioned, indexed coloring $(\sigma,J)$ with indexing set $\bJ$, and $L$ is $\bJ$-null-homologous. If $(\cH,J_s)$ and $(\cH',J_s')$ are two choices of diagrams and almost complex structures for $(Y,\bL)$, then the transition map
\[
\Phi_{(\cH,J_s)\to (\cH',J_s')}\colon \cCFL^\infty_{J_s}(\cH,\sigma,\frs)\to \cCFL^\infty_{J_s'}(\cH',\sigma,\frs)
\]
 preserves the relative Alexander multi-grading over $\Z^{\bJ}$. Similarly, assuming instead that the coloring is type-partitioned and $c_1(\frs)$ (resp. $c_2(\frs-\PD[L])$) is torsion, the transition map preserves the relative $\gr_{\ws}$ (resp. $\gr_{\zs}$) grading.
\end{lem}

\begin{proof} To verify the claim, one must prove that the relative gradings are preserved by the transition maps associated to the following Heegaard moves: isotopies and handleslides of the $\as$- and $\bs$ curves, index $1/2$-(de)stabilizations, isotopies of the Heegaard surface inside $Y$, and changes of the almost complex structure.

We will focus on showing that the transition map associated to a handleslide or isotopy of the $\as$ curves preserves the relative Alexander multi-grading. The transition maps associated to an isotopy or handleslide of the $\as$ curves can be computed by counting holomorphic triangles. Furthermore, an arbitrary isotopy or handleslide of the $\as$ curves can be computed as a sequence of holomorphic triangle maps, such that in each Heegaard triple $(\Sigma, \as',\as,\bs,\ws,\zs)$, the sets $\as'$ and $\as$ satisfy $|\alpha_i'\cap \alpha_j|=2\delta_{ij}$, and there is a unique intersection point $\Theta^+_{\a'\a}\in \bT_{\alpha'}\cap \bT_{\alpha}$ which is the highest $\gr_{\ws}$ and $\gr_{\zs}$ graded intersection point.

Suppose that $\ve{x},\ve{x}'\in \bT_{\alpha}\cap \bT_{\beta}$ are two intersection points with $\frs_{\ws}(\xs)=\frs_{\ws}(\xs')=\frs$, and $\psi\in \pi_2(\Theta_{\a'\a}^+,\ve{x},\ve{y})$ and $\psi'\in \pi_2(\Theta_{\a'\a}^+,\ve{x}',\ve{y}')$ are two homology classes of triangles which are counted by the map $\Phi^{\as\to \as'}_{\bs}$. The triangle classes $\psi$ and $\psi'$ both represent the restriction of $\frs$ to $X_{\a'\a\b}$, under the  inclusion $X_{\a'\a\b}\hookrightarrow [0,1]\times Y$ from Lemma~\ref{lem:uniqueembeddingintoW}. By \cite{OSDisks}*{Proposition~8.5}, it follows  that there are homology classes 
\[
\phi_{\a'\a}\in \pi_2(\Theta_{\a'\a}^+,\Theta_{\a'\a}^+),\qquad \phi_{\a\b}\in \pi_2(\ve{x}',\ve{x})\qquad \text{and} \qquad \phi_{\a'\b}\in \pi_2(\ve{y},\ve{y}')
\] 
such that
\[
\psi'=\psi+\phi_{\a'\a}+\phi_{\a\b}+\phi_{\a'\b}.
\]
 Hence
\begin{equation}
\begin{split}
A( U_{\ve{w}}^{n_{\ve{w}}(\psi)} V_{\ve{z}}^{n_{\ve{z}}(\psi)}\cdot \ve{y},U_{\ve{w}}^{n_{\ve{w}}(\psi')} V_{\ve{z}}^{n_{\ve{z}}(\psi')}\cdot \ve{y}')_j&=(n_{\ve{z}}-n_{\ve{w}})_j(\phi_{\a'\b}+\psi-\psi')\\
&=(n_{\ve{z}}-n_{\ve{w}})_j(-\phi_{\a\b})-(n_{\ve{z}}-n_{\ve{w}})_j(\phi_{\a'\a})\\
&=A(\ve{x},\ve{x}')_j,
\end{split}
\label{eq:invarianceofrelativeAlexander}
\end{equation} since $(n_{\ve{z}}-n_{\ve{w}})_j(\phi_{\a'\a})=A(\Theta_{\a'\a}^+, \Theta^+_{\a'\a})=0$.

Invariance from the Alexander gradings under moves of the $\bs$ curves is handled similarly.  Invariance of the relative Alexander grading under index 1/2 stabilization can be handled as follows. Suppose that $\cH'=(\Sigma\# \bT^2,\as\cup \{\alpha_0\},\bs\cup \{\beta_0\},\ws,\zs)$ is obtained from $\cH=(\Sigma,\as,\bs,\ws,\zs)$ by a stabilization,  and let $c\in \alpha_0\cap \beta_0$ denote the new intersection point. If $\xs\in \bT_{\alpha}\cap \bT_{\beta}$, the transition map $\Phi_{(\cH,J_s)\to (\cH',J_s')}$ sends $\xs$ to $\xs\times \{c\}$, for an appropriately stretched $J_s'$. If $\phi\in \pi_2(\xs,\xs')$ is a class of disks on $\cH$, then we can construct a class $\phi'\in \pi_2(\xs\times \{c\}, \xs'\times \{c\})$ which agrees with $\phi$ outside of $\bT^2$, and has constant multiplicity in $\bT^2$. We note that
\begin{equation}
A(\xs\times \{c\}, \xs'\times \{c\})_j=(n_{\zs}-n_{\ws})_j(\phi')=(n_{\zs}-n_{\ws})_j(\phi)=A(\xs,\xs')_j. \label{eq:Alexandergradingpreservedbystabilization}
\end{equation}
Invariance of the relative Alexander grading under isotopies of the Heegaard surface inside of $Y$ is a tautology.

Invariance of the relative Alexander multi-graded chain homotopy type from the choice of almost complex structure is proven similarly to invariance under moves of the $\as$ and $\bs$ curves, since the transition map 
 \[
 \Phi_{J_s\to J_s'}\colon \cCFL^\infty_{J_s}(\cH,\frs)\to \cCFL^\infty_{J_s'}(\cH,\frs)
 \]
 can be computed by counting index 0 holomorphic disks in $\Sym^n(\Sigma)$, for path of paths of almost complex structures on $\Sym^n(\Sigma)$, connecting $J_s$ and $J_s'$.

 Invariance of the Maslov grading from moves of the $\as$ and $\bs$ curves follows by adapting Equation~\eqref{eq:invarianceofrelativeAlexander} using the definition of the relative Maslov gradings, the fact that the transition maps count holomorphic triangles of Maslov index 0, and that the Maslov index is additive under juxtaposition of triangle and disk classes. Invariance of the Maslov index under index 1/2 stabilization follows by adapting Equation~\eqref{eq:Alexandergradingpreservedbystabilization}, noting that $\mu(\phi')=\mu(\phi)$, $n_{\ws}(\phi')=n_{\ws}(\phi)$ and $n_{\zs}(\phi')=n_{\zs}(\phi)$. 
\end{proof}

\subsection{Two simple examples}

We briefly give two examples, illustrating the gradings when some components of $L$ have non-trivial homology class.

\begin{example}\label{ex-s1xs2}Consider $Y=S^1\times S^2$, with $K=S^1\times \{pt\}$. A Heegaard diagram is shown in Figure~\ref{fig::17}. For $\frs$ the torsion $\Spin^c$ structure, we see that
\[
\cCFL^-(Y,\bL,\frs)\iso \left(\bF_2[U,V]\xrightarrow{1+V} \bF_2[U,V]\right).
\] 
The homology is $\bF_2[U,V]/(1+V)$. The grading $\gr_{\ve{w}}$ is defined, but $\gr_{\ve{z}}$  cannot be defined since $V$ acts by the identity on homology. Note that $\frs-\PD[K]$ is not torsion.
\end{example}

 \begin{figure}[ht!]
 \centering
\begingroup%
  \makeatletter%
  \providecommand\color[2][]{%
    \errmessage{(Inkscape) Color is used for the text in Inkscape, but the package 'color.sty' is not loaded}%
    \renewcommand\color[2][]{}%
  }%
  \providecommand\transparent[1]{%
    \errmessage{(Inkscape) Transparency is used (non-zero) for the text in Inkscape, but the package 'transparent.sty' is not loaded}%
    \renewcommand\transparent[1]{}%
  }%
  \providecommand\rotatebox[2]{#2}%
  \newcommand*\fsize{\dimexpr\f@size pt\relax}%
  \newcommand*\lineheight[1]{\fontsize{\fsize}{#1\fsize}\selectfont}%
  \ifx\svgwidth\undefined%
    \setlength{\unitlength}{195.95248326bp}%
    \ifx\svgscale\undefined%
      \relax%
    \else%
      \setlength{\unitlength}{\unitlength * \real{\svgscale}}%
    \fi%
  \else%
    \setlength{\unitlength}{\svgwidth}%
  \fi%
  \global\let\svgwidth\undefined%
  \global\let\svgscale\undefined%
  \makeatother%
  \begin{picture}(1,0.41778018)%
    \lineheight{1}%
    \setlength\tabcolsep{0pt}%
    \put(0,0){\includegraphics[width=\unitlength,page=1]{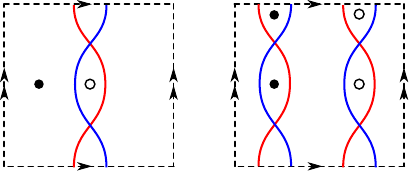}}%
    \put(0.10027537,0.23148711){\color[rgb]{0,0,0}\makebox(0,0)[lt]{\lineheight{1.25}\smash{\begin{tabular}[t]{l}$w$\end{tabular}}}}%
    \put(0.2182224,0.23148711){\color[rgb]{0,0,0}\makebox(0,0)[lt]{\lineheight{1.25}\smash{\begin{tabular}[t]{l}$z$\end{tabular}}}}%
  \end{picture}%
\endgroup%

 \caption{\textbf{Two diagrams for $S^1\times S^2$, for knots or links with components which are not null-homologous.} On the left is the knot $K=S^1\times \{p\}$ and on the right is $K=S^1\times \{p_1,p_2\}$. All intersection points are mapped to the torsion $\Spin^c$ structure by $\frs_{\ve{w}}$.\label{fig::17}}
 \end{figure}

\begin{example}Consider  $Y=S^1\times S^2$ and $L=S^1\times \{p_1,p_2\}$. We place two basepoints on each  component of $L$, and orient the two components to intersect the sphere $\{pt\}\times S^2$ with opposite sign. A diagram is shown on the right side of Figure~\ref{fig::17}. The intersection points represent the torsion $\Spin^c$ structure with respect to $\frs_{\ve{w}}$. It  is easy to see that the collapsed Alexander grading can be defined, but a two component Alexander grading cannot be defined.
\end{example}

\subsection{Absolute gradings on $\cCFL^\infty$ for unlinks in $(S^1\times S^2)^{\# k}$}
\label{sec:absolutegradingsunknots}

As a step toward describing the absolute gradings in general, we fix the absolute gradings for unlinks in $(S^1\times S^2)^{\# k}$.

\begin{lem}\label{lem:admittopdegreegenerator} Suppose $\bU$ is a multi-based unlink in $(S^1\times S^2)^{\# k}$, with an arbitrary configuration of basepoints. The $\bF_2$-module $\Hat{\HFL}((S^1\times S^2)^{\# k}, \bU,\frs_0)$ has rank $2^{|\ve{w}|+k-1}$. Furthermore, $\Hat{\HFL}((S^1\times S^2)^{\# k}, \bU,\frs_0)$ has a top degree generator with respect to each of the gradings $\gr_{\ve{w}}$ and $\gr_{\ve{z}}$, for which we write $\Theta^{\ve{w}}$ or $\Theta^{\ve{z}}$.
\end{lem}

\begin{proof}By Lemma~\ref{lem:changeofdiagramsrelativegradings}, the relatively graded isomorphism type of the group $\Hat{\HFL}((S^1\times S^2)^{\# k}, \bU,\frs_0)$ is an invariant, so we need only check the claim for a particular diagram.

We start with the case that each component of $\bU$ contains exactly two basepoints. In this case, we can pick a diagram $\cH=(\Sigma,\as,\bs,\ws,\zs)$ where the $\as$ curves are small Hamiltonian translates of the $\bs$ curves, and the $\ws$ and $\zs$ basepoints come in pairs of adjacent basepoints on $\Sigma\setminus (\as\cup \bs)$. For such a diagram, the holomorphic disks which do not pass over any of the basepoints come in canceling pairs, and there is an isomorphism of groups
\[
\Hat{\CFL}(\cH,\frs_0)=\Hat{\HFL}(\cH,\frs_0):= \bigotimes^{k+|\bU|-1}_{i=1} V,
\]
where $V$ is a 2-dimensional vector space over $\bF_2$ with two generators which have relative $(\gr_{\ws},\gr_{\zs})$-bigrading which differ by $(1,1)$. Hence $\Hat{\HFL}(\cH,\frs_0)$ has an element $\Theta^+=\Theta^{\ws}=\Theta^{\zs}$, which is maximally graded with respect to both $\gr_{\ws}$ and $\gr_{\zs}$. This verifies the claim when $\bU$ has exactly two basepoints.

To verify the claim when $\bU$ has more than two basepoints, we proceed by induction. Supposing the claim holds for an unlink $\bU$ with some configuration of basepoints, we will show that it also holds for the link $\bU'$ obtained by adding two extra basepoints to $\bU$. Adding two basepoints can be achieved by the \emph{quasi-stabilization} operation \cite{MOIntSurg}*{Section~6}. We will consider the quasi-stabilization operation in more detail later; See Figure~\ref{fig::52} for a Heegaard diagrammatic description. If $\cH$ is a diagram for $\bU$, and $\cH'$ is a quasi-stabilization, then there is a relatively graded isomorphism
\begin{equation}
\Hat{\CFL}(\cH',\frs_0)\iso \Hat{\CFL}(\cH,\frs)\otimes_{\bF_2} V',\label{eq:quasi-stabilizedisomorphism}
\end{equation}
where $V'$ is 2-dimensional vector space whose generators have $(\gr_{\ws},\gr_{\zs})$ bi-grading $(\tfrac{1}{2}, -\tfrac{1}{2})$ and $(-\tfrac{1}{2}, \tfrac{1}{2})$. By \cite{ZemQuasi}*{Proposition~5.3}, for an appropriate choice of almost complex structures, the isomorphism in Equation~\eqref{eq:quasi-stabilizedisomorphism} is an isomorphism of chain complexes (viewing $V'$ as having vanishing differential). In particular, if $\Hat{\HFL}(\cH,\frs_0)$ has distinguished generators $\Theta^{\ws}$ and $\Theta^{\zs}$, then so does $\Hat{\HFL}(\cH',\frs_0)$, completing the proof.
\end{proof}

Lemma~\ref{lem:admittopdegreegenerator} allows us to declare absolute lifts of the Maslov gradings on $\Hat{\CFL}((S^1\times S^2)^{\# k}, \bU,\frs)$ by setting
\begin{equation}
\tilde{\gr}_{\ve{w}}(\Theta^{\ve{w}}):=\tilde{\gr}_{\ve{z}}(\Theta^{\ve{z}}):=\frac{1}{2}(k+|\ve{w}|-1)=\frac{1}{2}(k+|\ve{z}|-1).\label{eq:grwtop}
\end{equation}
The declaration in Equation~\eqref{eq:grwtop} specifies the gradings $\tilde{\gr}_{\ws}$ and $\tilde{\gr}_{\zs}$ uniquely on all intersection points representing $\frs_0$. We extend these to $\cCFL^\infty$ by declaring the $U_{\ws}$ variables to have $(\tilde{\gr}_{\ws},\tilde{\gr}_{\zs})$ bi-grading $(-2,0)$, and declaring the $V_{\zs}$ variables to have $(\tilde{\gr}_{\ws},\tilde{\gr}_{\zs})$ bi-grading $(0,-2)$.

In a similar manner, we can declare an absolute lift of the Alexander multi-grading. We index the link $\bU=(U,\ws,\zs)$ using  its set of components, i.e., we set $\bJ_0=C(U)$ and let $J_0\colon U\to \bJ_0$ be the natural map. If $K$ is a component of $U$, we write $n_K$ for one half the total number of basepoints on $K$ (i.e. the number of $\ws$ basepoints on $K$, or the number of $\zs$ basepoints). We set
\begin{equation}
\tilde{A}(\Theta^{\ws})_K:=\frac{1}{2}(n_K-1).\label{eq:AlexandergradinThetaw}
\end{equation}
It is straightforward to see that Equation~\eqref{eq:AlexandergradinThetaw} implies
\[
\tilde{A}(\Theta^{\zs})_K=-\frac{1}{2}(n_K-1).
\]
Writing $\tilde{A}$ for the collapsed Alexander grading, it is straightforward to see that
\begin{equation}
\tilde{A}=\frac{1}{2}(\tilde{\gr}_{\ve{w}}-\tilde{\gr}_{\ve{z}}),\label{eq:collapsedAlexanderMaslov}
\end{equation}

For an arbitrary indexing $J\colon U\to \bJ$, we define the Alexander grading over $\bJ$ by collapsing the Alexander grading which is indexed over $\bJ_0$, i.e., for $j\in \bJ$ we define
\[
\tilde{A}(\xs)_j:=\sum_{K\subset J^{-1}(j)}\tilde{A}(\xs)_K.
\]

\subsection{Transitive systems of gradings}
\label{sec:transitivesystemsofgradings}
Since there are many different diagrams for a pair $(Y,\bL)$, we cannot specify the objects $\cCFL^\infty(Y,\bL^\sigma,\frs)$ as concrete chain complexes. Instead, using naturality (see Proposition~\ref{prop:naturality}) they are \emph{transitive systems} in the category of $\cR_{\bmP}$-equivariant, $\Z^{\bmP}$-filtered chain complexes. Hence, to define the notion of a grading on $\cCFL^\infty(Y,\bL^\sigma,\frs)$, we need an analogous notion of  a transitive system of gradings.

If $(\sigma,J)$ is a type-partitioned, indexed coloring of the link $\bL$ in $Y$, which is $\bJ$-null-homologous and $\cH$ is a diagram for $(Y,\bL)$, we define $\bA(\cH,\sigma,J,\frs)$ to be the set of absolute  lifts to $\Q^{\bJ}$ of the relative Alexander gradings on $\cCFL^\infty(\cH,\sigma,\frs)$, described in Section~\ref{sec:relativegradings}. The set $\bA(\cH,\sigma,J,\frs)$ is an affine space over $\Q^{\bJ}$. 

Similarly, if $\sigma$ is a type partitioned coloring of $\bL$, we define $\bG_{\ws}(\cH,\sigma,\frs)$ and $\bG_{\zs}(\cH,\sigma,\frs)$ to be the set of absolute lifts to $\Q$ of the relative gradings $\gr_{\ws}$ and $\gr_{\zs}$. The sets $\bG_{\ws}(\cH,\sigma,\frs)$ and $\bG_{\zs}(\cH,\sigma,\frs)$ are affine spaces over $\Q^{|Y|}$.

In this section, we prove the following naturality result for gradings:

\begin{prop}\label{prop:changeofdiagramsmapsgradings} Suppose that $\bL$ is a multi-based link in $Y$, with a type-partitioned, indexed coloring $(\sigma,J)$, and $L$ is $\bJ$-null-homologous. If $\cH$ and $\cH'$ are two admissible diagrams, then there is a well defined transition map
$F_{\cH\to \cH'}\colon \bA(\cH,\sigma,J,\frs)\to \bA(\cH',\sigma,J,\frs).$ Furthermore,  the following are satisfied:
\begin{enumerate}
\item $F_{\cH\to \cH}=\id$.
\item $F_{\cH'\to \cH''}\circ F_{\cH\to \cH'}=F_{\cH\to \cH''}$.
\item\label{eq:transitionmapsinteractgrading} $(F_{\cH\to \cH'}(A))(\Phi_{\cH\to \cH'}(\xs))=A(\xs).$
\end{enumerate}

Similarly, if $\sigma$ is a type-partitioned coloring of $\bL$, and $c_1(\frs)$ (resp. $c_1(\frs-\PD[L])$) is torsion, then there are well defined transition maps $F_{\cH\to \cH'}\colon \bG_{\ws}(\cH,\sigma,\frs)\to \bG_{\ws}(\cH',\sigma,\frs)$ (resp.  $F_{\cH\to \cH'}\colon \bG_{\zs}(\cH,\sigma,\frs)\to \bG_{\zs}(\cH',\sigma,\frs)$), satisfying the same axioms.
\end{prop}

Note that Proposition~\ref{prop:changeofdiagramsmapsgradings} implies that we can view the sets $\bA(\cH,\sigma,J,\frs)$ as fitting into a transitive system indexed by the set of admissible diagrams of $(Y,\bL)$. We define $\bA(Y,\bL^{(\sigma,J)},\frs)$ as the transitive limit, i.e., the set of tuples 
\[
(A_{\cH})_{\cH\in \cD(Y,\bL,\frs)}\in \prod_{\cH\in \cD(Y,\bL,\frs)} A(\cH,\sigma,J,\frs)
\]
 satisfying $F_{\cH\to \cH'}(A_{\cH})=A_{\cH'}$ for all $\cH$ and $\cH'$, where $\cD(Y,\bL,\frs)$ denotes the set of admissible diagrams. We define the transitive limits $\bG_{\ws}(Y,\bL^\sigma,\frs)$ and $\bG_{\zs}(Y,\bL^\sigma,\frs)$ similarly.

The rest of the section is devoted to proving Proposition~\ref{prop:changeofdiagramsmapsgradings}. First, we define the maps $F_{\cH\to \cH'}$ when $\cH'$ and $\cH$ differ by an elementary Heegaard move, and then we verify that there is no monodromy around loops in the space of Heegaard diagrams, adapting the strategy of \cite{JTNaturality} for the transition maps on the Heegaard Floer complexes.

Suppose that $\cH'=(\Sigma,\as',\bs,\ws,\zs)$ and $\cH=(\Sigma,\as,\bs,\ws,\zs)$ are related by a handleslide or isotopy of the $\as$ curves, and the triple $(\Sigma,\as',\as,\bs,\ws,\zs)$ is admissible. In this case, we write $F_{\cH\to \cH'}=F^{\as\to \as'}_{\bs}$ for the map
\[
F_{\bs}^{\as\to \as'}\colon \bA(\cH,\sigma,J,\frs)\to \bA(\cH',\sigma,J,\frs),
\]
defined as follows. First note that the indexing $J$ of $\bL$ induces an indexing of the unlink $L_{\a'\a}\subset Y_{\a'\a}$. If $A\in \bA(\cH,\sigma,J,\frs)$, we define
\begin{equation}
F_{\bs}^{\as\to \as'}(A)(\ve{y})_j:= A(\ve{x})_j+\tilde{A}(\Theta)_j+(n_{\ve{w}}-n_{\ve{z}})_j(\psi),\label{eq:changeofdiagramsmapgradingsasmove}
\end{equation}
for any homology class of triangles $\psi\in \pi_2(\Theta,\ve{x},\ve{y})$ with $\frs_{\ve{w}}(\psi)=\iota^*(\frs)$ where  $\iota_*\colon X_{\a'\a\b}\to [0,1]\times Y$ denotes the inclusion from Lemma~\ref{lem:uniqueembeddingintoW}. The grading $F_{\bs}^{\as\to \as'}(A)$ is independent of the choice of the intersection points $\ve{x}$ and $\Theta$, since splicing in homology classes of disks into the ends of $\psi$ does not affect Equation~\eqref{eq:changeofdiagramsmapgradingsasmove}. Similarly, the grading is independent of the homology class $\psi$, since by \cite{OSDisks}*{Proposition~8.5} any other homology class also representing $\frs$ can be obtained from $\psi$ by splicing in homology classes of disks on the diagrams $(\Sigma, \as',\as)$, $(\Sigma, \as,\bs)$, and $(\Sigma,\as',\bs)$.

Analogously, if $\bs'$ differs from $\bs$ by a sequence of handleslides or isotopies, and $(\Sigma, \as,\bs,\bs',\ws,\zs)$ is admissible, we define the transition map $F_{\bs\to \bs'}^{\as}$ via the formula
\[
F_{\bs\to \bs'}^{\as}(A)(\ys)_j:=A(\xs)_j+\tilde{A}(\Theta)_{j}+(n_{\ws}-n_{\zs})_j(\psi),
\]
for a choice of $\psi\in \pi_2(\xs,\Theta,\ys)$.

Similarly, if $\cH'=(\Sigma,\as',\bs,\ws,\zs)$ and $\cH=(\Sigma,\as,\bs,\ws,\zs)$ are related by a handleslide or isotopy of the $\as$ curves, then we define maps 
\[
F_{\bs}^{\as\to \as'}\colon \bG_{\ws}(\cH,\sigma,\frs)\to \bG_{\ws}(\cH',\sigma,\frs)\qquad \text{and} \qquad F_{\bs}^{\as\to \as'}\colon \bG_{\zs}(\cH,\sigma,\frs)\to \bG_{\zs}(\cH',\sigma,\frs),
\] whenever $(\Sigma, \as',\as,\bs,\ws,\zs)$ is admissible, as follows. If $g\in \bG_{\ws}(\cH,\sigma,\frs)$ and $\psi\in \pi_2(\Theta,\xs,\ys)$ we define
\begin{equation}
F_{\ve{\bs}}^{\as\to \as'}(g)(\ys):=g(\ve{x})+\tilde{\gr}_{\ws}(\Theta)-\frac{1}{2}(k+|\ve{w}|-1)-\mu(\psi)+2n_{\ws}(\psi),\label{eq:betamoveMaslovgrading}
\end{equation}
for any choice of $\xs\in \bT_{\alpha}\cap \bT_{\beta}$ and homology class of triangles $\psi\in \pi_2(\Theta,\xs,\ys)$. In Equation~\eqref{eq:betamoveMaslovgrading}, $\tilde{\gr}_{\ws}$ denotes the absolute grading described in Section~\ref{sec:absolutegradingsunknots}.

 By replacing each instance of $\ws$ with $\zs$, we obtain the analogous map from $\bG_{\zs}(\cH,\sigma,\frs)$ to $\bG_{\zs}(\cH',\sigma,\frs)$.

By adapting the argument given above for the Alexander grading, it is straightforward to see that the grading defined in Equation~\eqref{eq:betamoveMaslovgrading} does not depend on the choice of $\psi$, $\xs$ or $\Theta$. 

Next, we suppose that $\cH'$ is a stabilization of $\cH$. If $\xs$ is an intersection point of $\cH$, then we let $\xs\times \{c\}$ denote the product of $\xs$ with the intersection point of the new $\as$ and $\bs$ curves on $\cH'$. We define
\[
F_{\cH\to \cH'}(A)(\xs\times \{c\})_j:=A(\xs)_j.
\]
We define the destabilization map $F_{\cH'\to \cH}$ as the inverse of $F_{\cH\to \cH'}$. An analogous transition map for the $\gr_{\ws}$ and $\gr_{\zs}$ gradings is defined similarly.

Finally, if $\cH'=(\Sigma',\as',\bs',\ws,\zs)$ is obtained by an isotopy of the diagram $\cH=(\Sigma,\as,\bs,\ws,\zs)$ within $Y$, we define the map $F_{\cH\to \cH'}$ via tautology.

\begin{lem}\label{lem:canonicalgradingsaretransitive} Suppose that $\bU$ is a multi-based link in $(S^1\times S^2)^{\# k}$, with any configuration of basepoints, and $\cH$ and $\cH'$ are two diagrams for $((S^1\times S^2)^{\# k}, \bU)$. If $\tilde{A}_{\cH_1}\in \bA(\cH_1,\frs_0)$ and $\tilde{A}_{\cH_2}\in \bA(\cH_2,\frs_2)$ denote the two gradings defined in Section~\ref{sec:absolutegradingsunknots}, and $F_{\cH_1\to \cH_2}$ denotes the transition map (defined using any sequence of elementary Heegaard moves), then
\[
F_{\cH_1\to \cH_2}(\tilde{A}_{\cH_1})=\tilde{A}_{\cH_2}.
\] The same statement holds for the Maslov gradings.
\end{lem}
\begin{proof}We recall the absolute Alexander grading for unlinks in $(S^1\times S^2)^{\# k}$ in Section~\ref{sec:absolutegradingsunknots}  was defined by fixing the grading of the distinguished elements $\Theta^{\ws}$ and $\Theta^{\zs}$. By Lemma~\ref{lem:changeofdiagramsrelativegradings}, the maps $\Phi_{\cH_1\to \cH_2}$ preserve the relatively graded chain homotopy type of $\Hat{\CFL}$, and hence must satisfy
\begin{equation}
\Phi_{\cH_1\to \cH_2}(\Theta^{\ws}_{\cH_1})=\Theta^{\ws}_{\cH_2}, \label{eq:PhiHH'preserveThetaws}
\end{equation}
and similarly for $\Theta^{\zs}$. Next, we note that Part~\eqref{eq:transitionmapsinteractgrading} of Proposition~\ref{prop:changeofdiagramsmapsgradings} is a tautology, and can be verified for each individual Heegaard move. Hence it follows that
\begin{align*}
F_{\cH_1\to \cH_2}(\tilde{A}_{\cH_1})(\Theta^{\ws}_{\cH_2})_j&=F_{\cH_1\to \cH_2}(\tilde{A}_{\cH_1})(\Phi_{\cH_1\to \cH_2}(\Theta^{\ws}_{\cH_1}))_j\\
&=\tilde{A}_{\cH_1}(\Theta^{\ws}_{\cH_1})_j.
\end{align*}
Hence it follows that $F_{\cH_1\to \cH_2}(\tilde{A}_{\cH_1})=\tilde{A}_{\cH_2}$, completing the proof.
\end{proof}

We have the following:
\begin{lem}\label{lem:commutativesquares} Suppose that $\bL$ is a multi-based link in $Y$. The following statements hold for the transition maps on the sets of Alexander, $\gr_{\ws}$  and $\gr_{\zs}$ gradings, whenever they are defined:
\begin{enumerate} 
\item\label{claim:square1} Suppose that $(\Sigma, \as'',\as',\as,\bs,\ws,\zs)$ is an admissible quadruple and $\as,$ $\as'$ and $\as''$ are all related to each other by a sequence of handleslides and isotopies.  Then 
	\[
	F_{\bs}^{\as\to \as''}=F_{\bs}^{\as'\to \as''}\circ F^{\as\to \as'}_{\bs}. 
	\]
	A similar statement holds for an admissible quadruple $(\Sigma,\as,\bs,\bs',\bs'',\ws,\zs)$, where $\bs,$ $\bs'$ and $\bs''$ are related by a sequence of handleslides and isotopies.
\item\label{claim:square2} Suppose $(\Sigma, \as',\as,\bs,\ws,\zs)$ is an admissible Heegaard triple and $\as'$ is related to $\as$ by a sequence of handleslides and isotopies, and suppose $(\Sigma, \hat{\as}',\hat{\as},\hat{\bs},\ws,\zs)$ is a triple obtained by stabilizing $(\Sigma, \ve{\alpha}',\as,\bs,\ws,\zs)$ at a point in $\Sigma\setminus (\as'\cup \as\cup \bs)$. Writing $F_S$  for the transition map associated to stabilization, we have
	\[
	F_S\circ F^{\as\to \as'}_{\bs}=F^{\hat{\as}\to \hat{\as}'}_{\hat{\bs}}\circ F_S.
	\]	
		\item Suppose $(\Sigma, \as,\bs,\ve{w},\ve{z})$ is an admissible diagram for $(Y,\bL)$ and $\phi\colon (\Sigma,\ve{w}\cup \ve{z})\to (\Sigma,\ve{w}\cup \ve{z})$ is a diffeomorphism which is isotopic to $\id_\Sigma$, relative to $\ve{w}\cup \ve{z}$. Then
	\[
	F^{\as\to \phi(\as)}_{\phi(\bs)}\circ F_{\bs\to \phi(\bs)}^{\as}=\phi_*.
	\]	
	\item Suppose that $(\Sigma, \as',\as,\bs,\bs',\ve{w},\ve{z})$ is an admissible quadruple, such that $\as'$ is related to $\as$ by a sequence of handleslides and isotopies, and $\bs'$ is related to $\bs$ by a sequence of handleslides or isotopies. Then
	\[
	F^{\as\to \as'}_{\bs'}\circ F^{\as}_{\bs\to \bs'}=F^{\as'}_{\bs\to \bs'}\circ F^{\as\to \as'}_{\bs}.
	\]
	\item If $S$ and $S'$ are two disjoint stabilizations, then
	\[
	F_S\circ F_{S'}=F_{S'}\circ F_S.
	\]
	\item\label{claim:square6} If $S$ is a stabilization, and $\phi\colon (Y,L)\to (Y,L)$ is a diffeomorphism fixing $\ws\cup \zs$, then
	\[
	\phi_*\circ F_{S}=F_{\phi(S)}\circ \phi_*.	
	\]
\end{enumerate}
\end{lem}

	\begin{proof}We consider Claim~\eqref{claim:square1}, focusing on Alexander gradings. Suppose that $A\in \bA(\Sigma,\as,\bs,\sigma,J,\frs)$, and $\ys_{\a''\b}\in \bT_{\alpha''}\cap \bT_{\beta}$. Pick intersection points $\Theta_{\a''\a'}\in \bT_{\alpha''}\cap \bT_{\alpha'}$, $\Theta_{\a'\a}\in \bT_{\alpha'}\cap \bT_{\alpha}$, $\Theta_{\a''\a}\in \bT_{\alpha''}\cap \bT_{\alpha}$, $\ve{x}_{\a\b}\in \bT_{\alpha}\cap \bT_{\beta}$ and $\ve{x}_{\a'\b}\in \bT_{\alpha'}\cap \bT_{\beta}$, such that $\Theta_{\a''\a},$ $\Theta_{\a'\a}$ and $\Theta_{\a''\a'}$ represent the torsion $\Spin^c$ structures. Pick homology classes of triangles $\psi_{\a''\a'\a}\in \pi_2(\Theta_{\a''\a'}, \Theta_{\a'\a},\Theta_{\a''\a}),$ $\psi_{\a''\a\b}\in \pi_2(\Theta_{\a''\a}, \xs_{\a\b}, \ys_{\a''\b})$, and $\psi_{\a'\a\b}\in \pi_2(\Theta_{\a'\a}, \xs_{\a\b}, \xs_{\a'\b}),$ and $\psi_{\a''\a'\b}\in \pi_2(\Theta_{\a''\a'}, \xs_{\a'\b}, \ys_{\a''\b})$, such that
	\begin{equation}
	\psi_{\a''\a\b}+\psi_{\a''\a'\a}=\psi_{\a'\a\b}+\psi_{\a''\a'\b}.\label{eq:trianglesaddup}
	\end{equation} By definition,
	\begin{equation}
	F^{\as\to \as''}_{\bs}(A)(\ys_{\a''\b})_j=A(\ve{x}_{\a\b})_j+\tilde{A}(\Theta_{\a''\a})_j+(n_{\ve{w}}-n_{\ve{z}})_j(\psi_{\a''\a\b})\label{eq:singletransitionamove}
	\end{equation}
	and
	\begin{equation}
	\begin{split}
	&(F_{\bs}^{\as'\to \as''}\circ F^{\as\to \as'}_{\bs})(A)(\ys_{\a''\b})_j\\
	=&A(\ve{x}_{\a\b})_j+\tilde{A}(\Theta_{\a'\a})_j+\tilde{A}(\Theta_{\a''\a'})_j+ (n_{\ve{w}}-n_{\ve{z}})_j(\psi_{\a''\a'\b}+\psi_{\a'\a\b}).
\end{split}	
	\label{eq:twotransitionamove}
	\end{equation}
	Equation~\eqref{eq:twotransitionamove} minus Equation~\eqref{eq:singletransitionamove} is
	\begin{equation}
	-\tilde{A}(\Theta_{\a''\a})_j+\tilde{A}(\Theta_{\a'\a})_j+\tilde{A}(\Theta_{\a''\a'})_j+(n_{\ve{w}}-n_{\ve{z}})_j(\psi_{\a''\a'\a}).\label{eq:differenceofFAs}
	\end{equation}

Note that the expression in Equation~\eqref{eq:differenceofFAs} is equal to 
\begin{equation}
F_{\as}^{\as'\to \as''}(\tilde{A})(\Theta_{\a''\a})_j-\tilde{A}(\Theta_{\a''\a})_j.\label{eq:pushforwardgradingthesameascanonical}
\end{equation}

By Lemma~\ref{lem:canonicalgradingsaretransitive}, the grading $\tilde{A}$ on unknots in $(S^1\times S^2)^{\# k}$ is a transitive grading, and hence  Equation~\eqref{eq:pushforwardgradingthesameascanonical} must vanish. It follows that  Equation~\eqref{eq:differenceofFAs} vanishes as well, establishing Claim~\eqref{claim:square1} for the Alexander multi-grading.

The proof of Claim~\eqref{claim:square1} for Maslov gradings, as well as the proofs of Claims~\eqref{claim:square2}--\eqref{claim:square6}, are straightforward modifications of the above argument.
	\end{proof}

There is an important class of loops in the set of Heegaard diagrams, called \emph{simple handleswaps}. We refer the reader to \cite{JTNaturality}*{Definition~2.31} for a precise description. We state the following version of handleswap invariance for gradings (compare \cite{JTNaturality}*{Proposition~9.25}):

\begin{lem}\label{lem:handleswapinvariance}If $\cH_1\xrightarrow{e} \cH_2\xrightarrow{f} \cH_3\xrightarrow{g} \cH_1$ is a simple handleswap, then
	\[
	F_g\circ F_f\circ F_e=\id,
	\]
	 as maps on the set of Alexander gradings or Maslov gradings on $\cH_1$.
	\end{lem}
\begin{proof} The maps $F_e$ and $F_f$ are maps induced by handleslides of the $\as$ and $\bs$ curves, respectively. The map $F_g$ is induced by a diffeomorphism. In our context, the maps $F_e$ and $F_g$ are computed  by picking any homology class of triangles. Hence the argument can be proven by adapting the standard proof of handleswap invariance \cite{JTNaturality}*{Proposition~9.25}, noting that in our context, we just need to check the claim for any two homology classes of triangles (one to compute $F_e$ and one to compute $F_f$). 
	\end{proof}
	
We  now prove that the maps $F_{\cH\to \cH'}$ give each of  $\bA(\cH,\sigma,J,\frs)$, $\bG_{\ws}(\cH,\sigma,\frs)$ and $\bG_{\zs}(\cH,\sigma,\frs)$ the structure of a transitive system.

	\begin{proof}[Proof of Proposition~\ref{prop:changeofdiagramsmapsgradings}] It is sufficient to verify that the sets of gradings together with the transition maps we've defined satisfy the axioms of a \emph{strong Heegaard invariant} \cite{JTNaturality}*{Definition 2.33}. Lemmas~\ref{lem:commutativesquares} and \ref{lem:handleswapinvariance} imply that the sets of gradings, together with the transition maps  we previously associated to elementary Heegaard moves, satisfy the axioms of \cite{JTNaturality}*{Definition 2.33}. Hence, by \cite{JTNaturality}*{Theorem~2.38},  the map $F_{\cH\to \cH'}$, defined using elementary moves between  Heegaard diagrams, does not depend on the choice of elementary Heegaard moves between $\cH$ and $\cH'$.
	
	Finally, it remains to show Claim \eqref{eq:transitionmapsinteractgrading}, i.e., that
	\begin{equation}
(F_{\cH\to \cH'}(A))(\Phi_{\cH\to \cH'}(\xs))=A(\xs).\label{eq:changeofdiagramsinteractcomplgrading}
	\end{equation}
 Equation~\eqref{eq:changeofdiagramsinteractcomplgrading} can be checked for each elementary Heegaard move, and is  a tautology from the definitions.
		\end{proof}

\subsection{Definition of the absolute gradings}
\label{sec:defabsolutegradings}

In this section, we give the definition of the absolute gradings. In Section~\ref{sec:invarianceofgradings}, we prove that these gradings are well defined.

Pick a parametrized Kirby diagram $\bP=(\phi_0,\lambda,\bS_1,f)$ for $(Y,L)$. The parametrized Kirby diagram specifies an unlink $U$ in $S^3$, as well as a framed link $\bS_1\subset S^3\setminus U$, and a diffeomorphism $\hat{f}$ between $S^3(\bS_1)$ and $Y$, which maps $U$ to $L$. Abusing notation slightly, let us write $\ws$ and $\zs$ for the basepoints on $U$ obtained by pulling back $\ws\cup \zs\subset L$ under $\hat{f}$. Let $\bU$ denote the multi-based link
\[
\bU:=(U,\ws,\zs).
\]
 We pick a $\beta$-bouquet $\cB$ for $\bS_1$, as well as a Heegaard triple $\cT=(\Sigma,\as,\bs,\bs',\ws,\zs)$ subordinate to $\cB$.

Note that by definition $(\Sigma,\as,\bs,\ws,\zs)$ is a diagram for $(S^3,\bU)$, and $(\Sigma,\as,\bs',\ws,\zs)$ is a diagram for $(Y,\bL)$. The diagram $(\Sigma,\bs,\bs',\ws,\zs)$ is a diagram for an unlink $\bU_{\b\b'}$ in $(S^1\times S^2)^{\#k}$, with exactly two basepoints per link component.

Suppose that $(\sigma,J)$ is a type-partitioned, indexed coloring of $\bL$, with index set $\bJ$, and suppose that $\bL$ is $\bJ$-null-homologous. Let $S=(S_j)_{j\in \bJ}$ be a generalized $\bJ$-Seifert surface of $\bL$. Note that $(\sigma,J)$  induces a type-partitioned, indexed coloring of both $\bU$ and $\bU_{\b\b'}$, for which we will also write $(\sigma,J)$.

Write $W(S^3,\bS_1)$ for the 2-handle cobordism from $S^3$ to $Y$ obtained by attaching 2-handles to $[0,1]\times S^3$ along $\{1\}\times \bS_1$. Let $\Sigma_j$ denote the surface $[0,1]\times U_j\subset W(S^3,\bS_1)$, and let $\hat{\Sigma}_j$ denote the integral 2-cycle obtained by capping off $\Sigma_j$ with $\{1\}\times f^{-1}(S_j)\subset \{1\}\times S^3(\bS_1)$, as well as  an arbitrary Seifert surface of $U_j$ in $\{0\}\times S^3$. Let $[\hat{\Sigma}]$ denote the integral 2-cycle $[\hat{\Sigma}]:=\sum_{j\in \bJ} [\hat{\Sigma}_j].$

If $\ys\in \bT_{\alpha}\cap \bT_{\beta'}$ is an intersection point with $\frs_{\ws}(\ys)=\frs$ and $j\in \bJ$,  we  pick intersection points $\xs\in \bT_{\alpha}\cap \bT_{\beta}$ and $\Theta\in \bT_{\beta}\cap \bT_{\beta'}$, as well as a homology class of triangles $\psi\in \pi_2(\ve{x},\Theta,\ve{y})$, and set
\begin{equation}
A_{S}(\ve{y})_j:=\tilde{A}(\ve{x})_j+\tilde{A}(\Theta)_j+(n_{\ve{w}}-n_{\ve{z}})_j(\psi)+\frac{\langle c_1(\frs_{\ve{w}}(\psi)),[\hat{\Sigma}_j]\rangle -[\hat{\Sigma}]\cdot [\hat{\Sigma}_j]}{2}.\label{eq:defabsgrading}
\end{equation}

As defined above, the grading $A_{S}$ is an element of $\bA(\cH,\sigma,J,\frs)$, where $\cH=(\Sigma,\as,\bs',\ws,\zs)$. However, there is a canonical isomorphism $\bA(\cH,\sigma,J,\frs)\to \bA(Y,\bL^{(\sigma,J)},\frs)$, so there is an induced transitive grading in $\bA(Y,\bL^{(\sigma,J)}, \frs)$. Well definedness of $A_{S}$ amounts to showing that the induced element $A_{S}$ in $\bA(Y,\bL^{(\sigma,J)},\frs)$ is independent from the choice of $\xs,$ $\Theta$, $\psi$,  $\cT$ and $\bP$. This will be addressed in Section~\ref{sec:invarianceofgradings}.

We define the absolute Maslov gradings in a similar fashion. Assuming that $\sigma$ is a type-partitioned coloring of $\bL$, and $c_1(\frs_{\ws}(\ys))$ is torsion, we define
\begin{equation}
\begin{split}
\gr_{\ws}(\ys):=&\tilde{\gr}_{\ws}(\xs)+\tilde{\gr}_{\ws}(\Theta)-\frac{1}{2}(k+|\ws|-1)-\mu(\psi)+2n_{\ws}(\psi)\\
&+\frac{c_1(\frs_{\ws}(\psi))^2-2\chi(W(S^3,\bS_1))-3\sigma(W(S^3,\bS_1))}{4}.
\end{split}
\label{def:absolutegrw}
\end{equation}
We define $\gr_{\zs}$ similarly, by replacing each instance of $n_{\ws},$  $\gr_{\ws}$ or $\frs_{\ws}$  in Equation~\eqref{def:absolutegrw} with $n_{\zs},$ $\gr_{\zs}$  or $\frs_{\zs}$, respectively.

Note that we can immediately prove Part~\eqref{thm1.1a'} of Theorem~\ref{thm:1}, that the Alexander grading $(A_S)_j$ takes values in $\Z+\tfrac{1}{2} \lk(L\setminus L_j, L_j)$:
\begin{proof}[Proof of  Part \eqref{thm1.1a'} of Theorem~\ref{thm:1}] Since $c_1(\frs_{\ws}(\psi))$ is a characteristic vector of $Q_W$, it follows that $\langle c_1(\frs_{\ws}(\psi)),[\hat{\Sigma}_j]\rangle -[\hat{\Sigma}_j]\cdot [\hat{\Sigma}_j]$ is an even integer. Hence, modulo $\Z$, the expression in Equation~\eqref{eq:defabsgrading} is equal to
\begin{equation}
\frac{1}{2} \left([\hat{\Sigma}\setminus \hat{\Sigma}_j]\cdot [\hat{\Sigma}_j]\right).\label{eq:intersectionoflinkcobordismpieces}
\end{equation}
Since the link cobordism surfaces $\Sigma_j$ and $\Sigma_i$ are disjoint whenever $i\neq j$, it is straightforward to see that the expression in Equation~\eqref{eq:intersectionoflinkcobordismpieces} is $\pm \tfrac{1}{2} \#((L\setminus L_j)\cap S_j)$, which is, by definition, $\tfrac{1}{2} \lk(L\setminus L_j,L_j)$.
\end{proof}

\subsection{Rationally null-homologous links and relative cyclic gradings}

There are several additional situations where one can define versions of the Alexander and Maslov gradings.

The first is when $L$ is \emph{rationally null-homologous}, i.e., $[L]=0\in H_1(Y;\Q)$. In this case, Equation~\eqref{eq:nz-nw=intersectionnumber} implies that the relative Alexander grading is still well defined. In fact, by picking a rational 2-chain $S$ with boundary $-L$, the techniques of this paper still give a well defined $\Q$-valued Alexander grading $A_S$. We will focus on integrally null-homologous links, for notational simplicity.

More generally, if $c_1(\frs)$ is non-torsion, then Equation~\eqref{eq:chernclassformula} implies that there is a  $\Z/\frd(\frs) \Z$ valued relative Maslov grading $\gr_{\ws}$, where
\[
\frd(\frs)=\gcd_{\xi\in H_2(Y;\Z)} \langle c_1(\frs),\xi\rangle.
\]
Similarly, there is a $\Z/\frd(\frs-\PD[L]) \Z$ valued relative Maslov grading $\gr_{\zs}$.

Similarly, by examining Equation~\eqref{eq:nz-nw=intersectionnumber} we see that when $L\neq 0\in H_1(Y;\Q)$ there is still a $\Z/\frd(L)\Z$ valued relative Alexander grading, where
\[
\frd(L):=\gcd_{\xi\in H_2(Y;\Z)} \langle \PD[L],\xi\rangle.
\]
The techniques of this paper do not give lifts the relative cyclic gradings to absolute cyclic gradings.

\section{Invariance of the absolute gradings}
\label{sec:invarianceofgradings}

In this section, we prove that the absolute lifts of the Alexander and Maslov gradings defined in Section~\ref{sec:defabsolutegradings} do not depend on the choices made in the construction.

\subsection{Invariance of the absolute Alexander grading}
\label{sec:invariancealexandergrading}
\begin{lem}\label{lem:indoftriangle}Suppose that $(Y,\bL)$ is a multi-based link, and $\bP$ is a parametrized Kirby decomposition with framed link $\bS_1$, $\cB^\beta$ is a $\beta$-bouquet for $\bS_1$, and $\cT=(\Sigma,\as,\bs,\bs',\ws,\zs)$ is a Heegaard triple subordinate to $\cB^\beta$. If $\ys\in \bT_{\alpha}\cap \bT_{\beta'}$, the expression for $A_S(\ys)_j$ in Equation~\eqref{eq:defabsgrading} is independent of the choice of $\xs\in \bT_{\alpha}\cap \bT_{\beta}$, $\Theta\in \bT_{\beta}\cap \bT_{\beta'}$ and $\psi\in \pi_2(\xs,\Theta,\ys)$.
	\end{lem}
	\begin{proof} 
	We first show that $A_S$ is independent of the triangle $\psi$, for fixed $\xs$ and $\Theta$. If $\psi,\psi'\in \pi_2(\ve{x},\Theta,\ve{y})$ are two homology classes, we can write
	\[
	\psi'=\psi+\cP,
	\]
	 for a triply periodic domain $\cP$. By \cite{OSDisks}*{Proposition 8.5},  
	\[
	\frs_{\ve{w}}(\psi')=\frs_{\ve{w}}(\psi)+q_*\PD[H(\cP)],
	\] 
	where $H(\cP)$ is the integral 2-cycle obtained by capping off the triply periodic domain $\cP$, and 
	\[
	q_*\colon H^2(X_{\alpha\beta\beta'},\d X_{\alpha\beta\beta'};\Z)\to H^2(X_{\alpha\beta\beta'};\Z)
	\] is the map in the long exact sequence of cohomology.
	
	Let $A_{S}^{(\xs,\Theta,\psi)}$ denote the grading defined with  $\xs$, $\Theta$ and $\psi$, and let $A_{S}^{(\xs,\Theta,\psi')}$ denote the grading defined with $\xs$, $\Theta$ and $\psi'$. We compute 
		\begin{align*}&A_{S}^{(\xs,\Theta,\psi')}(\ve{y})_j\\
		=&\tilde{A}(\ve{x})_j+\tilde{A}(\Theta)_j+(n_{\ve{w}}-n_{\ve{z}})_j(\psi')+\frac{\langle c_1(\frs_{\ve{w}}(\psi')),[\hat{\Sigma}_j]\rangle -[\hat{\Sigma}]\cdot [\hat{\Sigma}_j]}{2}\\
		=&\tilde{A}(\ve{x})_j+\tilde{A}(\Theta)_j+(n_{\ve{w}}-n_{\ve{z}})_j(\psi)+(n_{\ve{w}}-n_{\ve{z}})_j(\cP)+\frac{\langle c_1(\frs_{\ve{w}}(\psi))+2q_*\PD[H(\cP)],[\hat{\Sigma}_j]\rangle -[\hat{\Sigma}]\cdot [\hat{\Sigma}_j]}{2}\\
		=&\tilde{A}(\ve{x})_j+\tilde{A}(\Theta)_j+(n_{\ve{w}}-n_{\ve{z}})_j(\psi) +\frac{\langle c_1(\frs_{\ve{w}}(\psi)),[\hat{\Sigma}_j]\rangle -[\hat{\Sigma}]\cdot [\hat{\Sigma}_j]}{2}\\
		&+(n_{\ve{w}}-n_{\ve{z}})_j(\cP)+\langle q_*\PD[H(\cP)], \hat{\Sigma}_j \rangle \\
		=&A_{S}^{(\xs,\Theta,\psi)}(\ve{y})_j
		\end{align*} since 
		\begin{align*}&(n_{\ve{w}}-n_{\ve{z}})_j(\cP)+\langle q_*\PD[H(\cP)], \hat{\Sigma}_j \rangle\\
		=&(n_{\ve{w}}-n_{\ve{z}})_j(\cP)+\langle \PD[H(\cP)], \Sigma_j \rangle\\
		=&(n_{\ve{w}}-n_{\ve{z}})_j(\cP)+\langle \PD[H(\cP)], (\Sigma_{\alpha\beta\beta'})_j \rangle\\
		=&0\end{align*} by Equation \eqref{eq:intersecttriplyperiodicdomain}.
			
	The independence of $A_S(\ys)_j$ from $\xs$ and $\Theta$ is handled similarly. If $\xs'$ is another choice of intersection point in $\bT_{\alpha}\cap \bT_{\beta}$, then we pick a homology class $\phi\in \pi_2(\xs',\xs)$. We compute directly from the definition that
	\begin{align*}
&A_{S}^{(\xs',\Theta,\psi+\phi)}(\ys)_j-A_{S}^{(\xs,\Theta,\psi)}(\ys)_j\\
=&\tilde{A}(\xs')_j-\tilde{A}(\xs)_j+(n_{\ws}-n_{\zs})_j(\phi)\\
=&0
	\end{align*} 
	since $\tilde{A}(\xs')_j-\tilde{A}(\xs)_j=(n_{\zs}-n_{\ws})_j(\phi)$, by definition. Independence from $\Theta$ is proven analogously. 
			
	\end{proof}

Next, for fixed $\bP$ and $\cB^\beta$, we consider the dependence on the triple $\cT$ subordinate to $\cB^\beta$:

\begin{lem}\label{lem:Ainvariantformsurgerytriple}For fixed parametrized Kirby diagram $\bP$ with framed link $\bS_1$, and fixed $\beta$-bouquet $\cB^\beta$ for $\bS_1$, the Alexander grading $A_{S}$ defined in Equation~\eqref{eq:defabsgrading} is independent of the choice of Heegaard triple $\cT=(\Sigma,\as,\bs,\bs',\ws,\zs)$ subordinate to $\cB^\beta$. 
\end{lem}

	\begin{proof} Suppose $\cT_1=(\Sigma_1,\as_1,\bs_1,\bs_1',\ws,\zs)$ and $\cT_2=(\Sigma_2,\as_2,\bs_2,\bs_2',\ws,\zs)$ are two triples subordinate to $\cB^\beta$. Let $A_{S,\cT_1}$ denote the grading defined with the triple $\cT_1$, and let $A_{S,\cT_2}$ be the grading defined with $\cT_2$. Write $\cH_1=(\Sigma_1,\as_1,\bs'_1,\ws,\zs)$ and $\cH_2=(\Sigma_2,\as_2,\bs_2',\ws,\zs)$. By definition, we need to show that
	\begin{equation}
A_{S,\cT_2}=F_{\cH_1\to \cH_2}(A_{S,\cT_1}),	\label{eq:transitionmaponalexandertriples}
	\end{equation}
	where $F_{\cH_1\to \cH_2}$ is the transition map on sets of Alexander gradings defined in Section~\ref{sec:transitivesystemsofgradings}.
	
	 Any two triples subordinate to $\cB^\beta$ can be connected by a sequence of the six moves of Lemma~\ref{lem:movesbetweenbetasubordinatediagrams}, so it is sufficient to prove Equation~\eqref{eq:transitionmaponalexandertriples} when $\cT_1$ and $\cT_2$ differ by one of the moves on the list.
		
		We consider Move~\eqref{move:connecttwotriples1} first, when  $\cT_2$ is obtained from $\cT_1$ by a  handleslide or isotopy of the $\as$ curves. In this case, let us write $\cT_1=(\Sigma, \as,\bs,\bs')$ and $\cT_2=(\Sigma,\as',\bs,\bs')$.

		Suppose that $\ys_{\a'\b'}\in \bT_{\alpha'}\cap \bT_{\beta'}$. We make choices of the following:
	\begin{enumerate}
	\item  $\xs_{\a'\b}\in \bT_{\a'}\cap \bT_{\b}$.
	\item $\Theta_{\b\b'}\in \bT_{\b}\cap \bT_{\b'}$, representing the torsion $\Spin^c$ structure.
	\item $\psi_{\a'\b\b'}\in \pi_2(\xs_{\a'\b}, \Theta_{\b\b'}, \ys_{\a'\b'})$. 
	\item $\ys_{\a\b'}\in \bT_{\alpha}\cap \bT_{\beta'}$ such that $\frs_{\ws}(\ys_{\a\b'})=\frs_{\ws}(\ys_{\a'\b'})\in \Spin^c(Y)$.
	\item $\psi_{\alpha\beta\beta'}\in \pi_2(\xs_{\a\b}, \Theta_{\b\b'}, \ys_{\a\b'})$.
	\item $\Theta_{\a'\a}\in \bT_{\a'}\cap \bT_{\a}$, representing the torsion $\Spin^c$ structure.
		\end{enumerate}
	It follows from \cite{OSDisks}*{Proposition~8.5} that by adding triply periodic domains into $\psi_{\a'\b\b'}$ and $\psi_{\a\b\b'}$, we can assume  
	\begin{equation}
	\frs_{\ws}(\psi_{\a'\b\b'})=\frs_{\ws}(\psi_{\a\b\b'})\label{eq:equalityofspincstructures}
	\end{equation}
	 under the canonical inclusions of $X_{\a'\b\b'}$ and $X_{\a\b\b'}$ into $W(S^3,\bS_1)$ from Lemma~\ref{lem:uniqueembeddingintoW}.
	
Similar to the map from homology classes of triangles to $\Spin^c$ structures discussed in Section~\ref{sec:heegaardtriplesandspinc}, there is a $\Spin^c$ map on quadrilateral classes
\[
\frs_{\ws}\colon \pi_2(\Theta_{\a'\a}, \xs_{\a\b}, \Theta_{\b\b'}, \ys_{\a'\b'})\to \Spin^c(X_{\a'\a\b\b'}).
\]	

	We note that after filling in $Y_{\a'\a}$ with 3-handle and 4-handles, the manifolds $X_{\a'\a\b'}$ and $X_{\a'\a\b}$ become $[0,1]\times S^3$ and $[0,1]\times Y^3$ respectively. Hence we can find triangle classes $\psi_{\a'\a\b}\in \pi_2(\Theta_{\a'\a},\xs_{\a\b},\xs_{\a'\b})$ and $\psi_{\a'\a\b'}\in \pi_2(\Theta_{\a'\a}, \ys_{\a\b'},\ys_{\a'\b'})$ such that
	\[
\frs_{\ws}(\psi_{\a'\a\b'}+\psi_{\a\b\b'})=\frs_{\ws}(\psi_{\a'\b\b'}+\psi_{\a'\a\b})\in \Spin^c(X_{\a'\a\b\b'})\iso \Spin^c(W(S^3,\bS_1)).
	\]
By \cite{OSDisks}*{Section~8.1.5}, it follows that the quadrilateral class $\psi_{\a'\a\b'}+\psi_{\a\b\b'}$ can be obtained from $\psi_{\a'\b\b'}+\psi_{\a'\a\b}$ by splicing homology classes of disks into the four ends. By splicing these four disk classes into the triangle classes $\psi_{\a'\b\b'}$ and $\psi_{\a'\a\b}$, we may simply assume that
\begin{equation}
\psi_{\a'\a\b'}+\psi_{\alpha\beta\beta'}=\psi_{\a'\b\b'}+\psi_{\a'\a\b},\label{eq:relationoftrianglesforassoc}
\end{equation}		
as homology classes of quadrilaterals.

By expanding out the definitions of the gradings, and simplifying slightly using Equation~\eqref{eq:relationoftrianglesforassoc}, we obtain the equality
\begin{equation}
\begin{split}
&A_{S,\cT_2}(\ys_{\a'\b'})_j-F_{\cH_1\to \cH_2}(A_{S,\cT_1})(\ys_{\a'\b'})_j
\\=&\tilde{A}(\xs_{\a'\b})_j-\tilde{A}(\xs_{\a\b})_j-\tilde{A}(\Theta_{\a'\a})_j-(n_{\ws}-n_{\zs})_j(\psi_{\a'\a\b})\\
&+\frac{\langle c_1(\frs_{\ws}(\psi_{\a'\b\b'})), [\hat{\Sigma}_j]\rangle -[\hat{\Sigma}]\cdot [\hat{\Sigma}_j] }{2}-\frac{\langle c_1(\frs_{\ws}(\psi_{\alpha\beta\beta'})), [\hat{\Sigma}_j]\rangle -[\hat{\Sigma}]\cdot [\hat{\Sigma}_j] }{2}.
\end{split}
\label{eq:indofTmoveofalphacurves}
\end{equation}

The expression
\[
\tilde{A}(\xs_{\a'\b})_j-\tilde{A}(\xs_{\a\b})_j-\tilde{A}(\Theta_{\a'\a})_j-(n_{\ws}-n_{\zs})_j(\psi_{\a'\a\b}),
\]
in Equation~\eqref{eq:indofTmoveofalphacurves} vanishes because it is equal to $\tilde{A}(\xs_{\a'\b})_j-F_{\bs}^{\as\to \as'}(\tilde{A})(\xs_{\a'\b})_j$, which vanishes by Lemma~\ref{lem:canonicalgradingsaretransitive}.

By Equation~\eqref{eq:equalityofspincstructures}, the two summands involving Chern classes and self intersection numbers also vanish. Hence the entirety of Equation~\eqref{eq:indofTmoveofalphacurves} vanishes.

Invariance from Moves~\eqref{move:connecttwotriples2}, \eqref{move:connecttwotriples4} and \eqref{move:connecttwotriples5} also amount to proving invariance from a sequence of isotopies or handleslides of some of the attaching curves, and are proven similarly to Move~\eqref{move:connecttwotriples1}. Invariance under Move~\eqref{move:connecttwotriples3}, (de)stabilization, is an easy computation. Finally, invariance under Move~\eqref{move:connecttwotriples6}, isotopies of the Heegaard surface $\Sigma$ within $S^3$, is a tautology.

\end{proof}

	Next, we address independence from the $\beta$-bouquet $\cB^\beta$:

\begin{lem}\label{lem:independentofbouquet}For a fixed parametrized Kirby diagram for $(Y,\bL)$, the Alexander multi-grading $A_{S}$ is invariant from the $\beta$-bouquet $\cB^\beta$ for the framed link $\bS_1$ of $\bP$.
\end{lem}

\begin{proof}This follows from an adaptation of the original argument that the 2-handle maps are invariant of the choice of $\beta$-bouquet \cite{OSTriangles}*{Lemma~4.8}. As argued therein, if $\cB^\beta_1$ and $\cB^{\beta}_2$ are two bouquets which differ by replacing a single arc with another, then Heegaard triples $\cT_1=(\Sigma,\as,\bs_1,\bs'_1,\ws,\zs)$ and $(\Sigma,\as,\bs_2,\bs_2',\ws,\zs)$ can be constructed so that $\bs_2$ is obtained from $\bs_1$ via a sequence of handleslides and isotopies, and $\bs_2'$ is obtained from $\bs_1'$ via a sequence of handleslides and isotopies. Adapting the argument from Lemma~\ref{lem:Ainvariantformsurgerytriple} for associativity on the level of homology classes yields the statement.
\end{proof}

\begin{lem}\label{lem:invariantfromhandleslides} The Alexander multi-grading is invariant under Move~\ref{move:L1}, handleslides amongst the components of $\bS_1$.
\end{lem}

\begin{proof}This follows by adapting the proof of invariance of the 2-handle maps from handleslides \cite{OSTriangles}*{Lemma~4.14}. Handlesliding a component of $\bS_1$ across another can be realized as a sequence of several handles of the $\bs$ curves over each other, and several handleslides of the $\bs'$ curves over each other. An argument using associativity on the level of homology classes as in Lemma~\ref{lem:Ainvariantformsurgerytriple} shows invariance.
\end{proof}

We now consider Move~\ref{move:L2}, invariance under blowing-up or down:

\begin{lem}\label{lem:blowupawayfromL} Suppose that $K$ is an unknot in $S^3$ which is contained in a ball which is disjoint from $\bS_1$ and $U$. Suppose that $\bP=(\phi_0,\lambda, \bS_1,f)$ is a parametrized Kirby diagram for $(Y,\bL)$ and let $\bP'=(\phi_0,\lambda,\bS_1\cup \{K\},f_K)$ denote the parametrized Kirby diagram obtained by adding $K$ to $\bS_1$ with framing $\pm 1$, and let $f_K$ be the induced diffeomorphism. The gradings $A_{S,\bP}$ and $A_{S,\bP'}$ agree.
\end{lem}
\begin{proof}The proof is similar to the standard proof of the blow-up formula \cite{OSTriangles}*{Section~6}. If $\cT$ is subordinate to a bouquet for $\bS_1$, then we can construct a triple $\cT^+$ which is subordinate to a bouquet for $\bS_1\cup \{K\}$ by taking the connected sum of $\cT$ with one of the two diagrams shown in Figure~\ref{fig::36}, depending on whether we are taking a positive or negative blow-up.

\begin{figure}[ht!]
\centering
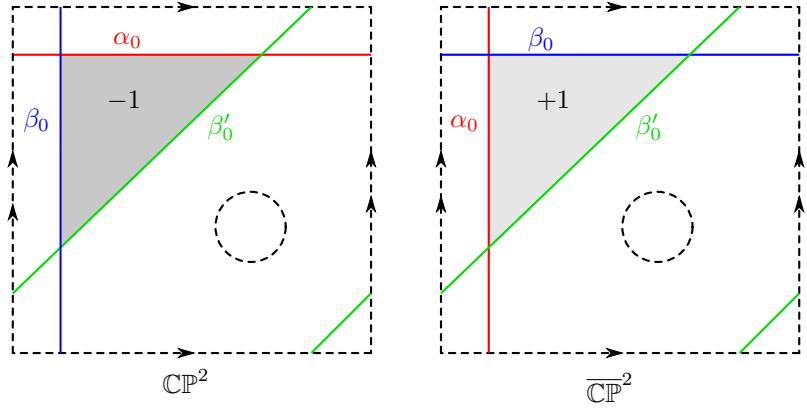
\caption{\textbf{A Heegaard triple for surgery after a blow-up (Move~\ref{move:L1}).} Taking connected sum of a surgery triple with one of the two triples shown above, at the dashed circle shown, results in a surgery triple for blowing up away from $\Sigma$. Multiplicities of a homology triangles are shown.}\label{fig::36}
\end{figure}

For both positive and negative blow-ups, if $\psi$ is a class of triangles on $(\Sigma, \as,\bs,\bs',\ws,\zs)$, we can make $\psi$ have multiplicity 0 at the connected sum point by splicing $\psi$ with the class $k\cdot [\Sigma]$ (where $\Sigma$ denotes the Heegaard surface).  We can construct a homology class of triangles $\psi^+$ by taking the product of $\psi$ and the triangle class shown in Figure~\ref{fig::36}. We note that
\begin{equation}\frs_{\ve{w}}(\psi^+)=\frs_{\ve{w}}(\psi)\# \frs'\label{eq:Alexanderchangeblowup0}\end{equation} for some $\frs'$ in $\Spin^c(\CP^2)$ or $\Spin^c(\bar{\CP}^2)$.  Equation~\eqref{eq:Alexanderchangeblowup0} implies that 
\[
\langle c_1(\frs_{\ve{w}}(\psi^+)),[\hat{\Sigma}_j] \rangle=\langle c_1(\frs_{\ve{w}}(\psi)), [\hat{\Sigma}_j]\rangle.
 \]
  Furthermore, $[\hat{\Sigma}]\cdot [\hat{\Sigma}_j]$ is easily seen to be unchanged, since the blow-up occurs away from $[0,1]\times U\subset W(S^3,\bS_1)$. Since  $(n_{\ve{w}}-n_{\ve{z}})_j(\psi)$ is also unchanged, it follows that Equation~\eqref{eq:defabsgrading} is unchanged, so $A_{S,\bP}$ and $A_{S,\bP'}$ agree.
\end{proof}

We now consider Move~\ref{move:L3}, when the new component is given framing $-1$:

\begin{lem}\label{lem:invariantO3'-1framing}Suppose that $\bP=(\phi_0,\lambda,\bS_1,f)$ is a parametrized Kirby diagram for $(Y,\bL)$  and that $K\subset S^3\setminus N(U)$ is a meridian of a single component of $U$, as in Move~\ref{move:L3}, and suppose $K$ is given framing $-1$. Let $\bP'=(\phi_0,\lambda',\bS_1\cup \{K\},f_K)$ where $f_K$ is the induced diffeomorphism, and  $\lambda'$ is the new framing on $L$. The gradings $A_{S,\bP}$ and $A_{S,\bP'}$ agree.
\end{lem}
\begin{proof}Given a Heegaard triple $\cT$ subordinate to a bouquet for $\bS_1$, we can construct a Heegaard triple $\cT^+$ subordinate to a bouquet for $\bS_1\cup \{K\}$ by taking the connected sum of the genus 1 Heegaard triple on the right side of Figure~\ref{fig::36} with $\cT$, near a basepoint $z$ on the link component which $K$ is a meridian of, and then moving $z$ into the position shown in  Figure~\ref{fig::24}. Let $j\in \bJ$ denote the index $j:=J(z)$.

\begin{figure}[ht!]
\centering
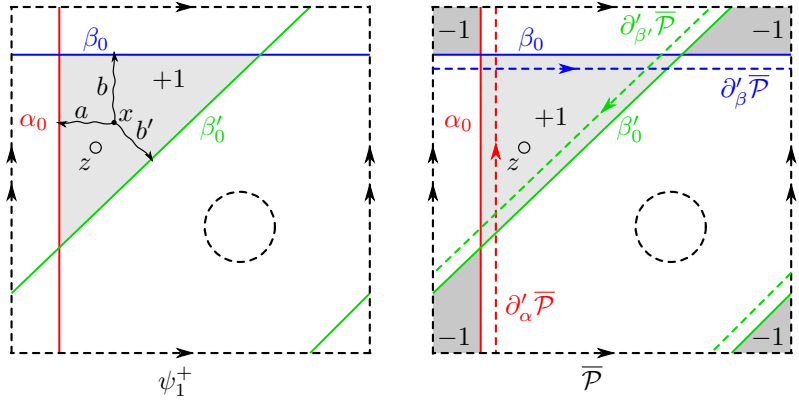
\caption{\textbf{The Heegaard triple $\cT^+$ for surgery on $\bS_1\cup \{K\}$ where $K$ has framing $-1$.} This corresponds to Move~\ref{move:L3}. On the left is the homology class $\psi_1^+$, and a dual spider, with arcs $a$, $b$ and $b'$. On the right is the triply periodic domain $\bar{\cP}$ with $H(\bar{\cP})\in H_2(\bar{\CP}^2;\Z)$ a generator. Also on the right are the translates $\d'_\alpha \bar{\cP},$ $\d'_\beta \bar{\cP}$ and $\d'_{\b'} \bar{\cP}$.}\label{fig::24}
\end{figure}

We note that a there is a diffeomorphism between $W(S^3,\bS_1\cup \{K\})$ and $ W(S^3,\bS_1)\# \bar{\CP}^2$ which is the identity outside  of $[0,1]\times B\subset W(S^3,\bS_1)$ for a ball $B\subset S^3$ containing $K$. Under this connected sum decomposition, we can describe $H_2(W(S^3,\bS_1\cup \{K\});\Z)$ as  $H_2(W(S^3,\bS_1);\Z)\oplus \Z$ where the new copy of $\Z$ is generated by an embedded sphere $E$ in $W(S^3,\bS_1\cup \{K\})$ formed by taking a Seifert disk for $K$ in $S^3\times \{1\}$ and gluing on the core of the 2-handle attached along $K$. Let $\Sigma_j$ denote the link cobordism surface $[0,1]\times U_j\subset W(S^3,\bS_1)$ and let $\Sigma'_j$ denote the analogous surface in $W(S^3,\bS_1\cup \{K\})$.

We can view the original link cobordism surface $\Sigma_j$ as a surface in $W(S^3,\bS_1\cup \{K\})\iso W(S^2,\bS_1)\# \bar{\CP}^2$, by first pushing $\Sigma_j$ outside of the connected sum region before we take the connected sum. We note that, by inspection, the class $E$ is equal to $H(\bar{\cP})$,  the homology class obtained by capping off the triply periodic domain $\bar{\cP}$ shown in Figure~\ref{fig::24}. Note that clearly
\begin{equation}
[\hat{\Sigma}_j']=[\hat{\Sigma}_j]+a\cdot E,
\end{equation}
for some $a\in \Z$. In fact, we can compute that $a=-1$ by applying Equation~\eqref{eq:intersecttriplyperiodicdomain} to the triply periodic domain $\bar{\cP}$ in Figure~\ref{fig::24}:
\[
1=(n_{\ve{z}}-n_{\ve{w}})_j( \bar{\cP})=\#(\Sigma'_j\cap H( \bar{\cP}))=[\hat{\Sigma}_j+a H( \bar{\cP})]\cdot[H(\bar{\cP})]=-a.
\] 

Now let $\psi\in \pi_2(\xs,\Theta,\ys)$ be any homology class of triangles on $\cT$, such that $\Theta$ represents the torsion $\Spin^c$ structure. By splicing in the class of the Heegaard surface to $\psi$, we can assume that $n_z(\psi)=0$.

We construct a class $\psi_1^+\in \pi_2(\xs\times x_{\a_0\b_0},\Theta\times \theta_{\b_0\b_0'}, \ys\times y_{\a_0\b_0'})$, by taking the product of $\psi$ and the small triangle class shown in Figure~\ref{fig::24}.

We now compute $\langle c_1(\frs_{\ve{w}}(\psi_1^+)), H(\bar{\cP})\rangle$. According to \cite{OSTriangles}*{Proposition~6.3}, we have 
\begin{equation}
\langle c_1(\frs_{\ve{w}}(\psi_1^+)),H(\bar{\cP})\rangle=e(\bar{\cP})+\# (\d \bar{\cP})-2n_{\ve{w}}(\bar{\cP})+2\sigma(\psi_1^+,\bar{\cP}),\label{eq:chernclassformulaheegaardtriple}
\end{equation}
 where $e(\bar{\cP})=0$ denotes the Euler measure, $\# \d \bar{\cP}=3$ is the number of boundary components of $\bar{\cP}$, and $\sigma(\psi_1,\bar{\cP})$ is the \emph{dual spider number} \cite{OSTriangles}*{Section~6.1}. We recall briefly the construction of the dual spider number $\sigma(\psi_1^+,\bar{\cP})$. Let $u\colon \Delta\to \Sym^n(\Sigma_{\cT^+})$ denote a topological representative of the class $\psi_1^+$. Let $x\in \Delta$ be a generic point, and let $a$, $b$ and $b'$ be three paths in $\Delta$ from $x$ to the $\alpha_0$, $\beta_0$ and $\b'_0$ boundaries of $\Delta$, respectively. Perturbing $u$ slightly if necessary, we can view $u(x)$  as an $n$-tuple of points on $\Sigma$, and $u(a)$, $u(b)$ and $u(b')$  as integral 1-chains on $\Sigma$. The dual spider number is defined as
\begin{equation}
\sigma(\psi_1^+,\bar{\cP}):=n_{u(x)}(\bar{\cP})+\#(u(a)\cap \d'_\a \bar{\cP})+\# (u(b)\cap \d'_\beta \bar{\cP})+\# (u(b')\cap \d'_{\b'} \bar{\cP}),\label{eq:dualspidernumber}
\end{equation}
where  $\d'_\tau \bar{\cP}$ denotes the translation of the boundary component $\d_\tau \bar{\cP}$ of $\bar{\cP}$, in the direction of the inward normal vector field, according to the periodic domain $\bar{\cP}$.

Using Equation~\eqref{eq:dualspidernumber}, we compute  $\sigma(\psi_1^+,\bar{\cP})=1-1-1-1=-2$. Computing the remainder of the terms in Equation~\eqref{eq:chernclassformulaheegaardtriple}, we see
\[
\langle c_1(\frs_{\ve{w}}(\psi_1^+)),H(\bar{\cP}) \rangle=-1.
\] 
We note
\[
n_{\ws}(\psi_1^+)_j=n_{\ws}(\psi)_j\qquad  \text{and} \qquad n_{\zs}(\psi_1^+)_j=n_{\zs}(\psi)_j+1.
\]
Using $\psi$ to compute the Alexander grading $A_{S,\bP}$, and using $\psi_1^+$ to compute the Alexander grading $A_{S,\bP'}$, we see that the difference between the expressions defining $A_{S,\bP'}(\ys\times y_{\a_0\b_0'})_j$ and $A_{S,\bP}(\ys)$ is
\[
\frac{\langle c_1(\frs_{\ve{w}}(\psi_1^+)),-H(\bar{\cP})\rangle-[-H(\bar{\cP})]\cdot [-H(\bar{\cP})]}{2}-n_z(\psi_1^+)=0.
\]

In a similar way to Lemma~\ref{lem:blowupawayfromL}, it is straightforward to see that the other components of the Alexander grading are unchanged, completing the proof.
\end{proof}

We now consider Move~\ref{move:L3}, when the new link component is given framing $+1$:

\begin{lem}\label{lem:invariantO3'+1framing}Suppose that $\bP=(\phi_0,\lambda,\bS_1,f)$ is a parametrized Kirby diagram for $(Y,\bL)$ and suppose that $K$ is a meridian of a single component of $U$, as in Move~\ref{move:L3}, and suppose $K$ is given framing $+1$. Let $\bP'=(\phi_0,\lambda',\bS_1\cup \{K\},f_K)$, where $f_K$ is the induced diffeomorphism, and $\lambda'$ is the new framing on $L$. The two gradings $A_{S,\bP}$ and $A_{S,\bP'}$ agree.
\end{lem}
\begin{proof} The proof is similar to the proof of Lemma~\ref{lem:invariantO3'-1framing}. Let $\cT^+$ be a triple subordinate to a bouquet for $\bS_1\cup \{K\}$, constructed as in the proof of Lemma~\ref{lem:invariantO3'-1framing}; See Figure~\ref{fig::37}. Let $j$ denote the index of the grading that $K$ is assigned to. Let $\psi_{-1}^+$ and $\cP$ be the homology class of triangles and triply periodic domain shown in Figure~\ref{fig::37}.

\begin{figure}[ht!]
\centering
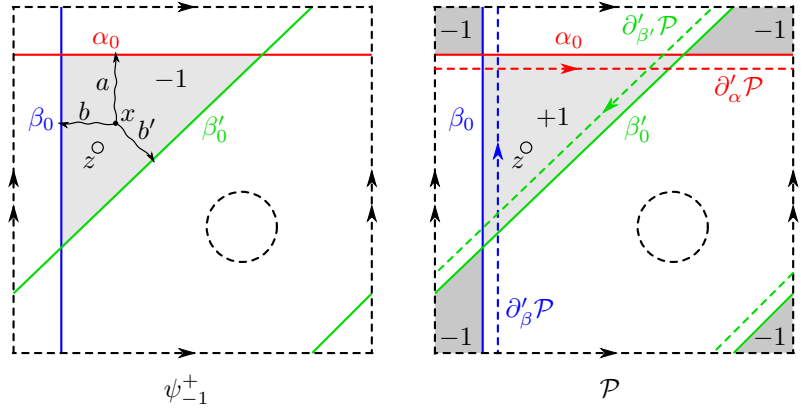

\caption{\textbf{The Heegaard triple $\cT^+$ for surgery on $\bS_1\cup \{K\}$ where $K$ has framing $+1$.} This corresponds to Move~\ref{move:L3}. On the left is the homology class $\psi_1^+$, and a dual spider, with arcs $a$, $b$ and $b'$. On the right is the triply periodic domain $\cP$ with $H(\cP)\in H_2(\CP^2;\Z)$ a generator. Also on the right are the translates $\d'_\alpha \cP,$ $\d'_\beta \cP$ and $\d'_{\b'} \cP$.}
\label{fig::37}
\end{figure}

Let $\Sigma'_j$ denote the surface $[0,1]\times U_j\subset W(S^3,\bS_1\cup \{K\})$ and let $\Sigma_j$ denote the surface $[0,1]\times U_j\subset W(S^3,\bS_1)$, which we can also view as being a surface in $W(S^3,\bS_1\cup \{K\})$. Let $\hat{\Sigma}'_j$ and $\hat{\Sigma}_j$ denote the integral 2-cycles obtained by capping off $\Sigma'_j$ and $\Sigma_j$.  As in Lemma~\ref{lem:invariantO3'-1framing}, we can write $[\hat{\Sigma}'_j]=[\hat{\Sigma}_j]+ a H(\cP)$. Arguing as before, 
\[
1=(n_{\ve{z}}-n_{\ve{w}})_j(\cP)=\#(\Sigma'_j\cap H(\cP))= a [H(\cP)]\cdot[ H(\cP)]=a,
\] 
implying that $[\Sigma'_j]=[\Sigma_j] +H(\cP)$. Computing using Equations~\eqref{eq:chernclassformulaheegaardtriple} and \eqref{eq:dualspidernumber} we have
\[
\sigma(\psi_{-1}^+,\cP)=-2\qquad \text{and} \qquad \langle c_1(\frs_{\ve{w}}(\psi_{-1}^+)),H(\cP)\rangle=-1.
\]
Furthermore
\[
n_{\ws}(\psi_{-1}^+)_j=n_{\ws}(\psi)_j\qquad \text{and} \qquad n_{\zs}(\psi_{-1}^+)_j=n_{\zs}(\psi)_j-1.
\]

Hence, the difference between $A_{S,\bP'}(\ys\times y_{\a_0\b'_0})_j$ and $A_{S,\bP}(\ys)_j$ is 
\[
\frac{\langle c_1(\frs_{\ve{w}}(\psi_{-1}^+)),H(\cP)\rangle-[H(\cP)]\cdot [H(\cP)]}{2}-n_{\zs}(\psi_{-1}^+)_j+n_{\zs}(\psi)_j=0.
\]
 Furthermore, arguing analogously to Lemma~\ref{lem:invariantO3'-1framing}, the Alexander grading is also unchanged in the components other than $j$.
\end{proof}

We now consider Move~\ref{move:L4}, corresponding to changing the identification $\phi_0$ of $U$ with $L$.

\begin{lem}\label{lem:invariantfromphi0}Suppose that $\bP=(\phi_0,\lambda, \bS_1,f)$ is a parametrized Kirby diagram of $(Y,\bL)$ and that $\psi_0\colon U\to U$ is a diffeomorphism, which is extended by $\psi\colon (S^3,U)\to (S^3,U)$, and $\psi$ is orientation preserving for both $S^3$ and $U$. Then $A_{S,\bP}$ and $A_{S,\bP'}$ agree, where $\bP'=(\phi_0\circ \psi_0^{-1}, \lambda,\psi(\bS_1), f\circ (\psi^{\bS_1})^{-1})$.
\end{lem}

\begin{proof}The proof is essentially a tautology. We take a triple $\cT=(\Sigma, \as,\bs,\bs',\ws,\zs)$ subordinate to a $\beta$-bouquet of $\bS_1$. The triple $\psi_* \cT=(\psi(\Sigma), \psi(\as),\psi(\bs),\psi(\bs'),\ws,\zs)$ is subordinate to a $\beta$-bouquet for $\psi(\bS_1)$. Note that the resulting Heegaard surface for $(Y,\bL)$ from both the pairs $(\bP,\cT)$ and $(\bP',\psi_*\cT)$ is $(f(\Sigma),f(\as),f(\bs'),\ws,\zs)$.  In particular, the pairs $(\bP,\cT)$ and $(\bP',\psi_*\cT)$ define an absolute grading on the same chain complex.  Furthermore,the expression defining the Alexander grading in Equation~\eqref{eq:defabsgrading} is unchanged, since we can simply push forward a triangle class on $\cT$ under the diffeomorphism $\psi$ to get a triangle class on $\psi(\cT)$. 
\end{proof}

Combining the results of this section, we can prove part~\eqref{thm1.1a} of Theorem~\ref{thm:1}:

\begin{customthm}{\ref{thm:1} Part \eqref{thm1.1a}}\label{thm:Alexandergradingsindependent} Suppose that $\bL$ is a multi-based link in $Y$, and $(\sigma, \bJ)$ is a type-partitioned, indexed coloring of $\bL$,  $L$ is $\bJ$-null-homologous, and that $S$ is a generalized $\bJ$-Seifert surface of $L$. Then the chain complex $\cCFL^\infty(Y,\bL^{\sigma}, \frs)$ admits an absolute Alexander multi-grading $A_{S}$ which takes values in $\Q^{\bJ}$. The multi-grading is additive with respect to collapsing indices.
\end{customthm}
\begin{proof} \emph{A-priori}, the grading $A_S$ depends on the choice of parametrized Kirby decomposition $\bP$ and Heegaard triple $\cT$. By Lemma~\ref{lem:Ainvariantformsurgerytriple} the grading is independent of the choice of Heegaard triple $\cT$, subordinate to a fixed bouquet of the framed link $\bS_1$ of $\bP$. By Lemma~\ref{lem:independentofbouquet}, the grading is independent of the choice of bouquet subordinate to $\bS_1$. Hence the grading $A_S$ depends at most on the $\bP$ (and $Y$, $\bL$ and $S$).

By  Proposition~\ref{prop:connecttwoparamsurgdecomps}, any two choices of $\bP$ can be connected by Moves~\ref{move:L0}, \ref{move:L1}, \ref{move:L2}, \ref{move:L3} and \ref{move:L4}. Invariance from Move~\ref{move:L0} (isotopies of $f$ or $\bS_1$, fixing $\d (S^3_U(\bS_1))$) is automatic, using naturality of Heegaard Floer homology. Invariance under Move~\ref{move:L1} follows from Lemma~\ref{lem:invariantfromhandleslides}. Invariance from Move~\ref{move:L2} follows from Lemma~\ref{lem:blowupawayfromL}. Invariance from Move~\ref{move:L3} follows from Lemmas~\ref{lem:invariantO3'-1framing} and \ref{lem:invariantO3'+1framing}. Finally, invariance from Move~\ref{move:L4} follows from Lemma~\ref{lem:invariantfromphi0}.

The final claim, that the Alexander grading is additive with respect to collapsing components of the index set $\bJ$, is straightforward, since each summand in the expression in Equation~\eqref{eq:defabsgrading} is additive under collapsing gradings.
\end{proof}

\subsection{Dependence on the Seifert surface $S$}

We now prove part \eqref{thm1.1b} of Theorem~\ref{thm:1}:

\begin{customthm}{\ref{thm:1} Part \eqref{thm1.1b}}If $S$ and $S'$ are two generalized $\bJ$-Seifert surfaces for a $\bJ$-null-homologous link $\bL$, and $\frs\in \Spin^c(Y)$, then 
\[
(A_{S'})_j-(A_{S})_j=\frac{\langle c_1(\frs),[S_j'\cup -S_j]\rangle}{2}.
\]
 In particular, if $c_1(\frs)$ is torsion, then the absolute Alexander grading does not depend on the choice of generalized $\bJ$-Seifert surface.
\end{customthm}

\begin{proof} Let $S$ and $S'$ be two choices of generalized $\bJ$-Seifert surfaces. Pick a parametrized Kirby diagram of $(Y,\bL)$ with framed 1-dimensional link $\bS_1$, and pick a Heegaard triple subordinate to a $\beta$-bouquet of $\bS_1$. Let $\Sigma_j\subset W(S^3,\bS_1)$ denote the surface $[0,1]\times U_j$ and let $\hat{\Sigma}_j$ and $\hat{\Sigma}_j'$ denote the closed 2-chain obtained by capping $\Sigma$ with a Seifert surface of $U$ in $\{0\}\times S^3$, and $S_j$ or $S_j'$, respectively, in $Y$. Write $[F_j]=[S_j'\cup -S_j]$. As elements of $H_2(W(S^3,\bS_1);\Z)$, we have $[\hat{\Sigma}'_j]=[\hat{\Sigma}_j]+[F_j]$. Let $[\hat{\Sigma}]$, $[\hat{\Sigma}']$ and $[F]$ denote the sum over $j\in\bJ$ of the classes $[\hat{\Sigma}_j]$, $[\hat{\Sigma}_j']$, and $[F_j]$ respectively. Using the definition of the absolute grading in Equation~\eqref{eq:defabsgrading}, we see
\begin{align*}(A_{S'})_j-(A_{S})_j&=\frac{\langle c_1(\frs), [\hat{\Sigma}'_j]\rangle -[\hat{\Sigma}']\cdot [\hat{\Sigma}'_j]}{2}-\frac{\langle c_1(\frs),[\hat{\Sigma}_j]\rangle- [\hat{\Sigma}]\cdot [\hat{\Sigma}_j]}{2}\\
&=\frac{\langle c_1(\frs),[F_j]\rangle}{2}+\frac{-[F]\cdot [\hat{\Sigma}_j]-[\hat{\Sigma}]\cdot [F_j]-[F]\cdot [F_j]}{2}\\
&=\frac{\langle c_1(\frs),[F_j]\rangle}{2},\end{align*} since $[F]$ and $[F_j]$ are in the image of the inclusion map $H_2(\d W(S^3,\bS_1);\Z)\to H_2(W(S^3,\bS_1);\Z)$, and hence the intersection number of either with anything in $H_2(W(S^3,\bS_1);\Z)$ vanishes.
\end{proof}

We illustrate Part~\eqref{thm1.1b} of Theorem~\ref{thm:1} when $c_1(\frs)$ is torsion with the following example:

\begin{example}Consider an unknot $\bU$ in $S^1\times S^2$ with two basepoints. For convenience, we  view $\bU$ as being embedded in $\{pt\}\times S^2$, so that there are two distinguished Seifert disks, $D_1$ and $D_2$, for $\bU$. Note that $D_2\cup (-D_1)=\{pt\}\times S^2$. In Figure~\ref{fig::18}, two diagrams $\cH_1=(\Sigma,\alpha,\beta,w,z)$ and $\cH_2=(\Sigma,\alpha',\beta,w,z)$ are shown. We can assume that $D_1$ and $D_2$ intersect $\Sigma$ in an embedded arc, connecting $w$ to $z$. Furthermore, we assume that $D_1$ is disjoint from $\alpha\cup \beta$, and $D_2$ is disjoint from $\alpha'\cup \beta$.

\begin{figure}[ht!]
\centering
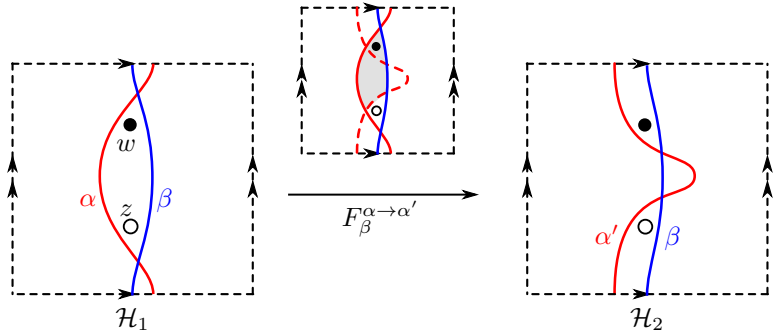
\caption{\textbf{Two diagrams $\cH_1$ and $\cH_2$ for $(S^1\times S^2,\bU)$.} In both diagrams, $\frs_{w}(x)$ is non-torsion for any intersection point $x$. The two Seifert disks, $D_1$ and $D_2$, can be picking a path between the basepoints the two diagrams. Pushing this arc into the handlebodies sweeps out a half disk. The orientation of the Heegaard surface is clockwise with respect to the page.}\label{fig::18}
\end{figure}

It is straightforward to see from the definition of the grading that $A_{D_1}\equiv 0$ on $\Hat{\CFL}(\cH_1)$ and $A_{D_2}\equiv 0$ on $\Hat{\CFL}(\cH_2)$. As transitive gradings, however, we claim that 
\[
A_{D_2}-A_{D_1}=\tfrac{1}{2}\langle c_1(\frs_{w}(x)), [D_2\cup -D_1] \rangle =-1,
\]
where $x\in \alpha\cap \beta$ is either intersection point.

Note that to compare $A_{D_2}$ and $A_{D_1}$, we need to use the transition maps on sets of Alexander gradings. Using the triangle class shown in Figure~\ref{fig::18}, we compute
\[
F_{\beta}^{\alpha\to \alpha'}(A_{D_1})(y)=A_{D_1}(x)+(n_w-n_z)(\psi)=1,
\]
for some (and hence any) choice of $x\in \alpha\cap \beta$ and $y\in \alpha'\cap \beta$. Hence, as transitive gradings, we have
\begin{equation}
A_{D_2}-A_{D_1}\equiv -1.\label{eq:gradingsdifferby1}
\end{equation}

On the other hand, from our orientation conventions (recall that we are using the outward normal first convention for the boundary orientation, Seifert surfaces are oriented so that $\d S=-L$ and that in the handlebody $U_{\b}$ the knot goes from $z$ to $w$), we have that
\[
 H(\cP)=[D_2\cup (-D_1)],
 \] 
  where $\cP$ is the periodic domain on $(\Sigma,\alpha,\beta,w,z)$ which has multiplicity $+1$ in the bigon containing $w$ and $z$, multiplicity $-1$ in the empty bigon, and multiplicity 0 elsewhere. Using Equation~\eqref{eq:chernclassformula}, we see that
\begin{equation}
\langle c_1(\frs_{w}(x)), H(\cP)\rangle =-2,\label{eq:chernclassesformulaworks}
\end{equation}
for any intersection point $x\in \alpha\cap \beta$.

We finally remark that Equations~\eqref{eq:gradingsdifferby1} and \eqref{eq:chernclassesformulaworks} are in accordance with Part~\eqref{thm1.1b} of Theorem~\ref{thm:1}.
\end{example}

\subsection{Invariance of the absolute Maslov gradings}
\label{sec:absolutemaslov}
In this section, we prove Parts~\eqref{thm1.1c} and \eqref{thm1.1d} of Theorem~\ref{thm:1}, and sketch the proof of the well-definedness of the two absolute Maslov gradings, $\gr_{\ve{w}}$ and $\gr_{\ve{z}}$. The proof of well-definedness of $\gr_{\ve{w}}$ and $\gr_{\ve{z}}$ is mostly analogous to the proof of the analogous result for the absolute grading on $\HF^-$ from \cite{OSTriangles}, as well as the proof of Part~\ref{thm1.1a} of Theorem~\ref{thm:1} from Section~\ref{sec:invariancealexandergrading}. Hence we only sketch the details the proof of invariance. A helpful formula is the following:

\begin{lem}\label{lem:sarkarsformula}Suppose that $\cP$ is triply periodic domain and $\psi$ is a homology class of triangles on the Heegaard triple $(\Sigma,\as,\bs,\gs,\ve{w},\ve{z})$. Then 
	\[
	\mu(\psi+\cP)-\mu(\psi)=2n_{\ve{w}}(\cP)+\frac{c_1(\frs_{\ve{w}}(\psi+\cP))^2-c_1(\frs_{\ve{w}}(\psi))^2}{4}.
	\]
	\end{lem}
The proof of Lemma~\ref{lem:sarkarsformula} can be found in \cite{SMaslov}*{Section 5.1}.

We now prove part \eqref{thm1.1c} of Theorem~\ref{thm:1}, which we rephrase as follows:

\begin{customthm}{\ref{thm:1} Part~\eqref{thm1.1c}} If $c_1(\frs)$ is torsion, the absolute grading $\gr_{\ve{w}}$ defined in Equation~\eqref{def:absolutegrw} is a well defined transitive grading, and  is independent of the intersection points $\xs$ and $\Theta$, the homology class $\psi$, the Heegaard triple $\cT$, and the parametrized Kirby diagram $\bP$. Similarly, when $c_1(\frs-\PD[L])$ is torsion, the grading $\gr_{\zs}$ is well defined.
\end{customthm}
\begin{proof}Most of the proof proceeds as in the proof of Part~\eqref{thm1.1a} of Theorem~\ref{thm:1}. The only major difference is in the proof that the formula from Equation~\eqref{def:absolutegrw} is independent from the homology class $\psi$ and intersection points $\xs$ and $\Theta$. Independence of the absolute grading from $\ve{x}$ and $\Theta$ can be proven by splicing in disks on those ends and seeing that the formula does not change. Next, to see that the formula is invariant under the choice of $\psi$, we note that any two homology classes of triangles with the same endpoints differ by a triply periodic domain. Suppose $\psi$ is a homology class of triangles and $\cP$ is a triple periodic domain. Letting $\gr_{\ve{w}}^{\psi}(\ve{y})$ and $\gr_{\ve{w}}^{\psi+\cP}(\ve{y})$ denote the gradings, computed with the classes $\psi$ or $\psi+\cP$, respectively, we observe that the difference is
\[
\gr_{\ve{w}}^{\psi+\cP}(\ve{y})-\gr_{\ve{w}}^{\psi}(\ve{y})=-\mu(\psi+\cP)+\mu(\psi)+2n_{\ve{w}}(\cP)+\frac{c_1(\frs_{\ve{w}}(\psi+\cP))^2-c_1(\frs_{\ve{w}}(\psi))^2}{4},
\]
 which is zero by Sarkar's formula from Lemma~\ref{lem:sarkarsformula}.
 
 Independence from the choice of Heegaard triple $\cT$ and bouquet $\cB^\beta$ can be proven by adapting Lemmas~\ref{lem:Ainvariantformsurgerytriple} and \ref{lem:independentofbouquet}. Independence  from the parametrized Kirby diagram follows by adapting Lemmas~\ref{lem:invariantfromhandleslides}, \ref{lem:blowupawayfromL}, \ref{lem:invariantO3'-1framing}, \ref{lem:invariantO3'+1framing} and \ref{lem:invariantfromphi0}. 
\end{proof}

The following is part \eqref{thm1.1d} of Theorem~\ref{thm:1}:

\begin{customthm}{\ref{thm:1} Part~\eqref{thm1.1d}} The absolute Maslov and collapsed Alexander gradings satisfy
\[
A=\frac{1}{2}(\gr_{\ve{w}}-\gr_{\ve{z}}).
\]
\end{customthm}
\begin{proof} We pick a parametrized Kirby diagram and associated Heegaard triple $\cT=(\Sigma,\as,\bs,\bs',\ws,\zs)$ for $(Y,\bL)$. We compute $\tfrac{1}{2}$ times the difference between the expressions defining $\gr_{\ve{w}}(\ve{x})$ and $\gr_{\ve{z}}(\ve{x})$ in Equation~\eqref{def:absolutegrw}. By Equation~\eqref{eq:collapsedAlexanderMaslov}, the formula $\tilde{A}=\tfrac{1}{2}(\tilde{\gr}_{\ve{w}}-\tilde{\gr}_{\ve{z}})$ holds for the gradings associated to unlinks in $(S^1\times S^2)^{\# k}$. By Lemma~\ref{lem:poincaredualofsurface},  
\[
\frs_{\ve{w}}(\psi)-\frs_{\ve{z}}(\psi)=\PD[\Sigma_{\alpha\beta\beta'}].
\]
 Under the inclusion $\iota_*\colon X_{\alpha\beta\beta'}\hookrightarrow W(Y,\bS_1)$, we have that $\PD[\Sigma_{\alpha\beta\beta'}]=\iota^*\PD[\Sigma]$. Noting that
\[
\frac{c_1(\frs_{\ve{w}}(\psi))^2-c_1(\frs_{\ve{z}}(\psi))^2}{4}=c_1(\frs_{\ve{w}}(\psi))\cup \PD[\Sigma]-\PD[\Sigma]\cup \PD [\Sigma]=\langle c_1(\frs_{\ve{w}}(\psi)),[\hat{\Sigma}] \rangle-[\hat{\Sigma}]\cdot [\hat{\Sigma}],
\]
 we see that the expression defining $\tfrac{1}{2}(\gr_{\ve{w}}(\ve{x})-\gr_{\ve{z}}(\ve{x}))$ becomes exactly the expression defining $A(\ve{x})$. 
\end{proof}

\section{Link cobordisms and absolute gradings}
\label{sec:gradingchange}

 In this section we prove the grading formulas for link cobordisms stated in Theorems~\ref{thm:maingradingchangeformula} and \ref{thm:generalAlexandergradingformula}. In Section~\ref{sec:overviewofmaps} we  outline the construction of cobordism maps from \cite{ZemCFLTQFT}. In Section~\ref{subsec:gradingchangeelementary} we compute the grading change associated to each elementary link cobordism, and in Section~\ref{subsection:proofofgradingformula} we prove the general grading formula.
 
 \subsection{Overview of the link cobordism maps}
 \label{sec:overviewofmaps}
 Before we prove Theorem~\ref{thm:generalAlexandergradingformula}, we need to give a brief summary of the construction of the link cobordism maps from \cite{ZemCFLTQFT}. The maps are defined as a composition of maps for elementary link cobordisms of the following form:
 \begin{enumerate}
 \item $0$- and $4$-handles, which contain a standardly embedded disk with a dividing set consisting of a single arc.
\item Cobordisms obtained by attach a 1-handle or 3-handle to $Y\setminus L$.
\item Cobordisms obtained by attaching a collection of 2-handles along a framed 1-dimensional link in $Y\setminus L$.
\item Elementary saddle link cobordisms in $[0,1]\times Y$ (see Figure~\ref{fig::51}).
\item Cylindrical link cobordisms that add a pair of adjacent basepoints to a link (see Figure~\ref{fig::6}).
 \end{enumerate}

We now briefly describe the maps associated to each of the five elementary link cobordisms. We refer the reader to \cite{ZemCFLTQFT} for further details.

We begin with the 0-handle and 4-handle maps. A 0-handle cobordism is a decorated link cobordism $(W,\cF)$ from $(Y,\bL)$ to $(Y\sqcup S^3,\bL\sqcup \bU)$, where $\bU$ is a doubly based unknot in $S^3$,  $W=([0,1]\times Y)\sqcup B^4$ and $\cF$ consists of the surface $([0,1]\times L)\sqcup D^2$, where $D^2\subset B^4$ is a standardly embedded slice disk of an unknot. A 4-handle cobordism is defined analogously. We let $(S^2,w,z)$ be the Heegaard diagram for $(S^3,\bU)$ with no $\as$ or $\bs$ curves.  On the level of diagrams, if $(\Sigma,\as,\bs,\ws,\zs)$ is a diagram for $(Y,\bL)$, then $(\Sigma\sqcup S^2,\as,\bs,\ws\cup \{w\}, \zs\cup \{z\})$ is a diagram for $(Y\sqcup S^3,\bL\sqcup \bU)$. Noting that the tori $\bT_{\alpha}$ and $\bT_{\beta}$ coincide between the two diagrams, we define the 0-handle and 4-handle maps as the identity on the level of intersection points, with respect to these diagrams.

The 1-handle and 3-handle maps are defined similarly to the constructions in \cite{OSTriangles} and \cite{JCob}. Given a pair of points $\bS_0=\{p_1,p_2\}$ in $Y\setminus L$ (thought of as a 0-sphere), we pick a Heegaard diagram $(\Sigma,\as,\bs,\ws,\zs)$ such that $p_1,p_2\in \Sigma\setminus (\as\cup \bs).$ A Heegaard diagram for the surgered manifold $Y(\bS_0)$ can be obtained by removing two small disks centered at $p_1$ and $p_2$, and connecting the resulting boundary components with an annulus $A$. Inside of $A$, we add two new curves, $\alpha_0$ and $\beta_0$, which are both homologically essential in $A$, and intersect in a pair of points $\{\theta^+,\theta^-\}$. The points $\theta^+$ and $\theta^-$ are distinguished by the relative Maslov grading. The 1-handle map is then defined by the formula
\[
F_{Y,\bL,\bS_0,\frs}(\xs) := \xs\otimes \theta^+,
\]
extended equivariantly over the ring $\cR_{\bmP}^\infty$. The 3-handle map is defined similarly, and is the dual of the 1-handle map.

The 2-handle maps are defined similarly to \cite{OSTriangles}. If $\bS_1$ is a framed 1-dimensional link in $Y\setminus L$, and $\frs\in \Spin^c(W(Y,\bS_1))$, then the 2-handle map $F_{Y,\bL,\bS_1,\frs}$ is defined by picking a $\beta$-bouquet in $Y$ for $\bS_1$, as well as a Heegaard triple $\cT=(\Sigma,\as,\bs,\bs',\ws,\zs)$ which is subordinate to $\bS_1$. In this case, the Floer homology $\cHFL^-(\Sigma,\bs,\bs',\ws,\zs,\sigma,\frs_0)$ has a distinguished element $\Theta^+_{\beta\beta'}$, for any coloring $\sigma$ of $\ws\cup \zs$. The 2-handle map is then defined by counting holomorphic triangles via the formula
\[
F_{Y,\bL,\bS_1,\frs}(\xs):=\sum_{\ys\in \bT_{\alpha}\cap \bT_{\beta'}}\sum_{\substack{
\psi\in \pi_2(\xs,\Theta_{\beta\beta'}^+,\ys)\\
\mu(\psi)=0\\ \frs_{\ws}(\psi)=\frs}} \# \Hat{\cM}(\psi)\cdot U_{\ws}^{n_{\ws}(\psi)} V_{\zs}^{n_{\zs}(\psi)}\cdot \ys.
\]

Next we discuss the band maps. The band maps are described in detail in \cite{ZemCFLTQFT}*{Section~6}. We refer the reader back to Definitions~\ref{def:alphaband} and \ref{def:subordinatetoalphaband} for the definition of an oriented $\alpha$-band, and the definition of a Heegaard triple subordinate to an $\alpha$-band. If $(\Sigma,\as',\as,\bs,\ws,\zs)$ is subordinate to an $\alpha$-band, then $(\Sigma,\as',\as,\ws,\zs)$ represents an unlink $\bU_{\a'\a}$ in $(S^1\times S^2)^{\# k}$. Furthermore, all components of $\bU_{\a'\a}$ have exactly two basepoints, except for one component, which has four.

According to \cite{ZemCFLTQFT}*{Lemma~3.7}, there are two distinguished elements 
\[
\Theta^{\ws}_{\a'\a}, \Theta^{\zs}_{\a'\a}\in \cHFL^-(\Sigma,\as',\as,\ws,\zs,\sigma,\frs_0)
,
\]
 for an appropriate coloring $\sigma$.  The element $\Theta^{\ws}_{\a'\a}$ is a generator of the top $\gr_{\ws}$ graded subset of $\cHFL^-(\Sigma,\as',\as,\ws,\zs,\sigma,\frs_0)$, while $\Theta^{\zs}_{\a'\a}$ is a generator of the top $\gr_{\zs}$ graded subset.
 
 There are two band maps
 \begin{equation}
F_{B}^{\ws}, F_{B}^{\zs}\colon \cCFL^\infty(Y,\bL^\sigma, \frs)\to \cCFL^\infty (Y,\bL(B)^\sigma,\frs), \label{eq:bandmapsdef}
 \end{equation}
 defined by counting holomorphic triangles via the formulas
 \[
F_B^{\ws}(\xs):=F_{\a'\a\b,\frs}( \Theta^{\zs}_{\a'\a}\otimes \xs) \qquad \text{and} \qquad F_B^{\zs}(\xs):=F_{\a'\a\b,\frs}(\Theta^{\ws}_{\a'\a}\otimes\xs).
 \]
The map $F_B^{\ws}$ is the map induced by a  decorated saddle cobordism inside of $[0,1]\times Y$, where all divides go from $\{0\}\times L$ to $\{1\}\times L(B)$. Furthermore, if $f$ denotes the Morse function $(t,y)\mapsto t$, then $f$ restricts to the link cobordism surface to be Morse and have a single critical point, which is of index 1, and occurs inside of $\Sigma_{\ws}$. Furthermore, $f$ restricts to a Morse function on the dividing set with no critical points. The map $F_{B}^{\zs}$ corresponds to a similar decorated saddle cobordism,  where the index 1 critical point occurs inside of $\Sigma_{\zs}$. These are illustrated in Figure~\ref{fig::51}.  

The requirement on the coloring $\sigma$ for one of the band maps in Equation~\eqref{eq:bandmapsdef} to be defined corresponds exactly to the requirement that $\sigma$ is induced by the appropriate decorated link cobordism from Figure~\ref{fig::51}.

\begin{figure}[ht!]
\centering
\begingroup%
  \makeatletter%
  \providecommand\color[2][]{%
    \errmessage{(Inkscape) Color is used for the text in Inkscape, but the package 'color.sty' is not loaded}%
    \renewcommand\color[2][]{}%
  }%
  \providecommand\transparent[1]{%
    \errmessage{(Inkscape) Transparency is used (non-zero) for the text in Inkscape, but the package 'transparent.sty' is not loaded}%
    \renewcommand\transparent[1]{}%
  }%
  \providecommand\rotatebox[2]{#2}%
  \newcommand*\fsize{\dimexpr\f@size pt\relax}%
  \newcommand*\lineheight[1]{\fontsize{\fsize}{#1\fsize}\selectfont}%
  \ifx\svgwidth\undefined%
    \setlength{\unitlength}{401.96166464bp}%
    \ifx\svgscale\undefined%
      \relax%
    \else%
      \setlength{\unitlength}{\unitlength * \real{\svgscale}}%
    \fi%
  \else%
    \setlength{\unitlength}{\svgwidth}%
  \fi%
  \global\let\svgwidth\undefined%
  \global\let\svgscale\undefined%
  \makeatother%
  \begin{picture}(1,0.31589271)%
    \lineheight{1}%
    \setlength\tabcolsep{0pt}%
    \put(0,0){\includegraphics[width=\unitlength,page=1]{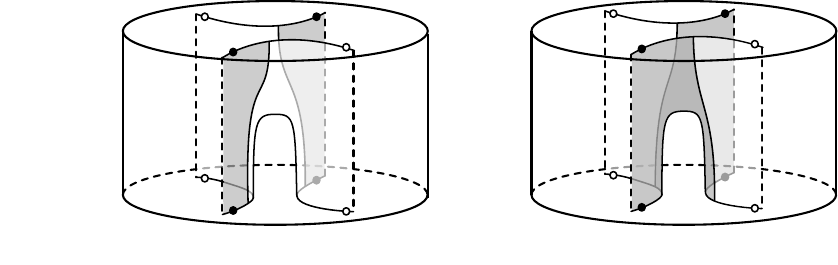}}%
    \put(0.07232596,0.07676952){\color[rgb]{0,0,0}\makebox(0,0)[t]{\lineheight{1.25}\smash{\begin{tabular}[t]{c}$(Y,\bL)$\end{tabular}}}}%
    \put(0.07232596,0.26522032){\color[rgb]{0,0,0}\makebox(0,0)[t]{\lineheight{1.25}\smash{\begin{tabular}[t]{c}$(Y,\bL(B))$\end{tabular}}}}%
    \put(0,0){\includegraphics[width=\unitlength,page=2]{fig51.pdf}}%
    \put(0.32857018,0.00586722){\color[rgb]{0,0,0}\makebox(0,0)[t]{\lineheight{1.25}\smash{\begin{tabular}[t]{c}$F_{B}^{\zs}$\end{tabular}}}}%
    \put(0.82456458,0.00940464){\color[rgb]{0,0,0}\makebox(0,0)[t]{\lineheight{1.25}\smash{\begin{tabular}[t]{c}$F_{B}^{\ws}$\end{tabular}}}}%
  \end{picture}%
\endgroup%

\caption{\textbf{Decorated link cobordisms corresponding to $F_{B}^{\zs}$ and $F_{B}^{\ws}$.} The underlying 4-manifold is $[0,1]\times Y$. Outside of the region shown, all of the dividing arcs are of the form $[0,1]\times \{p\}$, for  a point $p\in L\setminus (\ws\cup \zs)$.\label{fig::51}}
\end{figure}

Finally, we need to describe the cobordism maps which add or remove an adjacent pair of basepoints on the link. These are the \emph{quasi-stabilization} maps, and are described in the context of the link Floer TQFT in \cite{ZemCFLTQFT}*{Section~4}. Suppose that $\cH=(\Sigma,\as,\bs,\ws,\zs)$ is a diagram for $\bL$ and that $C$ is a component of $L\setminus (\ws\cup \zs)$ which is contained in the $\as$ handlebody $U_{\a}$. Suppose that $\d C=\{w',z'\}$. Let $w$ and $z$ be two new basepoints, contained in $C$. Let $\bL^+_{w,z}$ denote the link $(L,\ws\cup \{w\}, \zs\cup \{z\})$. We can form a diagram $\cH^+$ for $(Y,\bL_{w,z}^+)$, as follows.  Let $A\subset \Sigma\setminus \as$ be the connected component containing $w'$ and $z'$. Let $\alpha_s$ be a simple closed curve in $A$ which divides $A$ into two components, one of which contains $w'$, and the other contains $z'$. We then pick an arbitrary point on $\alpha_s$ (away from the $\bs$ curves), and add a very small $\beta_0$ curve, which intersects $\alpha_s$ in two points, and bounds a small disk. The disk bounded by $\beta_0$ is split into two bigons by $\alpha_s$. In one bigon, we put $w$, and in the other, we put $z$. The curves $\alpha_s$ and $\beta_0$ intersect in two points, which are distinguished by their $\gr_{\ws}$ and $\gr_{\zs}$ gradings. Let us write $\theta^{\ws}$ for the top $\gr_{\ws}$ graded intersection point, and $\theta^{\zs}$ for the top $\gr_{\zs}$ graded intersection point. The curves $\alpha_s$, $\beta_0$, and the intersection points $\theta^{\ws}$ and $\theta^{\zs}$ are shown in Figure~\ref{fig::52}.

\begin{figure}[ht!]
\centering
\begingroup%
  \makeatletter%
  \providecommand\color[2][]{%
    \errmessage{(Inkscape) Color is used for the text in Inkscape, but the package 'color.sty' is not loaded}%
    \renewcommand\color[2][]{}%
  }%
  \providecommand\transparent[1]{%
    \errmessage{(Inkscape) Transparency is used (non-zero) for the text in Inkscape, but the package 'transparent.sty' is not loaded}%
    \renewcommand\transparent[1]{}%
  }%
  \providecommand\rotatebox[2]{#2}%
  \newcommand*\fsize{\dimexpr\f@size pt\relax}%
  \newcommand*\lineheight[1]{\fontsize{\fsize}{#1\fsize}\selectfont}%
  \ifx\svgwidth\undefined%
    \setlength{\unitlength}{135.02655489bp}%
    \ifx\svgscale\undefined%
      \relax%
    \else%
      \setlength{\unitlength}{\unitlength * \real{\svgscale}}%
    \fi%
  \else%
    \setlength{\unitlength}{\svgwidth}%
  \fi%
  \global\let\svgwidth\undefined%
  \global\let\svgscale\undefined%
  \makeatother%
  \begin{picture}(1,0.62979738)%
    \lineheight{1}%
    \setlength\tabcolsep{0pt}%
    \put(0,0){\includegraphics[width=\unitlength,page=1]{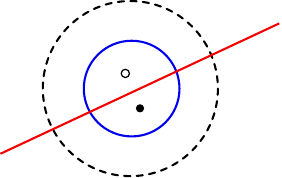}}%
    \put(0.46259129,0.38703895){\color[rgb]{0,0,0}\makebox(0,0)[lt]{\lineheight{0}\smash{\begin{tabular}[t]{l}$z$\end{tabular}}}}%
    \put(0.51537227,0.25458235){\color[rgb]{0,0,0}\makebox(0,0)[lt]{\lineheight{0}\smash{\begin{tabular}[t]{l}$w$\end{tabular}}}}%
    \put(0.29992594,0.22437993){\color[rgb]{0,0,0}\makebox(0,0)[rt]{\lineheight{1.25}\smash{\begin{tabular}[t]{r}$\theta^{\ws}$\end{tabular}}}}%
    \put(0.6292629,0.40808606){\color[rgb]{0,0,0}\makebox(0,0)[lt]{\lineheight{1.25}\smash{\begin{tabular}[t]{l}$\theta^{\zs}$\end{tabular}}}}%
    \put(0.84231979,0.42843075){\color[rgb]{1,0,0}\makebox(0,0)[lt]{\lineheight{1.25}\smash{\begin{tabular}[t]{l}$\alpha_s$\end{tabular}}}}%
    \put(0.54621567,0.4913336){\color[rgb]{0,0,1}\makebox(0,0)[lt]{\lineheight{1.25}\smash{\begin{tabular}[t]{l}$\beta_0$\end{tabular}}}}%
  \end{picture}%
\endgroup%

\caption{\textbf{A local picture of a quasi-stabilized Heegaard diagram.} }\label{fig::52}
\end{figure}
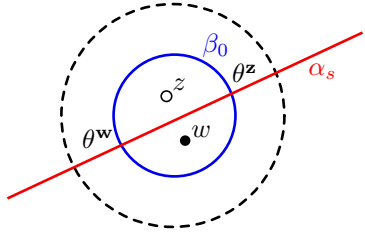

For appropriately chosen colorings $\sigma$ and $\sigma'$, there are two positive quasi-stabilization maps
\begin{equation}
S_{w,z}^+, T_{w,z}^+\colon \cCFL^\infty(Y,\bL^{\sigma},\frs)\to \cCFL^\infty(Y,(\bL_{w,z}^+)^{\sigma'}, \frs),\label{eq:positivequasistabilization}
\end{equation}
defined by the formulas
\begin{equation}
S_{w,z}^+(\xs):=\xs\otimes \theta^{\ws}\qquad \text{and} \qquad T_{w,z}^+(\xs)=\xs\otimes \theta^{\zs},\label{eq:quasi-stabilizationformula1}
\end{equation}
extended linearly over the ring $\cR_{\bmP}^\infty$.

There are also two negative quasi-stabilization maps, with the opposite domain and codomain as $S_{w,z}^+$ and $T_{w,z}^+$, defined via the formulas
\begin{equation}
S_{w,z}^-(\xs\otimes \theta^{\ws})=0,\quad S_{w,z}^-(\xs\otimes \theta^{\zs})=\xs,\quad T_{w,z}^-(\xs\otimes \theta^{\ws})=\xs \quad \text{and} \quad T_{w,z}^-(\xs\otimes \theta^{\zs})=0. \label{eq:quasi-stabilizationformula2}
\end{equation}
The quasi-stabilization maps correspond to the decorated link cobordisms shown in Figure~\ref{fig::6}. The requirement on the colorings $\sigma$ and $\sigma'$ in Equation~\eqref{eq:positivequasistabilization} corresponds exactly to $\sigma$ and $\sigma'$ being induced by a coloring of the associated decorated link cobordism in Figure~\ref{fig::6}. See \cite{ZemCFLTQFT}*{Corollary~4.4} for more on the coloring requirements and the quasi-stabilization maps. 

\begin{figure}[ht!]
\centering
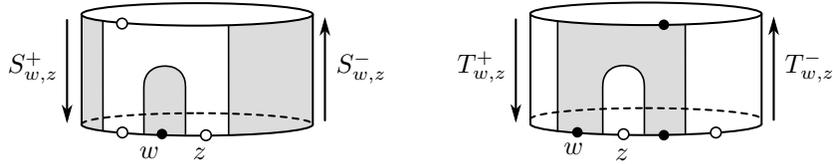
\caption{\textbf{Decorated link cobordisms for the quasi-stabilization maps $S_{w,z}^{+}$, $S_{w,z}^-$,  $T_{w,z}^{+}$ and $T_{w,z}^-$.} The underlying 4-manifolds are $[0,1]\times Y$.}\label{fig::6}
\end{figure}

\subsection{Grading changes of elementary link cobordism maps}

\label{subsec:gradingchangeelementary}

In this section, we compute the grading changes induced by the link cobordism maps for elementary link cobordisms.

We begin with the quasi-stabilization maps:

\begin{lem}\label{lem:quasi-stabgradingchange}The quasi-stabilization maps are graded and satisfy
\[
A_S(S_{w,z}^\circ(\xs))_j-A_S(\xs)_j=\tfrac{1}{2}\delta(J(K),j)\qquad \text{and}\qquad  A_S(T_{w,z}^\circ(\xs))_j-A_S(\xs)_j=-\tfrac{1}{2}\delta(J(K),j),
\]
for $\circ\in \{+,-\}$, and where $K$ denotes the component of $L$ which contains $w$ and $z$. Furthermore, 
\[
\gr_{\ws}(S_{w,z}^\circ(\xs))-\gr_{\ws}(\xs)=\gr_{\zs}(T_{w,z}^\circ(\xs))-\gr_{\zs}(\xs)=+\tfrac{1}{2}
\]
and
\[
\gr_{\zs}(S_{w,z}^\circ(\xs))-\gr_{\zs}(\xs)=\gr_{\ws}(T_{w,z}^\circ(\xs))-\gr_{\ws}(\xs)=-\tfrac{1}{2}
\]

\end{lem}
\begin{proof}We will focus on the Alexander grading formula, since the proof of the Maslov grading formula is similar.

We start with a parametrized Kirby diagram $\bP=(\phi_0,\lambda,\bS_1,f)$ for $(Y,\bL)$ and a Heegaard triple $\cT=(\Sigma,\as,\bs,\bs',\ws,\zs)$ subordinate to a $\beta$-bouquet of $\bS_1$. We can  quasi-stabilize $\cT$ to get a triple $\cT^+=(\Sigma,\as\cup \{\alpha_s\}, \bs\cup \{\beta_0\}, \bs'\cup \{\beta_0'\}, \ws\cup \{w\}, \zs\cup \{z\})$ which is subordinate to the same $\beta$-bouquet for $\bS_1$. In $\cT^+$, the curves $\beta_0$ and $\beta_0'$ both bound small disks, and intersect each other in two points, $\theta^+_{\beta_0\beta_0'}$ and $\theta^-_{\beta_0\beta_0'}$. The configuration is shown in Figure~\ref{fig::6}.

Suppose $\ys_{\a\b'}\in \bT_{\alpha}\cap \bT_{\beta'}$. Pick $\xs_{\alpha\beta}\in \bT_{\alpha}\cap \bT_{\beta}$, $\Theta_{\beta\beta'}\in \bT_{\beta}\cap \bT_{\beta'}$, and a homology class $\psi\in \pi_2(\xs_{\alpha\beta}, \Theta_{\beta\beta'}, \ys_{\a\b'})$. 

To compute the Alexander grading of $S_{w,z}^+(\ys_{\a\b'})=\ys_{\a\b'}\times \theta^{\ws}_{\alpha_s\beta_0'}$, we will use the homology class $\psi^+\in \pi_2(\xs_{\alpha\beta}\times \theta^{\ws}_{\alpha_s\beta_0}, \Theta_{\beta\beta'}\times \theta^+_{\beta_0\beta_0'}, \ys_{\a\b'}\times \theta^{\ws}_{\alpha_s\beta_0'})$ shown in Figure~\ref{fig::53}. The homology class $\psi^+$ is formed by taking a class $\psi\in \pi_2(\xs_{\a\b}, \Theta_{\b\b'}, \ys_{\a\b'})$, and adjoining a small triangle in the quasi-stabilization region.

\begin{figure}[ht!]
\centering
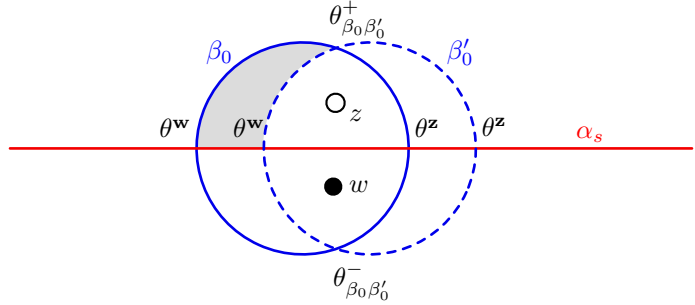
\caption{\textbf{The quasi-stabilization region of the triple $\cT^+$, and the triangle class $\psi^+$.} The subscripts on the intersection points labeled $\theta^{\ws}$ and $\theta^{\zs}$ have been suppressed.}\label{fig::53}
\end{figure}

 By explicit examination of the formula defining the Alexander grading in Equation~\eqref{eq:defabsgrading}, we  compute
 \begin{equation}
\begin{split}
&A(S_{w,z}^+(\ys_{\a\b'}))_j-A(\ys_{\a\b'})_j\\
=&A(\ys_{\a\b'}\times \theta^{\ws}_{\alpha_s\beta_0'})_j-A(\ys_{\a\b'})_j
\\=&\tilde{A}(\xs_{\a\b}\times \theta^{\ws}_{\a_s\b_0})_j-\tilde{A}(\xs_{\a\b})_j+\tilde{A}(\Theta_{\b\b'}\times \theta^{+}_{\b_0\b'_0})_j-\tilde{A}(\Theta_{\b\b'})_j\\
 &+(n_{\ws\cup \{w\}}-n_{\zs\cup \{z\}})_j(\psi^+)
-(n_{\ws}-n_{\zs})_j(\psi)\\
=&\tilde{A}(\xs_{\a\b}\times \theta^{\ws}_{\a_s\b_0})_j-\tilde{A}(\xs_{\a\b})_j+\tilde{A}(\Theta_{\b\b'}\times \theta^{+}_{\b_0\b'_0})_j-\tilde{A}(\Theta_{\b\b'})_j.\label{eq:gradingchangeafterquasistabilization}
\end{split}
\end{equation}
However by the definition of $\tilde{A}$ from Section~\ref{sec:absolutegradingsunknots}, the last line of Equation~\eqref{eq:gradingchangeafterquasistabilization} is $+\tfrac{1}{2}$ if $J$ maps $w$ and $z$ to $j$, and 0 otherwise.

The above argument can be modified to compute the grading changes of the other quasi-stabilization maps $S_{w,z}^-$, $T_{w,z}^+$ and $T_{w,z}^-$. The Maslov grading formulas are proven using a similar argument.
\end{proof}

We now consider the grading changes associated to the band maps. Suppose that $L$ is a link, $J\colon L\to \bJ$ is an indexing and $S$ is a generalized $\bJ$-Seifert surface. If $B$ is an oriented band, whose ends are on components of $L$ which are given the same index by $J$, then $S\cup B$ is a generalized $\bJ$-Seifert surface for $L(B)$. Note that $S\cup B$ may not be an embedded Seifert surface, though that's not a requirement for a generalized $\bJ$-Seifert surface (see Definition~\ref{def:generalizedJSeifertsruface}).

\begin{lem}\label{lem:bandmapsinducedcorrectgrading}Suppose that $(Y,\bL)$ is a multi-based link with an indexed, type-partitioned coloring $(\sigma,J)$, with index set $\bJ$, and $B$ is an $\alpha$-band for $\bL$ in $Y$. If $(\sigma,J)$ is compatible with one of the decorated link cobordisms shown in Figure~\ref{fig::51}, then the maps $F_{B}^{\ws}$ and $F_{B}^{\zs}$ are graded with respect to the Alexander multi-grading over $\bJ$, and satisfy 
\[
A_{S\cup B}(F_{B}^{\zs}(\xs))_j-A_S(\xs)_j=+\frac{1}{2} \delta(j,j_0) \qquad \text{and} \qquad A_{S\cup B}(F_{B}^{\ws}(\xs))_j-A_S(\xs)_j=-\frac{1}{2} \delta(j,j_0),
\]
where $j_0$ denotes the index assigned to the link components that $B$ is attached to. If $\sigma$ is a type-partitioned coloring of $\bL$ which is compatible with the appropriate decorated link cobordism in Figure~\ref{fig::51}, then the band maps are graded with respect to the Maslov gradings, and satisfy
\[
\gr_{\ws}(F_{B}^{\zs}(\xs))-\gr_{\ws}(\xs)=\gr_{\zs}(F_B^{\ws}(\xs))-\gr_{\zs}(\xs)=0,
\]
and
\[
\gr_{\ws}(F_B^{\ws}(\xs))-\gr_{\ws}(\xs)=\gr_{\zs}(F_B^{\zs}(\xs))-\gr_{\zs}(\xs)=-1.
\]
\end{lem}

\begin{proof} We will focus on the Alexander grading change of $F_B^{\zs}$.

First,  in general, if $(\Sigma,\as',\as,\bs,\ws,\zs)$ is any triple subordinate to a band $B$ attached to $\bL$ in $Y$, and $\ys_{\a\b}\in \bT_{\a}\cap \bT_{\b}$ and $\ys_{\a'\b}\in \bT_{\a'}\cap \bT_{\b}$ are intersection points representing the same $\Spin^c$ structure in $Y$, then there is a homology class $\psi_{\a'\a\b}\in \pi_2( \Theta_{\a'\a}^{\ws}, \ys_{\a\b}, \ys_{\a'\b})$ by \cite{OSDisks}*{Proposition~8.5}. Recalling from Lemma~\ref{lem:uniqueembeddingintoW} that $X_{\a'\a\b}$ becomes $[0,1]\times Y$ after filling $Y_{\a'\a}$ with 3- and 4-handles, any other triangle in $\pi_2( \Theta_{\a'\a}^{\ws}, \ys_{\a\b}, \ys_{\a'\b})$ can be obtained by splicing disks into the ends of $\psi_{\a'\a\b}$. Hence, it follows that the quantity
\begin{equation}
A_{S\cup B}(\ys_{\a'\b})_j+(n_{\zs}-n_{\ws})_j(\psi_{\a'\a\b'})-A_S(\ys_{\a\b})_j \label{eq:formalgradinchangeofband}
\end{equation}
is independent of the intersection points $\ys_{\a'\b}$, $\ys_{\a\b}$ and the triangle $\psi_{\a'\a\b}$. Furthermore, any two Heegaard triples subordinate to $B$ can related by a set of moves similar to those in Lemma~\ref{lem:movesbetweenbetasubordinatediagrams} (see \cite{ZemCFLTQFT}*{Lemma~6.3}) and hence an associativity argument like the one in Lemma~\ref{lem:Ainvariantformsurgerytriple} implies that the quantity in Equation~\eqref{eq:formalgradinchangeofband} is independent of the Heegaard triple $(\Sigma,\as',\as,\bs,\ws,\zs)$ subordinate to $B$.

We now claim that we can choose a parametrized Kirby diagram $\bP$ so that the set $\bL_{\a\b}\cup B_0$ can be isotoped into a plane in $S^3$. To construct such a $\bP$, we work backwards, and start with an unlink $U$ embedded in a plane in $S^3$, and define $B_0$ to be a planar band connecting or separating two components of $U$. We let $S^3_{U\cup B_0}$ denote $S^3\setminus (N(U)\cup N(B_0))$, and  $Y_{L\cup B}$ denote $Y\setminus N(L\cup B)$.  The 3-manifolds $S^3_{U\cup B_0}$ and $Y_{L\cup B}$ each have one genus 2 boundary component, as well as the same number of torus boundary components. Let $D$ and $D'$ denote two compressing disks in $N(L\cup B)$  which are disjoint from $L$ and $L(B)$, respectively. Let $D_0$ and $D_0'$ denote analogous compressing disks in $N(U\cup B_0)$. We pick a diffeomorphism 
\[
\phi\colon \d S^3_{U\cup B_0}\to \d  Y_{L\cup B}
\] which maps meridians of $U$ to meridians of $L$, and sends $\d D$ to $\d D_0$ and sends $\d D'$ to $\d D_0'$. We now pick a parametrized surgery datum (Definition~\ref{def:parsurgdata}) for the pair $(S^3_{U\cup B_0}, Y_{L\cup B}, \phi)$, which then induces a parametrized Kirby diagram $\bP$ with the stated properties.

Using the parametrized Kirby diagram $\bP$ constructed in the previous paragraph, we construct a Heegaard quadruple $(\Sigma,\as',\as,\bs,\bs',\ws,\zs)$ such that
\begin{enumerate}
\item $(\Sigma,\as,\bs,\bs',\ws,\zs)$ is subordinate to the framed link $\bS_1\subset S^3\setminus U$.
\item $(\Sigma,\as',\as,\bs',\ws,\zs)$ is subordinate to the band $B\subset Y$.
\end{enumerate}
Since $B$ is contained in the $\alpha$-handlebody, there is an induced band $B_0$ for the link $\bL_{\a\b}$ inside of $Y_{\a\b}\iso Y_{\a'\b}\iso S^3$. We note that $\bL_{\a\b}(B_0)=\bL_{\a'\b}$. Furthermore, the triple $(\Sigma,\as',\as,\bs,\ws,\zs)$ 
is subordinate to the band $B_0$.

 Let $\ys_{\a\b'}\in \bT_{\a}\cap \bT_{\b'}$,  $\Theta^{\ws}_{\a'\a}\in \bT_{\a'}\cap \bT_{\a}$, and let $\psi_{\a'\a\b'}\in \pi_2(\Theta_{\a'\a}^{\ws},\ys_{\a\b'},\ys_{\a'\b'})$ be any homology class of triangles (such as one which might be counted by the map $F_{B}^{\ws}$). Next, pick intersection points $\xs_{\a\b}\in \bT_{\a}\cap \bT_{\b}$ and $\Theta_{\b\b'}\in \bT_{\b}\cap\bT_{\b'}$ with $\frs_{\ws}(\Theta_{\b\b'})$ torsion, as well as a class $\psi_{\a\b\b'}\in \pi_2(\xs_{\a\b}, \Theta_{\b\b'}, \ys_{\a\b'})$. 

Arguing as in Lemma~\ref{lem:Ainvariantformsurgerytriple}, we can find an intersection point $\xs_{\a'\b}\in \bT_{\a'}\cap \bT_{\b}$, as well as homology classes $\psi_{\a'\a\b}\in \pi_2(\Theta_{\a'\a}^{\ws}, \xs_{\a\b}, \xs_{\a'\b})$ and $\psi_{\a'\b\b'}\in \pi_2(\xs_{\a'\b}, \Theta_{\b\b'}, \ys_{\a'\b'})$ such that
\begin{equation}
\psi_{\a'\a\b}+\psi_{\a'\b\b'}=\psi_{\a'\a\b'}+\psi_{\a\b\b'}.\label{eq:equlityoftrianglesbandgrad}
\end{equation}
We note that $(\Sigma,\as',\bs,\bs',\ws,\zs)$ is a Heegaard triple subordinate to the framed link $\bS_1$, and $(\Sigma,\as',\bs,\ws,\zs)$ is a diagram for a multi-based unlink in $S^3$. Hence $(\Sigma,\as',\bs,\bs',\ws,\zs)$ can be used to compute the Alexander grading of the intersection point $\ys_{\a'\b'}$.

Arguing as in Lemma~\ref{lem:Ainvariantformsurgerytriple}, using the definition of the Alexander grading, we compute that
\begin{equation}
\begin{split} A_{S\cup B}(\ys_{\a'\b'})_j-A_S(\ys_{\a\b'})_j=&\tilde{A}(\xs_{\a'\b})_j+\tilde{A}(\Theta_{\b\b'})_j+(n_{\ws}-n_{\zs})_j(\psi_{\a'\b\b'})\\
&-\tilde{A}(\xs_{\a\b})_j-\tilde{A}(\Theta_{\b\b'})_j-(n_{\ws}-n_{\zs})_j(\psi_{\a\b\b'}).
\end{split}\label{eq:equlityoftrianglesbandgrad2}
\end{equation}
Combining Equations~\eqref{eq:equlityoftrianglesbandgrad} and \eqref{eq:equlityoftrianglesbandgrad2} and rearranging, we see that
\begin{equation}
A_{S\cup B}(\ys_{\a'\b'})_j+(n_{\zs}-n_{\ws})_j(\psi_{\a'\a\b'})-A_S(\ys_{\a'\b})_j=\tilde{A}(\xs_{\a'\b})_j+(n_{\zs}-n_{\ws})_j(\psi_{\a'\a\b})-\tilde{A}(\xs_{\a\b})_j.\label{eq:FBwsgradingchange}
\end{equation}
We note that quantity on the left side of Equation~\eqref{eq:FBwsgradingchange} is exactly the formal grading change of the map $F_{B}^{\zs}$.

Next, we note that the quantity on the right side of  Equation~\eqref{eq:FBwsgradingchange} is exactly the formal Alexander grading change of the map $F_{B_0}^{\zs}$, for the band $B_0$ attached to $\bU\subset S^3$.

As we argued earlier, the formal grading change of $F_{B_0}^{\zs}$ is independent of the choice of homology class of triangle, as well as the Heegaard triple subordinate to $B_0$, and hence we can pick any convenient Heegaard triple and compute the grading change for any convenient homology class of triangle. We perform the model computation  in Figure~\ref{fig::54}. We can pick a triple $(\Sigma,\as',\as,\bs,\ws,\zs)$ subordinate to the band $B_0$ such there is an annular subregion which appears as in Figure~\ref{fig::54}, and outside the annular region, the $\as'$ curves are all small Hamiltonian isotopies of the $\as$ curves. We pick a triangle class $\psi\in \pi_2(\Theta_{\a'\a}^{\ws}, \Theta_{\a\b}^{\ws},\Theta_{\a'\b}^{\ws})$. Outside of the annular region shown in Figure~\ref{fig::54}, we can assume that the class $\psi$ consists of only small triangles, and inside the annular region, we assume that the class $\psi$ is one of the two shown in Figure~\ref{fig::54}, depending on whether $B_0$ splits a component of $U$ into two components, or connects two components. It is straightforward to compute in both cases, using Figure~\ref{fig::54} and the definition of the Alexander grading $\tilde{A}$ for unlinks in $S^3$ from Section~\ref{sec:absolutegradingsunknots}, that
\begin{equation}
\tilde{A}(\Theta^{\ws}_{\a'\b})_j+(n_{\zs}-n_{\ws})_j(\psi)-\tilde{A}(\Theta^{\ws}_{\a\b})_j=\frac{1}{2}\delta(j,j_0).\label{eq:gradingchangeformulafrommodelcomputation}
\end{equation}
As we described previously, Equation~\eqref{eq:gradingchangeformulafrommodelcomputation} is equal to the right hand side of Equation~\eqref{eq:FBwsgradingchange}, and the left hand side of Equation~\eqref{eq:FBwsgradingchange} is exactly the formal Alexander grading change of $F_{B}^{\zs}$, establishing the grading formula for $F_{B}^{\zs}$. 

A symmetrical argument can be used to compute the Alexander grading change of the band maps $F_{B}^{\ws}$.  Finally, we note that a simple modification of the above argument can be used to compute the Maslov grading changes.

\end{proof}

\begin{figure}[ht!]
	\centering
	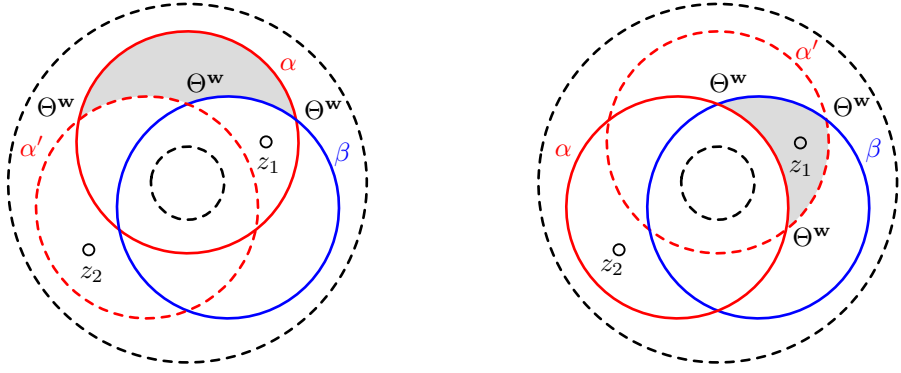
	\caption{\textbf{A model computation to compute the grading change of the band map $F_{B_0}^{\zs}$.} The left side corresponds to the case that $B_0$ is connecting to components of $L$, and the right side corresponds to the case that $B_0$ is splitting a component of $L$ into two components.}\label{fig::54}
\end{figure}

\subsection{Proof of Theorem~\ref{thm:generalAlexandergradingformula}}
\label{subsection:proofofgradingformula}
We can now prove our main grading theorem:

\begin{proof}[Proof of Theorem~\ref{thm:generalAlexandergradingformula}] 
Firstly, it is straightforward to check that each of the formulas we've described for the Alexander and Maslov gradings are additive under composition of cobordism. Hence, it is sufficient to check the grading change formulas for the elementary cobordisms in Section~\ref{sec:overviewofmaps}.

We will focus on proving the theorem for the Alexander grading, since the proof of the Maslov grading formulas is an easy modification.

We begin with the 0-handle and 4-handle maps. Clearly the 0-handle and 4-handle maps induce grading change zero. To verify our grading formula, note that the link cobordism surface consists of a cylindrical decorated link cobordism inside of $[0,1]\times Y$, together with a disk in $B^4$ which is split into two components by its single dividing arc. 
It is straightforward to compute that the expected Alexander grading change in the theorem statement is also 0, agreeing with the actual grading change.

Next, we consider 1-handle and 3-handle cobordisms. For simplicity, we will focus on the 1-handle maps, since the 3-handle maps are dual to the 1-handle maps. For a 1-handle which connects two components $(Y_1,\bL_1)$ and $(Y_2,\bL_2)$, we can obtain a parametrized Kirby diagram $\bP$ for $(Y_1\# Y_2, \bL_1\cup \bL_2)$ by taking the connected sum of a parametrized Kirby diagram $\bP_1$ for $(Y_1,\bL_1)$, and a parametrized Kirby diagram $\bP_2$ for $(Y_2,\bL_2)$. A Heegaard triple $\cT$ for $\bP$ can be obtained by taking the connected sum of a triple $\cT_1=(\Sigma_1,\as_1,\bs_1,\bs_1')$ for $\bP_1$ and a triple $\cT_2=(\Sigma_2,\as_2,\bs_2,\bs_2')$ for $\bP_2$. In the connected sum annulus, we add three new curves to the Heegaard triple, $\alpha_0$, $\beta_0$ and $\b'_0$, which are homologically essential in the annulus and such that $\alpha_0\cap \beta_0,$ $\beta_0\cap \beta_0'$ and $\alpha_0\cap \beta_0'$ each consists of exactly two points. We will write $\theta_{\a_0\b_0}^+$ and $\theta_{\a_0\b_0}^-$ for the two points of $\alpha_0\cap \beta_0$, and similarly for $\beta_0\cap \beta_0'$ and $\alpha_0\cap \beta_0'$. We pick classes $\psi_{\a_1\b_1\b_1'}\in \pi_2(\xs_{\a_1\b_1}, \Theta_{\b_1\b_1'},\ys_{\a_1\b_1'})$ and
$\psi_{\a_2\b_2\b_2'}\in \pi_2(\xs_{\a_2\b_2}, \Theta_{\b_2\b_2'},\ys_{\a_2\b_2'})$, assuming  $\Theta_{\b_1\b_1'}$ and $\Theta_{\b_2\b_2'}$ represent the torsion $\Spin^c$ structure. For convenience, we can assume $\psi_{\a_1\b_1\b_1'}$ and $\psi_{\a_2\b_2\b_2'}$ have zero multiplicity at the connected sum points. We construct a triangle class
\[
\psi^+\in \pi_2(\xs_{\a_1\b_1}\times \theta_{\a_0\b_0}^+\times \xs_{\a_2\b_2}, \Theta_{\b_1\b_1'}\times \theta_{\b_0\b_0'}^+\times \Theta_{\b_2\b_2'},\ys_{\a_1\b_1'}\times \theta_{\a_0\b_0'}^+\times \ys_{\a_2\b_2'})
\]
on $\cT^+$ by taking the connected sum of $\psi_{\a_1\b_1\b_1'}$ and $\psi_{\a_2\b_2\b_2'}$ and inserting a small triangle class in the connected sum region.  A straightforward computation using $\psi^+,$ $\psi_{\a_1\b_1\b_1'}$ and $\psi_{\a_2\b_2\b_2'}$ shows that
\[
A_{S_1\cup S_2}(\ys_{\a_1\b_1'}\times \theta^+_{\a_0\b_0'}\times \ys_{\a_2\b_2'})_j=A_{S_1}(\ys_{\a_1\b_1'})_j+A_{S_2}(\ys_{\a_2\b_2'})_j,
\]
whenever $S_1$ and $S_2$ are generalized $\bJ$-Seifert surfaces of $\bL_1$ and $\bL_2$ in $Y_1$ and $Y_2$, respectively. It follows that the 1-handle maps are 0-graded with respect to the Alexander grading, agreeing with the formula in the theorem statement. A similar argument works for the 3-handle maps, when a 3-handle splits a component of $Y$ into two components.

For a 1-handle which is attached with both feet on a single component of $Y$, we argue as follows. If $\bP$ is a parametrized Kirby diagram for $(Y,\bL)$, with framed link $\bS_1$, then a parametrized Kirby diagram $\bP'$ for $(Y\# (S^1\times S^2),\bL)$ can be obtained by adding a 0-framed unknot to $\bS_1$, which is unlinked from $\bS_1\cup U$. If $\cT=(\Sigma,\as,\bs,\bs',\ws,\zs)$ is a surgery triple for $\bP$, a surgery triple for $\bP'$ can be obtained by taking the connected sum of a genus one Heegaard triple $(T^2,\alpha_0,\beta_0,\beta_0')$, where $\beta_0$ and $\beta_0'$ are small Hamiltonian isotopies of each other, intersecting twice, and $\alpha_0$ is a curve on $T^2$ which intersects each of $\beta_0$ and $\beta_0'$ exactly once. An easy computation shows that if $\ys_{\a\b'}\in \bT_{\a}\cap \bT_{\b'}$, and $S$ is a generalized $\bJ$-Seifert surface for $\bL$ in $Y$, then
\[
A_S(\ys_{\a\b'}\times \theta^+_{\b\b'})_j=A_S(\ys_{\a\b'})_j,
\]
showing that the 1-handle maps induce Alexander grading change 0, agreeing with the formula from the theorem statement.

We now consider the 2-handle maps. The fact that the 2-handle maps induce the stated grading change is essentially a tautology, and follows from an associativity argument similar to the proof of Lemma~\ref{lem:bandmapsinducedcorrectgrading}.

Finally, the quasi-stabilization and band maps have the expected Alexander grading changes by Lemmas~\ref{lem:quasi-stabgradingchange} and \ref{lem:bandmapsinducedcorrectgrading} (note that the corresponding decorated link cobordisms are shown in Figures~\ref{fig::51} and \ref{fig::6}).

Having established the Alexander grading change for each elementary link cobordism, the grading change follows for a general link cobordism. The formula for the Maslov grading change is a straightforward modification.
\end{proof}

\section{Equivalence with Ozsv\'{a}th and Szab\'{o}'s construction}
\label{sec:equivalence}

In this section, we prove the following:

\begin{prop}\label{prop:agreeswithOSconstruction}If $\bL$ is a multi-based link in $S^3$, and each component of $\bL$ has exactly two basepoints, then the Alexander multi-grading $A$ defined in Section~\ref{sec:defabsolutegradings} coincides with the Alexander multi-grading defined by Ozsv\'{a}th and Szab\'{o} in \cite{OSLinks}.
\end{prop}

We warn the reader that Proposition~\ref{prop:agreeswithOSconstruction} is stated only for the Alexander multi-grading, and not the Maslov gradings. For links in $S^3$, there is a canonical choice of absolute Alexander grading, characterized by a conjugation symmetry property. For the Maslov gradings, there are several natural normalization conventions, depending on ones perspective. Nonetheless, for doubly based knots in $S^3$, the Maslov grading $\gr_{\ws}$ coincides with Ozsv\'{a}th and Szab\'{o}'s homological grading; See Section~\ref{subsec:usersguide}.

\subsection{Gradings using $\alpha$-bouquets}
\label{subsec:usealphabouqeuts}
We defined the absolute grading in Section~\ref{sec:defabsolutegradings} using $\beta$-bouquets of framed links in $S^3$, but it will be useful to know that $\alpha$-bouquets can be used as well. We will write $A^{\alpha}_{S}$ for the gradings defined using $\alpha$-bouquets, and $A^{\beta}_S$ for the gradings defined using $\beta$-bouquets. Similarly there are Maslov gradings $\gr_{\ws}^{\alpha},$ $\gr_{\ws}^{\beta}$, $\gr_{\zs}^{\alpha}$ and $\gr_{\zs}^{\beta}$.

\begin{lem}\label{lem:leftsubandrightsubgradingsequal}The absolute gradings satisfy
\[
\gr_{\ws}^\alpha=\gr_{\ws}^\beta,\qquad \gr_{\zs}^{\alpha}=\gr_{\zs}^{\beta}\qquad \text{and} \qquad A_S^\alpha=A_S^\beta.
\]
\end{lem}
\begin{proof}We focus on the equality for the Alexander gradings; the equality for the Maslov gradings can be proven similarly.

The key idea is that given a framed link $\bS_1$ in $Y\setminus L$, we can define two ``cobordism maps''
\begin{equation}
F_{Y,\bL,\bS_1,S,S'}^\alpha,\quad  F_{Y,\bL,\bS_1,S,S'}^\alpha \colon \bA(Y,\bL,\frs)\to \bA(Y(\bS_1),\bL,\frs'),
\end{equation}
whenever $\frs\in \Spin^c(Y)$ and $\frs'\in \Spin^c(Y(\bS_1))$ are $\Spin^c$ structures which have a common extension over the 2-handle cobordism $W(Y,\bS_1)$, and $S$ and $S'$ are generalized $\bJ$-Seifert surfaces in $Y$ and $Y(\bS_1)$.

To define the map $F_{Y,\bL,\bS_1,S,S'}^\beta$, we pick a $\beta$-bouquet $\cB^\beta$ for $\bS_1$, as well as a triple $\cT=(\Sigma,\as,\bs,\bs',\ws,\zs)$ which is subordinate to $\cB^\beta$. We pick a homology class $\psi\in \pi_2(\xs,\Theta,\ys)$ where $\frs_{\ws}(\xs)=\frs$, $\frs_{\ws}(\ys)=\frs'$ and $\frs_{\ws}(\Theta)$ is torsion. If $A\in \bA(Y,\bL^{(\sigma,\bJ)},\frs)$,  we define
\[
F_{Y,\bL,\bS_1,S,S'}^\beta(A)(\ve{y})_j:=A(\ve{x})_j+\tilde{A}(\Theta)_j+(n_{\ve{w}}-n_{\ve{z}})_j(\psi)+\frac{\langle c_1(\frs),[\hat{\Sigma}_j]\rangle -[\hat{\Sigma}]\cdot [\hat{\Sigma}_j]}{2}
\]
 where $\hat{\Sigma}_j$ is formed by capping off the surface $[0,1]\times  L_j\subset W(Y,\bS_1)$ with $S_j$ and $-S'_j$. A map $F_{Y,\bL,\bS_1,S,S'}^{\alpha}$ is defined analogously, using $\alpha$-bouquets.
 
 The proofs of Lemmas~\ref{lem:indoftriangle},~\ref{lem:Ainvariantformsurgerytriple} and \ref{lem:independentofbouquet} adapt to show invariance of the maps $F_{Y,\bL,S,S'}^\alpha$ and $F_{Y,\bL,S,S'}^{\beta}$ from the choice of bouquet and  Heegaard triple.

 Furthermore, we claim that the cobordism maps on gradings satisfy the following ``composition law'':
 \begin{equation}
 F_{Y,\bL,\bS_1,S,S'}^\beta(A^\beta_{S})=A^{\beta}_{S'}\qquad \text{and}\qquad F_{Y,\bL,\bS_1,S,S'}^\alpha(A^\alpha_{S})=A^\alpha_{S'}.\label{eq:abcobordismmapspreservegradings}
 \end{equation}
 Equation~\eqref{eq:abcobordismmapspreservegradings} is proven similar to the standard composition for the link cobordism maps. If $\bP$ is a parametrized Kirby diagram for  $(Y,\bL)$, with framed link $\bS_1'\subset S^3\setminus U$, one takes a Heegaard quadruple $(\Sigma,\as,\bs,\bs',\bs'',\ws,\zs)$ such that $(\Sigma,\as,\bs,\bs')$ is subordinate to a bouquet of $\bS_1'$, and $(\Sigma,\as,\bs',\bs'')$ is subordinate to a bouquet for $\bS_1\subset Y$, and $(\Sigma,\as,\bs,\bs'')$ is subordinate to a bouquet for $\bS_1'\cup \bS_1\subset S^3$. Using an associativity argument like the one in Lemma~\ref{lem:Ainvariantformsurgerytriple}, it is straightforward to establish Equation~\eqref{eq:abcobordismmapspreservegradings}.
 
 If $\bS_1$ and $\bS_1'$ are two framed links in $Y$, then adapting the associativity argument from Lemma~\ref{lem:Ainvariantformsurgerytriple} also shows that
 \begin{equation}
 F_{Y(\bS_1), \bL,\bS_1',S'', S'}^\beta\circ F_{Y,\bL,\bS_1,S,S''}^\alpha= F_{Y(\bS_1'),\bL,\bS_1,S''',S'}^\alpha\circ F_{Y,\bL,\bS_1',S,S'''}^\beta\label{eq:composingalphabetagradingmaps}
 \end{equation}
 whenever $S''$ and $S'''$ are generalized $\bJ$-Seifert surfaces for $\bL$ inside of $Y(\bS_1)$ and $Y(\bS_1')$, respectively (compare \cite{OSTriangles}*{Lemma~5.2}).
 
 If $U$ is an unlink in any $(S^1\times S^2)^{\# k}$, we will write (abusing notation slightly) $S_0$ for any Seifert surface for $U$. Recall that $\tilde{A}$ denotes the Alexander grading on the link Floer homology of the unlink in $(S^1\times S^2)^{\# k}$, which we declared in Section~\ref{sec:absolutegradingsunknots}.
  We note that it is easy to compute from the definition that
 \begin{equation}
  \tilde{A}=A^\alpha_{S_0}=A^\beta_{S_0}. \label{eq:canonicalgradingsforunknots}
 \end{equation} 
 
 We pick a parametrized Kirby diagram $\bP$ of $(Y,L)$, with framed link $\bS_1\subset S^3\setminus U$, and a diffeomorphism between $(S^3(\bS_1),U)$ and $(Y,L)$.  Let $\bS_1'$ denote the framed link consisting of  a 0-framed meridian for each component of $\bS_1$.
 
  By Equation~\eqref{eq:abcobordismmapspreservegradings}, we have
 \begin{equation}
\tilde{A}=(F^\beta_{Y, \bL, \bS_1',S,S_0}\circ F^\beta_{S^3,\bU, \bS_1, S_0, S})(\tilde{A})=F^\beta_{Y, \bL, \bS_1',S,S_0} (A_S^\beta),\label{eq:firstmanipulationofcompgradings}
 \end{equation}
where $\bU$ denotes the unlink $U$ in $S^3$, decorated with the basepoints from $\bL$ induced by $\bP$ 
 
 Since the ``cobordism maps'' on the set of Alexander gradings are isomorphisms of affine sets over $\Q^{\bJ}$,  to establish that $A^\beta_{S}=A^\alpha_S$, it is sufficient to show that they have the same evaluation under any 2-handle cobordism map. Hence, by Equation~\eqref{eq:firstmanipulationofcompgradings}, it is sufficient to show that
 \begin{equation}
 \tilde{A}=F^\beta_{Y, \bL, \bS_1',S,S_0} (A_S^\alpha).
 \end{equation}

Importantly, we note that since $\bS_1'\subset S^3$ consists of 0-framed unknots which are unlinked from $\bU$, the pair $(S^3(\bS_1'),\bU)$ is an unlink in $(S^1\times S^2)^{\# |\bS_1'|}$. Hence, by Equation~\eqref{eq:canonicalgradingsforunknots}
\[
F^\beta_{S^3,\bU,\bS_1',S_0,S_0}(\tilde{A})=F^\alpha_{S^3,\bU,\bS_1',S_0,S_0}(\tilde{A})=\tilde{A}.
\] 
 Hence we compute that
 \begin{align*}
F^\beta_{Y, \bL, \bS_1',S,S_0} (A_S^\alpha)&=(F^\beta_{Y, \bL, \bS_1',S,S_0}\circ F_{S^3,\bU,\bS_1,S_0,S}^\alpha)(\tilde{A})&&\text{(Equation~\eqref{eq:abcobordismmapspreservegradings})}\\
&=(F_{S^3(\bS_1'),\bL,\bS_1,S_0,S_0}^\alpha\circ F_{S^3,\bU,\bS_1',S_0,S_0}^\beta)(\tilde{A})&&\text{(Equation~\eqref{eq:composingalphabetagradingmaps})}\\
&=F_{S^3(\bS_1'),\bL,\bS_1,S_0,S_0}^\alpha(\tilde{A})&&\text{(Equation~\eqref{eq:canonicalgradingsforunknots})}\\
&=\tilde{A},&& \text{(Equation~\eqref{eq:abcobordismmapspreservegradings})}
 \end{align*}
completing the proof.
\end{proof}

\subsection{Conjugation symmetry}

As a step towards proving the equivalence of our construction with Ozsv\'{a}th and Szab\'{o}'s, we will analyze the interaction of our gradings with the conjugation action on link Floer homology.

If $\bL=(L,\ve{p},\ve{q})$ is a multi-based link, then we let $\bar{\bL}=(L,\ve{q},\ve{p})$ denote the same link, but with the designation of the basepoints as type-$\ws$ or type-$\zs$ switched.  There is a natural conjugation action on $\Spin^c(Y)$. If $\frs\in \Spin^c(Y)$ corresponds to the vector field $v$, then the conjugate $\Spin^c$ structure $\bar{\frs}$ corresponds to the vector field $-v$.

We now describe the conjugation action $\eta$ on $\Hat{\CFL}$. Given a Heegaard diagram $\cH=(\Sigma, \as,\bs,\ve{p},\ve{q})$ for $\bL=(L,\ve{p},\ve{q})$, we consider the conjugate diagram $\bar{\cH}=(-\Sigma, \bar{\bs},\bar{\as},\ve{q},\ve{p})$ for $(Y,\bar{\bL})$. There is a natural correspondence between the intersection points on $\cH$ and $\bar{\cH}$. We note however, that $\bar{\frs_{\cH,\ve{p}}(\ve{x})}=\frs_{\bar{\cH},\ve{p}}(\eta(\ve{x}))$, by explicit examination of the vector fields constructed by Ozsv\'{a}th and Szab\'{o} in \cite{OSDisks}. On the other hand, 
\begin{equation}
\frs_{\bar{\cH},\ve{q}}(\eta(\ve{x}))=\bar{\frs_{\cH,\ve{p}}(\ve{x})}+\PD[L], \label{eq:conjugationchangeofspinc}
\end{equation} 
from Lemma~\ref{lem:spincpoincareduallink}. By Equation~\eqref{eq:conjugationchangeofspinc}, we see that $\eta$ maps $\Hat{\CFL}(Y,\bL,\frs)$ to $\Hat{\CFL}(Y,\bar{\bL},\bar{\frs}+\PD[L])$. 

By extending $\eta$ linearly over the ring $\bF_{2}[U_{\ve{p}}, V_{\ve{q}}]$, the map $\eta$ induces an equivariant, filtered chain homotopy equivalence
\[
\cCFL^\infty(Y,\bL,\frs)\to \cCFL^\infty(Y,\bar{\bL}, \bar{\frs}+\PD[L]).
\]
In this section, we prove the following:

\begin{prop}\label{lem:conjugationactiongraded}Suppose that $\bL=(L,\ve{p},\ve{q})$ is a multi-link in $Y$ with an indexed, type partitioned coloring $(\sigma,J)$, and  $L$ is $\bJ$-null-homologous. If $S$ is a choice of generalized $\bJ$-Seifert surface for $\bL$, then the map $\eta\colon \cCFL^\infty(Y,\bL,\frs)\to \cCFL^\infty(Y,\bar{\bL},\bar{\frs})$ satisfies
\[
A_{S}(\eta(\ve{x}))=-A_{S}(\ve{x}),
\]
for homogeneously graded $\xs$. If in addition $c_1(\frs)$ is torsion, then
\[
 \gr_{\ve{w}}(\eta(\ve{x}))=\gr_{\ve{z}}(\ve{x})\qquad \text{and}\qquad \gr_{\ve{z}}(\eta(\ve{x}))=\gr_{\ve{w}}(\ve{x}).
\]
\end{prop}

\begin{proof}It is sufficient to show the claim for $\Hat{\CFL}$, since the map $\eta$ is $\bF_2[U_{\ve{p}},V_{\ve{q}}]$-equivariant.

We take a parametrized Kirby diagram $\bP=(\phi_0,\lambda,\bS_1,f)$ and a surgery triple $\cT=(\Sigma, \as,\bs,\bs',\ve{p},\ve{q})$, which is subordinate to a $\beta$-bouquet for $\bS_1$. If $\ve{y}\in \bT_{\alpha}\cap \bT_{\b'}$, and $\psi\in \pi_2(\ve{x},\Theta,\ve{y})$ is a homology class of triangles, then the absolute Alexander grading $A_S$ on $\Hat{\CFL}(\Sigma,\as,\bs',\ve{w},\ve{z})$ is defined by the formula
\[
A_{S}(\ve{y})_j=\tilde{A}(\ve{x})_j+\tilde{A}(\Theta)_j+(n_{\ve{p}}(\psi)-n_{\ve{q}}(\psi))_j+\frac{\langle c_1(\frs_{\ve{p}}(\psi)), [\hat{\Sigma}_j]\rangle-[\hat{\Sigma}]\cdot [\hat{\Sigma}_{j}]}{2}.
\] 

We can form the conjugate triple $\bar{\cT}=(-\Sigma, \bar{\bs}',\bar{\bs},\bar{\as},\ve{q},\ve{p})$ of $\cT$.  The $\beta$-bouquet for $\bS_1$ now becomes an $\alpha$-bouquet for $\bS_1$, and $\bar{\cT}$ is now subordinate to this $\alpha$-bouquet for the same framed link $\bS_1$ in $S^3\setminus U$.

Using Lemma~\ref{lem:leftsubandrightsubgradingsequal}, we can use $\bar{\cT}$ to compute the grading of $\eta(\ve{y})$. Notice that there is a canonical, orientation preserving diffeomorphism
\begin{equation}
X_{\alpha\beta\b'}\iso X_{\bar{\b}'\bar{\b}\bar{\a}}.\label{eq:conjugationdiffeosXs}
\end{equation}
The identification in Equation~\eqref{eq:conjugationdiffeosXs} respects the embedding of both $X_{\a\b\b'}$ and $X_{\bar{\b}'\bar{\b}\bar{\a}}$ into the 2-handle cobordism $W(S^3,\bS_1)$ from Lemma~\ref{lem:uniqueembeddingintoW}. The diffeomorphism from Equation~\eqref{eq:conjugationdiffeosXs} restricts to an orientation preserving diffeomorphism
\[
\Sigma_{\alpha\beta\b'}\iso \Sigma_{\bar{\b}'\bar{\b}\bar{\a}}.
\]
The homology class $\psi$ induces a class $\bar{\psi}$ on the conjugate Heegaard triple. 

The Alexander grading of $\eta(\ve{y})$ can be computed using the triple $\bar{\cT}$ and the triangle class $\bar{\psi}$, and indeed we see that
\[
A(\eta(\ve{y}))_j=\tilde{A}(\eta(\ve{x}))_j+\tilde{A}(\eta(\Theta))_j+(n_{\ve{q}}(\bar{\psi})-n_{\ve{p}}(\bar{\psi}))_j+\frac{\langle c_1(\frs_{\ve{q}}(\bar{\psi})), [\hat{\Sigma}_{j}]\rangle-[\hat{\Sigma}]\cdot [\hat{\Sigma}_{j}]}{2}.
\]

We will show that $A(\eta(\ve{y}))_j=-A(\ve{y})_j$.

 First observe that 
 \begin{equation}
 \tilde{A}(\eta(\ve{x}))_j=-\tilde{A}(\ve{x})_j\qquad\text{and} \qquad \tilde{A}(\eta(\Theta))_j=-\tilde{A}(\Theta)_j,\label{eq:conjugationcomp1}
 \end{equation}
  using the definition of $\tilde{A}$ from Section~\ref{sec:absolutegradingsunknots} and an easy model computation.

Note that the roles of $\ve{p}$ and $\ve{q}$ as type-$\ws$ and type-$\zs$ is reversed in $\bar{\cT}$. Correspondingly
\begin{equation}
(n_{\ve{q}}(\psi)-n_{\ve{p}}(\psi))_j=-(n_{\ve{p}}(\bar{\psi})-n_{\ve{q}}(\bar{\psi}))_j.\label{eq:conjugationcomp2}
\end{equation}

We now consider the homological terms involving the homology class of the surface $\Sigma_{\alpha\beta\b'}$ appearing in the formula for the grading. We note that $\frs_{\ve{p}}(\bar{\psi})=\bar{\frs_{\ve{p}}(\psi)}$. Similarly $\frs_{\ve{q}}(\bar{\psi})=\frs_{\ve{p}}(\bar{\psi})+\PD[\Sigma_{\bar{\b}'\bar{\b}\bar{\a}}]$ by Lemma~\ref{lem:poincaredualofsurface}. Note that after filling in $(Y_{\b\b'}, L_{\b\b'})$ with 3-handles and 4-handles containing standardly embedded slice disks of $L_{\b\b'}$, the pair $(X_{\a\b\b'}, \Sigma_{\a\b\b'})$ becomes $ (W(S^3,\bS_1),\Sigma)$ by Lemma~\ref{lem:uniqueembeddingintoW}.  Hence we  compute
\begin{equation}
\begin{split}\langle c_1(\frs_{\ve{q}}(\bar{\psi})), [\hat{\Sigma}_j]\rangle-[\hat{\Sigma}]\cdot [\hat{\Sigma}_j]&=\langle c_1(\frs_{\ve{p}}(\bar{\psi})+\PD[\Sigma]), [\hat{\Sigma}_j]\rangle-[\hat{\Sigma}]\cdot [\hat{\Sigma}_j]\\
&=-\langle c_1(\frs_{\ve{p}}(\psi)),[\hat{\Sigma}_j] \rangle+2\langle \PD[\Sigma],[\hat{\Sigma}_j]\rangle -[\hat{\Sigma}]\cdot [\hat{\Sigma}_j]\\
&=-\big(\langle c_1(\frs_{\ve{p}}(\psi)),[\hat{\Sigma}_j] \rangle-[\hat{\Sigma}]\cdot [\hat{\Sigma}_j]\big).
\end{split}
\label{eq:conjugationcomp3}
\end{equation}

Combining Equations~\eqref{eq:conjugationcomp1}, \eqref{eq:conjugationcomp2} and \eqref{eq:conjugationcomp3}, we see that each summand of $A(\ve{y})_j$ is changed to its negative in $A(\eta(\ve{y}))_j$, from which we conclude that $A(\eta(\ve{y}))_j=-A(\ve{y})_j$.

The claim about the Maslov gradings is proven similarly.
\end{proof}

\subsection{Proof of the equivalence}
We can now prove that our gradings coincide with Ozsv\'{a}th and Szab\'{o}'s:

\begin{proof}[Proof of Proposition~\ref{prop:agreeswithOSconstruction}]
By Proposition~\ref{lem:conjugationactiongraded}, the map $\eta$ induces an isomorphism  \begin{equation}
 \Hat{\HFL}(S^3,\bL)_{s}\iso \Hat{\HFL}(S^3,\bar{\bL})_{-s},\label{eq:conjugationisomorphism}
 \end{equation} where $s\in \Q^{L}$ denotes our Alexander multi-grading.   Since $\bL$ and $\bar{\bL}$ are isotopic links inside of $S^3$ (they are related by a half twist on each link component), $\Hat{\HFL}(S^3,\bL)$ and $\Hat{\HFL}(S^3,\bar{\bL})$ are isomorphic as multi-graded groups. By \cite{OSLinks}*{Equation~25}, the gradings defined by Ozsv\'{a}th and Szab\'{o} also satisfy Equation~\eqref{eq:conjugationisomorphism}. Since the hat version of link Floer homology groups for links in $S^3$ are non-vanishing and finitely generated over $\bF_2$, it is straightforward to see that our definition of the Alexander multi-gradings must coincide with theirs.
\end{proof}

 \section{Computations of the link cobordism maps}
 \label{sec:computations}
In this section, we perform some computations of the link cobordism maps in certain special cases. We focus on computing the map when $\cF$ is obtained by puncturing a closed surface, or computing the induced map on $\cHFL^\infty$ for more general link cobordisms.

A key computational tool is \cite{ZemCFLTQFT}*{Theorem~C}, stated below as Theorem~\ref{thm:generalreductionformula},  which computes the map induced by a decorated link cobordism when we algebraically forget about either the $\ws$ basepoints, or the $\zs$ basepoints.

Throughout this section, we will focus on colorings of links  $\sigma\colon \ws\cup \zs\to \bmP$ where $\bmP$ has exactly two colors, and all $\ws$ basepoints are assigned the variable $U$, and all $\zs$ basepoints are assigned the variable $V$. For such colorings, we write $\cR^-$ for the ring 
\[
\cR^-:=\cR^-_{\bmP}\iso \bF_2[U,V].
\]

For notational simplicity, we omit the coloring $\sigma$ from the notation in this section.

\subsection{A distinguished element of $\bF_2[U]\otimes \Lambda^*(H_1(\Sigma;\bF_2))$}

 In this section we describe a distinguished element of $\bF_2[U]\otimes_{\bF_2} \Lambda^*(H_1(\Sigma;\bF_2))$, which appears frequently in our computations, and is a familiar expression from Seiberg--Witten theory.

\begin{define}Suppose $\Sigma$ is an oriented surface of genus $g$, with either zero or one boundary component. We say a collection of simple closed curves $A_1,\dots, A_g,$ $B_1,\dots, B_g $ on $\Sigma$ form a \emph{geometric symplectic} basis of $H_1(\Sigma;\Z)$ if the following hold: 
\begin{enumerate}
\item $\{[A_1],\dots, [A_g], [B_1],\dots,[B_g]\}$ is a basis of $H_1(\Sigma;\Z)$. 
\item The geometric intersection number of $A_i$ and $B_j$ is $ \delta_{i,j}$.
\end{enumerate}
\end{define}
Given an oriented, connected surface with zero or one boundary component, we consider the element
 \begin{equation}
\xi(\Sigma):=\prod_{j=1}^{g(\Sigma)}  (U+[A_j]\wedge [B_j])\in \bF_2[U]\otimes_{\bF_2} \Lambda^*(H_1(\Sigma;\bF_2)).\label{eq:defxiSigma}
 \end{equation}

Although the element $\xi(\Sigma)$ is defined by picking a geometric symplectic basis, we in fact have the following:

\begin{prop}\label{prop:fixedbymappingclassgroup}The element $\xi(\Sigma)$ is independent of the choice of geometric symplectic basis of $H_1(\Sigma;\Z)$. Furthermore, the action of $\MCG(\Sigma)$ on $\bF_2[U]\otimes \Lambda^*(H_1(\Sigma;\bF_2))$ fixes $\xi(\Sigma)$.
\end{prop}

\begin{proof} If $\d \Sigma\neq \varnothing$, let $\hat{\Sigma}$ denote the surface obtained by capping off the boundary of $\Sigma$ with a disk. Since $H_1(\Sigma;\Z)\to H_1(\hat{\Sigma};\Z)$ is an isomorphism, it is sufficient to show the analogous statement for the element $\xi(\hat{\Sigma})\in \bF_2[U]\otimes \Lambda^*(H_1(\hat{\Sigma};\bF_2))$, defined using a geometric symplectic basis of $H_1(\hat{\Sigma};\Z)$. 

Suppose that $\{[A_1],\dots, [A_g], [B_1],\dots, [B_g]\}$ and $\{[A_1'],\dots, [A_g'],[B_1'],\dots, [B_g']\}$ are two choices of geometric symplectic bases. We can pick a single automorphism $\phi$ of $\Sigma$ such that for all $i\in\{1,\dots, g\}$ we have  $\phi(F_i)=F_i'$, where $F_i$ and $F_i'$ are the punctured tori
 \[
 F_i:=N(A_i\cup B_i) \qquad \text{and} \qquad F_i':=N(A_i'\cup B_i').
 \]
  Furthermore, we can arrange that $\phi(A_i)=A_i'$ and $\phi(B_i)=B_i'$ (up to orientation reversal).  Hence, it is sufficient to show that $\xi(\hat{\Sigma})$, computed with the basis $A_1,\dots, A_g,$ $B_1,\dots, B_g$, is fixed by $\MCG(\hat{\Sigma})$. According to \cite{Lickorish}, the group $\MCG(\hat{\Sigma})$ is generated by Dehn twists along the curves $a_1,\dots, a_g,$ $b_1,\dots, b_g$, $c_1,\dots, c_{g-1}$  shown in Figure~\ref{fig::44}.  We can assume $a_i=A_i$ and $b_i=B_i$. 

 \begin{figure}[ht!]
 	\centering
\begingroup%
  \makeatletter%
  \providecommand\color[2][]{%
    \errmessage{(Inkscape) Color is used for the text in Inkscape, but the package 'color.sty' is not loaded}%
    \renewcommand\color[2][]{}%
  }%
  \providecommand\transparent[1]{%
    \errmessage{(Inkscape) Transparency is used (non-zero) for the text in Inkscape, but the package 'transparent.sty' is not loaded}%
    \renewcommand\transparent[1]{}%
  }%
  \providecommand\rotatebox[2]{#2}%
  \newcommand*\fsize{\dimexpr\f@size pt\relax}%
  \newcommand*\lineheight[1]{\fontsize{\fsize}{#1\fsize}\selectfont}%
  \ifx\svgwidth\undefined%
    \setlength{\unitlength}{280.98547591bp}%
    \ifx\svgscale\undefined%
      \relax%
    \else%
      \setlength{\unitlength}{\unitlength * \real{\svgscale}}%
    \fi%
  \else%
    \setlength{\unitlength}{\svgwidth}%
  \fi%
  \global\let\svgwidth\undefined%
  \global\let\svgscale\undefined%
  \makeatother%
  \begin{picture}(1,0.20637305)%
    \lineheight{1}%
    \setlength\tabcolsep{0pt}%
    \put(0,0){\includegraphics[width=\unitlength,page=1]{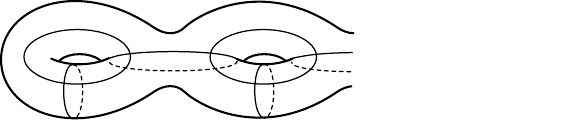}}%
    \put(0.10626533,0.02596726){\color[rgb]{0,0,0}\makebox(0,0)[rt]{\lineheight{1.25}\smash{\begin{tabular}[t]{r}$a_1$\end{tabular}}}}%
    \put(0.10252637,0.16791179){\color[rgb]{0,0,0}\makebox(0,0)[lt]{\lineheight{1.25}\smash{\begin{tabular}[t]{l}$b_1$\end{tabular}}}}%
    \put(0.30744645,0.13031815){\color[rgb]{0,0,0}\makebox(0,0)[lt]{\lineheight{1.25}\smash{\begin{tabular}[t]{l}$c_1$\end{tabular}}}}%
    \put(0,0){\includegraphics[width=\unitlength,page=2]{fig44.pdf}}%
    \put(0.9074185,0.16028666){\color[rgb]{0,0,0}\makebox(0,0)[lt]{\lineheight{1.25}\smash{\begin{tabular}[t]{l}$b_g$\end{tabular}}}}%
    \put(0.89924752,0.03057313){\color[rgb]{0,0,0}\makebox(0,0)[lt]{\lineheight{1.25}\smash{\begin{tabular}[t]{l}$a_g$\end{tabular}}}}%
    \put(0.75234193,0.1281546){\color[rgb]{0,0,0}\makebox(0,0)[rt]{\lineheight{1.25}\smash{\begin{tabular}[t]{r}$c_{g-1}$\end{tabular}}}}%
  \end{picture}%
\endgroup%

 	\caption{\textbf{Generators of \textbf{$\MCG(\hat{\Sigma})$.}}}\label{fig::44}
 \end{figure}

We now consider invariance of $\xi(\hat{\Sigma})$ under Dehn twists along $a_i$. The only curve in our basis which is changed by a Dehn twist along $a_i$ is $B_i$. A Dehn twist along $a_i$ sends the class $[B_i]$ to $[B_i]+[A_i]$. This does not change the element $\xi(\hat{\Sigma})$, since
\[
U+[A_i]([A_i]+[B_i])=U+[A_i][B_i].
\]
Invariance under Dehn twists along a curve $b_i$ is similar.

We finally consider a Dehn twist along a curve $c_i$, which we note intersects both $B_i$ and $B_{i+1}$, and none of the other $A_i$ or $B_j$ curves. A Dehn twist along $c_i$ sends  the homology class $[B_i]$ to $[B_i]+[A_i]+[A_{i+1}]$ and sends $[B_{i+1}]$ to $[B_{i+1}]+[A_i]+[A_{i+1}]$. To show invariance of $\xi(\hat{\Sigma})$, we compute
\begin{align*}
&\big(U+[A_i]([B_i]+[A_i]+[A_{i+1}])\big)\big(U+[A_{i+1}]([B_{i+1}]+[A_i]+[A_{i+1}])\big)\\
=&U^2+\big([A_i][B_i]+[A_{i+1}][B_{i+1}]\big)U+[A_i][B_i][A_{i+1}][B_{i+1}]\\
=&(U+[A_i][B_i])(U+[A_{i+1}][B_{i+1}]).
\end{align*}

It follows that $\xi(\hat{\Sigma})$ is preserved by the action of $\MCG(\hat{\Sigma})$, and hence is independent of the choice of geometric symplectic basis. As we remarked earlier, this implies that $\xi(\Sigma)$ is also independent of the choice of geometric symplectic basis, and consequently is fixed by the action of $\MCG(\Sigma)$.
\end{proof}

\subsection{The algebraic reductions of the link cobordism maps}

We recall that if $(\Sigma,\as,\bs,\ws)$ is a multi-pointed Heegaard diagram for $(Y,\ws)$, the module $\CF^-(Y,\ws,\frs)$ is the free $\bF_2[U]$-module generated by intersection points $\xs\in \bT_{\alpha}\cap \bT_{\beta}$. By counting holomorphic curves, one can define a differential on $\CF^-(Y,\ws,\frs)$, similar to the expression which appears in Equation~\eqref{eq:del}.

For links equipped with a coloring with exactly two colors, such that the $\ws$ basepoints are given the same color, and the $\zs$ basepoints are given the other, we can recover $\CF^-$ from  $\cCFL^-$ via either of the two natural isomorphisms
\begin{equation}
\begin{split}
\cCFL^-(Y,\bL,\frs)\otimes_{\cR^-} \cR^-/(V-1) \iso \CF^-(Y,\ws,\frs)\qquad \text{and}\\
 \cCFL^-(Y,\bL,\frs)\otimes_{\cR^-} \cR^-/(U-1) \iso \CF^-(Y,\zs,\frs-\PD[L]).
 \end{split}
 \label{eq:tensorreductions}
\end{equation}
As a general algebraic fact, if $F\colon M_1\to M_2$ is a map of $R$-modules, and $N$ is an $R$-module, there is an induced map $F\otimes \id_N\colon M_1\otimes_R N\to M_2\otimes_R N$. Hence, given a decorated link cobordism $(W,\cF)\colon (Y_1,\bL_1)\to (Y_2,\bL_2)$,  we obtain two maps: 
\[
F_{W,\cF,\frs}|_{V=1}\colon \CF^-(Y_1,\ws_1,\frs|_{Y_1})\to \CF^-(Y_2,\ws_2,\frs|_{Y_2})
\]
and
\[
F_{W,\cF,\frs}|_{U=1}\colon \CF^-(Y_1,\zs_1,\frs|_{Y_1}-\PD[L_1])\to \CF^-(Y_2,\zs_2,\frs|_{Y_2}-\PD[L_2]).
\]

In \cite{ZemGraphTQFT}, the author constructs a ``graph TQFT'' for $\CF^-$. The objects of the associated cobordism category are closed 3-manifolds with collections of basepoints. A cobordism from $(Y_1,\ws_1)$ to $(Y_2,\ws_2)$ consists of a pair $(W,\Gamma)$ such that $W$ is a compact, oriented 4-manifold with $\d W=-Y_1\sqcup Y_2$, and $\Gamma\subset W$ is a finite, embedded graph satisfying the following:
\begin{enumerate}
\item $\Gamma\cap Y_i=\ws_i$.
\item Each basepoint in $\ws_i$ has valence 1 in $\Gamma$.
\item $\Gamma$ is decorated which a choice of cyclic ordering at each of its vertices.
\end{enumerate}

Given a ribbon graph cobordism $(W,\Gamma)\colon (Y_1,\ws_1)\to (Y_2,\ws_2)$, there are two maps
\[
F_{W,\Gamma,\frs}^A,\,\,\, F_{W,\Gamma,\frs}^B \colon \CF^-(Y_1,\ws_1,\frs|_{Y_1})\to \CF^-(Y_2,\ws_2, \frs|_{Y_2}).
\]
 The construction from \cite{ZemGraphTQFT} corresponds to the type-$A$ maps. The type-$B$ maps are a simple variation, which are described in \cite{HMZConnectedSums}*{Section~3}.

 The type-$A$ maps and the type-$B$ maps satisfy the relation
\begin{equation}
F_{W,\Gamma,\frs}^A\simeq F_{W,\bar{\Gamma},\frs}^B,\label{eq:switchAtoB}
\end{equation}
where $\bar{\Gamma}$ is the graph obtained by reversing the cyclic orders of $\Gamma$ \cite{HMZConnectedSums}*{Lemma~5.9}.

To relate the link cobordism maps to the graph cobordism maps, we need the following notion:

\begin{define} Suppose $(W,\cF)\colon (Y_1,\bL_1)\to (Y_2,\bL_2)$ is a decorated link cobordism with type-$\ws$  subsurface $\Sigma_{\ws}$. If $\Gamma\subset W$ is a ribbon graph, we say that $\Gamma$ is a \emph{ribbon 1-skeleton} of $\Sigma_{\ws}$ if the following hold:
\begin{enumerate}
\item $\Gamma\subset \Sigma_{\ws}$.
\item $\Gamma\cap Y_i=\ws_i$.
\item $\Sigma_{\ws}$ is a regular neighborhood of $\Gamma$ inside of $\Sigma$.
\item The ribbon structure of $\Gamma$ is compatible with the orientation of $\Sigma$.
\end{enumerate}
\end{define}

The following general reduction theorem is proven in \cite{ZemCFLTQFT}:

\begin{thm}[\cite{ZemCFLTQFT}*{Theorem~C}]\label{thm:generalreductionformula} If $(W,\cF)$ is a decorated link cobordism, and $\Gamma_{\ws}\subset \Sigma_{\ws}$ and $\Gamma_{\zs}\subset \Sigma_{\zs}$ are ribbon 1-skeleta, then
\[
F_{W,\cF,\frs}|_{V=1}\simeq F_{W,\Gamma_{\ws},\frs}^B \qquad \text{and} \qquad F_{W,\cF,\frs}|_{U=1}\simeq F_{W,\Gamma_{\zs},\frs-\PD[\Sigma]}^A.
\]
\end{thm}

We now describe the maps induced by several simple graph cobordism, based on several computations from \cite{ZemDualityMappingTori}.

 Appearing in our formulas are two natural endomorphisms of $\CF^-(Y,\ws,\frs)$. The first endomorphism is the action of $\Lambda^* (H_1(Y;\Z)/\Tors)$, described by Ozsv\'{a}th and Szab\'{o} \cite{OSDisks}*{Section~4.2.5}. If $\gamma$ is a closed loop in $Y$, we will write $A_\gamma$ for the  map
\[
A_{\gamma}(\xs):=\sum_{\substack{\phi\in \pi_2(\xs,\ys)\\ \mu(\phi)=1}} a(\gamma,\phi) \# \Hat{\cM}(\phi)\cdot U^{n_{\ws}(\phi)}\cdot \ys,
\] 
 where $a(\gamma,\phi)$ denotes the intersection number of $\phi$ with  $\gamma$ (appropriately interpreted).
 
  The second endomorphism  is the map $\Phi_w$, defined via the formula
\[
\Phi_{w}(\xs):=U^{-1} \sum_{\substack{\phi\in \pi_2(\xs,\ys)\\ \mu(\phi)=1}} n_{w}(\phi) \# \Hat{\cM}(\phi)\cdot  U^{n_{\ws}(\phi)} \cdot \ys.
\]
The map $\Phi_w$  is considered in \cite{ZemGraphTQFT}.

 We consider the four ribbon graph cobordisms $([0,1]\times Y, \Gamma_i)$ shown in Figure~\ref{fig::41}. The homology classes of various loops in the graphs are labeled.

 \begin{figure}[ht!]
 	\centering
 	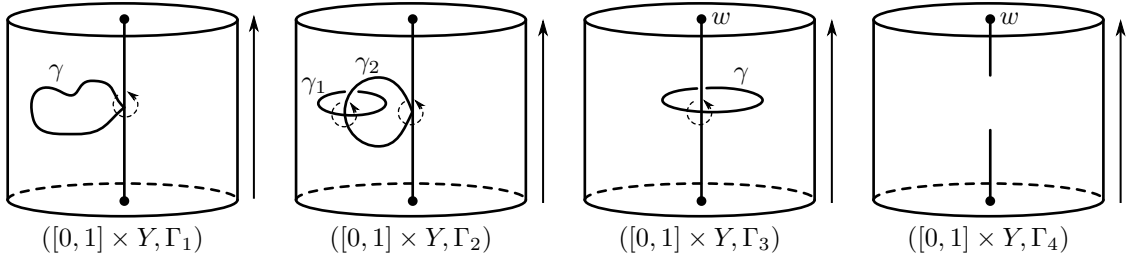
 	\caption{\textbf{The graph cobordisms $([0,1]\times Y, \Gamma_i)$ for $i\in \{1,2,3,4\}$ considered in Proposition~\ref{prop:graphcobordismcomp}.}}\label{fig::41}
 \end{figure}

\begin{prop}\label{prop:graphcobordismcomp}The graph cobordism maps for $([0,1]\times Y, \Gamma_i)$ satisfy
\begin{enumerate}
\item $F_{[0,1]\times Y, \Gamma_1,\frs}\simeq A_\gamma$.
\item $F_{[0,1]\times Y, \Gamma_2,\frs}\simeq U+A_{\gamma_1} A_{\gamma_2}$.
\item $F_{[0,1]\times Y, \Gamma_3,\frs}\simeq A_\gamma+U\Phi_w$.
\item $F_{[0,1]\times Y, \Gamma_4,\frs}\simeq \Phi_w$.
\end{enumerate}
The above relations hold for both the type-$A$ and $B$ versions of the graph cobordism maps.
\end{prop}

\begin{proof}The computation of the maps for $\Gamma_1$ and $\Gamma_2$ is performed in \cite{ZemDualityMappingTori}*{Proposition~4.6}. We note that \cite{ZemDualityMappingTori}*{Proposition~4.6} is stated only in the case that $[\gamma_1]=[\gamma_2]=0\in H_1(Y;\Z)$ (so that the induced map is the action of $U$). Nonetheless, the proof given in \cite{ZemDualityMappingTori} demonstrates the stated formula for general $\gamma_1$ and $\gamma_2$, and then specializes to the case that $[\gamma_1]=[\gamma_2]=0$. The computation of the map for $\Gamma_4$ is performed in \cite{ZemDualityMappingTori}*{Lemma~4.5}.

 The computation of the cobordism map for $\Gamma_3$ follows from the computation of the map for $\Gamma_1,$ and $\Gamma_4$, using the \emph{vertex breaking relation}, which describes the effect of changing the relative order of two edges adjacent to a vertex in the graph. The relation is shown in Figure~\ref{fig::45}, and is proven in \cite{ZemDualityMappingTori}*{Lemma~4.4}. To obtain the stated formula for the graph cobordism map for $\Gamma_3$, we apply the vertex breaking relation at the valence 4 vertex, showing that the induced map is a sum of a graph cobordism for the graph $\Gamma_1$, as well as $U$ times the graph cobordism map for $\Gamma_4$.

Finally, we briefly describe why the type-$A$ and $B$ maps coincide for the graphs $\Gamma_1,$ $\Gamma_2,$ $\Gamma_3$ and  $\Gamma_4$. For a rigorous proof, we refer the reader to the proof of \cite{ZemDualityMappingTori}*{Proposition~4.6},  where the claim is proven for $\Gamma_1$ and $\Gamma_2$. The proof therein extends easily to $\Gamma_3$. We note that Equation~\eqref{eq:switchAtoB} implies that switching from type-$A$ to type-$B$ corresponds to reversing the cyclic orders (immediately implying that the type-$A$ and $B$ maps agree for $\Gamma_4$). Hence, on a more conceptual level, one can use several applications of the vertex breaking relation in Figure~\ref{fig::45} to relate the type-$A$ the graph cobordism map for $\Gamma_i$ to the type-$A$ graph cobordism map for $\bar{\Gamma}_i$, for $i\in \{1,2,3,4\}$, though we leave this exercise to the reader.
\end{proof}

 \begin{figure}[ht!]
 	\centering
 	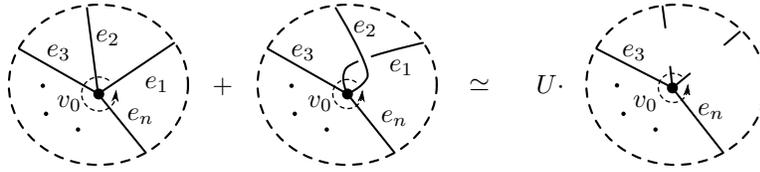
 	\caption{\textbf{The vertex breaking relation.} The vertex breaking relation illustrates the effect of changing the relative ordering of two edges adjacent at a vertex $v_0$. Note that the embedding of the actual edges at a vertex does not effect the induced map, though it makes it clearer to switch the embedding between the first and second pictures.}\label{fig::45}
 \end{figure}

The graph cobordism computations from Proposition~\ref{prop:graphcobordismcomp} provide a useful description of certain reductions of the link cobordism maps:

\begin{lem}\label{lem:computereductionsconnectedsurface}Suppose that $(W,\cF)\colon (Y_1,\bK_1)\to (Y_2,\bK_2)$ is a decorated knot cobordism, such that $\cF$ is a connected, decorated surface whose dividing set consists of exactly two arcs, $a_1$ and $a_2$, which divide $\cF$ into two connected components. Then
\[
F_{W,\cF,\frs}|_{V=1}(-)\simeq F_{W,c_{\ws},\frs}(\iota_*\xi(\Sigma_{\ws})\otimes -)\qquad \text{and} \qquad F_{W,\cF,\frs}|_{U=1}(-)\simeq  F_{W,c_{\zs},\frs-\PD[\Sigma]}(\iota_* \xi(\Sigma_{\zs})\otimes -),
\]
where $c_{\ws}$ and $c_{\zs}$ denote any choice of paths in $W$ formed by concatenating either of the arcs, $a_1$ or $a_2$, with subarcs of $\Sigma_{\ws}\cap \bK_i$ or $\Sigma_{\zs}\cap \bK_i$, respectively. Here $\iota\colon \Sigma\hookrightarrow W$ denotes inclusion.
\end{lem}

\begin{proof}  From Theorem~\ref{thm:generalreductionformula} we know that
\begin{equation}
F_{W,\cF,\frs}|_{V=1}\simeq F_{W,\Gamma_{\ws},\frs}^B, \label{eq:linkcobtographred1}
\end{equation}
where $\Gamma_{\ws}$ is a ribbon 1-skeleton of the type-$\ws$ subsurface $\Sigma_{\ws}$ of $\cF$. We will use the graph cobordism computations from Proposition~\ref{prop:graphcobordismcomp}. The key observation is that a ribbon 1-skeleton for $\Sigma_{\ws}$ can be constructed from choice of geometric symplectic basis of $H_1(\Sigma_{\ws};\Z)$, as well as an additional arc, which will be the path $c_{\ws}$. Let $c_{\ws}$ be one of the two embedded paths on $\Sigma_{\ws}$, which connect the two $\ws$ basepoints and run parallel to $\d \Sigma_{\ws}$. Suppose $A_1,\dots, A_g, $ $B_1,\dots, B_g$  is a symplectic basis of $H_1(\Sigma_{\ws};\Z)$. We can now construct a ribbon 1-skeleton for $\Sigma_{\ws}$ by isotoping each $A_i$ so that it intersects $c_{\ws}$ non-transversally at a single point. This procedure is shown in Figure~\ref{fig::46}.  Write $\Gamma_{\ws}$ for the ribbon 1-skeleton formed via this procedure.

 \begin{figure}[ht!]
 	\centering
\begingroup%
  \makeatletter%
  \providecommand\color[2][]{%
    \errmessage{(Inkscape) Color is used for the text in Inkscape, but the package 'color.sty' is not loaded}%
    \renewcommand\color[2][]{}%
  }%
  \providecommand\transparent[1]{%
    \errmessage{(Inkscape) Transparency is used (non-zero) for the text in Inkscape, but the package 'transparent.sty' is not loaded}%
    \renewcommand\transparent[1]{}%
  }%
  \providecommand\rotatebox[2]{#2}%
  \newcommand*\fsize{\dimexpr\f@size pt\relax}%
  \newcommand*\lineheight[1]{\fontsize{\fsize}{#1\fsize}\selectfont}%
  \ifx\svgwidth\undefined%
    \setlength{\unitlength}{274.2311966bp}%
    \ifx\svgscale\undefined%
      \relax%
    \else%
      \setlength{\unitlength}{\unitlength * \real{\svgscale}}%
    \fi%
  \else%
    \setlength{\unitlength}{\svgwidth}%
  \fi%
  \global\let\svgwidth\undefined%
  \global\let\svgscale\undefined%
  \makeatother%
  \begin{picture}(1,0.3728839)%
    \lineheight{1}%
    \setlength\tabcolsep{0pt}%
    \put(0,0){\includegraphics[width=\unitlength,page=1]{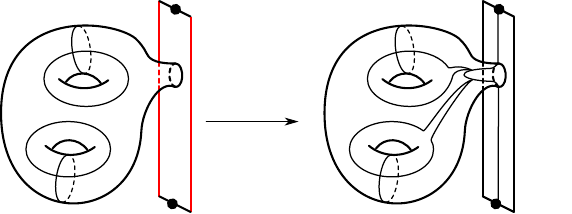}}%
    \put(0.92085829,0.31246524){\color[rgb]{0,0,0}\makebox(0,0)[lt]{\lineheight{1.25}\smash{\begin{tabular}[t]{l}$c_{\ws}$\end{tabular}}}}%
    \put(0,0){\includegraphics[width=\unitlength,page=2]{fig46.pdf}}%
    \put(0.34162257,0.27206134){\color[rgb]{1,0,0}\makebox(0,0)[lt]{\lineheight{1.25}\smash{\begin{tabular}[t]{l}$a_1$\end{tabular}}}}%
    \put(0.27280649,0.32089914){\color[rgb]{1,0,0}\makebox(0,0)[rt]{\lineheight{1.25}\smash{\begin{tabular}[t]{r}$a_2$\end{tabular}}}}%
  \end{picture}%
\endgroup%

 	\caption{\textbf{Constructing a ribbon 1-skeleton of $\Sigma_{\ws}$ from a geometric symplectic basis of $H_1(\Sigma_{\ws};\Z)$, and the additional arc $c_{\ws}$.} On the left are the curves in the symplectic basis. The two solid dots are the $\ws$ basepoints. On the right, a ribbon 1-skeleton has been constructed, using the symplectic basis and the arc $c_{\ws}$. The dividing arcs of $\cF$, labeled  $a_1$ and $a_2$, are shown in red on the left side.}\label{fig::46}
 \end{figure}

We can now use Proposition~\ref{prop:graphcobordismcomp} to compute the graph cobordism map $F_{W,\Gamma_{\ws},\frs}^B$. We decompose the 4-manifold $W$ into a sequence of handle attachments, so that the 1-handles occur before the 2-handles, which occur before the 3-handles. It is straightforward to arrange that all of the loops of $\Gamma_{\ws}$ occur in a product cobordism $[0,1]\times Y\subset W$ which occurs between the 1-handles and the 2-handles. Furthermore, since $W$ is a 4-manifold, it is straightforward to arrange $\Gamma_{\ws}\cap ([0,1]\times Y)$ so that it is a composition of graph cobordisms, each with two loops, and with the same configuration as the graph cobordism $([0,1]\times Y,\Gamma_2)$, shown in Figure~\ref{fig::41}. Using the computation from Proposition~\ref{prop:graphcobordismcomp}, combined with the composition law, it follows that
\begin{equation}
F_{W,\Gamma_{\ws},\frs}^B(-)\simeq F_{W,c_{\ws},\frs}(\xi(\Sigma_{\ws})\otimes -).\label{eq:linkcobtographred2}
\end{equation}
Combining Equation~\eqref{eq:linkcobtographred1} and \eqref{eq:linkcobtographred2} implies the stated formula for the $V=1$ reduction of $F_{W,\cF,\frs}$.

A similar argument works for the $U=1$ reduction.
\end{proof}

\subsection{The cobordism maps for closed surfaces}
\label{sec:cobordismmapsforclosedsurfaces}
In this section, we compute the cobordism maps for link cobordism obtained by puncturing a closed, decorated surface inside of a 4-manifold $W$ with $\d W=-Y_1\sqcup Y_2$.

Suppose that $\cF$ is a closed, decorated surface in a cobordism $W\colon Y_1\to Y_2$. Suppose furthermore that $W$, $Y_1$ and $Y_2$ are nonempty and connected. Let $D_1$ and $D_2$ be two embedded disks in $\cF$, which each intersect the dividing set of $\cF$ in a single arc.  Let $\cF_0$ be a properly embedded decorated surface in $W$, obtained by isotoping $\cF$ so that it intersects $Y_i$ along $D_i$, and then removing $D_1$ and $D_2$ from $\cF$. Let $\bU_i$ denote the decorated unknot in $Y_i$ obtained by adding two basepoints to $\d D_i$ in such a way that $(W,\cF_0)$ becomes a decorated link cobordism from $(Y_1,\bU_1)$ to $(Y_2,\bU_2)$.

Let  $\bF_2[\hat{U}]$ denote the free polynomial ring in the variable $\hat{U}$. We give the ring $\bF_2[U,V]$ an action of $\bF_2[\hat{U}]$ by declaring $\hat{U}$ to act by $UV$.

If $\bU$ is an unknot in $Y$ with exactly two basepoints, $w$ and $z$, as well as a distinguished Seifert disk $D$, then we can restrict attention to diagrams $\cH$ for $(Y,\bU)$, where the disk $D$ intersects the Heegaard surface in an arc connecting $w$ and $z$ which is disjoint from the  $\as$ and $\bs$ curves on $\cH$. If $p$ denotes the center of $D$, then by viewing $\CF^-(Y,p,\frs)$ as an $\bF_2[\hat{U}]$-module, we obtain a canonical isomorphism
\[
\cCFL^-(Y,\bU,\frs)\iso \CF^-(Y,p,\frs)\otimes_{\bF_2[\hat{U}]} \bF_2[U,V].
\]
In particular, if $(W,\Gamma)\colon (Y_1,p_1)\to (Y_2,p_2)$ is a graph cobordism, the graph cobordism map $F_{W,\Gamma,\frs}$ determines a map from $\cCFL^-(Y_1,\bU_1,\frs|_{Y_1})$ to $\cCFL^-(Y_2,\bU_2,\frs|_{Y_2})$, for which we write $F_{W,\Gamma,\frs}|^{\bF_2[U,V]}$. To be explicit, if the expression $F_{W,\Gamma,\frs}(\xs)$ contains the summand $\hat{U}^\ell \cdot \ys$, then $F_{W,\Gamma,\frs}|^{\bF_2[U,V]}(\xs)$ contains the summand $U^\ell V^\ell \cdot \ys$.

\begin{prop}\label{prop:reductionformulasforpuncturedsurfaces}Suppose that $\cF=(\Sigma,\cA)$ is a closed, decorated surface inside of the cobordism $W\colon Y_1\to Y_2$, and let $(W,\cF_0)$ denote the decorated link cobordism obtained by isotoping $\cF$ so that it intersects $Y_1$ and $Y_2$ in two disks, and then removing those two disks, as described above. Let $\Delta A$ denote the quantity
\[
\Delta A=\frac{\langle c_1(\frs),\Sigma\rangle-[\Sigma]\cdot [\Sigma]}{2}+\frac{\chi(\Sigma_{\ws})-\chi(\Sigma_{\zs})}{2}.
\]
\begin{enumerate}
\item\label{prop:red:claim1}  In general,
\[
F_{W,\cF_0,\frs}\simeq \begin{cases}  V^{\Delta A} \cdot F_{W,\Gamma_{\ws},\frs}^B|^{\bF_2[U,V]}& \text{ if }\quad \Delta A\ge 0,\\
U^{-\Delta A}\cdot F_{W,\Gamma_{\zs},\frs-\PD[\Sigma]}^A|^{\bF_2[U,V]}& \text{ if }\quad \Delta A\le 0,
\end{cases}
\]
where $\Gamma_{\ws}\subset \Sigma_{\ws}$ and $\Gamma_{\zs}\subset \Sigma_{\zs}$ are ribbon 1-skeleta.

\item \label{prop:red:claim2} If $\Sigma$ is connected and $\cA$ consists of a single closed curve $a$, dividing $\Sigma$ into two connected components, then
\[
F_{W,\cF_0,\frs}\simeq \begin{cases} V^{\Delta A}\cdot F_{W,c,\frs}(\iota_* \xi(\Sigma_{\ws})\otimes -)|^{\bF_2[U,V]} & \text{ if } \quad \Delta A\ge 0
\\ U^{-\Delta A} \cdot F_{W,c,\frs-\PD[\Sigma]}(\iota_*\xi(\Sigma_{\zs})\otimes -)|^{\bF_2[U,V]} & \text{ if } \quad \Delta A\le 0,
\end{cases}
\]
where $c$ is one of the two dividing arcs in $\cF_0$, viewed as a path from $Y_1$ to $Y_2$.
\end{enumerate}
\end{prop}

\begin{proof}We consider Claim~\eqref{prop:red:claim1} first.  Let $D_1$ and $D_2$ denote the two disks of $\Sigma$, which are pushed into $Y_1$ and $Y_2$, and let $U_1$ and $U_2$ denote $\d  D_1$ and $\d D_2$. We pick Heegaard diagrams $\cH_1=(\Sigma_1,\as_1,\bs_1,w_1,z_1)$ and $\cH_2=(\Sigma_2,\as_2,\bs_2,w_2,z_2)$ such that  the disk $D_i$ intersects $\Sigma_i$ along an arc connecting $w_i$ and $z_i$ which is disjoint from $\as_i$ and $\bs_i$. On this diagram, the Alexander grading on $\cCFL^-(Y_i,\bU_i,\frs|_{Y_i})$ with respect to the Seifert disk $D_i$ is given by 
\[A_{D_i}(U^m V^n \cdot \xs)=(n-m).\] In particular,  the map $F_{W,\cF,\frs}$ is determined entirely by the reduction $F_{W,\cF,\frs}|_{V=1}$ and the Alexander grading change $\Delta A$. If $\Delta A\ge 0$, then $F_{W,\cF,\frs}$ can be recovered from its $V=1$ reduction by correcting powers of $V$ via the formula
\begin{equation}
F_{W,\cF_0,\frs}\simeq V^{\Delta A}\cdot (F_{W,\cF_0,\frs}|_{V=1})|^{\bF_2[U,V]}.\label{eq:extensionofV=1reduction}
\end{equation}
Note that $\Delta A$ must be nonnegative for Equation~\eqref{eq:extensionofV=1reduction} to be satisfied, since otherwise the right hand side is not well defined up to filtered, equivariant chain homotopy.

If $\Delta A\le 0$, a similar argument shows that
\begin{equation}
F_{W,\cF_0,\frs}\simeq U^{-\Delta A}\cdot (F_{W,\cF_0,\frs}|_{U=1})|^{\bF_2[U,V]}\label{eq:extensionofU=1reduction}.
\end{equation}
Combining Equations~\eqref{eq:extensionofV=1reduction} and \eqref{eq:extensionofU=1reduction} with Theorem~\ref{thm:generalreductionformula} concludes the proof of Claim~\eqref{prop:red:claim1}.

Claim~\eqref{prop:red:claim2} follows from Claim~\eqref{prop:red:claim1}, together with Lemma~\ref{lem:computereductionsconnectedsurface}.
\end{proof}

\subsection{The cobordism maps on $\cHFL^\infty$}
\label{sec:4-manifoldmaps}

In this section we compute the maps associated to knot cobordisms on the level of $\cHFL^\infty$ when the dividing set is relatively simple, focusing on the case that the 4-manifold is negative definite.

When $\bL$ is not an unlink, there is in general no way to recover the $\Z\oplus \Z$-filtered complex $\cCFL^\infty(Y,\bL,\frs)$ from $\CF^\infty(Y,\ws,\frs)$, however we can recover $\cHFL^\infty(Y,\bL,\frs)$ from $\HF^\infty(Y,\ws,\frs)$, as we now describe.

As a general algebraic principle, if $R$ and $R_0$ are two unital rings such that $R_0$ has an action of $R$, and $N$ is an $R$-module, then there is a natural map
\[
N\to N\otimes_R R_0,
\]
which is simply $n\mapsto n\otimes 1_{R_0}$. 

Hence, using the isomorphisms in Equation~\eqref{eq:tensorreductions}, we obtain chain maps
\begin{equation}
R_{\ws}\colon \cCFL^\infty(Y,\bL,\frs)\to \CF^\infty(Y,\ws,\frs) \quad \text{and} \quad R_{\zs}\colon \cCFL^\infty(Y,\bL,\frs)\to \CF^\infty(Y,\zs,\frs-\PD[L]), \label{eq:reductionmaps}
\end{equation}
 by mapping $V$ to $1$, or mapping $U$ to $1$, respectively.

Suppose $\bL$ is a null-homologous link in $Y$, and $\frs\in \Spin^c(Y)$ is torsion. We can decompose $\cCFL^\infty(Y,\bL,\frs)$ as a direct sum over (collapsed) Alexander gradings
\[
\cCFL^\infty(Y,\bL,\frs)=\bigoplus_{i\in \Z} \cCFL^{\infty}(Y,\bL,\frs)_i.
\]
Note that the direct sum is not of $\bF_2[U,V,U^{-1},V^{-1}]$-modules, but instead of $\bF_2[\hat{U}, \hat{U}^{-1}]$-modules.

We define maps 
\[
(R_{\ws})_i\colon \cCFL^\infty(Y,\bL,\frs)_i\to \CF^\infty(Y,\ws,\frs) \qquad \text{and} \qquad (R_{\zs})_i\colon \cCFL^\infty(Y,\bL,\frs)_i\to \CF^\infty(Y,\zs,\frs), 
\]
as the restriction of the reduction maps $R_{\ws}$ and $R_{\zs}$ from Equation~\eqref{eq:reductionmaps} to $\cCFL^\infty(Y,\bL,\frs)_i$. 

\begin{lem}\label{lem:reductionmapischainisomorphism}Suppose that $\bL$ is a null-homologous link in $Y$. The map
\[
(R_{\ws})_i\colon \cCFL^\infty(Y,\bL,\frs)_i\to  \CF^\infty(Y,\ws,\frs),
\]
is an isomorphism of chain complexes over $\bF_2[\hat{U},\hat{U}^{-1}]$.
\end{lem}
\begin{proof} Suppose $\cH=(\Sigma,\as,\bs,\ws,\zs)$ is a fixed diagram of $(Y,\bL)$. We define an inverse $(Q_{\ws})_i$ of $(R_{\ws})_i$ via the formula
\[
(Q_{\ws})_i(\hat{U}^{j}\cdot \xs)=U^j V^{j+i-A(\xs)} \cdot \xs.
\]
It is straightforward to see that $(Q_{\ws})_i$ is a chain map, and that $(R_{\ws})_i$ and $(Q_{\ws})_i$ are inverses of each other.
\end{proof}

Note that the map $(Q_{\ws})_i$ in Lemma~\ref{lem:reductionmapischainisomorphism} does not necessarily respect the $\Z\oplus \Z$-filtration on $\cCFL^\infty(Y,\bL,\frs)$.

Using Lemma~\ref{lem:reductionmapischainisomorphism}, we can define a chain isomorphism
\[
\ve{R}_{\ws}\colon \cCFL^\infty(Y,\bL,\frs)\to \bigoplus_{i\in \Z} \CF^\infty(Y,\ws,\frs),
\]
by taking the direct sum of all the $(R_{\ws})_i$. The map $\ve{R}_{\ws}$ intertwines the action of $V$ on $\cCFL^\infty(Y,\bL,\frs)$ with the endomorphism of  $\bigoplus_{i\in \Z} \CF^\infty(Y,\ws,\frs)$ which shifts the index $i$ by $+1$.

\begin{thm}\label{thm:mostgeneralcompHFLinfty} Suppose that $(W,\cF)\colon (Y_1,\bK_1)\to (Y_2,\bK_2)$ is a knot cobordism.  Write $\cF=(\Sigma,\cA)$.
\begin{enumerate}
\item \label{prop:infinitycomp:claim1}If $\Gamma_{\ws}$ and $\Gamma_{\zs}$ are ribbon 1-skeleta of $\Sigma_{\ws}$ and $\Sigma_{\zs}$, then
\[
F_{W,\cF,\frs}=V^{\Delta A}\cdot \ve{R}_{\ws}^{-1}\circ F_{W,\Gamma_{\ws},\frs}^B\circ \ve{R}_{\ws}
\] and
\[
F_{W,\cF,\frs}=U^{-\Delta A}\cdot \ve{R}_{\zs}^{-1}\circ F_{W,\Gamma_{\zs},\frs-\PD[\Sigma]}^A\circ \ve{R}_{\zs}
\]
as maps from $\cHFL^\infty(Y_1,\bK_1,\frs|_{Y_1})$ to $\cHFL^\infty(Y_2,\bK_2,\frs|_{Y_2})$.  Here 
\[
\Delta A:=\frac{\langle c_1(\frs), \Sigma\rangle -[\Sigma]\cdot [\Sigma]}{2}+\frac{\chi(\Sigma_{\ws})-\chi(\Sigma_{\zs})}{2}
\]
denotes the Alexander grading change.
\item \label{prop:infinitycomp:claim2} Suppose further that each $\bK_i$ is a null-homologous knot with two basepoints and $\cA$ consists of two arcs running from $\bK_1$ to $\bK_2$ (necessarily implying that $\Sigma,$ $\Sigma_{\ws}$ and $\Sigma_{\zs}$ are connected). If $Y_1 $ and $Y_2$ are rational homology 3-spheres and $b_1(W)=b_2^+(W)=0$,  then the induced map on homology
\[
F_{W,\cF,\frs}\colon \cHFL^\infty(Y_1,\bK_1,\frs|_{Y_1})\to \cHFL^\infty(Y_2,\bK_2,\frs|_{Y_2})
\]
 is an isomorphism. 
 \item \label{prop:infinitycomp:claim3} If $(W,\cF)\colon (S^3,\bK_1)\to (S^3,\bK_2)$ is a knot cobordism such that $\cF$ is decorated as in Part~\eqref{prop:infinitycomp:claim2} and $b_1(W)=b_2^+(W)=0$, then under canonical the identification 
 \[
 \cHFL^\infty(S^3,\bK_i)\iso \cR^\infty=\bF_2[U,U^{-1},V,V^{-1}]
 \]
  given by the gradings, the map $F_{W,\cF,\frs}$ is equal to the map
\[
1\mapsto U^{-d_1/2} V^{-d_2/2},
\]
 where
\[
d_1=\frac{c_1(\frs)^2-2\chi(W)-3\sigma(W)}{4}-2g(\Sigma_{\ve{w}})
\]
 and
\[
d_2=\frac{c_1(\frs-\PD[\Sigma])^2-2\chi(W)-3\sigma(W)}{4}-2g(\Sigma_{\ve{z}}).
\]
\end{enumerate}
\end{thm}

\begin{proof} We consider Claim~\eqref{prop:infinitycomp:claim1} first. It follows from Theorem~\ref{thm:generalreductionformula} and the Alexander grading formula that \[
(R_{\ws})_{i+\Delta A}\circ (F_{W,\cF,\frs})_i=F_{W,\Gamma_{\ws},\frs}^B\circ (R_{\ws})_i,
\]
where $(F_{W,\cF,\frs})_i$ denotes $F_{W,\cF,\frs}$ restricted to the $i^{\text{th}}$ Alexander grading. Noting that $(R_{\ws})_{i+\Delta A}=(R_{\ws})_{i}\circ V^{-\Delta A}$, we obtain the formula
\[
(F_{W,\cF,\frs})_i=V^{\Delta A}\cdot (R_{\ws})_i^{-1}\circ F_{W,\Gamma_{\ws},\frs}^B\circ (R_{\ws})_i
\] 
Taking the direct sum over Alexander gradings yields the first formula in Claim~\eqref{prop:infinitycomp:claim1}. This strategy also adapts to prove the stated formula for $\ve{R}_{\zs}$ and $F_{W,\Gamma_{\zs},\frs-\PD[\Sigma]}^A$ as well.

We now consider Claim~\eqref{prop:infinitycomp:claim2}. By Lemma~\ref{lem:computereductionsconnectedsurface}
\[
F_{W,\cF,\frs}|_{V=1}\simeq F_{W,c_{\ws},\frs}(\iota_*\xi(\Sigma_{\ws})\otimes -),
\]
for some path $c_{\ws}$ from $Y_1$ to $Y_2$. According to the proof of \cite{OSIntersectionForms}*{Theorem 9.6},  the map $F_{W,c_{\ws},\frs}$ is an isomorphism on $\HF^\infty$ for all $\frs\in \Spin^c(W)$. Since $b_1(W)=0$, the actions of the elements in $\Lambda^*(H_1(W;\Z)/\Tors)$ in the formula for $\xi(\Sigma_{\ws})$ vanish, and hence 
\[
F_{W,c_{\ws},\frs}(\iota_*\xi(\Sigma_{\ws})\otimes -)=U^{g(\Sigma_{\ws})}\cdot F_{W,c_{\ws},\frs}(-).
\]
 Since the reduction maps $\ve{R}_{\ws}$ are isomorphisms on homology by Lemma~\ref{lem:reductionmapischainisomorphism}, by applying Claim~\eqref{prop:infinitycomp:claim1} we conclude that $F_{W,\cF,\frs}$ is an isomorphism on $\cHFL^\infty$, since it is a composition of isomorphisms.

Finally, we note that Claim~\eqref{prop:infinitycomp:claim3} follows from Claim~\eqref{prop:infinitycomp:claim2}, as well as our grading formula from Theorem~\ref{thm:maingradingchangeformula}, since the map $1\mapsto U^{-d_1/2} V^{-d_2/2}$ is the unique map from $\cR^\infty$ to $\cR^\infty$ which is an isomorphism of groups, and induces the correct grading change.
\end{proof}

\begin{customcor}{\ref{cor:2knotcomp}} Suppose that $\Sigma\subset S^4$ is a closed, oriented and connected surface, and $\cA$ is a simple closed curve on $\Sigma$ which divides $\Sigma$ into two connected subsurfaces, $\Sigma_{\ve{w}}$ and $\Sigma_{\ve{z}}$. The link cobordism map
\[
F_{S^4,\cF,\frs_0}\colon \cCFL^-(\varnothing)\to \cCFL^-(\varnothing)
 \] 
 is equal to the map 
\[
1\mapsto U^{g(\Sigma_{\ve{w}})} V^{g(\Sigma_{\ve{z}})},
\]
under the canonical identification of $\cCFL^-(\varnothing)\iso \bF_2[U,V]$.
\end{customcor}

In particular the map for the link cobordism obtained by puncturing a 2-knot in any homotopy $S^4$ is the identity map.

\begin{rem}Combining Corollary~\ref{cor:2knotcomp} with the composition law, we can compute the effect of taking the connected sum of a link cobordism $(W,\cF)$ with an oriented surface $\Sigma$ contained in a ball in $W$ which is disjoint from $\cF$. If we add $\Sigma$ to a type-$\ws$ region, then the effect is to multiply $F_{W,\cF,\frs}$ by $U^{g(\Sigma)}$. If we add $\Sigma$ to a type-$\zs$ region, then the effect is to multiply $F_{W,\cF,\frs}$ by $V^{g(\Sigma)}$. 
\end{rem}

\section{Link cobordism proofs of bounds on $\tau(K)$ and $V_k(K)$}
\label{sec:bounds}
In this section, we show that the link cobordism maps give simple proofs on the bounds on $\tau(K)$ and $V_k(K)$ stated in the introduction.

Let $\CFK^-(K)$ denote the module
\[
\CFK^-(K):=\cCFL^-(S^3,\bK,\frs_0)\otimes_{\bF_2[U,V]} \bF_2[U,V]/(V),
\]
i.e.,  $\CFK^-(K)$ is the free $\bF_2[U]$-module generated by intersection points $\xs\in \bT_{\alpha}\cap \bT_{\beta}$, and the differential counts disks $\phi$ with $n_z(\phi)=0$. We let $\HFK^-(K)$ denote the homology of $\CFK^-(K)$. We note that $\HFK^-(K)$ is isomorphic to $\bF_2[U]\oplus T$ for some torsion $\bF_2[U]$-module $T$. 

We note that by \cite{OSTLegendrian}*{Lemma~A.2},
\[
\tau(K)=-\max\{A(x): x\in \HFK^-(K) \text{ is homogeneous and non-torsion}\}.
\]

Similarly, there is an $\infty$ flavor of $\CFK^-(K)$. To avoid confusion with the bifiltered $\CFK^\infty(K)$, which is the zero Alexander graded part of $\cCFL^\infty(S^3,\bK)$, we will write $U^{-1}\CFK^-(K)$ for the module $\cCFL^-(S^3,\bK)\otimes_{\bF_2[U,V]} \bF_2[U,U^{-1},V]/(V)$, and $\HFK^\infty(K)$ for the homology group $H_*(U^{-1} \CFK^-(K))$. We note that
\[
\HFK^\infty(K)\iso \bF_2[U,U^{-1}].
\]
There is a natural map $\HFK^-(K)\to \HFK^\infty(K)$. An element $x\in \HFK^-(K)$ is non-torsion if and only if its image in $\HFK^\infty(K)$ is non-vanishing.

Suppose that $W$ is an oriented 4-manifold with boundary equal to two rational homology spheres, and $b_2^+(W)=b_1(W)=0$. If $[\Sigma]\in H_2(W,\d W;\Z)$ is a class whose image in $H_1(\d W;\Z)$ vanishes, then we can uniquely pull $[\Sigma]$ back to a class in $H_2(W;\Z)/\Tors$. With this in mind, we define 
\[
|[\Sigma]|:=\max_{\substack{C\in \Char(Q_W)\\ C^2=-b_2(W)}} \langle [\Sigma], H\rangle.
\]

If $\d W=S^3\sqcup S^3$, we can use Donaldson's diagonalization theorem to pick an orthonormal basis $e_1,\dots, e_b$  of $H^2(W;\Z)/\Tors$, and write $\PD[\Sigma]=s_1\cdot e_1+\cdots +s_b\cdot e_b.$ In this case, we have
\[
\big|[\Sigma]\big|=|s_1|+\cdots +|s_b|.
\]

We now give our link cobordism proof of Ozsv\'{a}th and Szab\'{o}'s bound:
\begin{thm}[\cite{OS4ballgenus}*{Theorem 1.1}]\label{thm:OSboundtauproof}Suppose that $(W,\Sigma)\colon (S^3,K_1)\to (S^3,K_2)$ is an oriented knot cobordism with $b_2^+(W)=b_1(W)=0$. Then 
\[
\tau(K_2)\le \tau(K_1)-\frac{\big|[\Sigma]\big|+[\Sigma]\cdot [\Sigma]}{2}+g(\Sigma).
\]
\end{thm}
\begin{proof}Let $\bK_1$ and $\bK_2$ denote $K_1$ and $K_2$ decorated with two basepoints. Construct a decorated link cobordism $(W,\cF)\colon (S^3, \bK_1)\to (S^3,\bK_2)$ by letting $\cF$ be obtained by adding two parallel dividing arcs to $\Sigma$, so that $\Sigma_{\zs}$ is a disk, and $\Sigma_{\ws}$ is a genus $g(\Sigma)$ surface with one boundary component. We have a commutative diagram
\begin{equation}
\begin{tikzcd} \cHFL^-(S^3,\bK_1)\arrow{d}\arrow{r}{F_{W,\cF,\frs}} &\cHFL^-(S^3,\bK_2)\arrow{d}\\
\HFK^-(K_1)\arrow{r}{F_{W,\cF,\frs}}& \HFK^-(K_2),
\end{tikzcd}
\label{eq:commutativediagramHFLinftytau}
\end{equation}
induced by the natural maps $\cCFL^-(S^3,\bK_i)\to \CFK^-(K_i).$ There is a similar commutative diagram involving $\cHFL^\infty$ and $\HFK^\infty.$ 

By Part~\eqref{prop:infinitycomp:claim3} of Theorem~\ref{thm:mostgeneralcompHFLinfty}, the map $F_{W,\cF,\frs}$ on $\cHFL^\infty$ will be multiplication by $U^{-d_1/2}V^{-d_2/2}$, where 
\[
d_1=\frac{c_1(\frs)^2+b_2(W)}{4}-2g(\Sigma)\qquad\text{and}\qquad  d_2=\frac{c_1(\frs-\PD[\Sigma])^2+b_2(W)}{4}. 
\]
From Equation~\eqref{eq:commutativediagramHFLinftytau} and its analog for $\cHFL^\infty$ and $\HFK^\infty$, the map $F_{W,\cF,\frs}$ will be an isomorphism on $\HFK^\infty$ if and only if $c_1(\frs-\PD[\Sigma])^2+b_2(W)=0$. Using Theorem~\ref{thm:maingradingchangeformula}, the Alexander grading change of $F_{W,\cF,\frs}$ is (after manipulating the expression slightly)
\[
\frac{\langle c_1(\frs-\PD[\Sigma]), [\Sigma]\rangle +[\Sigma]\cdot [\Sigma]}{2}-g(\Sigma).
\] 
Since $F_{W,\cF,\frs}$ maps non-torsion elements of $\HFK^-(K_1)$ to non-torsion elements of $\HFK^-(K_2)$ when $d_2=0$, it follows that
\begin{equation}
\tau(K_2)\le \tau(K_1)-\frac{\langle c_1(\frs-\PD[\Sigma]), [\Sigma]\rangle +[\Sigma]\cdot [\Sigma]}{2}+g(\Sigma),\label{eq:boundforparticularspincstructure}
\end{equation}
 for any $\frs$ with $c_1(\frs-\PD[\Sigma])^2+b_2(W)=0$. Minimizing Equation~\eqref{eq:boundforparticularspincstructure} over such $\Spin^c$ structures, we obtained the bound in the theorem statement.
\end{proof}

Next, we consider Rasmussen's local $h$-invariants, and the bounds on the slice genus. Recall that the standard, $\Z \oplus \Z$ filtered, full knot Floer complex $\CFK^\infty(K)$ is isomorphic to the subset of $\cCFL^\infty(S^3,\bK)$ in zero Alexander grading. Given an integer $k\ge 0$, we define the sub-complex of $\cCFL^\infty(S^3,\bK)$
\[
A_k(K):=\Span_{\bF_2}\{ U^i V^j\cdot \xs: A(\xs)+j-i=0, \quad i\ge 0,\quad  j\ge -k\}.
\]
Writing $\hat{U}$ for the  product $UV$, $A_k^-(K)$ is a $\bF_2[\hat{U}]$-module. The invariant $V_k(K)$ is defined as
\begin{equation}
V_k(K):=-\frac{1}{2}\max \{\gr(x): x\in H_*(A_k(K)) \text{ is homogeneous and non-torsion}\}.\label{eq:Vkdef}
\end{equation}
In the above expression, $\gr$ denotes either $\gr_{\ws}$ or $\gr_{\zs}$; they are equal when $A=0$.

We note that Rasmussen's original definition \cite{RasmussenKnots}*{Definition~7.1} for $V_k$ was in terms of the $d$-invariants of large surgeries on $K$, though using the large surgery formula \cite{OSKnots}*{Section~4},  Equation~\eqref{eq:Vkdef} is equivalent.

We now give a simple proof of Rasmussen's bound:
\begin{thm}[\cite{RasmussenGodaTeragaito}*{Theorem~2.3}]\label{thm:Rasmussensbound} If $K$ is a knot in $S^3$, then 
\[V_k(K)\le \left\lceil \frac{g_4(K)-k}{2} \right\rceil,\] for any $0\le k\le g_4(K)$.
\end{thm}
\begin{proof}Suppose that $\Sigma$ is a surface in $B^4$ with $\d \Sigma=K$. After puncturing $(B^4,\Sigma)$ along $\Sigma$, we obtain a genus $g$ knot cobordism $([0,1]\times S^3, \Sigma_0)$ from $(S^3,U)$ to $(S^3,K)$. By stabilizing $\Sigma_0$ with a null-homologous torus, if necessary, we may assume that $g(\Sigma_0)-k$ is even. We decorate $\Sigma_0$ with a dividing set $\cA$ consisting of two arcs, both running from $U$ to $K$, which divide $\Sigma_0$ into two connected components, $\Sigma_{\ws}$ and $\Sigma_{\zs}$, such that
\[
g(\Sigma_{\ws})=\frac{g(\Sigma)-k}{2}\qquad \text{and} \qquad g(\Sigma_{\zs})=\frac{g(\Sigma)+k}{2}.
\]
Let $(W,\cF)$ denote $([0,1]\times S^3,(\Sigma_0,\cA))$. On $\cHFL^\infty$, the map $F_{W,\cF,\frs}$ is multiplication by $U^{g(\Sigma_{\ws})} V^{g(\Sigma_{\zs})}$ by Part~\eqref{prop:infinitycomp:claim3} of Theorem~\ref{thm:mostgeneralcompHFLinfty}, and it changes Alexander grading by $g(\Sigma_{\zs})-g(\Sigma_{\ws})=k$. The map $F_{W,\cF,\frs}$ does not map $A_0^-(U)\iso \bF_2[\hat{U}]$ into $A_0^-(K)$, instead $F_{W,\cF,\frs}$ maps $\bF_2[\hat{U}]$ into the subset of $\cCFL^-(S^3,\bK)$ of Alexander grading $+k$. The subset of $\cCFL^-(S^3,\bK)$ in Alexander grading $k$ is canonically isomorphic as an $\bF[\hat{U}]$-module to $A_k^-(K)$; the isomorphism is given by the action of $V^{-k}$. Hence $V^{-k}\cdot F_{W,\cF,\frs}(1)$ is an element of $H_*(A_k^-(K))$ which is non-torsion. Furthermore, $V^{-k} \cdot F_{W,\cF,\frs}(1)$ has Maslov grading $-(g(\Sigma)-k)$ since the $\gr_{\ws}$-grading change of $V^{-k}\cdot F_{W,\cF,\frs}$ is $-(g(\Sigma)-k)$, and $\gr_{\ws}=\gr_{\zs}$ on $A_k^-(K)$. It follows that
\[
V_k(K)\le \frac{g(\Sigma)-k}{2},
\]
completing the proof.
\end{proof}

\section{$t$-modified knot Floer homology and a bound on the $\Upsilon_K(t)$ invariant}
\label{sec:Upsilon}
In \cite{OSSUpsilon}, Ozsv\'{a}th, Stipsicz and Szab\'{o} define a homomorphism from the smooth concordance group to the group of piecewise linear functions from $[0,2]$ to $\R$. In this section,  we prove our bound on $\Upsilon_K(t)$, Theorem~\ref{thm:upsiloninvariant}.

We recall the construction of $\Upsilon_K(t)$. Suppose that $K\subset S^3$ is an oriented knot, and  $\cH=(\Sigma, \as,\bs,w,z)$ is a diagram for $(S^3,K,w,z)$. If $t\in [0,2]$, we define the $t$-grading on intersection points $\ve{x}\in \bT_\alpha\cap \bT_\beta$ by
\[
\gr_t(\ve{x})=\left(1-\frac{t}{2}\right)\gr_{\ve{w}}(\ve{x})+\frac{t}{2} \gr_{\ve{z}}(\ve{x}).
\] 
If $t=\frac{m}{n}$, with $m$ and $n$ relatively prime,   $\tCFK^-(K)$ is the free $\bF_2[v^{1/n}]$-module generated by intersection points $\ve{x}\in \bT_{\alpha}\cap \bT_{\beta}$. The module $\tCFK^-(K)$ has an endomorphism
\[
\d(\ve{x})=\sum_{\ve{y}\in \bT_{\alpha}\cap \bT_{\beta}}\sum_{\substack{\phi\in\pi_2(\ve{x},\ve{y})\\ \mu(\phi)=1}} \# \Hat{\cM}(\phi)\cdot  v^{tn_z(\phi)+(2-t)n_w(\phi)}\cdot \ve{y},
\]
which squares to zero. The module $\tHFK^-(K)$ is defined as the homology of $(\tCFK^-(K),\d)$.

 The grading $\gr_t$ induces a grading on $\tCFK^-(K)$. The differential lowers degree by $1$, and the action of $v$ also lowers degree by $1$. The number $\Upsilon_K(t)\in \R$ is defined as the maximal $\gr_t$-grading of any homogeneous non-torsion element of $\tHFK^-(K)$.

We first need to understand the relationship between $\tCFK^-(K)$ and $\cCFL^-(S^3,\bK)$. We define the rings
\[
\cR_t^-=\bF_2[U,V,v^{1/n}]/(U-v^{2-t}, V-v^t)
\] 
and
\[
\cR_t^{\infty}:=\bF_2[U,U^{-1},V,V^{-1},v^{1/n}, v^{-1/n}]/(U-v^{2-t}, V-v^{t}).
\]

\begin{lem}\label{lem:tensorproductoftHFK}If $\bK=(K,w,z)$ is a doubly based knot in $S^3$, there are canonical isomorphisms 
\[
\cCFL^-(S^3,\bK)\otimes_{\bF_2[U,V]} \cR_t^-\iso \tCFK^-(K)\quad \text{and}\quad \cCFL^\infty(S^3,\bK)\otimes_{\bF_2[U,V,U^{-1},V^{-1}]} \cR_t^\infty\iso \tCFK^\infty(K).
\]
\end{lem}

\begin{proof}We focus on the first isomorphism, involving the minus flavors. We first describe an isomorphism between the rings $\cR_t^-$ and $\bF_2[v^{1/n}]$. Noting that $t=\tfrac{m}{n}$, we define a  map from $\bF_2[U,V,v^{1/n}]/(U-v^{2-t}, V-v^t)$ to $\bF_2[v^{1/n}]$ by the formula
\[
U^i V^j v^{s}\mapsto v^{i(2-t)+jt+s}
\]
 and a map in the opposite direction by the formula
\[
v^{s}\mapsto v^{s}.
\] To define maps between the chain complexes, we use the above maps on rings, extended over linear combinations of intersection points. That these maps are chain maps is immediate. It is also clear that these two maps are inverses of each other.

Essentially the same argument works for the $\infty$ flavors of the complexes.
\end{proof}

Phrased another way, $\cR_t^-$ is isomorphic to $\bF_2[v^{1/n}]$ with a module action of $\bF_2[U,V]$ declared. 

As a consequence, if $(W,\cF)\colon (S^3,\bK_1)\to (S^3,\bK_2)$ is a link cobordism and $\frs\in \Spin^c(W)$, then the link cobordism map $F_{W,\cF,\frs}$ induces a map 
\[ 
\tF_{W,\cF,\frs}\colon \tCFK^-(K_1)\to \tCFK^-(K_2).
\]

\begin{lem}\label{lem:tHFLmapnontrivialontpwers} Suppose that $(W,\cF)$ is a decorated knot cobordism from $(S^3,\bK_1)$ to $(S^3,\bK_2)$ and $\frs\in \Spin^c(W)$. Then the map $\tF_{W,\cF,\frs}\colon \tHFK^-(K_1)\to \tHFK^-(K_2)$ maps non-torsion elements to non-torsion elements if and only if the induced map $F_{W,\cF,\frs}\colon \cHFL^\infty(S^3,\bK_1)\to \cHFL^\infty(S^3,\bK_2)$ is an isomorphism.
\end{lem}
\begin{proof}Since $\cHFL^\infty(S^3,\bK_i)\iso \bF_2[U,V,U^{-1}, V^{-1}]$ and the map $F_{W,\cF,\frs}$ is graded and $\bF_2[U,V, U^{-1}, V^{-1}]$-equivariant, it follows that $F_{W,\cF,\frs}$ is an isomorphism on $\cHFL^\infty$ if and only if it is nonzero.

Similarly, the map $\tF_{W,\cF,\frs}$ maps non-torsion elements of $\tHFK^-(K_1)$ to non-torsion elements of $\tHFK^-(K_2)$ if and only if  the induced map on $\tHFK^\infty$ is non-zero.

Hence it is sufficient to show that the map $F_{W,\cF,\frs}$ is non-zero on $\cHFL^\infty$ if and only if $\tF_{W,\cF,\frs}$ is non-zero on $\tHFK^\infty$.

Using the $\gr_{\ws}$ and $\gr_{\zs}$ gradings, we can canonically identify $\cHFL^\infty(S^3,\bK_i)$ as $\bF_2[U,V,U^{-1},V^{-1}]$. Similarly, we can canonically identify $\tHFK^\infty(K_i)$ with $\bF_2[v^{-1/n},v^{1/n}]$. Since the maps $F_{W,\cF,\frs}$ and $\tF_{W,\cF,\frs}$ are graded, under the above identifications, they must be equal to multiplication by $c\cdot U^i V^j$ and $c'\cdot  v^\ell$, respectively, for $i,j\in \Z$, $\ell\in \R$ and $c,c'\in \bF_2$. We have a commutative diagram
\[
\begin{tikzcd}\cHFL^\infty(S^3,\bK_1)\arrow{d}\arrow{r}{F_{W,\cF,\frs}} &\cHFL^\infty(S^3,\bK_2)\arrow{d}\\
\tHFK^\infty(K_1)\arrow{r}{\tF_{W,\cF,\frs}}& \tHFK^\infty(K_2).
\end{tikzcd}
\]
Since the two vertical arrows are non-zero (they are the natural maps from $\bF_2[U,V,U^{-1},V^{-1}]$ to $\bF_2[U,V,U^{-1}, V^{-1}]\otimes \cR^\infty_t$), and the top horizontal arrow is identified with multiplication by $c\cdot  U^i V^j$ and the bottom arrow is identified with multiplication by $c'\cdot v^{\ell}$, we conclude that $c=c'$ and $\ell=i(2-t)+jt$. In particular, the map $F_{W,\cF,\frs}$ is nonzero if and only if $c=1$, which occurs if and only if $c'=1$, which occurs if and only if  $\tF_{W,\cF,\frs}$ is nonzero, completing the proof.
\end{proof}

If $(W,\cF)$ is a link cobordism, the $\gr_t$-grading change of the map $\tF_{W,\cF,\frs}$ can be computed using the $\gr_{\ve{w}}$- and $\gr_{\ve{z}}$-grading change formula from Theorem~\ref{thm:maingradingchangeformula}. If $\ve{x}$ is a homogeneously graded element, then
 \begin{equation}
 \begin{split}
 &\gr_{t}(F_{W,\cF,\frs}(\ve{x}))-\gr_t(\ve{x})\\
 =&\frac{c_1(\frs)^2-2\chi(W)-3\sigma(W)}{4}+t\cdot \left(\frac{-\langle c_1(\frs), \hat{\Sigma}\rangle+[\hat{\Sigma}]\cdot[\hat{\Sigma}]}{2}\right)+\left(1-\frac{t}{2}\right)\cdot \tilde{\chi}(\Sigma_{\ve{w}})+\frac{t}{2}\cdot\tilde{\chi}(\Sigma_{\ve{z}}).
 \end{split}
 \label{eq:grtgradingchange}
 \end{equation}

We now proceed to prove Theorem~\ref{thm:upsiloninvariant}. Recall from the introduction that if $W:S^3\to S^3$ is a cobordism and $[\Sigma]\in H_2(W,\d W;\Z)$ is a class, we can uniquely pull back $[\Sigma]$ to an element of $H_2(W;\Z)$, for which we will also write $\Sigma$. We define the quantity
\[
M_{[\Sigma]}(t):= \max_{C\in \Char(Q_W)} \frac{C^2+b_2(W)-2 t\cdot \langle C, [\Sigma]\rangle+2t\cdot [\Sigma]\cdot [\Sigma]}{4},
\]

 \begin{customthm}{\ref{thm:upsiloninvariant}}Suppose that $(W,\Sigma)\colon (S^3,K_1)\to (S^3,K_2)$ is an oriented knot cobordism with $b_2^+(W)=b_1(W)=0$. Then 
 \[
 \Upsilon_{K_2}(t)\ge \Upsilon_{K_1}(t)+ M_{[\Sigma]}(t)+ g(\Sigma)\cdot (|t-1|-1).
 \]
 \end{customthm}

 \begin{proof}We form the surface with divides $\cF$ by decorating $\Sigma$ with two parallel arcs, both running from $K_1$ to $K_2$, so that both $\Sigma_{\ws}$ and $\Sigma_{\zs}$ are connected and
 \[
g(\Sigma_{\ws})=0\qquad \text{and}\qquad  g(\Sigma_{\zs})=g(\Sigma). 
 \]
 By Part~\eqref{prop:infinitycomp:claim3} of Theorem~\ref{thm:mostgeneralcompHFLinfty}, the induced map $F_{W,\cF,\frs}$ from $\cHFL^\infty(S^3,\bK_1)$ to $\cHFL^\infty(S^3,\bK_2)$ can be identified with multiplication by $U^{-d_1/2}V^{-d_2/2}$, where $d_1$ is the change in $\gr_{\ws}$ grading, and $d_2$ is the change in $\gr_{\zs}$ grading. In particular, the map $F_{W,\cF,\frs}$ is nonzero on $\cHFL^\infty$. By Lemma~\ref{lem:tHFLmapnontrivialontpwers}, the induced map $\tF_{W,\cF,\frs}$ sends non-torsion elements of $\tHFK^-(K_1)$ to non-torsion elements of $\tHFK^-(K_2)$. Using the formula from Equation~\eqref{eq:grtgradingchange} for the change in the $\gr_t$ grading, we conclude that
 \[
\Upsilon_{K_2}(t)\ge \Upsilon_{K_1}(t)+G_{\frs}(t)-t\cdot g(\Sigma), 
 \]
 where 
 \[
 G_{\frs}(t):=\frac{c_1(\frs)^2-2\chi(W)-3\sigma(W)}{4}+t\cdot \left(\frac{-\langle c_1(\frs), \hat{\Sigma}\rangle+[\hat{\Sigma}]\cdot[\hat{\Sigma}]}{2}\right).
 \] Using the fact that $b_1(W)=b_2^+(W)=0$, we compute
 \[
G_{\frs}(t)=\frac{c_1(\frs)^2+b_2(W)-2t\cdot \langle c_1(\frs), [\Sigma]\rangle +2t\cdot [\Sigma]\cdot [\Sigma]}{4}.
 \]
 Taking the maximum over all $\frs\in \Spin^c(W)$, we obtain that
 \begin{equation}
\Upsilon_{K_2}(t)\ge \Upsilon_{K_1}(t)+M_{[\Sigma]}(t)-t\cdot g(\Sigma). \label{eq:Upsilonbound1}
 \end{equation}
 
 An easy computation shows that  $G_{\frs}(t)=G_{\bar{\frs}+\PD[\Sigma]}(2-t)$, so
 \[
M_{[\Sigma]}(t)=M_{[\Sigma]}(2-t). 
 \]
 On the other hand $\Upsilon_{K_i}(t)=\Upsilon_{K_i}(2-t)$ by \cite{OSSUpsilon}*{Proposition~1.2}, so from Equation~\eqref{eq:Upsilonbound1} we obtain
 \begin{equation}
 \Upsilon_{K_2}(t)\ge \Upsilon_{K_1}(t)+M_{[\Sigma]}(t)-(2-t)\cdot g(\Sigma). \label{eq:Upsilonbound2}
 \end{equation}
 Combining Equation~\eqref{eq:Upsilonbound1} and \eqref{eq:Upsilonbound2} yields the theorem statement.
 
 \end{proof}

\begin{rem}One could hope to refine the above bound even further, by considering different surfaces with divides and trying to optimize the expression
 \begin{equation}
 \left(1-\frac{t}{2}\right)\cdot \tilde{\chi}(\Sigma_{\ve{w}})+\frac{t}{2}\cdot \tilde{\chi}(\Sigma_{\ve{z}}).\label{eq:surfacecontributiongrading}
 \end{equation}
It is straightforward to see that if we restrict to dividing sets on $\Sigma$ consisting of two arcs and $\Sigma_{\ws}$ and $\Sigma_{\zs}$ are both connected, then the expression in Equation~\eqref{eq:surfacecontributiongrading} is maximized when $g(\Sigma_{\ws})=g(\Sigma)$ or when $g(\Sigma_{\zs})=g(\Sigma)$, depending on the value of $t$. Furthermore, the maximum value is $g(\Sigma)\cdot (|t-1|-1)$. 

One could investigate more complicated dividing sets as well, though using Part~\eqref{prop:infinitycomp:claim1} of Theorem~\ref{thm:mostgeneralcompHFLinfty} we see that the induced map $F_{W,\cF,\frs}$ will be an isomorphism on $\cHFL^\infty$ if and only if the graph cobordism $F_{W,\Gamma_{\ws}, \frs}^B$ is an isomorphism on $\HF^\infty$, for a ribbon 1-skeleton $\Gamma_{\ws}$ of $\Sigma_{\ws}$. However, noting that $\Phi_w=0$ on $\HF^\infty$ for a singly based 3-manifold, it is not hard to use Proposition~\ref{prop:graphcobordismcomp} to show that if $b_1(W)=0$, then map $F_{W,\Gamma_{\ws},\frs}^B\colon \HF^\infty(S^3,w_1)\to \HF^\infty(S^3,w_2)$ is non-vanishing if and only if $\Sigma_{\ws}$ is connected and has exactly one boundary component (note that this condition is symmetric between $\Sigma_{\ws}$ and $\Sigma_{\zs}$, since $\Sigma_{\ws}$ is connected and has exactly one boundary component if and only if $\Sigma_{\zs}$ is connected and has exactly one boundary component).
 \end{rem}

\subsection{Additional examples of the bound}
 \label{subsec:examplesofbounds}
In this section, we describe how several properties of $\Upsilon_K(t)$ from \cite{OSSUpsilon} which are simple corollaries from Theorem~\ref{thm:1}.

 \begin{cor}[\cite{OSSUpsilon}*{Theorem~1.11}]If $(W,\Sigma)\colon (S^3,K_1)\to (S^3,K_2)$ is an oriented knot cobordism, and $W$ is a rational homology cobordism, then
 \[
 |\Upsilon_{K_2}(t)-\Upsilon_{K_1}(t)|\le t \cdot g(\Sigma).
 \]
 \end{cor}
 \begin{proof}This follows immediately from Theorem~\ref{thm:upsiloninvariant}.
 \end{proof}

 \begin{cor}[\cite{OSSUpsilon}*{Proposition 1.10}]
  If $K_-$ and $K_+$ are knots in $S^3$, which differ by a crossing change, then
 \[
 \Upsilon_{K_+}(t)\le \Upsilon_{K_-}(t)\le \Upsilon_{K_+}(t)+1-|t-1|.
 \]
 \end{cor}
 
 \begin{proof}Suppose that $K_-$ and $K_+$ are two knots which differ by a crossing change (and $K_-$ has the negative crossing and $K_+$ has the positive crossing). We can construct a negative definite link cobordism $(W_{1},\Sigma_{1})$ from $(S^3,K_-)$ to $(S^3,K_+)$ and also a negative definite link cobordism $(W_{2},\Sigma_{2})$ from $(S^3,K_+)$ to $(S^3,K_-)$. Each is formed by adding a 2-handle with framing $-1$, around the crossing, as shown in Figure~\ref{fig::14}.
 
 A generator $E_i$ of $H_2(W_i;\Z)$ is given by taking a Seifert disk of the unknot, and capping with the core of the 2-handle. Hence  the homology class of the surface $\Sigma_i$ in $W_i$ can be computed from the intersection number  of the knot $K_{\pm}$ on the incoming end of $W_i$ with the Seifert disk for the $-1$ framed unknot. 
 
 \begin{figure}[ht!]
 	\centering
\begingroup%
  \makeatletter%
  \providecommand\color[2][]{%
    \errmessage{(Inkscape) Color is used for the text in Inkscape, but the package 'color.sty' is not loaded}%
    \renewcommand\color[2][]{}%
  }%
  \providecommand\transparent[1]{%
    \errmessage{(Inkscape) Transparency is used (non-zero) for the text in Inkscape, but the package 'transparent.sty' is not loaded}%
    \renewcommand\transparent[1]{}%
  }%
  \providecommand\rotatebox[2]{#2}%
  \newcommand*\fsize{\dimexpr\f@size pt\relax}%
  \newcommand*\lineheight[1]{\fontsize{\fsize}{#1\fsize}\selectfont}%
  \ifx\svgwidth\undefined%
    \setlength{\unitlength}{196.4983592bp}%
    \ifx\svgscale\undefined%
      \relax%
    \else%
      \setlength{\unitlength}{\unitlength * \real{\svgscale}}%
    \fi%
  \else%
    \setlength{\unitlength}{\svgwidth}%
  \fi%
  \global\let\svgwidth\undefined%
  \global\let\svgscale\undefined%
  \makeatother%
  \begin{picture}(1,0.57416407)%
    \lineheight{1}%
    \setlength\tabcolsep{0pt}%
    \put(0,0){\includegraphics[width=\unitlength,page=1]{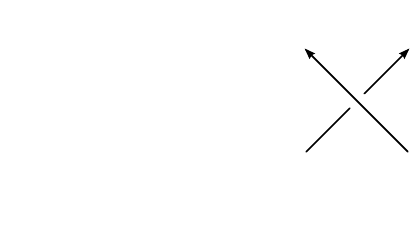}}%
    \put(0.85557377,0.16555341){\color[rgb]{0,0,0}\makebox(0,0)[lt]{\lineheight{0}\smash{\begin{tabular}[t]{l}$K_-$\end{tabular}}}}%
    \put(0,0){\includegraphics[width=\unitlength,page=2]{fig14.pdf}}%
    \put(0.09839506,0.16555341){\color[rgb]{0,0,0}\makebox(0,0)[lt]{\lineheight{0}\smash{\begin{tabular}[t]{l}$K_+$\end{tabular}}}}%
    \put(0,0){\includegraphics[width=\unitlength,page=3]{fig14.pdf}}%
    \put(0.62606365,0.51477607){\color[rgb]{0,0,0}\makebox(0,0)[lt]{\lineheight{0}\smash{\begin{tabular}[t]{l}$-1$\end{tabular}}}}%
    \put(0.41764118,0.4955503){\color[rgb]{0,0,0}\makebox(0,0)[rt]{\lineheight{0}\smash{\begin{tabular}[t]{r}$W_{1}=$\end{tabular}}}}%
    \put(0.46071681,0.06959555){\color[rgb]{0,0,0}\makebox(0,0)[rt]{\lineheight{0}\smash{\begin{tabular}[t]{r}$W_{2}=$\end{tabular}}}}%
    \put(0,0){\includegraphics[width=\unitlength,page=4]{fig14.pdf}}%
    \put(0.58262117,0.0768595){\color[rgb]{0,0,0}\makebox(0,0)[lt]{\lineheight{0}\smash{\begin{tabular}[t]{l}$-1$\end{tabular}}}}%
  \end{picture}%
\endgroup%

 	\caption{\textbf{The  knot cobordism $(W_{1},\Sigma_{1})$ from $(S^3,K_+)$ to $(S^3,K_-)$, and the knot cobordism $(W_2,\Sigma_2)$, in the opposite direction.}}\label{fig::14}
 \end{figure}

After orienting $E$ appropriately, we can take  $[\Sigma_1]=2\cdot E$ and $[\Sigma_2]=0$. Hence
 \[
 M_{[\Sigma_2]}(t)=0\qquad \text{and}\qquad M_{[\Sigma_1]}(t)=-(1-|t-1|).
 \] 
  Applying Theorem~\ref{thm:upsiloninvariant}, we see
\[
\Upsilon_{K_+}(t)\le \Upsilon_{K_-}(t)\le \Upsilon_{K_+}(t)+1-|t-1|.
\]
 \end{proof}
 
 \subsection{Positive torus knots}
 \label{sec:positivetorusknots}
 We show that the bound from Theorem~\ref{thm:upsiloninvariant} is sharp for torus knots, in the following sense:
 
 \begin{prop}\label{prop:generaltorusknots}Given any positive torus knot $T_{a,b}$, there is a knot cobordism $(W,\Sigma)\colon (S^3,U)\to (S^3,T_{a,b})$ with $g(\Sigma)=b_1(W)=b_2^+(W)=0$ and
 \[
\Upsilon_{T_{a,b}}(t)=M_{[\Sigma]}(t). 
 \]
 \end{prop}
 
 Before we prove Proposition~\ref{prop:generaltorusknots}, it is convenient for us to consider the torus knot $T_{n,n+1}$. Given an integer $n$, we define the quantity
\begin{equation}
m_n(t):=\max_{c\in 2\Z+1} \frac{-c^2+1+2tcn-2tn^2}{4}.\label{def:mnt}
\end{equation}
 
 \begin{lem}\label{lem:UpsilonTnn+1}We have $\Upsilon_{T_{n,n+1}}(t)=m_n(t)$.
 \end{lem}
 \begin{proof}We use the computation of $T_{n,n+1}$ due to Ozsv\'{a}th, Stipsicz and Szab\'{o} \cite{OSSUpsilon}*{Proposition~5}, which states that if  $t\in \left[\tfrac{2i}{n}, \tfrac{2i+2}{n}\right],$ then
 \begin{equation}
 \Upsilon_{T_{n,n+1}}(t)=-i(i+1)-\frac{1}{2}n(n-1-2i)t.\label{eq:OSScompUpsilonTnn1}
 \end{equation}
 Although one can directly compare this expression with $m_n(t)$, we also note that there is a negative definite knot cobordism $(W,\Sigma)$ from $(S^3,T_{n,1})$ (the unknot) to $(S^3, T_{n,n+1})$ which is obtained by embedding $T_{n,1}$ on a standard torus $T^2\subset S^3$, and performing $-1$ surgery on the core of one of the two handlebody components of $S^3\setminus T^2$. Theorem~\ref{thm:upsiloninvariant} implies 
 \begin{equation}
 \Upsilon_{T_{n,n+1}}(t)\ge M_{[\Sigma]}(t)= m_n(t).
 \label{eq:comareupsilon1}
 \end{equation}
On the other hand, if $t\in \left[\tfrac{2i}{n}, \tfrac{2i+2}{n}\right]$, we   plug $c=2n+1$ into the expression for $m_n(t)$ in Equation~\eqref{def:mnt}, and after an easy algebraic manipulation, we obtain
\begin{equation}
m_n(t)\ge -i(i+1)-\frac{1}{2}n(n-1-2i)t.\label{eq:plugintogetbound}
\end{equation}
 Combining   Equations~\eqref{eq:OSScompUpsilonTnn1}, \eqref{eq:comareupsilon1} and \eqref{eq:plugintogetbound}, we obtain the equality $m_n(t)=\Upsilon_{T_{n,n+1}}(t)$.
 
 \end{proof}
 
 \begin{proof}[Proof of Proposition~\ref{prop:generaltorusknots}]
  Suppose that $(a,b)$ is a pair of positive, relatively prime integers. Using the Euclidean algorithm, we can find a sequence $\{(a_i,b_i)\}_{i=1}^n$ such that \begin{enumerate}
   \item $(a_1,b_1)=(1,1)$,
	\item   $(a_n,b_n)=(a,b)$, and
   \item  $(a_{i+1}, b_{i+1})=(a_i+b_i, b_i)$ or $(a_{i+1},b_{i+1})=(a_i,b_i+a_i).$
   \end{enumerate}
   According to Feller and Krcatovich's recursive formula \cite{FellerKrcatovichUpsilon}*{Proposition 6}, we have

   \begin{equation}	\Upsilon_{T_{a_{i+1}, b_{i+1}}}(t)=\Upsilon_{T_{a_i,b_i}}(t)+\Upsilon_{T_{b_i,b_i+1}}(t)\qquad \text{or} \qquad 	\Upsilon_{T_{a_{i+1}, b_{i+1}}}(t)=\Upsilon_{T_{a_i,b_i}}(t)+\Upsilon_{T_{a_i,a_i+1}}(t)\label{eq:bothcasesupsilonrecursion}
   \end{equation}
depending on whether    $(a_{i+1},b_{i+1})=(a_i+b_i,b_i)$ or $(a_{i+1},b_{i+1})=(a_i,b_i+a_i)$ (respectively).

   On the other hand, there is a negative definite knot cobordism $(W_i,\Sigma_i)$ from $(S^3,T_{a_i,b_i})$ to $(S^3, T_{a_{i+1}, b_{i+1}})$. The surface $\Sigma_i$ is an annulus. Viewing the torus knots as being embedded on a standard torus $T^2$ in $S^3$, the cobordism $W_i$ is obtained by performing $-1$ surgery on an unknot which is the core of one of the two solid tori forming $S^3\setminus T^2$.

 We note that if $(a_{i+1},b_{i+1})=(a_i+b_i,b_i)$ then 
 $M_{[\Sigma_i]}=\Upsilon_{T_{b_i,b_i+1}}(t)$ by Lemma~\ref{lem:UpsilonTnn+1}, since $[\Sigma_i]=b_i\cdot E$, for a generator $E$ of $H_2(W_i;\Z)$. Similarly, if $(a_{i+1},b_{i+1})=(a_i,a_i+b_i)$, then $M_{[\Sigma_i]}(t)=\Upsilon_{T_{a_i,a_i+1}}(t)$. In both cases,  Equation~\eqref{eq:bothcasesupsilonrecursion} implies that
 \[
 \Upsilon_{T_{a_{i+1},b_{i+1}}}(t)=\Upsilon_{T_{a_i,b_i}}+M_{[\Sigma_i]}(t).
 \]
 Hence 
 \[
\Upsilon_{T_{a_n,b_n}}(t)=M_{[\Sigma_n]}(t)+\cdots + M_{[\Sigma_1]}(t)=M_{[\Sigma]}(t),
 \]
 where $(W,\Sigma)$ denotes the composition of all of the $(W_i,\Sigma_i)$, completing the proof.
 \end{proof}

  \begin{rem}We can give a more concrete description of the expression $M_{[\Sigma]}(t)$ when $\Sigma$ is a surface inside a cobordism $W\colon S^3\to S^3$ satisfying $b_2^+(W)=b_1(W)=0$. Using Donaldson's diagonalizability theorem, we can pick an orthonormal basis $e_1,\dots, e_n$ of $H^2(W;\Z)$, and write $\PD[\Sigma]=a_1 \cdot e_n+\cdots +a_n \cdot e_n$. The set of characteristic elements $C$ of $Q_W$ can be identified with the set of elements of the form $c_1\cdot e_1+\cdots+ c_n\cdot e_n$ where $c_i\in 2\Z+1$. Then
 \[
M_{[\Sigma]}(t)=\max_{(c_1,\dots, c_n)\in (2\Z+1)^n} \sum_{i=1}^n \frac{-c_i^2+1+2t c_i a_i-2t a_i^2}{4}. 
 \]
 Since each $c_i$ is involved in exactly one summand, we can commute the maximization and the summation to see that
 \[
M_{[\Sigma]}(t)=\sum_{i=1}^n m_{a_i}(t).
 \]
 \end{rem}

\section{Adjunction relations and inequalities}
\label{sec:adjunctionrelations}

As another application of our grading formula, we prove several adjunction relations for Heegaard Floer homology and link Floer homology. The version we prove for Heegaard Floer homology is a generalization of \cite{OSTrianglesandSymplectic}*{Theorem 3.1}. The version we prove for the link cobordism maps is new.

\subsection{The standard adjunction relation for Heegaard Floer homology}

In this section, we prove the following using the link cobordism maps:

\begin{customthm}{\ref{thm:7}}\label{thm:generaladjunctionrelation}Suppose that $\cF=(\Sigma,\cA)$ is an oriented, closed, decorated surface inside of a cobordism $W\colon Y_1\to Y_2$, with $W$, $Y_1$ and $Y_2$ connected. Write $\Sigma_{\ws}$ and $\Sigma_{\zs}$ for the type-$\ws$ and type-$\zs$ subsurfaces of $\cF$. Suppose $\cA$ consists of a simple closed curve, the subsurfaces $\Sigma_{\ws}$ and $\Sigma_{\zs}$ are both connected, and
\[
\langle c_1(\frs),[\Sigma]\rangle -[\Sigma]\cdot [\Sigma]+2g(\Sigma_{\zs})-2g(\Sigma_{\ws})=0.
\]
 Then
\[
F_{W,c,\frs}(\iota_*\xi(\Sigma_{\ws})\otimes -)\simeq F_{W,c,\frs-\PD[\Sigma]}(\iota_*\xi(\Sigma_{\zs})\otimes -),
\]
as maps on $\CF^-$, where  $c$ is any path from $Y_1$ to $Y_2$.
\end{customthm}

\begin{rem}When $g(\Sigma)=g(\Sigma_{\zs})>0$ and $g(\Sigma_{\ws})=0$, we recover \cite{OSTrianglesandSymplectic}*{Proposition 3.1}. This case is also an analog of the adjunction relation for the Seiberg--Witten invariant \cite{OSSymplecticThom}*{Theorem 1.3}.  If $\Sigma$ has genus zero, then an analogous result was proven by Fintushel and Stern \cite{FintushelSternImmersedSpheres}*{Lemma 5.2} for the Seiberg--Witten invariant.
\end{rem}

\begin{rem}Theorem~\ref{thm:generaladjunctionrelation} also holds for $\Hat{\CF}$, $\CF^\infty$ and $\CF^+$, since the chain homotopy between $F_{W,c,\frs}(\xi(\Sigma_{\ws})\otimes -)$ and $F_{W,c,\frs-\PD[\Sigma]}(\xi(\Sigma_{\zs})\otimes -)$ can be taken to be $U$-equivariant, and $\Hat{\CF}$, $\CF^\infty$ and $\CF^+$ can all be obtained algebraically from $\CF^-$ via a tensor product, or a quotient of a tensor product.
\end{rem}

\begin{proof}[Proof of Theorem~\ref{thm:generaladjunctionrelation}] We perform an isotopy of $\cF$ so that it intersects $\d W$ along two embedded disks, $D_1$ and $D_2$, in $Y_1$ and $Y_2$, and each $D_i$ intersects $\cA$ in a single arc.  Write $\cF_0$ for the resulting, properly embedded, decorated surface in $W$ obtained by removing these two disks from $\cF$. Let $ c$ denote one of the two dividing arcs on $\cF_0$. By choosing the isotopy of $\cF$ inside of $W$ appropriately, we can achieve any embedded path $c$ from $Y_1$ to $Y_2$.

 We now apply Proposition~\ref{prop:reductionformulasforpuncturedsurfaces}. Since 
 \[
 \Delta A= \frac{\langle c_1(\frs),[\Sigma]\rangle -[\Sigma]\cdot [\Sigma]}{2}+g(\Sigma_{\zs})-g(\Sigma_{\ws})=0,
 \]
  we conclude that
\begin{equation}
F_{W,c,\frs}(\iota_* \xi(\Sigma_{\ws})\otimes -)|^{\bF_2[U,V]}\simeq F_{W,c,\frs-\PD[\Sigma]}(\iota_* \xi(\Sigma_{\zs})\otimes -)|^{\bF_2[U,V]}, \label{eq:almostadjunctionrelation1}
\end{equation}
as maps between $\cCFL^-(Y_1,\bU_1,\frs|_{Y_1})$ and $\cCFL^-(Y_2,\bU_2,\frs|_{Y_2})$, where $|^{\bF_2[U,V]}$ is the operation defined in Section~\ref{sec:cobordismmapsforclosedsurfaces}, which amounts to replacing each instance of $\hat{U}$ with $UV$. If we restrict the maps in Equation~\eqref{eq:almostadjunctionrelation1} to the zero Alexander graded subsets of $\cCFL^-(Y_i,\bU_i,\frs|_{Y_i})$ (with respect to the gradings $A_{D_1}$ and $A_{D_2}$), which are canonically identified with $\CF^-(Y_i,w_i,\frs|_{Y_i})$, we obtain the relation
\[
F_{W,c,\frs}(\iota_*\xi(\Sigma_{\ws})\otimes -)\simeq F_{W,c,\frs-\PD[\Sigma]}(\iota_*\xi(\Sigma_{\zs})\otimes -),
\]
completing the proof.
\end{proof}

\begin{rem}More generally, one could put more exotic sets of divides on $\Sigma$ and consider the $U=1$ and $V=1$ reductions of the associated link cobordism maps, and try to recover the higher type adjunction relations from \cite{OSHigherTypeAdjunction}, which Ozsv\'{a}th and Szab\'{o} proved for the Seiberg--Witten invariant.
\end{rem}

We note that applying Theorem~\ref{thm:generaladjunctionrelation} to the identity cobordism recovers Ozsv\'{a}th and Szab\'{o}'s adjunction inequality for $\HF^+(Y,\frs)$ \cite{OSProperties}*{Theorem~7.1}:

\begin{cor}\label{cor:adjunctioninequality}If $\Sigma$ is a closed, oriented surface in $Y$ with $g(\Sigma)>0$ and $\HF^+(Y,\frs)\neq 0$, then
\[
|\langle c_1(\frs),\Sigma\rangle |\le 2g(\Sigma)-2.
\] 
\end{cor}
\begin{proof} Suppose $\Sigma$ is a closed, oriented surface in $Y$ which violates the inequality. We can reverse the orientation of $\Sigma$ if necessary, and add null-homologous handles so that 
\begin{equation}\langle c_1(\frs),\Sigma\rangle =-2g(\Sigma)\label{eq:chernclassgenusequation}\end{equation} (note that $\langle c_1(\frs),\Sigma\rangle$ is always even, since $W=[0,1]\times Y$). Decorating  $\Sigma$ with a single dividing curve so that $g(\Sigma_{\zs})=g(\Sigma)$ and $g(\Sigma_{\ws})=0$ and then applying Theorem~\ref{thm:generaladjunctionrelation} to the cobordism $W=[0,1]\times Y$, we obtain the relation that
\[
F_{W,\frs}\simeq [\iota_*(\xi(\Sigma))]\cdot F_{W,\frs-\PD[\Sigma]}.
\] 
Since $[\Sigma]=0\in H_2(W,\d W;\Z)\iso H^2(W;\Z)$, it follows that $\frs-\PD[\Sigma]=\frs$ on $W$, so  $F_{W,\frs-\PD[\Sigma]}=F_{W,\frs}=\id_{\HF^+(Y,\frs)}$. Hence
\[
\id_{\HF^+(Y,\frs)}=[\iota_*(\xi(\Sigma))].
\] 
 On the other hand, $\frs$ must be non-torsion for Equation \eqref{eq:chernclassgenusequation} to be satisfied, so  $U^n\cdot \HF^+(Y,\frs)=0$ for sufficiently large $n$ \cite{OSTrianglesandSymplectic}*{Lemma 2.3}. We also have that $A_{\gamma}\circ A_{\gamma}=0$ for any $\gamma\in H_1(Y;\Z)$, and also $A_{\gamma_1}\circ  A_{\gamma_2}=A_{\gamma_2}\circ A_{\gamma_1}$ \cite{OSDisks}*{Section~4.2.5}. Hence, if $n$ is sufficiently large, the action of $[\iota_*(\xi(\Sigma))]^n$ will be zero on $\HF^+(Y,\frs)$. Hence $\id=\id^n=[\iota_*(\xi(\Sigma))]^n=0$ as maps on $\HF^+(Y,\frs)$, so $\HF^+(Y,\frs)$ must vanish. 
\end{proof}

\subsection{An adjunction relation for the link cobordism maps}

In this section, we describe a generalization of Theorem~\ref{thm:generaladjunctionrelation} for closed surfaces in the complement of a link cobordism. Before we state the theorem, we need to recall some additional facts about link Floer homology.

Firstly, similar to the 3-manifold invariants, there is a homotopy action of $\Lambda^*(H_1(Y;\Z)/\Tors)$ on $\cCFL^-(Y,\bL,\frs)$. Its description is essentially the same as the action of $\Lambda^*(H_1(Y;\Z)/\Tors)$ on $\CF^-(Y,\ws,\frs)$. Given a  Heegaard diagram $(\Sigma,\as,\bs,\ws,\zs)$ for $(Y,\bL)$, and an element $h\in H_1(Y;\Z)$, we pick a representative $\gamma$ of $h$, which is an immersed, closed curve on the Heegaard surface $\Sigma$. The endomorphism $A_{h}$ is defined via the formula
\begin{equation}
A_h(\xs)=\sum_{\ys\in \bT_{\alpha}\cap \bT_{\beta}} \sum_{\substack{\phi\in \pi_2(\xs,\ys)\\ \mu(\phi)=1}} a(\phi,\gamma) \cdot \Hat{\cM}(\phi) U^{n_{\ws}(\phi)} V^{n_{\zs}(\phi)} \cdot \ys.\label{eq:homologyaction}
\end{equation}
The quantity $a(\phi,\gamma)$ is defined as the sum of the changes of the multiplicities of $\phi$ across each $\as$ curve, as one traverses $\gamma$ on $\Sigma$.
Since $a(\cP,\gamma)=0$ for any periodic class $\cP$ with boundary only on the $\as$ or only on the $\bs$ curves, a straightforward count of the ends of index 2 moduli spaces shows that $A_h$ is a chain map. Adapting the arguments from \cite{OSDisks}*{Section~4.2.5} shows that $A_h$ is independent of the representative curve $\gamma$ and is well defined on the level of transitive chain homotopy type invariants. Some additional details can be found in \cite{ZemGraphTQFT}*{Section~5} about the homology action in the context of multi-pointed Heegaard Floer complexes.

Given a surface $\Sigma$ with zero or one boundary components, as well as an element $T$ in an algebra $\cR$ over $\bF_2$, we define the element
\[
\xi_T(\Sigma)\in \cR\otimes_{\bF_2} \Lambda^*( H_1(\Sigma;\bF_2)),
\]
by picking a geometric symplectic basis of $H_1(\Sigma;\bF_2)$ and using the same formula as Equation~\eqref{eq:defxiSigma}, with the element $T$ in place of $U$.

The action of $\Lambda^*(H_1(Y;\Z)/\Tors)$ on $\cCFL^-(Y,\bL,\frs)$ can be incorporated into the link cobordism maps, in the sense that if $(W,\cF)\colon (Y_1,\bL_1)\to (Y_2,\bL_2)$ is a decorated link cobordism, then the cobordism maps from \cite{ZemCFLTQFT} induce maps
\[
F_{W,\cF,\frs}\colon \Lambda^*(H_1(W;\Z)/\Tors)\otimes_{\bF_2} \cCFL^-(Y_1,\bL_1,\frs|_{Y_1})\to \cCFL^-(Y_2,\bL_2,\frs|_{Y_2}).
\]

We can now state our adjunction relation for the link cobordism maps:
\begin{thm}\label{thm:adjunctionrelationlinkcob} Suppose that $(W,\cF)\colon (Y_1,\bL_1)\to (Y_2,\bL_2)$ is a decorated link cobordism, such that $\bL_1$ and $\bL_2$ are null-homologous and $\cS=(\Sigma,\cA)$ is a closed, decorated surface in the complement of $\cF$, such that $\cA$ consists of a single closed curve, dividing $\Sigma$ into two connected components, $\Sigma_{\ws}$ and $\Sigma_{\zs}$. If
\[
\langle c_1(\frs),[\Sigma]\rangle -[\Sigma]\cdot [\Sigma]+2g(\Sigma_{\zs})-2g(\Sigma_{\ws})=0, 
\]
then
\[
F_{W,\cF,\frs}(\iota_*\xi_{UV}(\Sigma_{\ws})\otimes -)\simeq F_{W,\cF,\frs-\PD[\Sigma]}(\iota_*\xi_{UV}(\Sigma_{\zs})\otimes -).
\]
\end{thm}

\begin{proof}The proof is similar to our proof of Theorem~\ref{thm:generaladjunctionrelation}.  

We pick two disks, $D_1$ and $D_2$, in $\cS$, which intersect $\cA$ along two arcs, then we isotope $\Sigma$ so that $\Sigma\cap Y_i=D_i$. Let $\cS_0$ denote the decorated surface obtained by removing $D_1$ and $D_2$ from $\cS$, after we perform the isotopy. Let $\bU_1$ and $\bU_2$ denote the unknots $\d D_1$ and $\d D_2$, in $Y_1$ and $Y_2$, respectively,  each decorated with two basepoints.

 Let $\sigma$ denote a coloring of $\cF$ which maps the type-$\ws$ subsurface to the variable $U$, and the type-$\zs$ subsurface to $V$. Let $\hat{\sigma}$ denote an extension of $\sigma$ to the decorated surface $\cF\cup \cS_0$, which assigns $\Sigma_{\ws}$ the variable $U'$, and $\Sigma_{\zs}$ the variable $V'$. We assume $\hat{\sigma}$ has a codomain $\bmP$ with $|\bmP|=4$, so that we can identify the ring $\cR^-_{\bmP}$ with $\bF_{2}[U,V,U',V']$.
 
We pick Seifert surfaces $S_1$ and $S_2$ for $\bL_1$ and $\bL_2$, and consider the link cobordism map
 \[
 F_{W,(\cF\cup \cS_0)^{\hat{\sigma}},\frs}\colon \cCFL^-(Y_1,(\bL_1\cup \bU_1)^{\hat{\sigma}},\frs|_{Y_1})\to \cCFL^-(Y_2,(\bL_2\cup \bU_2)^{\hat{\sigma}}, \frs|_{Y_2}).
 \]
 
 We give $\cF$ and $\cS_0$  an indexing $J\colon F\cup S_0\to \{1,2\}$ which sends $F$ to $1$ and $S_0$ to $2$. The indexing $J$ induces a two component Alexander grading $A=(A_1,A_2)$, where $A_1$ corresponds to the links $\bL_1$ and $\bL_2$, with Seifert surfaces $S_1$ and $S_2$, while $A_2$ corresponds to the unknots $\bU_1$ and $\bU_2$, with Seifert surfaces $D_1$ and $D_2$. The map $F_{W,(\cF\cup \cS_0)^{\hat{\sigma}},\frs}$ is bi-graded with respect to $A$. Furthermore, $F_{W,(\cF\cup \cS_0)^{\hat{\sigma}},\frs}$ induces an $A_2$ grading change of
 \begin{align*}
 \Delta A_2&= 
 \frac{\langle c_1(\frs),[\Sigma]\rangle -([\Sigma]+[\hat{F}])\cdot [\Sigma]}{2}+g(\Sigma_{\zs})-g(\Sigma_{\ws}) \\
&=\frac{\langle c_1(\frs),[\Sigma]\rangle -[\Sigma]\cdot [\Sigma]}{2}+g(\Sigma_{\zs})-g(\Sigma_{\ws})\\
&=0.
\end{align*}
 
We now claim that
\begin{equation}
F_{W,(\cF\cup \cS_0')^{\hat{\sigma}}, \frs}(\xi_{U'V'}(\Sigma_{\ws})\otimes -)\simeq F_{W, (\cF\cup \cS_0')^{\hat{\sigma}}, \frs-\PD[\Sigma]}(\xi_{U'V'}(\Sigma_{\zs})\otimes -), \label{eq:adjunctionrelationwithextratube}
\end{equation}
where $\cS_0'$ is a properly embedded, decorated annulus in $W$, with boundary $-\bU_1\sqcup \bU_2$, which is disjoint from $\cF$, and is obtained by adding a tube along one of the dividing arcs of $\cS_0$, and decorating it with two parallel dividing arcs, running from $\bU_1$ to $\bU_2$.

To prove Equation~\eqref{eq:adjunctionrelationwithextratube}, we adapt the proof of Proposition~\ref{prop:reductionformulasforpuncturedsurfaces}. We consider the $U'=1$ and $V'=1$ reductions of the map $F_{W,(\cF\cup \cS_0)^{\hat{\sigma}},\frs}$, which by a natural extension of Theorem~\ref{thm:generalreductionformula}, one can identify as the map induced by a cobordism $W$ containing the decorated surface $\cF$ as well as a ribbon graph. The cobordism with decorated surface and graph connects two 3-manifolds which contain multi-based links as well as free basepoints.  If $c_{\ws}$ (resp. $c_{\zs}$) is a path in $W$ connecting the $\ws$ basepoints (resp. $\zs$ basepoints) of $\bU_1$ and $\bU_2$, constructed as in the statement of Lemma~\ref{lem:computereductionsconnectedsurface}, then by adapting the proof thereof, we see that
\begin{equation}\begin{split}
F_{W,(\cF\cup \cS_0)^{\hat{\sigma}},\frs}|_{V'=1}&\simeq F_{W,\cF^\sigma\cup c_{\ws},\frs}(\xi_{U'}(\Sigma_{\ws})\otimes -)\\
 \text{and} \qquad F_{W,(\cF\cup \cS_0)^{\hat{\sigma}},\frs}|_{U'=1}&\simeq F_{W,\cF^\sigma\cup c_{\zs},\frs-\PD[\Sigma]}(\xi_{V'}(\Sigma_{\zs})\otimes -).
 \end{split}\label{eq:reductionstolink/graphcob}
\end{equation}
 
By identifying $w_i$ and $z_i$ with a single point $p_i\in Y_i$, we can view the two reductions of $F_{W,(\cF \cup \cS_0)^{\hat{\sigma}}, \frs}$ as having a single domain and range, namely $\cCFL^-(Y_1,\bL_1^\sigma\cup \{p_1\}, \frs|_{Y_1})$ and $\cCFL^-(Y_2,\bL_2^\sigma\cup \{p_2\},\frs|_{Y_2})$, both of which we view as modules over the ring $\bF_2[U,V,\hat{U}]$. Note that when we identify $w_i$ and $z_i$ with the point $p_i$, we can pick $c_{\ws}$ and $c_{\zs}$ to be isotopic relative to their endpoints. We write $c$ for the common path from $p_1$ to $p_2$.

 We can view $\hat{U}$ as acting on $\bF_2[U',V']$ as the product $U'V'$. Since $\Delta A_2=0$, by adapting the proof of Proposition~\ref{prop:reductionformulasforpuncturedsurfaces}, it follows that
\begin{equation}
F_{W,(\cF\cup \cS_0)^{\hat{\sigma}}, \frs}\simeq (F_{W,(\cF\cup \cS_0)^{\hat{\sigma}},\frs}|_{V'=1})|^{\bF_2[U',V']}\simeq (F_{W,(\cF\cup \cS_0)^{\hat{\sigma}},\frs}|_{U'=1})|^{\bF_2[U',V']}. \label{eq:reductions-toextensionsyieldoriginal}
\end{equation}
 
 Combining Equations~\eqref{eq:reductionstolink/graphcob} and \eqref{eq:reductions-toextensionsyieldoriginal} shows that
 \begin{equation}
 F_{W,\cF^\sigma\cup c,\frs}(\xi_{\hat{U}}(\Sigma_{\ws})\otimes -)|^{\bF_2[U',V']}\simeq F_{W,\cF^\sigma\cup c,\frs-\PD[\Sigma]}(\xi_{\hat{U}}(\Sigma_{\zs})\otimes -)|^{\bF_2[U',V']}.\label{eq:almostthereadjunctionrelationlinks}
 \end{equation} 
The notation $|^{\bF_2[U',V']}$ is described in Section~\ref{sec:cobordismmapsforclosedsurfaces}, and amounts to replacing each instance of $\hat{U}$ with $U'V'$. Finally, we  replace the path $c$ with a properly embedded annulus $\cS'_0$ contained in a small neighborhood of $c$, which we decorate with two dividing arcs running from $Y_1$ to $Y_2$. The effect on the cobordism maps is simply to replace any single basepoint $p$ on a Heegaard diagram or triple (with $p$ corresponding to a point on $c$) with a pair of basepoints $w$ and $z$, which are immediately adjacent on the Heegaard diagram or triple. It follows that the algebraic extension operation $|^{\bF_2[U',V']}$ amounts to just replacing the path $c$ with the annulus $\cS_0'$. In light of this, Equation~\eqref{eq:adjunctionrelationwithextratube} follows from Equation~\eqref{eq:almostthereadjunctionrelationlinks}.

 Note that the relation in Equation~\eqref{eq:adjunctionrelationwithextratube} persists when we set $U=U'$ and $V=V'$, since the chain homotopy is $R_{\bmP}^-$-equivariant.
 
We now construct two cobordisms $([0,1]\times Y_1, \cF_1')$ and $([0,1]\times Y_2, \cF_2')$, such that the following are satisfied:
 \begin{enumerate}
 \item The cobordism $([0,1]\times Y_1, \cF_1')\colon (Y_1,\bL_1)\to (Y_2,\bL_1\cup \bU_1)$ splits an unknot off of $\bL_1$, which becomes $\bU_1$.
 \item  The cobordism $([0,1]\times Y_2, \cF_2')\colon (Y_2, \bL_2\cup \bU_2)\to (Y_2,\bL_2)$ caps off $\bU_2$ with the disk $D_2$.
 \end{enumerate}
 
Note that the composition
\[ 
([0,1]\times Y_2, \cF_2')\circ (W,\cF\cup\cS_0') \circ([0,1]\times Y_1, \cF_1')
\] 
is diffeomorphic to $(W,\cF)$. Pre- and post-composing both sides of Equation~\eqref{eq:adjunctionrelationwithextratube} with the maps for $([0,1]\times Y_2, \cF_2')$ and $([0,1]\times Y_1,\cF_1')$, after having set $U=U'$ and $V=V'$, we obtain the theorem statement.
\end{proof}

\subsection{An adjunction inequality for the link cobordism maps}
\label{sec:adjunctioninequalityforlinkcobs}
In this section, we use Theorem~\ref{thm:adjunctionrelationlinkcob} to prove an adjunction inequality for the link cobordism maps. As in the previous section, we work over the ring $\cR^-:=\bF_2[U,V]$. The adjunction inequality we state concerns the version of link Floer homology
 \[
 \CFL^-(Y,\bL,\frs):= \cCFL^-(Y,\bL,\frs)\otimes_{\cR^-} \cR^-/(V).
 \]
 
 \begin{customthm}{\ref{thm:linkadjunctioninequality}}\label{thm:adunctioninequalitylinkcobs}Suppose that $(W,\cF)\colon (Y_1,\bL_1)\to (Y_2,\bL_2)$ is a link cobordism with $b_1(W)=0$, such that $\bL_1$ and $\bL_2$ are null-homologous in $Y_1$ and $Y_2$, respectively. Suppose that the induced map 
\[
F_{W,\cF,\frs}\colon \CFL^-(Y_1,\bL_1,\frs|_{Y_1})\to \CFL^-(Y_2,\bL_2,\frs|_{Y_2})
\]
 is not $\bF_2[U]$-equivariantly chain homotopic to the zero map. If $\Sigma$ is a closed, oriented surface in the complement of $\cF$ such that $g(\Sigma)>0$, then
 \[
|\langle c_1(\frs), [\Sigma]\rangle| +[\Sigma]\cdot [\Sigma] \le 2 g(\Sigma)-2.
 \] 
 \end{customthm}
\begin{proof}Suppose that $\Sigma$ is a surface with $|\langle c_1(\frs), [\Sigma]\rangle |+[\Sigma]\cdot [\Sigma]\ge 2 g(\Sigma)$ (note that $|\langle c_1(\frs), [\Sigma]\rangle| +[\Sigma]\cdot [\Sigma]$ is always even since $c_1(\frs)$ is a characteristic vector of  $Q_W$). By  reversing the orientation of $\Sigma$ if necessary, we may assume that  $|\langle c_1(\frs),[\Sigma]\rangle|=-\langle c_1(\frs),[\Sigma]\rangle$, so that
 \[
-\langle c_1(\frs), [\Sigma]\rangle +[\Sigma]\cdot [\Sigma] \ge 2 g(\Sigma).
 \]
 
 By taking the connected sum of $\Sigma$ with null-homologous tori to raise the genus, if necessary, we may assume that
 \[
-\langle c_1(\frs),[\Sigma]\rangle+[\Sigma]\cdot [\Sigma]=2g(\Sigma). 
 \]
 
 We now decorate $\Sigma$ with a dividing set consisting of a single closed curve, dividing $\Sigma$ into two components, such that $\Sigma_{\ws}$ is a disk, and $\Sigma_{\zs}$ is a connected surface of genus $g(\Sigma)$. Since
 \[
  \langle c_1(\frs),[\Sigma]\rangle -[\Sigma]\cdot [\Sigma]+2g(\Sigma_{\zs})-2g(\Sigma_{\ws})=0
 \]
 we can apply Theorem~\ref{thm:adjunctionrelationlinkcob} to see that
 \begin{equation}
F_{W,\cF,\frs}(-)\simeq F_{W,\cF,\frs-\PD[\Sigma]}(\iota_*\xi_{UV}(\Sigma_{\zs})\otimes -). \label{eq:pre-adjunctioninequality}
 \end{equation}
 Since $b_1(W)=0$, the element $\iota_*\xi_{UV}(\Sigma_{\zs})$ is simply $(UV)^{g(\Sigma_{\zs})}=(UV)^{g(\Sigma)}$. Since $V=0$ on $\CFL^-$, we obtain immediately from Equation~\eqref{eq:pre-adjunctioninequality} that $F_{W,\cF,\frs}$ is $\bF_2[U]$-equivalently chain homotopic to 0 on $\CFL^-$, completing the proof.
\end{proof}
 
 We now describe two examples of Theorem~\ref{thm:adjunctionrelationlinkcob}. 
 
\begin{example}\label{ex:blowup1} Suppose that $\bL$ is a multi-based link in $S^3$, and $(W,\cF)\colon (S^3,\bL)\to (S^3,\bL)$ is the link cobordism obtained by performing $-1$ surgery on an unknot $U\subset S^3$ which is unlinked from $\bL$. Let $E\subset W$ denote an exceptional sphere which is disjoint from $\cF$. Write $\frs_k^+$ for the $\Spin^c$ structure with $\langle c_1(\frs_k^+), E\rangle =2k+1$, for $k\ge 0$. Let $\frs_k^-$ be the $\Spin^c$ structure with $\langle c_1(\frs_k^-), E\rangle =-(2k-1)$, for $k\ge 0$.  Theorem~\ref{thm:linkadjunctioninequality} implies that  $F_{W,\cF,\frs_k^{\pm}}\simeq 0$ on $\CFL^-$ unless $k=0$. Indeed using Theorem~\ref{thm:adjunctionrelationlinkcob}, for appropriate stabilizations of $E$, we obtain the relations
\[
F_{W,\cF,\frs_k^+}\simeq (UV)^{k}\cdot  F_{W,\cF,\frs_{k-1}^+}\simeq\cdots\simeq  (UV)^{k(k+1)/2}\cdot F_{W,\cF,\frs_0^+},
\]
on $\cCFL^-(S^3,\bL,\frs_0)$.
 A similar relation holds for a $\Spin^c$ structure $\frs_k^-$ with $k\ge 1$. In analogy to the standard blow-up formula \cite{OSTriangles}*{Theorem~1.4}, we expect $F_{W,\cF,\frs_k^+}\simeq (UV)^{k(k+1)/2}\cdot \id_{\cCFL^-(S^3,\bL)}$, however for the sake of brevity, we will not endeavor to write down a proof.
\end{example}

 \begin{example} Let $(W_2,\cF)\colon (S^3,\bK_+)\to (S^3,\bK_-)$ denote a decorated version of the 2-handle knot cobordism $(W_2,\Sigma_2)$ shown in Figure~\ref{fig::14}, which changes a positive crossing to a negative one. The 4-manifold $W_2$ is obtained by attaching a 2-handle to an unknot with framing $-1$ which is linked with $K_+$. Define the $\Spin^c$ structures $\frs_k^{\pm}\in \Spin^c(W_2)$ as in Example~\ref{ex:blowup1}. There is an embedded torus $T$ in the complement of $\cF$ which is a generator of $H_2(W_2;\Z)$, and satisfies
 \[
|\langle c_1(\frs_k^{\pm}), T\rangle| +[T]\cdot [T]=2k.
 \]
  The torus $T$ is shown in Figure~\ref{fig::47}. By Theorem~\ref{thm:adunctioninequalitylinkcobs}, it follows that whenever $k>0$, the map $F_{W,\cF,\frs_{k}^{\pm}}$ vanishes on $\HFK^-$ and $\Hat{\HFK}$ by Theorem~\ref{thm:linkadjunctioninequality}.
 
 We note in \cite{OSKnots}*{Theorem~8.2}, Ozsv\'{a}th and Szab\'{o} give a proof of the Skein exact triangle for knot Floer homology, by adapting the proof of the standard surgery exact triangle for $\Hat{\HF}$ \cite{OSProperties}*{Theorem~9.16}. The map $F:=\sum_{\frs\in \Spin^c(W_2)} F_{W_2,\cF,\frs}$ is one of the maps in their Skein exact sequence for $\Hat{\HFK}$ . We note that in the statement of \cite{OSKnots}*{Theorem~8.2}, the map $F$ is not known to be a graded map with respect to the Maslov grading. Theorem~\ref{thm:adunctioninequalitylinkcobs} shows  $F$ in fact does preserve the Maslov grading, since the only $\Spin^c$ structures with non-trivial contribution are $\frs_0^+$ and $\frs_0^-$. 
 
 This is in contrast to \cite{OSSkeinExactSequence}, where Ozsv\'{a}th and Szab\'{o} give a proof of the Skein exact triangle, where all three maps in the triangle are graded. The proof they give is not an adaptation of the surgery exact sequence, but instead is a careful examination of certain specially constructed Heegaard diagrams.
 \end{example}

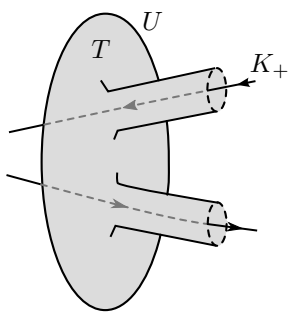
\begin{figure}[ht!]
	\centering
\begingroup%
  \makeatletter%
  \providecommand\color[2][]{%
    \errmessage{(Inkscape) Color is used for the text in Inkscape, but the package 'color.sty' is not loaded}%
    \renewcommand\color[2][]{}%
  }%
  \providecommand\transparent[1]{%
    \errmessage{(Inkscape) Transparency is used (non-zero) for the text in Inkscape, but the package 'transparent.sty' is not loaded}%
    \renewcommand\transparent[1]{}%
  }%
  \providecommand\rotatebox[2]{#2}%
  \newcommand*\fsize{\dimexpr\f@size pt\relax}%
  \newcommand*\lineheight[1]{\fontsize{\fsize}{#1\fsize}\selectfont}%
  \ifx\svgwidth\undefined%
    \setlength{\unitlength}{104.59414591bp}%
    \ifx\svgscale\undefined%
      \relax%
    \else%
      \setlength{\unitlength}{\unitlength * \real{\svgscale}}%
    \fi%
  \else%
    \setlength{\unitlength}{\svgwidth}%
  \fi%
  \global\let\svgwidth\undefined%
  \global\let\svgscale\undefined%
  \makeatother%
  \begin{picture}(1,1.07648621)%
    \lineheight{1}%
    \setlength\tabcolsep{0pt}%
    \put(0,0){\includegraphics[width=\unitlength,page=1]{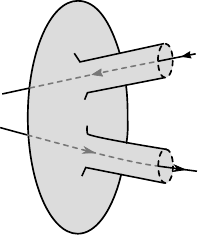}}%
    \put(0.87823096,0.86104734){\color[rgb]{0,0,0}\makebox(0,0)[lt]{\lineheight{1.25}\smash{\begin{tabular}[t]{l}$K_+$\end{tabular}}}}%
    \put(0.30736225,0.91265621){\color[rgb]{0,0,0}\makebox(0,0)[lt]{\lineheight{1.25}\smash{\begin{tabular}[t]{l}$T$\end{tabular}}}}%
    \put(0.48935184,1.01315806){\color[rgb]{0,0,0}\makebox(0,0)[lt]{\lineheight{1.25}\smash{\begin{tabular}[t]{l}$U$\end{tabular}}}}%
  \end{picture}%
\endgroup%

	\caption{\textbf{The embedded torus $T$ in $W_2\setminus \Sigma_2.$} The exterior circle is the unknot $U$ which we perform $-1$ surgery on to obtain $(S^3, K_-)$. The gray region shown is a punctured torus in $S^3$, with boundary on the unknot $U$. By attaching the core of the 2-handle which is attached along $U$, we obtain a torus with self intersection $-1$.}\label{fig::47}
\end{figure}

\bibliographystyle{custom}
\bibliography{biblio}

\end{document}